\documentclass[11pt,a4paper]{article}
\usepackage[left=2cm, right=2cm, top=3.00cm, bottom=3.00cm]{geometry}
\usepackage[utf8]{inputenc}
\usepackage[T1]{fontenc}
\usepackage{graphicx} 
\usepackage{amsmath}
\usepackage{amsthm}
\usepackage{amssymb}
\usepackage{enumerate}
\usepackage[table]{xcolor}
\usepackage{comment}
\usepackage{hyperref}
\usepackage{multicol}
\usepackage{multirow}
\usepackage{hhline}
\usepackage{longtable}
\usepackage{chngcntr}\counterwithin{table}{section}
\usepackage{tikz}\usetikzlibrary{math}

\newtheorem{theorem}{Theorem}[section]
\newtheorem{lemma}{Lemma}[section]

\newtheorem{observation}{Observation}[section]

\theoremstyle{definition}

\newtheorem{remark}{Remark}[section]

\newcommand{\Z}{\mathbb{Z}}
\newcommand{\x}{\times}
\newcommand{\ttk}[1]{\in \left[#1\right]}
\newcommand{\A}{\mathbf{v}_1}
\newcommand{\B}{\mathbf{v}_2}

\newcommand{\so}{S_{0}}
\newcommand{\ls}{L_{\so}}
\newcommand{\lc}{L_{C}}
\newcommand{\<}[1]{\left<{#1}\right>}
\newcounter{cn}

\newcommand{\tn}{\arabic{cn}\addtocounter{cn}{1}}
\newcommand{\rc}{\rowcolors{2}{black!5}{white}}
\newcommand{\fl}{
\hline
\rowcolor{black!10}
\setcounter{cn}{1} Row & \multicolumn{1}{c|}{$S_0$} & \multicolumn{1}{c|}{$C$} \\
\hline
\endhead}
\newcommand{\flap}{
\hline
\rowcolor{black!10}
\setcounter{cn}{1} $(|\so|,|C|,|\lc|)$ & \multicolumn{1}{c|}{$\lc$, $A$, $B$, $\so$, and $C$} & Reference & Result \\
\hline
\endhead}
\newcommand{\flp}{
\hline
\rowcolor{black!10}
\setcounter{cn}{1} $(|\so|,|C|,|\ls|)$ & \multicolumn{1}{c|}{$\ls$, $X$, $Y$, $\so$, and $C$} & Reference & Result \\
\hline
\endhead}
\newcommand{\Cay}{\mathrm{Cay}}
\newcommand{\Table}{\ref{t.p^3}}
\newcommand{\R}{This gives row \tn\, in Table \Table. }
\newcommand{\row}[1]{(row #1)}
\newcommand{\e}[1]{\mathbf{e}_{#1}}

\usepackage{xcolor}
\usepackage[normalem]{ulem}
\newcommand{\cha}[2]{\textcolor{blue}{}\textcolor{black}{#2}}

\usepackage{authblk}

\begin{document}
\baselineskip 15pt

\title{Perfect codes in Cayley graphs of Haj\'os groups}

\author[1,2]{Yusuf Hafidh\thanks{E-mail: \texttt{yhafidh@student.unimelb.edu.au}}}
\author[1]{Binzhou Xia\thanks{E-mail: \texttt{binzhoux@unimelb.edu.au}}}
\author[1]{Sanming Zhou\thanks{E-mail: \texttt{sanming@unimelb.edu.au}}}

\affil[1]{\small School of Mathematics and Statistics, The University of Melbourne, Parkville, VIC 3010, Australia}
\affil[2]{\small Department of Mathematics, Institut Teknologi Bandung, Indonesia}

\maketitle

\begin{abstract}
A perfect code in a graph $\Gamma$ is a subset $C$ of the vertex set of $\Gamma$ such that every vertex of $\Gamma$ outside $C$ has exactly one neighbour in $C$. A perfect code in a directed graph can be defined similarly by requiring that for every vertex $v$ outside $C$ there exists exactly one vertex $u$ in $C$ such that the arc from $u$ to $v$ exists in $\Gamma$. A subset $X$ of an abelian group $G$ is said to be periodic if there exists a non-identity element $g$ of $G$ such that $g + X = X$. A factorization of $G$ is a pair of nonempty subsets $(A, B)$ of $G$ such that every element $g$ of $G$ can be expressed uniquely as $g = a+b$ with $a \in A$ and $b \in B$. If for every factorization $(A, B)$ of an abelian group $G$ at least one of $A$ and $B$ is periodic, then $G$ is said to be a Haj\'os group. In this paper we classify all Cayley graphs (directed or undirected) of Haj\'os groups which admit perfect codes, and moreover we determine all perfect codes in such Cayley graphs.


\end{abstract}

\section{Introduction}

All groups considered in this paper are finite, and all graphs considered are finite and simple.
A $q$-ary code of length $n$ is a set of words of length $n$ over an alphabet of size $q$, where $q \ge 2$ and $n \ge 1$ are integers. Equivalently, a $q$-ary code of length $n$ is a subset $C$ of $\Z_q^n$ if the alphabet is set to be $\Z_q = \{0, 1, \dots, q-1\}$. The covering radius $\rho(C)$ of $C$ is the least integer $r$ such that every word of length $n$ is at distance no more than $r$ to at least one codeword of $C$, and the minimum distance $d(C)$ of $C$ is the minimum distance between any two distinct codewords of $C$. In coding theory, the distance used is usually the Hamming distance or Lee distance, which is the graph distance in the Hamming graph $H(n, q)$ or in the graph $L(n, q)$, respectively, where $H(n, q) = K_{q} \Box \cdots \Box K_{q}$ is the $n$-fold Cartesian product of the complete graph $K_q$ of order $q$, and $L(n, q) = C_{q} \Box \cdots \Box C_{q}$ is the $n$-fold Cartesian product of the cycle $C_q$ of length $q \ge 3$. A code $C$ is said to be $t$-error-correctng if $d(C) \ge 2t+1$, and is called a perfect $t$-code if $d(C) = 2t+1$ and $\rho(C) = t$, where $t \ge 1$ is an integer. Equivalently, $C$ is a \emph{perfect $t$-code} if every word in $\Z_q^n$ is at distance no more than $t$ to exactly one codeword of $C$. Ever since the beginning of information theory, perfect codes play an important role in coding theory due to their maximum capacity of error correction without ambiguity. The Hamming codes, the Golay code $G_{23}$, and the Golay code $G_{11}$ are well known examples of perfect codes, and they are the only nontrivial linear perfect codes according to a famous result \cite{Tiet73} which settled a long-standing open problem in coding theory, where a \emph{linear code} is a subspace of a linear space over a finite field. The reader is referred to \cite{vL} for an early survey on perfect codes and to \cite{Heden1} for a more recent survey on binary perfect $1$-codes.  

The notion of perfect $t$-codes can be generalized to graphs \cite{biggs, krat} simply by replacing the Hamming or Lee distance by the distance in a graph. More specifically, a subset $C$ of vertices of a graph $\Gamma$ is called a \emph{perfect $t$-code} if every vertex of $\Gamma$ is at distance no more than $t$ to exactly one vertex in $C$. Equivalently, $C$ is a perfect $t$-code in $\Gamma$ if the closed $t$-neighbourhoods $N_{t}[v], v \in C$, form a partition of the vertex set of $\Gamma$, where $N_{t}[v]$ is the set of vertices of $\Gamma$ at distance no more than $t$ to $v$. This definition applies to directed graphs $\Gamma$ when $N_{t}[v]$ is interpreted as the set of vertices of $\Gamma$ at (directed) distance no more than $t$ from $v$. In graph theory, a perfect $1$-code in $\Gamma$ is also called an efficient dominating set of $\Gamma$ \cite{eds1,eds2,eds3,ed1,ed2}. It is readily seen that perfect $t$-codes in $H(n, q)$ or $L(n, q)$ are precisely perfect $t$-codes under the Hamming distance or Lee distance in the classical setting. 

Perfect codes in Cayley graphs are especially interesting and significant objects of study due to their connections to coding theory and factorizations of groups. See \cite[Section 1]{22} and \cite[Section 1]{subgrup} for short reviews of related results and background information. Given a group $G$ with identity element $e$ and a subset $S$ of $G \setminus \{e\}$, define $\Cay(G,S)$ to be the directed graph with vertex set $G$ such that for $x, y \in G$ there is an arc from $x$ to $y$ if and only if $yx^{-1} \in S$. It is known that $\Cay(G,S)$ is connected if and only if $S$ generates $G$. In the case when $S$ is inverse-closed (that is, $g^{-1} \in S$ whenever $g \in S$), $\Cay(G,S)$ is regarded as an undirected graph as the arc from $x$ to $y$ exists if and only if the arc from $y$ to $x$ exists. For brevity, we call $\Cay(G,S)$ the \emph{Cayley graph} of $G$ with \emph{connection set} $S$ even when it is a directed graph (that is, $S$ is not inverse-closed). 
It can be verified that both $H(n, q)$ and $L(n, q)$ are Cayley graphs of $\Z_q^n$. As such perfect codes in Cayley graphs are natural generalizations of perfect codes in the classical setting. They can also be viewed as factorizations of the underlying group. A \emph{factorization} of $G$ (into two factors) \cite{book1} is a pair of subsets $(A, B)$ of $G$ such that every element of $G$ can be uniquely expressed as $ab$ with $a \in A$ and $b \in B$ (so that $G = AB$). If in addition $e \in A \cap B$, then this factorization a \emph{tiling} of $G$ \cite{Dinitz06}. As noted in \cite{subgrup}, if $(A, B)$ is a tiling of $G$ with $A^{-1} = A$, then $B$ is a perfect $1$-code in $\Cay(G, A \setminus \{e\})$, and conversely any perfect $1$-code $C$ in a Cayley graph $\Cay(G, S)$ gives rise to the tiling $(S_0, C)$ of $G$, where $S_0 = S \cup \{e\}$. A similar statement holds for perfect $t$-codes in Cayley graphs if one considers $S_0^t$ instead of $S_0$. 

Perfect codes in Cayley graphs have received considerable attention in recent years (see, for example, \cite{eds1, subgrup, 20}). In particular, the question about when a given subgroup of a group can be a perfect $1$-code in some Cayley graph of the group has been studied extensively. See \cite{cwx20, kak23, cm13, wwz20, zz21, zz21a} for recent results in this line of research. In view of potential applications, perfect codes in Cayley graphs of abelian groups are the most important. This is because for an abelian group $G = \Z_{q_1} \times \cdots \times \Z_{q_n}$ the codewords of a perfect code in a Cayley graph of $G$ are strings in the usual sense with $i^{\mathrm{th}}$ bit from alphabet $\Z_{q_i}$. Such a perfect code is a mixed perfect code \cite{Heden1} if not all $q_1, \ldots, q_n$ are equal. Unfortunately, even for Cayley graphs of cyclic groups, called \emph{circulant graphs}, we do not have a complete characterization of perfect $1$-codes in them, though a number of partial results exist in the literature (see \cite{eds2,eds3,22,kls22,MBG,ed1,ed2,20,wsxz24}). 

In this paper we study perfect $1$-codes in Cayley graphs of Haj\'os groups, where a group $G$ is \emph{Haj\'os} \cite{gag2} if $G$ is abelian and for every factorization $(A, B)$ of $G$, either $A$ or $B$ is periodic in the sense that it is fixed under the translation by some non-identity element of $G$. The classification of all Haj\'os groups was completed by Sands in 1962 \cite{gag2} (see Lemma \ref{goodabelian}). With the help of this result we classify all Cayley graphs of Haj\'os groups which admit perfect $1$-codes, and moreover we determine all perfect $1$-codes in such Cayley graphs. As far as we know, this is the first large class of abelian groups for which we have complete knowledge of perfect $1$-codes in their Cayley graphs. Since we will only consider perfect $1$-codes, for short we will call them \emph{perfect codes} from now on. Since each connected component of a Cayley graph $\Cay(G,S)$ is isomorphic to a Cayley graph of the subgroup of $G$ generated by $S$, and subgroups of Haj\'os groups are also Haj\'os, it suffices to consider connected Cayley graphs of Haj\'os abelian groups. Also, we do not need to consider complete graphs as each of their perfect codes consists of a single vertex. Our main results can be summarized as follows.

\begin{theorem}
\label{main}
Let $G$ be a Haj\'os group and $\Cay(G,S)$ a connected non-complete Cayley graph of $G$. Then a subset $C$ of $G$ is a perfect code in $\Cay(G,S)$ if and only if the triple $(G, S_0, C)$ lies in Table \ref{tablemain}, where $S_0=S\cup\{0\}$ with $0$ the identity element of $G$, $k\geq 1$ is an integer and $p$, $q$, $r$, and $s$ are different primes.
\end{theorem} 

We would like to emphasize that this result holds for both directed and undirected Cayley graphs $\Cay(G,S)$ of Haj\'os groups. 
More details on $(S_0, C)$ will be given in later sections as indicated in Table \ref{tablemain}. Theorem \ref{main} implies that $\Cay(G,S)$ admits a perfect code if and only if there exists a subset $C$ of $G$ such that $(G, S_0, C)$ is as shown in Table \ref{tablemain}, and moreover we obtain all perfect codes $C$ in $\Cay(G,S)$ in this manner. 

\rc\begin{longtable}{|c|c|c|c|}
\caption{\small Perfect codes in Cayley graphs of Haj\'os groups}
\label{tablemain}
\\
\hline
\rowcolor{black!10}
\setcounter{cn}{1} Row & \multicolumn{1}{c|}{$G$} & \multicolumn{1}{c|}{$S_0$ and $C$} & \multicolumn{1}{c|}{Condition} \\ \hline
\tn & $\Z_p\x\Z_p$ & Theorem \ref{as-bs} &  \\
\tn & $\Z_p\x\Z_q$ & Table \ref{t.p.q} &  \\
\tn & $\Z_p\x\Z_q\x\Z_q$ & Table \ref{t.p.q.q} & $p\geq 3\geq q$ \\
\tn & $\Z_{p^k}$ & Theorem \ref{p^k} &  \\
\tn & $\Z_{p^k}\x\Z_q$ & Theorem \ref{p^k.q} & $p,q \ne 2$ \\
\tn & $\Z_{2^k}\x\Z_2$ & Theorem \ref{2^k.2} &  \\
\tn & $\Z_{p^2}$ & Theorem \ref{p^2} &  \\
\tn & $\Z_{p^3}$ & Table \ref{t.p^3} &  \\
\tn & $\Z_{p^2}\x\Z_2$ & Table \ref{t.p^2.2} & $p\ne 2$ \\
\tn & $\Z_2\x\Z_2\x\Z_2$ & Theorem \ref{2.2.2} &  \\
\tn & $\Z_{p^3}\x\Z_2$ & Table \ref{t.p^3.2} & $p\ne 2$ \\
\tn & $\Z_{p^2}\x\Z_2\x\Z_2$ & Table \ref{t.p^2.2.2} & $p\ne 2$ \\
\tn & $\Z_{p}\x\Z_2\x\Z_2\x\Z_2$ & Table \ref{t.p.2.2.2} & $p\ne 2$ \\
\tn & $\Z_2\x\Z_2\x\Z_2\x\Z_2$ & Table \ref{t.2.2.2.2} &  \\
\tn & $\Z_{p^3}\x\Z_2\x\Z_2$ & Table \ref{t.p^3.2.2} & $p\ne 2$ \\
\tn & $\Z_{p^2}\x\Z_2\x\Z_2\x\Z_2$ & Table \ref{t.p^2.2.2.2} & $p\ne 2$ \\
\tn & $\Z_{p}\x\Z_2\x\Z_2\x\Z_2\x\Z_2$ & Table \ref{t.p.2.2.2.2} & $p\ne 2$ \\
\tn & $\Z_p\x\Z_q\x\Z_r$ & Table \ref{t.p.q.2} & \\
\tn & $\Z_p\x\Z_q\x\Z_2\x\Z_2$ & Table \ref{t.p.q.2.2} & $p,q\ne 2$ \\
\tn & $\Z_{p^2}\x\Z_p$ & Table \ref{t.4.2} & $p\leq 3$ \\
\tn & $\Z_p\x\Z_4$ & Table \ref{t.p.4} & $p\ne 2$ \\
\tn & $\Z_p\x\Z_4\x\Z_2$ & Table \ref{t.p.4.2} & $p\ne 2$ \\
\tn & $\Z_{p^2}\x\Z_{q^2}$ & \cite{eds2,eds3} \& Table \ref{t.p^2.q^2} &  \\
\tn & $\Z_{p^2}\x\Z_{q}\x\Z_{r}$ & \cite{eds2,eds3} \& Table \ref{t.p^2.q.r} &  \\
\tn & $\Z_{p}\x\Z_{q}\x\Z_{r}\x\Z_{s}$ & \cite{eds2,eds3} &  \\
\hline
\end{longtable}

There are 14 families of Haj\'os groups, of which four families consist of cyclic groups (items 9-12 in Lemma \ref{goodabelian}). So Cayley graphs of these four families of Haj\'os groups are circulant graphs. Circulant graphs having a perfect codes with size a prime or relative prime to the size of the connection set plus one in these circulant graphs have been classified in \cite{eds2,eds3}. In these two papers, they restrict the Cayley graphs to be undirected graphs, that is requiring $S=-S$. Although they have this restriction, they dont use the property $S=-S$ to get their results. This means the characterization they got also works for perfect codes in directed circulant graphs on Haj\'os groups. 
We will determine perfect codes in these circulant graphs with other sizes. Of course, we will also determine perfect codes in Cayley graphs of the other ten families of Haj\'os abelian groups. 

We divide the proof of Theorem \ref{main} into a number of theorems. In Section \ref{sec:prel}, we provide some background information and lemmas that will be used to prove our theorems. In Section \ref{sec:pqq}, we characterize Cayley graphs of $\Z_p\x\Z_q\x\Z_q$ and its subgroups that admit a perfect code. Section \ref{sec:p^k} focuses on Haj\'os groups that contain $\Z_{p^k}$ as a subgroup for a prime $p$. In Section \ref{sec:p^k2^l}, we consider Haj\'os groups of the form $\Z_{p^k} \times \Z_2^l$. The last section, Section \ref{sec:pq22}, covers the remaining Haj\'os groups, that is, $\Z_p\x\Z_q\x\Z_2\x\Z_2$, $\Z_p\x\Z_4\x\Z_2$, $\Z_4\x\Z_4$, $\Z_p^2\x\Z_q^2$, $\Z_{p^2}\x\Z_q\x\Z_r$ and their subgroups.

\section{Preliminaries}
\label{sec:prel}

\cha{Throughout this paper, all groups are considered finite and abelian. We use addition as the group operation and $0$ as the group identity. We denote $\mathbf{e}_i$ as the vector with $1$ in the $i^\mathrm{th}$ component and $0$ in the other components. For $n\in \Z^+$, let $[n]=\{0,\dots,n-1\}$.}{Additive notation will be used in the rest of the paper: the group operation will be $+$, the identity $0$, the inverse of $g$ will be $-g$, and $mg$ will be used instead of $g^m$. For a positive integer $n$, let $\Z_n$ be the cyclic group of integers modulo $n$ and set 
$$
[n]=\{0,1,\dots,n-1\}.
$$ 
For any abelian group written as a direct sum of cyclic groups, we use $\mathbf{e}_i$ to denote its element with $1$ in the $i^\mathrm{th}$ coordinate and $0$ in all other coordinates.} 

\cha{Let $u$ be a vertex of a connected graph $\Gamma$ and $r$ be a positive integer. The \textit{ball} with radius $r$ centered at $u$ is denoted by $B_r(u)$. This is the set of vertices with distance at most $r$ to $u$, i.e. $B_r(u)=\{v:d(v,u)\leq r\}$. A set of vertices $C$ of $\Gamma$ is called a \textit{perfect $r$-code} in $\Gamma$ if $\{B_r(u):u\in C\}$ is a partition of $V(\Gamma)$. When $r=1$, a perfect $1$-code is simply called a perfect code, and $B_1(u)$ is the closed neighbourhood of $u$, $N[u]$.}{}

\cha{Let $A$ and $B$ be subsets of a group $G$. The sum $A+B$ is the set of elements $a+b$ with $a\in A$ and $b\in B$.}{Let $G$ be a group, and let $A$ and $B$ be nonempty subsets of $G$.  The \textit{sum} of $A$ and $B$ is defined as}
$$
A+B = \{a+b\, | \, a \in A, b \in B\}. 
$$
\cha{We simply use $g+A$ for $\{g\}+A$.}{In particular, we write $g+A, A+g$ in place of $\{g\}+A, A+\{g\}$, respectively.} If every element of $A+B$ can be written uniquely as $a+b$ with $a\in A$ and $b\in B$, we say that the sum is a \textit{direct sum} and denote it by $A\oplus B$. \cha{We call $G=A\oplus B$ a factorization of $G$, and $A$ and $B$ are factors of $G$.}{If $G=A\oplus B$, we call this expression a \textit{factorization} of $G$ into \textit{factors} $A$ and $B$.} \cha{Let $G$ be a group and $S\subseteq G\setminus\{0\}$, the \textit{Cayley graph} $\Gamma=\Cay(G,S)$ is the graph with vertex set $G$ such that $x\in G$ is adjacent to $y\in G$ if and only if $y-x\in S$. The set $S$ is called the connection set of the graph. If $S$ is closed under taking inverse elements (i.e. $S=-S$), then $\Gamma$ is undirected, and if $S$ generates $G$, then $\Gamma$ is connected.}{Since we use additive notation, in a Cayley graph $\Cay(G,S)$ there is an arc from $x$ to $y$ if and only if $y - x \in S$. Set
$$
S_0 = S \cup \{0\}.
$$
The following lemma says that $\Cay(G, S)$ admits a perfect code if and only if $G$ admits a factorization with $S_0$ as one of the factors.}

\cha{Perfect codes on Cayley graphs are strongly related to the factorization of groups because $\Cay(G, S)$ admits a perfect code $C$ is equivalent to $G=C\oplus S_0$ where $S_0=S\cup\{0\}$.}{}

\begin{lemma}[{\cite[Lemma 2.3]{eds2, subgrup}}]
\label{lem:factor} 
\cha{Let $G$ be a group and $S\subseteq G\setminus\{0\}$, then $\Cay(G, S)$ admits a perfect code $C$ if and only if $G = C\oplus S_0$.}{Let $G$ be a group, $S$ an inverse-closed subset of $G\setminus\{0\}$, and $C$ a subset of $G$. Then $C$ is a perfect code in $\Cay(G, S)$ if and only if $G = S_0 \oplus C$.}
\end{lemma}

\cha{When a perfect code contains zero, we call the perfect code normalized. Note that a translation of any perfect code is still a perfect code because if $G=C\oplus S_0$, then $(C+g)\oplus S_0=G+g=G$. This means every perfect code is a translation of some normalized perfect code. To characterize all perfect codes of a Cayley graph, it is enough to only consider the normalized perfect codes.}{In particular, $\Cay(G, S)$ admits a perfect code if and only if there exists a subset $C$ of $G$ such that $G = S_0 \oplus C$. A perfect code in $\Cay(G, S)$ is said to be \textit{normalized} if it contains $0$. By Lemma \ref{lem:factor}, one can see that any translation of a perfect code in a Cayley graph is also a perfect code. Hence every perfect code in a Cayley graph is a translation of some normalized perfect code. Therefore, to classify all perfect codes in a Cayley graph it suffices to consider normalized perfect codes.}

\begin{lemma}[\cite{book1}]\label{oplus} 
    Let $A$ and $B$ be \cha{}{nonempty} subsets of an abelian group $G$. The following statements are equivalent\cha{.}{:}
    \begin{enumerate}[\rm (a)]
        \item $G=A\oplus B$\cha{.}{;}
        \item $G=A+B$ and $|G|=|A| |B|$\cha{.}{;}
        \item $|G|=|A| |B|$ and $(A-A)\cap(B\cap B)=\{0\}$\cha{.}{;}
        \item $A+B$ is a direct sum and $|G|=|A| |B|$.
    \end{enumerate}
\end{lemma}

Let $A$ be a nonempty subset of an abelian group $G$. Define
$$
L_A = \{g \in G\, | \, g+A=A\}.
$$ 
Then $L_A \ne \emptyset$ as $0 \in L_A$. Note that $L_A$ is the setwise stabilizer of $A$ in $G$ in the action of $G$ on itself by addition and hence is a subgroup of $G$. For any subset $X$ of $G$, set
$$
(X+L_A)/{L_A} = \left\{x+L_A\, | \, x \in X\right\}.
$$

\begin{lemma}[\cite{book1}]\label{period} 
\cha{Let $G$ be a finite abelian group and $A$ a subset of $G$. Let $L_A$ be the set of elements $g\in G$ such that $g+A=A$. Then $L_A$ is a subgroup of $G$ called the subgroups of periods of $A$, $A$ is the union of some cosets of $L_A$ ($A=L_A\oplus B$ for some set $B$), and $P_A=\bigcap_{a\in A}(A-a).$}{Let $G$ be an abelian group and $A$ a nonempty subset of $G$. Then $L_A$ is a subgroup of $G$, $A$ is the union of some cosets of $L_A$ (that is, $A = L_A \oplus D$ for some $D \subseteq G$), and $L_A=\bigcap_{a\in A}(A-a)$.}
\end{lemma}

\cha{}{We call $L_A$ the \textit{subgroup of periods} of $A$ \cite{book1}. If $L_A$ contains at least one non-zero element (called a \textit{period} of $A$ in $G$), then $A$ is called a \textit{periodic} subset of $G$; otherwise, $A$ is called an \textit{aperiodic} subset of $G$. In other words, $A$ is periodic or aperiodic depending on whether $L_A$ is a nontrivial or trivial subgroup of $G$. If in every factorization of an abelian group $G$ at least one factor is periodic, then $G$ is called a \textit{Haj\'os} group.} 

\begin{lemma}
[\cite{gag2}]
\label{goodabelian}
All Haj\'os groups are precisely the following groups and their subgroups (where $p$, $q$, $r$, and $s$ are different primes, and $k$ is a positive integer):
\begin{multicols}{3}
    \begin{enumerate}
        \item $\Z_{p}\x\Z_p$ \cite{gagpp}
        \item $\Z_{p}\x\Z_3\x\Z_3$ \cite{gag2}
        \item $\Z_{p}\x\Z_q\x\Z_2\x\Z_2$ \cite{gag2}
        \item $\Z_{p}\x\Z_4\x\Z_2$ \cite{gag2}
        \item $\Z_{p^3}\x\Z_2\x\Z_2$ \cite{gag2}
        \item $\Z_{p^2}\x\Z_2\x\Z_2\x\Z_2$ \cite{gag2}
        \item $\Z_{p}\x\Z_2\x\Z_2\x\Z_2\x\Z_2$ \cite{gag2}
        \item $\Z_{2^k}\x\Z_2$ \cite{gag3}
        \item $\Z_{p^k}\x\Z_q$ \cite{gag0}
        \item $\Z_{p^2}\x\Z_{q^2}$ \cite{gag1}
        \item $\Z_{p^2}\x\Z_q\x\Z_r$ \cite{gag1}
        \item $\Z_{p}\x\Z_q\x\Z_r\x\Z_s$ \cite{gag1}
        \item $\Z_{9}\x\Z_3$ \cite{gag3}
        \item $\Z_{4}\x\Z_4$ \cite{gag3}
    \end{enumerate}
\end{multicols}
\end{lemma}

\cha{We provide some lemmas that we will use to get our main results.}{We now present four lemmas that will be used in the proof of our main results.} The following lemma is a combination of Lemma 2.6 and Corollary 2.2 in \cite{book1} with some new results.

\begin{lemma}\label{quotient}
\cha{Let $G$ be a group and $A$, $B$ subsets of $G$. Let $G=A\oplus B$ be a factorization where $A$ is periodic and $A=L_A\oplus D$ where $L_A$ is the subgroups of periods of $A$. Then,}{Let $G$ be an abelian group, and let $G=A\oplus B$ be a factorization of $G$ such that $A$ is periodic. Then for any subset $D$ of $G$ with $A=L_A\oplus D$,}
$$
\cha{G/{L_A}=(D+L_A)/{L_A}\oplus (B+{L_A})/L_{A}}{}
$$
$$
G/{L_A} = ((D+L_A)/{L_A}) \oplus ((B+{L_A})/L_{A})
$$
is a factorization of $G/{L_A}$\cha{, where $ (H+L_A)/{L_A}=\left\{h+L_A:h\in H\right\}$, $H=B,D$}{}. Moreover, $(D+L_A)/L_A$ is aperiodic in $G/L_A$, $|(D+{L_A})/L_{A}|=|D|$, and $|(B+{L_A})/L_{A}|=|B|$.
\end{lemma}

\begin{proof}
\cha{Note that $G=L_A\oplus D\oplus B$, so for any $g+L_A \in G/L_A$,
    $g+L_A=l+d+b+L_A=(d+l+L_A) + (b+L_A)=(d+L_A)+(b+l_A)\in (D+L_A)/{L_A} + (B+{L_A})/L_{A}$ for some $l\in L_A$, $d\in D$, and $b\in B$. Let $g+L_A=(d_1+L_A)+(b_1+L_A)=(d_2+L_A)+(b_2+L_A) \in (D+L_A)/{L_A} + (B+{L_A})/L_{A}$, then $d_1+b_1+l_1=d_2+b_2+l_2$ for some $l_1,l_2\in L$. Since $G=L_A\oplus D\oplus B$, then $b_1=b_2$ and $d_1=d_2$ which means $G/{L_A}=(D+L_A)/{L_A}\oplus (B+{L_A})/L_{A}$.}{Since $G=A \oplus B = L_A\oplus D\oplus B$, for any $g+L_A \in G/L_A$, we have $g+L_A=l+d+b+L_A=(d+l+L_A) + (b+L_A)=(d+L_A)+(b+l_A)\in ((D+L_A)/{L_A}) + ((B+{L_A})/L_{A})$ for some $l\in L_A$, $d\in D$ and $b\in B$. If $g+L_A=(d_1+L_A)+(b_1+L_A)=(d_2+L_A)+(b_2+L_A) \in ((D+L_A)/{L_A}) + ((B+{L_A})/L_{A})$ for some $d_1, d_2 \in D$ and $b_1, b_2 \in B$, then $d_1+b_1+l_1=d_2+b_2+l_2$ for some $l_1, l_2\in L$. Since $G=L_A\oplus D\oplus B$, we must have $b_1=b_2$ and $d_1=d_2$. Since this holds for any $g+L_A \in G/L_A$, it follows that $G/{L_A}=((D+L_A)/{L_A}) \oplus ((B+{L_A})/L_{A})$.}

\cha{To prove $(D+L_A)/L_A$ is aperiodic in $G/L_A$, suppose $(g+L_A)+(D+L_A)/L_A=(D+L_A)/L_A$, we will prove $g\in L_A$ which means $(D+L_A)/L_A$ does not have a non-zero period in $G/L_A$, i.e. aperiodic. Let $D=\{d_1,d_2,\dots,d_k\}$, then $(D+L_A)/L_A=\left\{ d_1+L_A,\dots,d_k+L_A\right\}$ and $(g+L_A)+(D+L_A)/L_A=\left\{ g+d_1+L_A,\dots,g+d_k+L_A\right\}$. Now,}{We now prove $(D+L_A)/L_A$ is aperiodic in $G/L_A$. Suppose $(g+L_A)+(D+L_A)/L_A=(D+L_A)/L_A$. We will prove $g\in L_A$, which means $(D+L_A)/L_A$ has no period in $G/L_A$ and so is aperiodic. Set $D=\{d_1,d_2,\dots,d_k\}$. Then $(D+L_A)/L_A=\left\{ d_1+L_A,\dots,d_k+L_A\right\}$ and $(g+L_A)+(D+L_A)/L_A=\left\{ g+d_1+L_A,\dots,g+d_k+L_A\right\}$. Now,}
$$
\cha{g+A
        =g+D+L_A
        =\bigcup_{i=1}^k g+d_i+L_A
        =\bigcup_{i=1}^k d_i+L_A
        =D+L_A
        =A.}{}
$$
$$
\cha{}{g+A
        =g+D+L_A
        =\bigcup_{i=1}^k (g+d_i+L_A)
        =\bigcup_{i=1}^k (d_i+L_A)
        =D+L_A
        =A.}
$$
\cha{So $g$ is a period of $A$.}{So $g \in L_A$ as desired.}

\cha{We now prove $|(D+{L_A})/L_{A}|=|D|$ and $|(B+{L_A})/L_{A}|=|B|$. Since $A=L_A\oplus D$, then $\left(L_A\cap(D-D)\right) \subseteq \left((L_A-L_A)\cap(D-D)\right)=\{0\}$. Let $d_1,d_2\in D$,}{Since $A=L_A\oplus D$, we have $\left(L_A\cap(D-D)\right) \subseteq \left((L_A-L_A)\cap(D-D)\right)=\{0\}$. For $d_1,d_2\in D$, we have}
    $$
    d_1+L_A=d_2+L_A \Leftrightarrow d_1-d_2\in L_A \Leftrightarrow d_1-d_2\in \left(L_A\cap(D-D)\right)=\{0\} \Leftrightarrow d_1=d_2.
    $$
    So $|(D+{L_A})/L_{A}|=|D|$. Now, 
    $$
    |(B+{L_A})/L_{A}|=\frac{|G/L_A|}{|(D+L_A)/L_A|}=\frac{|G|}{|L_A|\cdot|D|}=\frac{|G|}{|A|}=|B|.
    $$
\cha{}{This completes the proof.}
\end{proof}

The following lemmas \cha{are going to}{will} be used repeatedly throughout this paper.

\begin{lemma}\label{d_i+l_i}
Let $G = A \oplus B$ and $A = L_A \oplus D$ be as in Lemma \ref{quotient}. Let $|(D+L_A)/L_A|=n$ and $|(B+L_A)/L_A|=m$.
If
$$
L_A=\{l_i\mid i\ttk{t}\}, \,
(D+L_A)/L_A=\{d_j+L_A \mid j\ttk{n}\}, \, \text{and} \,\,
(B+L_A)/L_A=\{b_k+L_A \mid k\ttk{m}\},
$$
then
$$
A=\{l_i+d_j\mid i\ttk{t}, j\ttk{n}\} \quad \text{and} \quad
B=\{b_k+\ell_k \mid k\ttk{m}\}
$$
for some $\ell_k\in L_A$.
\end{lemma}
\begin{proof}
    First we prove the lemma for $B$.
    From Lemma \ref{quotient} we have $|B|=|(B+L_A)/L_A|=m$, so let $B=\{x_k\mid k\ttk{m}\}$.
    By definition $(B+L_A)/L_A=\{x_k + L_A\mid k\ttk{m}\}$.
    This means there is a permutation $\sigma\in \mathrm{sym}(m)$ such that $b_k+L_A=x_{\sigma(k)}+L_A$ for all $k$ in $[m]$.
    Hence $x_{\sigma(k)}=b_k+\ell_k$ for some $\ell_k$ in $L_A$. 
    Therefore $B=\{x_k\mid k\ttk{m}\}$ $=\{x_{\sigma(k)}\mid k\ttk{m}\}$ $=\{b_k+\ell_k \mid k\ttk{m}\}$ for some $\ell_k$ in $L_A$.

    Now we prove the lemma for $A$.
    By following the same arguments for $B$, we will get $D=\{d_j+\ell_j \mid j\ttk{n}\}$ for some $\ell_j$ in $L_A$.
    As $L_A+\ell_j=L_A$, we have 
    \[
        A = L_A \oplus D
        = \bigcup_{j\ttk{n}} (L_A+d_j+\ell_j)
        = \bigcup_{j\ttk{n}} (L_A+d_j)
        = L_A \oplus \{d_j \mid j\ttk{n}\}
        = \{l_i + d_j \mid i\ttk{t}, j\ttk{n}\}. \qedhere
    \]
\end{proof}

\begin{lemma}\label{generate-zero}
Let $G = A \oplus B$ and $A = L_A \oplus D$ be as in Lemma \ref{quotient}. The following two statements hold.
    \begin{enumerate}[{\rm(a)}]
        \item $A$ generates $G$ if and only if $(D+L_A)/L_A$ generates $G/L_A$.
        \item $0\in A$ if and only if $L_A \in (D+L_A)/L_A$.
    \end{enumerate}
\end{lemma}

\begin{proof}
Suppose $A$ generates $G$ and let $g+L_A \in G/L_A$. Since $A$ generates $G$, \cha{then}{we have} $g=\sum_{i=1}^{n} a_i$ for some integer $n$ and \cha{}{all} $a_i\in A$. For every $i$, write $a_i=l_i+d_i$ where $l_i\in L_A$ and $d_i\in D$. Then $g+L_A=\sum_{i=1}^n (l_i+d_i) +L_A = \sum_{i=1}^n d_i +L_A$. This means $(D+L_A)/L_A$ generates $G/L_A$.

Now suppose $(D+L_A)/L_A$ generates $G/L_A$. Let $g\in G$ and write $g+L_A=\sum_{i=1}^m d_i+L_A$ for some integer $m$ and $d_i\in D$. Let $l=g-\sum_{i=1}^m d_i \in L_A$. Then $g=\sum_{i=1}^{m-1} (d_i+0) + (d_m+l) $. \cha{Note that $(d_i+0),(d_m+l)\in D\oplus L_A=A$. This means $A$ generates $G$.}{Since both $d_i+0$ and $d_m+l$ belong to $D\oplus L_A=A$ and $g$ is an arbitrary element of $G$, it follows that $A$ generates $G$.} This proves (a).

Suppose $0\in A$. Then $0 \in A = D\oplus L_A$ and hence $0=d+l$ for some $d\in D$ and $l\in L_A$. So $d=-l\in L_A$ and hence $L_A=d+L_A\in (D+L_A)/L_A$. Now suppose $L_A \in (D+L_A)/L_A$. This implies $L_A=d+L_A$ for some $d$ in $D$, and hence $d\in L_A$. As $L_A$ is a subgroup of $G$ by Lemma \ref{period}, $-d$ is in $L_A$. It follows that $0=d+(-d)\in D\oplus L_A=A$. This proves (b).
\end{proof}

Part (c) in the following lemma is also Theorem 3.1 part (b) in \cite{camyapzhou}. They also prove that for every Haj\'os group $G$, every coset with order $|G|/p$ where $p$ is a prime is a perfect code for some Cayley graph of $G$.

\begin{lemma}\label{ls<so}\label{prime deg}
Let $G$ be a Haj\'os group. Let $\Cay(G,S)$ be a connected but non-complete Cayley graph of $G$ which admits a perfect code $C$.
    \begin{enumerate}[{\rm(a)}]
        \item If $G$ is not a cyclic group, then $|\so|>2$.
        \item If $\so$ is periodic, then \cha{$|\ls|$ divides $|\so|$ and $|\ls|<|\so|$}{$1 < |\ls| < |\so|$ and $|\ls|$ divides $|\so|$}.
        \item If $|\so|$ is a prime, then $\so$ is aperiodic and \cha{$C=H+g$ where $H$ is a subgroup of $G$ and $g\in G$}{$C = g + L_C$ for some element $g$ of $G$}.
    \end{enumerate}
\end{lemma}

\begin{proof}
\cha{If $|\so|=2$, then $|S|=1$, which means $G$ is generated by a single element, i.e. a cyclic group. If $\so$ is periodic, then $\so=\ls\oplus D$ for some subset $D$, so $|\ls|$ divide $|\so|$. If $|\ls|=|\so|$, then $\so$ is a coset of $\ls$. Since both $\so$ and $\ls$ contains $0$, then $\so=\ls$. This means, $\so$ is a subgroup of $G$ that generates $G$, a contradiction, therefore $|\ls|<|\so|$.}{Since $\Cay(G,S)$ is connected, $S$ is a generating set of $G$. If $|\so|=2$, then $|S|=1$ and hence $G$ is a cyclic group. This proves part (a). Suppose $\so$ is periodic. Then $|\ls| > 1$ and $\so=\ls\oplus D$ for some subset $D$. So $|\ls|$ divides $|\so|$. If $|\ls|=|\so|$, then $\so$ is a coset of $\ls$. Since both $\so$ and $\ls$ contain $0$, we must have $\so=\ls$. Hence $\so$ is a subgroup of $G$ that generates $G$, which means $\so = G$, but this contradicts the assumption that $\Cay(G,S)$ is not a complete graph. Therefore, $|\ls|<|\so|$ and (b) is proved.}

\cha{Let $|\so|=p$, a prime. Suppose $S_0$ is periodic, from Lemma \ref{period} we have $|L_{S_0}|$ is greater than 1 and divides $|S_0|$, a prime, so $|L_{S_0}|=|S_0|$. Since $0\in L_{S_0}$, $0\in S_0$, and $S_0=L_{S_0}+B$ from Lemma \ref{period}, then $L_{S_0}=S_0$. This means ${S_0}$ is a proper subgroup of $G$ that generates $G$, a contradiction. So $S_0$ is aperiodic and since $G$ is a Haj\'os group, $C$ has to be periodic. Consider the factorization of the quotient group $G/L_C$ as in Lemma \ref{quotient}, $G/L_C=(D+L_C)/L_C\oplus (S_0+L_C)/L_C$ where $C=D\oplus L_C$. The quotient group $G/L_C$ is isomorphic to a subgroup of $G$, so it is also a Haj\'os group. From Lemma \ref{quotient}, we know that $(D+L_C)/L_C$ is aperiodic, so $(S_0+L_C)/L_C$ is periodic, and since $|(S_0+L_C)/L_C|=|S_0|$, $(S_0+L_C)/L_C$ is equal to its subgroup of periods in $G/L_C$. Since $S_0$ generates $G$, then $(S_0+L_C)/L_C$ generates $G/L_C$. So $(S_0+L_C)/L_C$ is a subgroup of $G/L_C$ that also generates $G/L_C$, this implies $(S_0+L_C)/L_C=G/L_C$.}{Now assume $|\so|=p$ is a prime. Suppose $S_0$ is periodic. Then by Lemma \ref{period} we have $|L_{S_0}| > 1$ and $|L_{S_0}|$ divides $|S_0| = p$. So $|L_{S_0}|=|S_0|$. Since by Lemma \ref{period}, $S_0$ is a union of some cosets of $L_{S_0}$, it follows that $S_0 = L_{S_0}$ and hence $S_0$ is a subgroup of $G$. Since $S_0$ also generates $G$, we have $S_0 = G$ and so $\Cay(G,S)$ is a complete graph, but this contradicts our assumption. Therefore, $S_0$ is aperiodic. Since $G$ is a Haj\'os group and $G = C \oplus S_0$ (as $C$ is a perfect code in $\Cay(G,S)$), $C$ must be periodic. Since $G = C \oplus S_0$ and $C$ is periodic, by Lemma \ref{quotient}, for any subset $D$ of $G$ with $C=L_C \oplus D$ we have}
$$
G/L_C=((D+L_C)/L_C) \oplus ((S_0+L_C)/L_C).
$$ 
\cha{}{Since $G$ is abelian, $G/L_C$ is isomorphic to a subgroup of $G$. Thus $G/L_C$ is also a Haj\'os group as $G$ is Haj\'os. Since, by Lemma \ref{quotient}, $(D+L_C)/L_C$ is aperiodic in $G/L_C$, $(S_0+L_C)/L_C$ in the factorization above must be periodic in the Haj\'os abelian group $G/L_C$. Since $|(S_0+L_C)/L_C|=|S_0|$, $(S_0+L_C)/L_C$ is equal to its subgroup of periods in $G/L_C$. Since $S_0$ generates $G$, $(S_0+L_C)/L_C$ generates $G/L_C$. Now that $(S_0+L_C)/L_C$ is a subgroup of $G/L_C$ which generates $G/L_C$, we must have $(S_0+L_C)/L_C=G/L_C$.} Hence
$$
|L_C|=\frac{|G|}{|G/L_C|}=\frac{|G|}{|(S_0+L_C)/L_C|}=\frac{|G|}{|S_0|}=|C|.
$$
\cha{So $C = L_C + g$}{Therefore, $|D| = 1$ and so $C = g + L_C$} for some $g\in G$.
\end{proof}

\section{Perfect codes in Cayley graphs of $\Z_p\x\Z_q\x\Z_q$ and their subgroups}
\label{sec:pqq}

In this section, we characterize Cayley graphs on Haj\'os groups of the form $\Z_p\x\Z_q\x\Z_q$ and their subgroups, where $p$ and $q$ are distinct primes. These Haj\'os groups are $\Z_p\x\Z_p$ with $p\geq3$, $\Z_p\x\Z_q$ with $p \ne q$, and $\Z_p\x\Z_q\x\Z_q$ with $p \geq 3 \geq q$ and $p\ne q$. The main result in this section are the following four theorems.

\begin{theorem}\label{as-bs}\label{p,p}
    Let $G=\Z_p\x\Z_p$ for an odd prime $p$, let $S$ be a proper subset of $G\setminus \{0\}$ that generates $G$, and let $S_0=S\cup\{(0,0)\}$. Then a subset $C$ of $G$ is a perfect code in $\Cay(G,S)$ if and only if $|S|=p-1$ and there exist $a,b\in\Z_p$ such that $\{as_1-bs_2\mid (s_1,s_2)\in S_0\}=\Z_p$ and $C=\{(bn+c,an+d)\mid n\in \Z_p\}$ for some $c,d\in\Z_p$.
\end{theorem}
\begin{proof}
    Suppose that $|S|=p-1$ and there exist $a,b\in\Z_p$ such that $\{as_1-bs_2\mid (s_1,s_2)\in S_0\}=\Z_p$ and $C=\{(bn+c,an+d)\mid n\in \Z_p\}$. 
    Then $|\{as_1-bs_2\mid (s_1,s_2)\in S_0\}|=|S_0|$.    
    To prove that $C$ is a perfect code in $\Cay(G,S)$, suppose for a contradiction that $(bn_1+c,an_1+d)+(s_1,s_2)=(bn_2+c,an_2+d)+(s_3,s_4)$ for distinct $(s_1,s_2),(s_3,s_4)\in S$.
    Then $b(n_1-n_2)=s_3-s_1$ and $a(n_1-n_2)=s_4-s_2$. It follows that    
    $a(s_3-s_1)=ab(n_2-n_1)=b(s_4-s_2)$, which implies $as_3-bs_4=as_1-bs_2$, contradicting $|\{as_1-bs_2\mid (s_1,s_2)\in S_0\}|=|S_0|$.
    Hence $C+S_0$ is a direct sum. Since $|S_0|\cdot|C|=p^2=|G|$, we conclude that $G=S_0\oplus C$.

    Conversely, suppose that $\Cay(G,S)$ admits a perfect code $C$, then $G=\so\oplus C$. From Lemma \ref{oplus}, we have $|S|={p-1}$. 
    Since $|S_0|=p$ is a prime, Lemma \ref{prime deg} implies that $C=L_C+g$ for some $g\in G$.
    Note that $L_C$ is a proper subgroup of $G$, so $C=\left<(b,a)\right>+(c,d)=\{(bn+c,an+d)\mid n\in \Z_p\}$ for some non-zero $(b,a)\in G$ and $c,d\in\Z_p$.
    It remains to prove that $\{as_1-bs_2\mid (s_1,s_2)\in S_0\}=\Z_p$, which is equivalent to that $as_1-bs_2$ is all different for all $(s_1,s_2)\in S_0$.
    Suppose there are $(s_1,s_2),(s_3,s_4)\in S_0$ such that $as_1-bs_2=as_3-bs_4$, so that $a(s_1-s_3)=b(s_2-s_4)$.
    If $a\ne 0$, then 
    \begin{align*}
        (s_1,s_2)+(c,d)=(s_3,s_4)+(s_1-s_3+c,s_2-s_4+d)
        &=\left(b(a^{-1}s_2-a^{-1}s_4)+c,a(a^{-1}s_2-a^{-1}s_4)+d\right)
    \end{align*}
    are two representations of $(s_1+c,s_2+d)$ as the sum of an elements of $S_0$ and an element of $C$.
    If $a=0$, then $b\ne 0$ and 
    \begin{align*}
        (s_1,s_2)+(c,d)=(s_3,s_4)+(s_1-s_3+c,s_2-s_4+d)
        &=\left(b(b^{-1}s_1-a^{-1}s_3)+c,a(b^{-1}s_1-a^{-1}s_3)+d\right)
    \end{align*}
    are two representations of $(s_1+c,s_2+d)$ as the sum of an elements of $S_0$ and an element of $C$.
    This contradicts $G=S_0\oplus C$, so $\{as_1-bs_2\mid (s_1,s_2)\in S_0\}=\Z_p$, which completes the proof.
\end{proof}

We give an equivalent form of Theorem \ref{as-bs} that requires less calculation to check whether $\Cay(G,S)$ admits a perfect code.

\begin{theorem}\label{as-s}
    Let $G=\Z_p\x\Z_p$ for an odd prime $p$ and let $S$ be a proper subset of $G\setminus \{0\}$ that generates $G$. Then a subset $C$ of $G$ is a perfect code in $\Cay(G,S)$ if and only if $|S|=p-1$ and $S$ satisfies: (1) $\{s_1\mid (s_1,s_2)\in S_0\}=\Z_p$, or (2) there is an element $a\in\Z_p$ such that $\{as_1-s_2\mid (s_1,s_2)\in S_0\}=\Z_p$.
    Moreover, the perfect codes are $\{(b,n)\mid n\in\Z_p\}$ for any $b\in\Z_p$ if $S$ satisfies (1), or $\{(n,an+b)\mid n\in \Z_p\}$ for any $b\in\Z_p$ if $S$ satisfies (2).
\end{theorem}
\begin{proof}
    This theorem follows from Theorem \ref{as-bs} and the fact that $\{as_1-bs_2\mid (s_1,s_2)\in S_0\}=\Z_p$ if and only if $\{\lambda as_1-\lambda bs_2\mid (s_1,s_2)\in S_0\}=\Z_p$ for any non-zero $\lambda \in \Z_p$. So, instead of checking if $(a,b)$ satisfies $\{as_1-bs_2\mid (s_1,s_2)\in S_0\}=\Z_p$ for all non-zero $(a,b)\in{\Z_p^2}$, we only need to check $(a,b)$ for $b=1$ and for $(a,b)=(1,0)$.
\end{proof}

\begin{remark}
    In Theorem \ref{as-s}, the two conditions of $S$ for $\Cay(\Z_p^2, S)$ to have a perfect code are not mutually exclusive. One example is $\Gamma=\Cay\left(\Z_5^2,\{(1,0),(2,2),(3,2),(4,1)\}\right)$, as $\Gamma$ satisfies both (1) $\{s_1\mid (s_1,s_2)\in S_0\}=\Z_5$, and (2) $\{3s_1-s_2\mid (s_1,s_2)\in S_0\}=\Z_5$. In this example, both (1) $\{(b,n)\mid n\in\Z_5\}$ and (2) $\{(n,3n+b)\mid n\in \Z_5\}$ are perfect codes in $\Gamma$ for any $b\in\Z_5$.
\end{remark}

Theorem \ref{as-bs} seems simpler compared to Theorem \ref{as-s} for characterizing $\Gamma=\Cay(\Z_p^2, S)$ that admits a perfect code, as there is only one condition for $S$. But the condition in Theorem \ref{as-bs} requires $O(p^2)$ computing time, and the condition in Theorem \ref{as-s} only requires $O(p)$ computing time. For this reason, we argue that Theorem \ref{as-s} is a better characterization for $\Gamma$ to have a perfect code.

Let $\A$ and $\B$ be two generators of $G=\Z_p\x\Z_p$.
If $C$ is a normalized perfect code in Theorem \ref{as-bs}, then $C=\lc=\<{(a,b)}=\<\A$. 
By setting $D=\{(0,0)\}$ in the factorization $G/\lc=((\so+\lc)/\lc) \oplus ((D+\lc)/\lc)$ given by Lemma \ref{quotient}, we have $(\so+\lc)/\lc=G/\lc=\{i\B+\lc \mid i\ttk{p}\}$. This implies $\so=\{\alpha_i\A+i\B \mid i\ttk{p}\}$, for some $\alpha_i\in \Z$.


\begin{theorem}\label{p.q}
Let $G=\Z_p\times\Z_q$ for distinct primes $p$ and $q$, and let $S$ be a proper subset of $G\setminus \{0\}$ that generates $G$. Let $0\in C\subseteq G$. Then $C$ is a perfect code in $\Cay(G,S)$ if and only if $(\so,C)$ is one of the pairs in Table \ref{t.p.q} where $\alpha_0=0$ and ${\alpha_i}$'s are in $\Z$.
\end{theorem} 
\rc\begin{longtable}{|c|l|l|}
\caption{Perfect codes in $\Cay(\Z_p\x\Z_q,\,S)$}\label{t.p.q}\\
\fl
\tn & $\{(i,\alpha_i) \mid i\ttk{p}\}$
& $\{(0,i) \mid i\ttk{q}\}$ \\
\tn & $\{(\alpha_i,i) \mid i\ttk{q}\}$
& $\{(i,0) \mid i\ttk{p}\}$ \\
\hline
\end{longtable}

\begin{proof}
    Suppose $\Cay(G,S)$ has a perfect code $C$. Since $\Cay(G,S)$ is not a trivial graph, we have $|\so|=p$ or $q$.
    If $|\so|=p$, then $C$ is a subgroup of order $q$ from Lemma \ref{prime deg}, so $C=\lc=\<{(0,1)}$.
    Recall the factorization $G/\lc=((\so+\lc)/\lc) \oplus ((D+\lc)/\lc)$ given in Lemma \ref{quotient}, where $C=D\oplus \lc$.
    By setting $D=\{(0,0)\}$, we have $C=D\oplus \lc$ and     $(\so+\lc)/\lc=G/\lc=(\Z_{p}\times\Z_q)/\<{(0,1)}=\{(i,0)+\lc \mid i\ttk{p}\}$. This implies $\so=\{(i,\alpha_i) \mid i\ttk{p}\}$, for some $\alpha_i\in \Z$, which gives the first pair in Table \ref{t.p.q}. Now if $(S_0,C)$ is the first pair in Table \ref{t.p.q}, then $|S_0|\cdot|C|=|G|$, and for any $(x,y)\in G$, $(x,y)=(x,\alpha_x)+(0,y-\alpha_x)\in S_0+C$. Therefore, $G=S_0\oplus C$. The same reason applies for $|\so|=q$ to get the second pair in Table \ref{t.p.q}.
\end{proof}

\begin{theorem}\label{p,q,q}\label{p.2.2}
Let $G=\Z_p\x\Z_q\x\Z_q$ for distinct primes $p$ and $q$, where $p\geq3\geq q$, and let $S$ be a proper subset of $G\setminus \{0\}$ that generates $G$. Let $0\in C\subseteq G$. Then $C$ is a perfect code in $\Cay(G,S)$ if and only if $(\so,C)$ is one of the pairs in Table \ref{t.p.q.q} where $\alpha_0=\beta_0=\alpha_{00}=0$ and $\alpha_i$'s, $\beta_i$'s, and $\alpha_{ij}$'s are in $\Z$.
\end{theorem}

\rc\begin{longtable}{|c|l|l|}
\caption{Perfect codes in $\Cay(\Z_p\x\Z_q\x\Z_q,S)$}\label{t.p.q.q}\\
\fl
\tn & $\{i\e1+\alpha_i\e2+\beta_i\e3 \mid i\ttk{p}\}$
& $\{i\e2+j\e3 \mid i,j\ttk{q}\}$ \\
\tn & $\{\alpha_i\e1 +\alpha_i\e2 + i\e3 \mid i\ttk{3}\}$
& $\{i\e1 + j\e2 \mid i\ttk{p},\, j\ttk{3}\}$ \\
\tn & $\{j\e1 + \alpha_i \e2 + i\e3 \mid i\ttk{3}, j\ttk{p}\}$
& $\{i\e2+\alpha_i \e3 \mid i\ttk{3}\}$ \\
\tn & $\{i\e1+j\e2+\alpha_i\e3 \mid i\ttk{p},\,j\ttk{q}\}$
& $\{\alpha_i\e2+i\e3 \mid i\ttk{q}\}$ \\
\tn & $\{i\e1+\alpha_{ij}\e2+j\e3 \mid i\ttk{p},\,j\ttk{q}\}$
& $\{i\e2 \mid i\ttk{q}\}$ \\
\tn & $\{\alpha_i\e1+j\e2+i\e3 \mid i,j\ttk{q}\}$
& $\{i\e1+\alpha_i \e2 \mid i\ttk{p}\}$ \\
\tn & $\{\alpha_{ij}\e1+i\e2+j\e3 \mid i,j\ttk{q}\}$
& $\{i\e1 \mid i\ttk{p}\}$ \\
\hline
\end{longtable}

\begin{proof}
    Suppose $\Cay(G,S)$ has a perfect code $C$. From Lemma \ref{ls<so}, we have $|\so|>2$.

    If $|\so|=p$, then from Lemma \ref{prime deg}, $\so$ is aperiodic and $C$ is a subgroup of order $q^2$, so $C=\lc=\{0\}\times\Z_q\times\Z_q$.
    From the factorization of $G/\lc$ given in Lemma \ref{quotient}, $(\so+\lc)/\lc=G/\lc=(\Z_{p}\times\Z_q\times\Z_q)/(\{0\}\times\Z_q\times\Z_q)=\{i\e1+\lc \mid i\ttk{p}\}$, and hence $\so=\{i\e1+\alpha_i\e2+\beta_i\e3 \mid i\ttk{p}\}$, $\alpha_i,\beta_i\in \Z$, which is pair 1 in Table \ref{t.p.q.q}.

    If $|\so|=q=3$, then from Lemma \ref{prime deg}, $\so$ is aperiodic and $C$ is a subgroup of order $pq$, so $C=\lc=\<{\e1+\e2}$.
    From the factorization of $G/\lc$ given in Lemma \ref{quotient}, $(\so+\lc)/\lc=G/\lc=\{i\e3+\lc \mid i\ttk{q}\}$, and hence $\so=\{\alpha_i\e1 +\alpha_i\e2 + i\e3 \mid i\ttk{q}\}$, $\alpha_i\in \Z$, which is pair 2 in Table \ref{t.p.q.q}.
    If $q=2$, then $S_0$ does not generate $G$.

    If $|\so|=pq$ and $\so$ is periodic, then $|\ls|=q$ or $p$ by Lemma \ref{ls<so}.
    Let $\so=D\oplus\ls$ for some subset $D$. From the factorization of $G/\ls$ given in Lemma \ref{quotient}, we have $G/\ls=((D+\ls)/\ls)\oplus ((C+\ls)/\ls)$.
    If $|\ls|=p$, then $\ls=\left<\e1\right>$.
    Note that $(D+\ls)/\ls$ is a aperiodic set of order $q$ and $G/\ls\cong \Z_q\times\Z_q$.
    If $q=2$, then from Lemmas \ref{generate-zero}, $(D+\ls)/\ls$ is a aperiodic set of order $2$ that generates $\Z_2\times\Z_2$, which is not possible by Lemma \ref{ls<so}.
    If $q=3$, from Theorem \ref{p,p}, we have $(C+\ls)/\ls=\{i\e2+\ls\mid i\ttk{3}\}$, and hence $(D+\ls)/\ls=\{\alpha_i \e2 + i\e3 \mid i\ttk{3}\}$. 
    Therefore $\so=D\oplus\ls=\{j\e1 + \alpha_i \e2 + i\e3 \mid i\ttk{3}, j\ttk{p}\}$ and $C=\{i\e2+\alpha_i \e3 \mid i\ttk{3}\}$, which is pair 3 in Table \ref{t.p.q.q}.

    If $|\so|=pq$, $\so$ is periodic, and $|\ls|=q$, then $\ls=\left<\e2\right>$.
    Note that $(D+\ls)/\ls$ is a aperiodic set of order $p$ and $G/\ls\cong \Z_{p}\times\Z_q$.
    From the case of $|\so|=p$ in Theorem \ref{p.q} (pair 1), we have $(D+\ls)/\ls=\{i\e1+\alpha_i\e3+\ls \mid i\ttk{p}\}$, where $\alpha_i\in \Z$ and  $(C+\ls)/\ls=\{i\e3+\ls \mid i\ttk{q}\}$.
    Therefore $\so=D\oplus\ls=\{i\e1+j\e2+\alpha_i\e3 \mid i\ttk{p},\,j\ttk{q}\}$ and $C=\{\alpha_i\e2+i\e3 \mid i\ttk{q}\}$, where $\alpha_i\in \Z$, which is pair 4 in Table \ref{t.p.q.q}.

    If $|\so|=pq$ and $\so$ is aperiodic, then $|C|=q$ and $C$ is periodic, and hence $C=\lc=\<{\e2}$.
    From the factorization of $G/\lc$ given in Lemma \ref{quotient}, $(\so+\lc)/\lc=G/\lc=\{i\e1+j\e3+\lc \mid i\ttk{p},\,j\ttk{q}\}$, and hence $\so=\{i\e1+\alpha_{ij}\e2+j\e3 \mid i\ttk{p},\,j\ttk{q}\}$, $\alpha_{ij}\in \Z$, which is pair 5 in Table \ref{t.p.q.q}.

    If $|\so|=q^2$ and $\so$ is periodic, then $|\ls|=q$ from Lemma \ref{ls<so} and hence $\ls=\left<\e2\right>$.
    Let $\so=D\oplus\ls$ for some subset $D$. From the factorization of $G/\ls$ given in Lemma \ref{quotient}, we have $G/\ls=((D+\ls)/\ls)\oplus ((C+\ls)/\ls)$.
    Note that $(D+\ls)/\ls$ is a aperiodic set of order $q$ and $G/\ls\cong \Z_p\times\Z_q$.
    From the case of $|S_0|=q$ in Theorem \ref{p.q} (pair 2), we have $(D+\ls)/\ls=\{\alpha_i\e1+i\e3+\ls \mid i\ttk{q}\}$, where $\alpha_i\in \Z$ and $(C+\ls)/\ls=\{i\e1+\ls \mid i\ttk{p}\}$.
    Therefore $\so=D\oplus\ls=\{\alpha_i\e1+j\e2+i\e3 \mid i,j\ttk{q}\}$ and $C=\{i\e1+\alpha_i \e2 \mid i\ttk{p}\}$ where $\alpha_i\in \Z$, which is pair 6 in Table \ref{t.p.q.q}.

    If $|\so|=q^2$ and $\so$ is aperiodic, then $|C|=p$ and $C$ is periodic, and hence $C=\lc=\<{\e1}$.
    From the factorization of $G/\lc$ given in Lemma \ref{quotient}, we have $(\so+\lc)/\lc=G/\lc=\{i\e2+j\e3+\lc \mid i,j\ttk{q}\}$, and hence $\so=\{\alpha_{ij}\e1+i\e2+j\e3 \mid i,j\ttk{q}\}$, $\alpha_{ij}\in \Z$, which is pair 7 in Table \ref{t.p.q.q}.

    Now, suppose $(S_0,C)$ is one of the pairs in Table \ref{t.p.q.q}. For every pair $(S_0,C)$, we have $|S_0|\cdot|C|=pq^2=|G|$. In every row in the table, we can see that for every $i$, exactly one of $S_0$ or $C$ has an independent variable ($i$ or $j$) before $\mathbf{e}_i$, and possibly a fixed integer ($\alpha_i$, $\beta_i$, or $\alpha_{ij}$) in the other set. This means that for every $\mathbf{x}=(x,y,z)\in G$, if we want $\mathbf{x}=\mathbf{s}+\mathbf{c}$ for $\mathbf{s}\in S_0$ and $\mathbf{c}\in C$,
    we can set the independent variable to be $x$ or $x-\alpha_{i'}$ (similarly for $y$ and $z$).
    For example, take the fourth row, we have
    $$(x,y,z)=\left(x\e1+(y-\alpha'_{z-\alpha_x})\e2+\alpha_x\e3\right)+\left(\alpha'_{z-\alpha_x}\e2+(z-\alpha_x)\e3\right).$$
    This means $G=S_0\oplus C$ and we are done.
\end{proof}

\begin{remark}
    In Theorem \ref{p,q,q}, we can change $\e2$ and $\e3$ into any two generators of $\{0\}\x\Z_q\x\Z_q$.
\end{remark}
\section{Perfect codes in Cayley graphs of good abelian groups containing $\Z_{p^k}$}
\label{sec:p^k}

In this section, we will focus on Haj\'os groups that contain $\Z_{p^k}$ as a subgroup, where $p$ is a prime and $k$ is a positive integer. These Haj\'os groups are $\Z_{p^k}$, $\Z_{p^k}\x\Z_q$, and $\Z_{2^k}\x\Z_2$.
From this section onward, we will only give the normalized form for the perfect codes. Recall that any perfect code is a translation of a normalized perfect code.

\begin{theorem}\label{p^k}
    Let $G=\Z_{p^k}$ for a prime number $p$ and a positive integer $k$. Let $S$ be a proper subset of $G\setminus \{0\}$ that generates $G$. Let $0\in C\subseteq G$. Then $C$ is a perfect code in $\Cay(G,S)$ if and only if there are integers $0=m_0<m_1<\cdots<m_n=k$
    and $1\leq l_t\leq m_{t+1}-m_t$ for $t\in[n]$, such that
        \begin{align}
            \label{S0form}
            \so=\left\{\sum_{t=0}^{n-1}\left(i_tp^{m_t}+\alpha_{i_0,\dots,i_t}p^{m_t+l_t}\right)\, \middle| \, i_t\in\left[p^{l_{t}}\right]\right\}
        \end{align}
        for some $\alpha_{i_0,\dots,i_t}$'s in $\Z_{p^k}$ and $\alpha_{0,\dots,0}=0$.
    Moreover, the normalized perfect codes in $\Cay(G,S)$ are in the following form
    \begin{align}
        \label{Cform}
        C = \left\{\sum_{t=0}^{n-1} \left( i_tp^{m_t+l_t}+\beta_{i_0,\dots,i_t}p^{m_{t+1}}\right) \, \middle| \, i_t\in\left[p^{m_{t+1}-m_t-l_t}\right] \right\}
    \end{align}
    for some $\beta_{i_0,\dots,i_t}\in \Z_{p^k}$ and $\beta_{0,\dots,0}=0$.
\end{theorem}

\begin{theorem}\label{p^k.q}
    Let $G=\Z_{p^k}\x\Z_q$ for distinct primes $p$, $q$, and a positive integer $k$. Let $S$ be a proper subset of $G\setminus \{0\}$ that generates $G$. Let $0\in C\subseteq G$. Then $C$ is a perfect code in $\Cay(G,S)$ if and only if there are integers $r\in[n]$, $0=m_0<m_1<\cdots<m_n=k$,
    and $1\leq l_t\leq m_{t+1}-m_t$ for $t\in[n]\setminus\{r\}$ and $0\leq l_r\leq m_{r+1}-m_r$ such that $(\so,C)$ is one of the following pairs where $\alpha_{i_0,\dots,i_t}$'s and $\beta_{i_0,\dots,i_t}$'s are in $\Z$ and $\alpha_{0,\dots,0}=\beta_{0,\dots,0}=0$.
    \begin{enumerate}[{\rm(a)}]
        \item $\begin{aligned}[t]
            \so &= \left\{ \left(\sum_{t=0}^{n-1} i_tp^{m_t}+ \sum_{t\in[n]\setminus\{r\}}\alpha_{i_0,\dots,i_t}p^{m_t+l_t},\, \alpha_{i_0,\dots,i_r}\right) \, \middle| \, i_t\in\left[p^{l_t}\right]\right\} \\
            C &= \left\{ \left(\sum_{t\in[n]\setminus\{r\}} i_tp^{m_t+l_t}+ \sum_{t=0}^{n-1}\beta_{i_0,\dots,i_t}p^{m_{t+1}},\, i_r\right) \, \middle| \, i_t\in\left[p^{m_{t+1}-m_t-l_t}\right] \text{ for } t\ne r,\ i_r\in[q]\right\}
        \end{aligned}$
        \vspace{0.2 cm}\\
        where $m_r+l_r=m_{r+1}$.\vspace{0.2 cm}
        
        \item $\begin{aligned}[t]
            \so &= \left\{\left(\sum_{t=0}^{n-1} \left( i_tp^{m_t}+\alpha_{i_0,\dots,i_t}p^{m_t+l_t}\right),\, \alpha_{i_0,\dots,i_r}\right) \, \middle| \, i_t\in\left[p^{l_t}\right]\right\} \\
            C &= \left\{ \left(\sum_{t=0}^{n-1} i_tp^{m_t+l_t} + \sum_{t=0}^{r-1} \beta_{i_0,\dots,i_t}p^{m_{t+1}} + \sum_{t=r}^{n-1} \beta_{i_0,\dots,i_t,j}\,p^{m_{t+1}},\,j\right) \, \middle| \, i_t\in\left[p^{m_{t+1}-m_t-l_t}\right],\,j\in[q] \right\}.
        \end{aligned}$\vspace{0.2 cm}\\
        
        \item $\begin{aligned}[t]
            \so &= \left\{ \left(\sum_{t\in[n]\setminus\{r\}} i_tp^{m_t}+ \sum_{t=0}^{n-1}\alpha_{i_0,\dots,i_t}p^{m_t+l_t},\, i_r\right) \, \middle| \, i_t\in\left[p^{l_t}\right] \text{ for } t\ne r,\, i_r\in[q]\right\} \\
            C &= \left\{ \left(\sum_{t=0}^{n-1} i_tp^{m_t+l_t}+ \sum_{t\in[n]\setminus\{r-1\}}\beta_{i_0,\dots,i_t}p^{m_{t+1}},\, \beta_{i_0,\dots,i_{r-1}} \right) \, \middle| \, i_t\in\left[p^{m_{t+1}-m_t-l_t}\right]\right\}
        \end{aligned}$\vspace{0.2 cm}\\
        where $l_r=0$. \vspace{0.2 cm}

        \item $\begin{aligned}[t]
            \so &= \left\{ \left(\sum_{t=0}^{n-1} i_tp^{m_t} + \sum_{t=0}^{r-1}\alpha_{i_0,\dots,i_t}p^{m_t+l_t} + \sum_{t=r}^{n-1}\alpha_{i_0,\dots,i_t,j}\,p^{m_t+l_t},\, j\right) \, \middle| \, i_t\in\left[p^{l_t}\right],\,j\in[q] \right\} \\
            C &= \left\{ \left(\sum_{t=0}^{n-1} \left(i_tp^{m_t+l_t}+\beta_{i_0,\dots,i_t}p^{m_{t+1}}\right),\, \beta_{i_0,\dots,i_{r-1}} \right) \, \middle| \, i_t\in\left[p^{m_{t+1}-m_t-l_t}\right] \right\}.
        \end{aligned}$

    \end{enumerate}
\end{theorem}

\begin{theorem}\label{2^k.2}
    Let $G=\Z_{2^k}\x\Z_2$ for a positive integer $k$ and $S$ be a proper subset of $G\setminus \{0\}$ that generates $G$. Let $0\in C\subseteq G$. Then $C$ is a perfect code in $\Cay(G,S)$ if and only if there are integers $r\in[n]$, $0=m_0<m_1<\cdots<m_n=k$, and $1\leq l_t\leq m_{t+1}-m_t$ for $t\in[n]\setminus\{r\}$ and $0\leq l_r\leq m_{r+1}-m_r$, such that $(\so,C)$ is one of the following pairs where $\alpha_{i_0,\dots,i_t}$'s, $\beta_{i_0,\dots,i_t}$'s, $\alpha_{i_0,\dots,i_t,j}$'s, and $\beta_{i_0,\dots,i_t,j}$'s are in $\Z$ and $\alpha_{0,\dots,0}=\beta_{0,\dots,0}=0$.
    \begin{enumerate}[\rm(a)]
        \item $\begin{aligned}[t]
            \so &= \left\{ \left(\sum_{t=0}^{n-1} i_t2^{m_t}+ \sum_{t\in[n]\setminus\{r\}}\alpha_{i_0,\dots,i_t}2^{m_t+l_t},\, \alpha_{i_0,\dots,i_r}\right) \, \middle| \, i_t\in\left[2^{l_t}\right]\right\} \\
            C &= \left\{ \left(\sum_{t\in[n]\setminus\{r\}} i_t2^{m_t+l_t}+ \sum_{t=0}^{n-1}\beta_{i_0,\dots,i_t}2^{m_{t+1}},\, i_r\right) \, \middle| \, i_t\in\left[2^{m_{t+1}-m_t-l_t}\right] \text{ for } t\ne r,\ i_r\in[2]\right\}
        \end{aligned}$\vspace{0.2 cm}\\
        where $m_r+l_r=m_{r+1}$.\vspace{0.2 cm}
        
        \item $\begin{aligned}[t]
            \so &= \left\{\left(\sum_{t=0}^{n-1} i_t2^{m_t}+\sum_{t\in[n]\setminus\{r\}} \alpha_{i_0,\dots,i_t}2^{m_t+l_t}+\alpha_{i_0,\dots,i_r}2^{m_r+l_r-1},\, \alpha_{i_0,\dots,i_r}\right) \, \middle| \, i_t\in\left[2^{l_t}\right]\right\} \\
            C &= \left\{ \left(\sum_{t\in[n]\setminus\{r\}}i_t2^{m_t+l_t} +\sum_{t=0}^{n-1}\beta_{i_0,\dots,i_t}2^{m_{t+1}}+i_r2^{m_r+l_r-1},\,i_r\right) \, \middle| \, i_t\in\left[2^{m_{t+1}-m_t-l_t}\right] \text{ for } t\ne r,\right.\\
            & \qquad \left. i_{r}\in\left[2^{m_{r+1}-m_r-l_r+1}\right] \right\}
        \end{aligned}$\vspace{0.2 cm}\\
        
        \item $\begin{aligned}[t]
            \so &= \left\{\left(\sum_{t=0}^{n-1} \left( i_t2^{m_t}+\alpha_{i_0,\dots,i_t}2^{m_t+l_t}\right),\, \gamma_{i_0,\dots,i_r}\right) \, \middle| \, i_t\in\left[2^{l_t}\right]\right\} \\
            C &= \left\{ \left(\sum_{t=0}^{n-1} i_t2^{m_t+l_t} + \sum_{t=0}^{r-1} \beta_{i_0,\dots,i_t}2^{m_{t+1}} + \sum_{t=r}^{n-1} \beta_{i_0,\dots,i_t,j}\,2^{m_{t+1}},\,j\right) \, \middle| \, i_t\in\left[2^{m_{t+1}-m_t-l_t}\right],\,j\in[2] \right\}
        \end{aligned}$\vspace{0.2 cm}\\
        
        \item $\begin{aligned}[t]
            \so &= \left\{ \left(\sum_{t\in[n]\setminus\{r\}} i_t2^{m_t}+ \sum_{t=0}^{n-1}\alpha_{i_0,\dots,i_t}2^{m_t+l_t},\, i_r\right) \, \middle| \, i_t\in\left[2^{l_t}\right] \text{ for } t\ne r,\, i_r\in[2]\right\} \\
            C &= \left\{ \left(\sum_{t=0}^{n-1} i_t2^{m_t+l_t}+ \sum_{t\in[n]\setminus\{r-1\}}\beta_{i_0,\dots,i_t}2^{m_{t+1}},\, \beta_{i_0,\dots,i_{r-1}} \right) \, \middle| \, i_t\in\left[2^{m_{t+1}-m_t-l_t}\right]\right\}
        \end{aligned}$\vspace{0.2 cm}\\
        where $l_r=0$.\vspace{0.2 cm}
        
        \item $\begin{aligned}[t]
            \so &= \left\{ \left(\sum_{t\in[n]\setminus\{r\}}i_t2^{m_t}+\sum_{t=0}^{n-1}\alpha_{i_0,\dots,i_t}2^{m_t+l_t}+i_r2^{m_r-1},i_r\right) \, \middle| \, i_t\in\left[2^{l_{t}}\right] \text{ for } t\ne r,\, i_r\in\left[2^{l_r+1}\right]\right\} \\
            C &= \left\{ \left(\sum_{t=0}^{n-1}i_t2^{m_t+l_t}+\sum_{t\in[n]\setminus\{r-1\}}\beta_{i_0,\dots,i_t}2^{m_{t+1}}+\beta_{i_0,\dots,i_{r-1}}2^{m_r-1},\beta_{i_0,\dots,i_{r-1}}\right) \, \middle| \, i_t\in\left[2^{m_{t+1}-m_t-l_t}\right] \right\}
        \end{aligned}$\vspace{0.2 cm}
        
        \item $\begin{aligned}[t]
            \so &= \left\{ \left(\sum_{t=0}^{n-1} i_t2^{m_t} + \sum_{t=0}^{r-1}\alpha_{i_0,\dots,i_t}2^{m_t+l_t} + \sum_{t=r}^{n-1}\alpha_{i_0,\dots,i_t,j}\,2^{m_t+l_t},\, j\right) \, \middle| \, i_t\in\left[2^{l_t}\right],\,j\in[2] \right\} \\
            C &= \left\{ \left(\sum_{t=0}^{n-1} \left(i_t2^{m_t+l_t}+\beta_{i_0,\dots,i_t}2^{m_{t+1}}\right),\, \gamma_{i_0,\dots,i_{r-1}} \right) \, \middle| \, i_t\in\left[2^{m_{t+1}-m_t-l_t}\right] \right\}
        \end{aligned}$
    \end{enumerate}
\end{theorem}

The rest of this section is devoted to the proof of Theorems \ref{p^k}, \ref{p^k.q}, and \ref{2^k.2}

\subsection{Sufficiency of Theorems \ref{p^k}, \ref{p^k.q}, and \ref{2^k.2}}

\begin{lemma}\label{G=so+c}
    Let $(\so,C)$ be one of the pairs in Theorem \ref{p^k}, \ref{p^k.q} or \ref{2^k.2} and $G$ the corresponding group. Then $G=\so\oplus C$.
\end{lemma}

\begin{proof}
    First, note that \begin{align*}
        \prod_{t=0}^{n-1} p^{l_t} \cdot \prod_{t=0}^{n-1}p^{m_{t+1}-m_t-l_t}= p^{m_n-m_0}=p^k,
    \end{align*}
    and hence $|\so|\cdot|C|=|G|$ for every pair.
    Let $s_1,s_2\in \so$ and $c_1,c_2\in C$ such that $s_1+c_1=s_2+c_2$. We will prove $s_1=s_2$ and $c_1=c_2$.

\smallskip
    \textsf{Case 1.} $(\so,C)$ is the pair in Theorem \ref{p^k}. 
\smallskip    

    Let $s_1=\sum_{t=0}^{n-1}\left(i_tp^{m_t}+\alpha_{i_0,\dots,i_t}p^{m_t+l_t}\right)$,  $s_2=\sum_{t=0}^{n-1}\left(j_tp^{m_t}+\alpha_{j_0,\dots,j_t}p^{m_t+l_t}\right)$ for one $i_t,j_t\in[p^{l_t}]$, and $c_1=\sum_{t=0}^{n-1}\left(m_tp^{m_t+l_t}+\beta_{m_0,\dots,m_t}p^{m_{t+1}}\right)$, $c_2=\sum_{t=0}^{n-1}\left(n_tp^{m_t+l_t}+\beta_{n_0,\dots,n_t}p^{m_{t+1}}\right)$ for one $m_t,n_t\in[p^{m_{t+1}-m_t-l_t}]$.
    Since $s_1+c_1=s_2+c_2$, 
    \begin{align*}
        &\ i_0+\alpha_{i_0}p^{l_0}+m_0p^{l_0}+\beta_{m_0}p^{m_{1}}+\sum_{t=1}^{n-1}\left(i_tp^{m_t}+\alpha_{i_0,\dots,i_t}p^{m_t+l_t}\right)+\sum_{t=1}^{n-1}\left(m_tp^{m_t+l_t}+\beta_{m_0,\dots,m_t}p^{m_{t+1}}\right)\\
        =\ &j_0+\alpha_{j_0}p^{l_0}+n_0p^{l_0}+\beta_{n_0}p^{m_{1}}+\sum_{t=1}^{n-1}\left(j_tp^{m_t}+\alpha_{j_0,\dots,j_t}p^{m_t+l_t}\right)+\sum_{t=1}^{n-1}\left(n_tp^{m_t+l_t}+\beta_{n_0,\dots,n_t}p^{m_{t+1}}\right).
    \end{align*}
    Note that $0=m_0<m_0+l_0\leq m_1<m_1+l_1\leq m_2<m_2+l_2\leq\cdots\leq m_n=k$.
    By taking the last equation in modulo $p^{l_0}$, we have $i_0 \equiv j_0 \pmod{p^{l_0}}$. Since $i_0,j_0\in[p^{l_0}]$, $i_0=j_0$ in $\Z_{p^k}$.
    If $l_0<m_1$, take the equation in modulo $p^{m_1}$ to get $m_0p^{l_0}\equiv n_0p^{l_0} \pmod{p^{m_1}}$.
    Since $m_0,n_0\in[p^{m_1-l_0}]$, $m_0=n_0$ in $\Z_{p^k}$. If $l_0=m_1$, then $m_0,n_0\in[p^{m_1-l_0}]=\{0\}$. Therefore, we have $i_0=j_0$ and $m_0=n_0$. Using a similar argument, by taking the equation in modulo $p^{m_x+l_x}$ and $p^{m_{x+1}}$ for $x=0,\dots,t$ in order, we will get $i_t=j_t$ and $m_t=n_t$, and hence $s_1=s_2$ and $c_1=c_2$.

\smallskip
    \textsf{Case 2.} $(\so,C)$ is pair (a) in Theorem \ref{p^k.q}, or pair (a) or (b) in Theorem \ref{2^k.2}.
\smallskip
    
    We use the same argument in Case 1 to get that $i_t$ in $s_1$ and $s_2$ are the same for $t\in[r+1]$ and $i_t$ in $c_1$ and $c_2$ are the same for $t\in[r]$.
    By looking at the second component of $s_1+c_1=s_2+c_2$, we get $i_r$ in $c_1$ and $c_2$ are equal.
    Since we already have $i_t$ in $s_1$ and $s_2$ and $i_t$ in $c_1$ and $c_2$ are the same for $t\in[r+1]$, the remaining equation in $s_1+c_1=s_2+c_2$ is the same as in Case 1, and this leads to $s_1=s_2$ and $c_1=c_2$.

\smallskip
    \textsf{Case 3.} $(\so,C)$ is pair (c) in Theorem \ref{p^k.q}, or pair (d) or (e) in Theorem \ref{2^k.2}.
\smallskip
    
    Similarly to Case 2, but $i_t$ in $s_1$ and $s_2$ and $i_t$ in $c_1$ and $c_2$ are the same for $t\in[r]$, and from the second component of $s_1+c_1=s_2+c_2$, we get $i_r$ in $s_1$ and $s_2$ are equal. The remaining arguments follow.

\smallskip
    \textsf{Case 4.} $(\so,C)$ is pair (b) or (d) in Theorem \ref{p^k.q}, or pair (c) or (f) in Theorem \ref{2^k.2}.
\smallskip
    
    Similarly to Cases 2 and 3, but from the second component of $s_1+c_1=s_2+c_2$, we get $j$ in $s_1$ and $s_2$ or $j$ in $c_1$ and $c_2$ are equal. The remaining arguments follow.
\end{proof}

\subsection{Proof of Theorem \ref{p^k}}

\begin{proof}[\textbf{Proof of Theorem \ref{p^k}}]
    Suppose $\so$ and $C$ are as in (\ref{S0form}) and (\ref{Cform}), respectively. Then $G=\so\oplus C$ by Lemma \ref{G=so+c}.
    Suppose $\Cay(G,S)$ has a perfect code $C$. Then $G=\Z_{p^k}=\so\oplus C$ and $|\so|=p^r$. We will prove (\ref{S0form}) and (\ref{Cform}) using induction on $r$. First, for $|\so|=p$, by Lemma \ref{prime deg}, $\so$ is aperiodic and $C=p\Z_{p^k}=\{ip+\beta_ip^k\mid i \in[p^{k-1}]\}$.
    In this case, $s_1\not\equiv s_2\pmod{p}$ for $s_1,s_2\in \so$, which means we can write $\so$ as $\left\{ i+\alpha_i p \, \middle| \, i\in[p] \right\}$ as desired.

    Now, suppose (\ref{S0form}) and (\ref{Cform}) are true for $|\so|<p^r$.

\smallskip
    \textsf{Case 1.} $\so$ is periodic. 
\smallskip
    
    Let $\so=L_{\so}\oplus D$ and $|L_{\so}|=p^s$. Then $0<s<r$ because $\so$ generates $G$.
    From Lemma \ref{quotient}, $G/{L_{\so}}=((D+L_{\so})/L_{\so})\oplus ((C+L_{\so})/L_{\so})$ and $|(D+L_{\so})/L_{\so}|=|D|<|\so|$.
    We also have $G/L_{\so}\cong \Z_{p^{k-s}}$.
    From the induction hypothesis, there are integers $0=m_0<m_1<\cdots<m_n=\log_p(|G/L_{\so}|)=k-s$
    and $1\leq l_t\leq m_{t+1}-m_t$ for $t\in[n]$ such that
        \begin{align*}
            (D+L_{\so})/L_{\so} &= \left\{ \sum_{t=0}^{n-1} \left(i_tp^{m_t}+\alpha_{i_0,\dots,i_t}p^{m_t+l_t}\right) +L_{\so} \, \middle| \, i_t\in\left[p^{l_t}\right]\right\},\ \text{and}\\
            (C+L_{\so})/L_{\so} &= \left\{ \sum_{t=0}^{n-1} \left( i_tp^{m_t+l_t}+\beta_{i_0,\dots,i_t}p^{m_{t+1}}\right) +L_{\so} \, \middle| \, i_t\in\left[p^{m_{t+1}-m_t-l_t}\right] \right\},
        \end{align*}
        where $\alpha_{i_0,\dots,i_t},\beta_{i_0,\dots,i_t}\in \Z_{p^k}$.
    Now $L_{\so}=p^{k-s}\Z_{p^k}=\{hp^{k-s}\mid h\in[p^s]\}$.
    Therefore, 
    \begin{align*}
        \so &= L_{\so}\oplus D\\
        &= \left\{ \sum_{t=0}^{n-1} \left(i_tp^{m_t}+\alpha_{i_0,\dots,i_t}p^{m_t+l_t}\right) + hp^{k-s} \, \middle| \, i_t\in \left[p^{l_t}\right], h\in[p^s] \right\} \\
        &= \left\{ \sum_{t=0}^{n-1} \left(i_tp^{m_t}+\alpha_{i_0,\dots,i_t}p^{m_t+l_t}\right) + hp^{k-s} + \alpha_{i_0,\dots,i_{n-1},h}\,p^{k} \, \middle| \, i_t\in\left[p^{l_t}\right], h\in[p^s] \right\} \\
        &= \left\{ \sum_{t=0}^{n} \left(i_tp^{m_t}+\alpha_{i_0,\dots,i_t}p^{m_t+l_t}\right) \, \middle| \, i_t\in\left[p^{l_t}\right]\right\},
    \end{align*}
    where $m_{n}=k-s$, $l_{n}=s$, and $i_{n}=h$. Moreover, $m_{n+1}=k>m_{n}$ and $1\leq l_{n}=m_{n+1}-m_{n}$.
    Since $G/L_{\so}\cong \Z_{p^{k-s}}$, $p^{k-s}+L_{\so}=L_{\so}$, so we have
    \begin{align*}
        & (C+L_{\so})/L_{\so} \\
        = & \left\{ \sum_{t=0}^{n-2} \left( i_tp^{m_t+l_t}+\beta_{i_0,\dots,i_t}p^{m_{t+1}}\right)+i_{n-1}p^{m_{n-1}+l_{n-1}}+\beta_{i_0,\dots,i_{n-1}}p^{m_{n}}+L_{\so} \, \middle| \, i_t\in\left[p^{m_{t+1}-m_t-l_t}\right] \right\}\\
        = & \left\{ \sum_{t=0}^{n-2} \left( i_tp^{m_t+l_t}+\beta_{i_0,\dots,i_t}p^{m_{t+1}}\right) +i_{n-1}p^{m_{n-1}+l_{n-1}}+L_{\so} \, \middle| \, i_t\in\left[p^{m_{t+1}-m_t-l_t}\right] \right\}.
    \end{align*}
    Therefore,
    \begin{align*}
        C &= \left\{ \sum_{t=0}^{n-2} \left( i_tp^{m_t+l_t}+\beta_{i_0,\dots,i_t}p^{m_{t+1}}\right) +i_{n-1}p^{m_{n-1}+l_{n-1}}+\beta_{i_0,\dots,i_{n-1}}p^{k-s} \, \middle| \, i_t\in\left[p^{m_{t+1}-m_t-l_t}\right] \right\}\\
        &= \left\{ \sum_{t=0}^{n-1} \left( i_tp^{m_t+l_t}+\beta_{i_0,\dots,i_t}p^{m_{t+1}}\right) \, \middle| \, i_t\in\left[p^{m_{t+1}-m_t-l_t}\right] \right\}.
    \end{align*}
    Note that in $C$ we can also set the sum from $t=0$ to $n$ because $m_n+l_n=m_{n+1}=k$.

\smallskip
    \textsf{Case 2.} $\so$ is aperiodic.
\smallskip
    
    In this case $C$ is periodic, so let $C=L_C\oplus D$.
    Consider the direct sum $G/{L_C}=((\so+L_C)/L_C)\oplus ((D+L_C)/L_C)$.
    If $|L_C|=|C|$, then $L_C=C=p^r\Z_{p^k}=\{ip^r \mid i\in[p^{k-r}]\}$.
    Since $C=L_C$, we can take $D=\{0\}$ and we have $G/{L_C}=(\so+L_C)/L_C$. This means $(\so+L_C)/L_C=\{i+L_C \mid i\in[p^r]\}$ and hence $\so=\{i+\alpha_ip^r \mid i\in[p^r]\}$.
    
    If $|L_C|<|C|$, then by Lemma \ref{quotient}, $(D+L_C)/L_C$ is aperiodic.
    Therefore $(\so+L_C)/L_C$ is periodic and $|(\so+L_C)/L_C|=|\so|$. Let $|L_C|=p^s$.
    From Case 1, there are integers $0=m_0<m_1<\cdots<m_{n+1}=k-s$ and $1\leq l_t\leq m_{t+1}-m_t$ for $t\in[n+1]$ such that
    \begin{align*}
        (\so+L_C)/L_C
        &= \left\{ \sum_{t=0}^{n} \left(i_tp^{m_t}+\alpha_{i_0,\dots,i_t}p^{m_t+l_t}\right) + L_C \, \middle| \, i_t\in\left[p^{l_t}\right]\right\},\ \text{and}\\
        (D+L_C)/L_C &= \left\{ \sum_{t=0}^{n-1} \left( i_tp^{m_t+l_t}+\beta_{i_0,\dots,i_t}p^{m_{t+1}}\right) +L_C \, \middle| \, i_t\in\left[p^{m_{t+1}-m_t-l_t}\right] \right\}.
    \end{align*}
    From the previous case, we also have $m_{n}+l_{n}=m_{n+1}=k-s$. Since $|G/L_C|=p^{k-s}$, $p^{k-s}+L_C=L_C$, so we have
    \begin{align*}
        (\so+L_C)/L_C
        &= \left\{ \sum_{t=0}^{n-1} \left(i_tp^{m_t}+\alpha_{i_0,\dots,i_t}p^{m_t+l_t}\right) + i_{n}p^{m_{n}} + L_C \, \middle| \, i_t\in\left[p^{l_t}\right] \right\}.
    \end{align*}
    Since $|L_C|=p^s$, $L_C=p^{k-s}\Z_{p^k}=\{hp^{k-s}\mid h\in\left[p^s\right]\}$. Therefore,
    \begin{align*}
        \so
        &= \left\{ \sum_{t=0}^{n-1} \left(i_tp^{m_t}+\alpha_{i_0,\dots,i_t}p^{m_t+l_t}\right) + i_{n}p^{m_{n}} + \alpha_{i_0,\dots,i_{n}}p^{k-s} \, \middle| \, i_t\in\left[p^{l_t}\right] \right\}\\
        &= \left\{ \sum_{t=0}^{n} \left(i_tp^{m_t}+\alpha_{i_0,\dots,i_t}p^{m_t+l_t}\right) \, \middle| \, i_t\in\left[p^{l_t}\right] \right\}\ \text{and}
    \end{align*}
    \begin{align*}
        C &= D\oplus L_C \\
        &= \left\{ \sum_{t=0}^{n-1} \left( i_tp^{m_t+l_t}+\beta_{i_0,\dots,i_t}p^{m_{t+1}}\right) + hp^{k-s} \, \middle| \, i_t\in\left[p^{m_{t+1}-m_t-l_t}\right], h\in\left[p^s\right] \right\}\\
        &= \left\{ \sum_{t=0}^{n-1} \left( i_tp^{m_t+l_t}+\beta_{i_0,\dots,i_t}p^{m_{t+1}}\right) + hp^{k-s} + \beta_{i_0,\dots,i_{n-1},h}p^{k} \, \middle| \, i_t\in\left[p^{m_{t+1}-m_t-l_t}\right], h\in\left[p^s\right] \right\}.
    \end{align*}
    By setting $i_n=h$ and a new value for $m_{n+1}$ to be equal to $k$, we have
    \begin{align*}
        C &= \left\{ \sum_{t=0}^{n} \left( i_tp^{m_t+l_t}+\beta_{i_0,\dots,i_t}p^{m_{t+1}}\right) \, \middle| \, i_t\in\left[p^{m_{t+1}-m_t-l_t}\right] \right\}.
    \qedhere
    \end{align*}

\end{proof}

\subsection{Proof of Theorem \ref{p^k.q}}

\begin{proof}[\textbf{Proof of Theorem \ref{p^k.q}}]

    Suppose $(\so,C)$ is one of the pairs in Theorem \ref{p^k.q}.
    Then $G=\so\oplus C$ by Lemma \ref{G=so+c}.    
    Suppose $\Cay(G,S)$ has a perfect code $C$.
    Then $G=\Z_{p^k}\x\Z_q=\so\oplus C$.
    We may write $|\so|=p^eq^f$ as $|\so|$ divides $|G|=p^kq$, where $f=0$ or $1$ and $0<e+f<k+1$.
    We will prove $(\so,C)$ is one of the four pairs in Theorem \ref{p^k.q} using induction on $e+f$.

    First, consider the case $|\so|=p$. From Lemma \ref{prime deg}, $\so$ is aperiodic and $C=p\Z_{p^k}\times\Z_q=\{\left(ip,\,j\right) \mid i \in[p^{k-1}],\,j\in[q]\}$.
    Write $C=L_C\oplus D$ where $D=\{0\}$, by Lemma \ref{quotient}, we have a direct sum $G/{L_C}=((\so+L_C)/L_C)\oplus ((D+L_C)/L_C)$.
    Since $D=\{0\}$, we have $G/{L_C}=(\so+L_C)/L_C$.
    This means $(\so+L_C)/L_C=\{(i,0)+L_C \mid i\in[p]\}$ and hence $\so=\{(i,0)+\alpha_i(p,1) \mid i\in[p]\}=\{(i+\alpha_ip,\alpha_i) \mid i\in[p]\}$, where $\alpha_i\in\Z$.
    Note that the pair $(\so,C)$ here is pair (b) in Theorem \ref{p^k.q} where $n=1$ and $r=0$.

    Now, for $|\so|=q$. From Lemma \ref{prime deg}, $\so$ is aperiodic and $C=\Z_{p^k}\times\{0\}=\{\left(i,\,0\right) \mid i \in[p^k]\}$. Using the same arguments as in the case of $|\so|=p$, $(\so+L_C)/L_C=G/L_C=\{(0,i)+L_C \mid i\in[q]\}$ and hence $\so=\{(0,i)+\alpha_i(1,0) \mid i\in[q]\}=\{(\alpha_i,i) \mid i\in[q]\}$, where $\alpha_i\in\Z$. Note that the pair $(\so,C)$ here is pair (c) in Theorem \ref{p^k.q} where $n=1$ and $r=0$.
    
    Suppose Theorem \ref{p^k.q} is true for $e+f<m$.
    We will prove that $(\so,C)$ is one of the four pairs in Theorem \ref{p^k.q} for $|\so|=p^m$ and $|\so|=p^{m-1}q$.

    \smallskip
    \textsf{Case 1.} $|\so|=p^m$ and $\so$ is periodic.
\smallskip
    
    Let $\so=L_{\so}\oplus D$ and $|L_{\so}|=p^s$, then $0<s<m$ because $\so$ generates $G$. 
    From Lemma \ref{quotient}, $G/{L_{\so}}=((D+L_{\so})/L_{\so})\oplus ((C+L_{\so})/L_{\so})$. Since $G/{L_{\so}}\cong\Z_{p^{k-s}}\times\Z_q$ and $|(D+L_{\so})/L_{\so}|=|D|=p^{m-t}$, from the induction hypothesis, there are integers $0=m_0<m_1<\cdots<m_n=k-s$
    and $1\leq l_t\leq m_{t+1}-m_t$ for $t\in[n]$, such that $(D+L_{\so})/L_{\so}$ and $(C+L_{\so})/L_{\so}$ are one of the four pairs in Theorem \ref{p^k.q}. Since $|D|=p^{r-t}$, they can only be pair (a) or (b).

    If $(D+L_{\so})/L_{\so}$ and $(C+L_{\so})/L_{\so}$ are pair (b), then
    \begin{align*}
        (D+L_{\so})/L_{\so} = & \left\{\left(\sum_{t=0}^{n-1} \left( i_tp^{m_t}+\alpha_{i_0,\dots,i_t}p^{m_t+l_t}\right),\, \alpha_{i_0,\dots,i_r}\right) +L_{\so} \, \middle| \, i_t\in\left[p^{l_t}\right]\right\},\ \text{and}\\
        (C+L_{\so})/L_{\so} = & \left\{ \left(\sum_{t=0}^{n-1} i_tp^{m_t+l_t} + \sum_{t=0}^{r-1} \beta_{i_0,\dots,i_t}p^{m_{t+1}} + \sum_{t=r}^{n-1} \beta_{i_0,\dots,i_t,j}\,p^{m_{t+1}},\,j\right) +L_{\so}\, \middle| \, \right.\\ 
        & \quad \left.  i_t\in\left[p^{m_{t+1}-m_t-l_t}\right],\,j\in[q] \right\}.
    \end{align*}
    Now, $L_{\so}=p^{k-s}\Z_{p^k}\times\{0\}=\left\{\left(hp^{k-s},\,0\right)\, \middle| \, h\in[p^s]\right\}$. 
    Therefore,
    \begin{align*}
        \so &= L_{\so}\oplus D\\
        &= \left\{ \left(\sum_{t=0}^{n-1} \left(i_tp^{m_t}+\alpha_{i_0,\dots,i_t}p^{m_t+l_t}\right),\, \alpha_{i_0,\dots,i_r}\right) + \left(hp^{k-s},\,0\right) \, \middle| \, i_t\in \left[p^{l_t}\right], h\in[p^s] \right\} \\
        &= \left\{ \left(\sum_{t=0}^{n-1} \left(i_tp^{m_t}+\alpha_{i_0,\dots,i_t}p^{m_t+l_t}\right),\, \alpha_{i_0,\dots,i_r}\right) + \left(hp^{k-s}+ \alpha_{i_0,\dots,i_{n-1},h}\,p^{k},\,0\right) \, \middle| \, i_t\in\left[p^{l_t}\right], h\in[p^s] \right\} \\
        &= \left\{ \left(\sum_{t=0}^{n} \left(i_tp^{m_t}+\alpha_{i_0,\dots,i_t}p^{m_t+l_t}\right),\, \alpha_{i_0,\dots,i_r}\right) \, \middle| \, i_t\in\left[p^{l_t}\right], h\in[p^s]\right\}
    \end{align*}
    where $m_{n}=k-s$, $l_{n}=s$, and $i_{n}=h$. Moreover, $m_{n+1}=k>m_{n}$ and $1\leq l_{n}=m_{n+1}-m_{n}$.
    Since $G/L_{\so}\cong \Z_{p^{k-s}}\times\Z_q$, $(p^{k-s},0)+L_{\so}=L_{\so}$, so we have
    \begin{align*}
        (C+L_{\so})/L_{\so}
        = &\left\{ \left(\sum_{t=0}^{n-2} i_tp^{m_t+l_t} + \sum_{t=0}^{r-1} \beta_{i_0,\dots,i_t}p^{m_{t+1}} + \sum_{t=r}^{n-2} \beta_{i_0,\dots,i_t,j}\,p^{m_{t+1}},\,j\right) \right.\\
        & +\left. \left(i_{n-1}p^{m_{n-1}+l_{n-1}},0\right)+L_{\so} \middle| \, i_t\in\left[p^{m_{t+1}-m_t-l_t}\right],\ j\in[q] \right\}.
    \end{align*}
    Therefore,
    \begin{align*}
        C = & \left\{ \left(\sum_{t=0}^{n-2} i_tp^{m_t+l_t} + \sum_{t=0}^{r-1} \beta_{i_0,\dots,i_t}p^{m_{t+1}} + \sum_{t=r}^{n-2} \beta_{i_0,\dots,i_t,j}\,p^{m_{t+1}},\,j\right) + \left(i_{n-1}p^{m_{n-1}+l_{n-1}}+ \beta_{i_0,\dots,i_{n-1},j}p^{k-s},0\right) \right.\\
        & \quad \left.\, \middle| \, i_t\in\left[p^{m_{t+1}-m_t-l_t}\right], j\in[q] \right\}\\
        = & \left\{ \left(\sum_{t=0}^{n-1} i_tp^{m_t+l_t} + \sum_{t=0}^{r-1} \beta_{i_0,\dots,i_t}p^{m_{t+1}} + \sum_{t=r}^{n-1} \beta_{i_0,\dots,i_t,j}\,p^{m_{t+1}},\,j\right) \, \middle| \, i_t\in\left[p^{m_{t+1}-m_t-l_t}\right], j\in[q] \right\}.
    \end{align*}
    Note that in $C$ we can also set the sum from $t=0$ to $n$ because $m_n+l_n=m_{n+1}=k$. The pair $(\so,C)$ here is also in the form of pair (b).

    If $(D+L_{\so})/L_{\so}$ and $(C+L_{\so})/L_{\so}$ are pair (a), the same algebraic operation we did for pair (b) is still applicable and we will get $(\so,C)$ to be another pair in the form of pair (a).

\smallskip
    \textsf{Case 2.} $|\so|=p^m$ and $\so$ is aperiodic.
\smallskip
    
    In this case, $C$ is periodic and $|C|=p^{k-m}q$. Let $C=L_C\oplus D$. Then $|L_C|=q$, $p^s$, or $p^sq$, for $0<s\leq k-m$. Consider the factorization given in Lemma \ref{quotient}, we have $G/{L_C}=((\so+L_C)/L_C)\oplus ((D+L_C)/L_C)$.

    \smallskip
    \textsf{Subcase 2.1.} $|L_C|=q$.
    In this case, $G/{L_C}\cong \Z_{p^k}$ and $(\so+L_C)/L_C$ is periodic. 
    From the case when $\so$ is periodic in Theorem \ref{p^k}, we have
    \begin{align*}
        (\so+L_C)/L_C
        &= \left\{ \left(\sum_{t=0}^{n} \left(i_tp^{m_t}+\alpha_{i_0,\dots,i_t}p^{m_t+l_t}\right),0\right) + L_C \, \middle| \, i_t\in\left[p^{l_t}\right]\right\},\ \text{and}\\
        (D+L_C)/L_C &= \left\{ \left(\sum_{t=0}^{n-1} \left( i_tp^{m_t+l_t}+\beta_{i_0,\dots,i_t}p^{m_{t+1}}\right),0 \right) +L_C \, \middle| \, i_t\in\left[p^{m_{t+1}-m_t-l_t}\right] \right\}.
    \end{align*}
    From the proof of Theorem \ref{p^k}, we also have $m_{n}+l_{n}=m_{n+1}=k$. Since $(p^{k},0)+L_C=L_C$,
    \begin{align*}
        (\so+L_C)/L_C
        &= \left\{ \left(\sum_{t=0}^{n-1} \left(i_tp^{m_t}+\alpha_{i_0,\dots,i_t}p^{m_t+l_t}\right),0\right) + \left(i_{n}p^{m_{n}},0\right) + L_C \, \middle| \, i_t\in\left[p^{l_t}\right] \right\}.
    \end{align*}
    Since $|L_C|=q$, $L_C=\{(0,j)\mid j\in[q]\}=\<{(0,1)}$.
    Therefore,
    \begin{align*}
        \so
        &= \left\{ \left(\sum_{t=0}^{n-1} \left(i_tp^{m_t}+\alpha_{i_0,\dots,i_t}p^{m_t+l_t}\right),0\right) + \left(i_{n}p^{m_{n}},0\right) + \alpha_{i_0,\dots,i_{n}}\left(0,1\right) \, \middle| \, i_t\in\left[p^{l_t}\right] \right\}\\
        &= \left\{ \left(\sum_{t=0}^{n} i_tp^{m_t}+\sum_{t=0}^{n-1}\alpha_{i_0,\dots,i_t}p^{m_t+l_t},\alpha_{i_0,\dots,i_{n}}\right) \, \middle| \, i_t\in\left[p^{l_t}\right] \right\}.
    \end{align*}
    By setting $r=n$, we obtain
    \begin{align*}
        \so 
        &= \left\{ \left(\sum_{t=0}^{n} i_tp^{m_t}+\sum_{t\in[n+1]\setminus \{r\}}\alpha_{i_0,\dots,i_t}p^{m_t+l_t},\alpha_{i_0,\dots,i_{r}}\right) \, \middle| \, i_t\in\left[p^{l_t}\right] \right\}.
    \end{align*}
    Now, we will prove that $C$ is also in the form of $C$ in pair (a).
    We have
    \begin{align*}
        C &= D\oplus L_C \\
        &= \left\{ \left(\sum_{t=0}^{n-1} \left( i_tp^{m_t+l_t}+\beta_{i_0,\dots,i_t}p^{m_{t+1}}\right),0\right) + (0,j) \, \middle| \, i_t\in\left[p^{m_{t+1}-m_t-l_t}\right], j\in\left[q\right] \right\}\\
        &= \left\{ \left(\sum_{t=0}^{n-1} \left( i_tp^{m_t+l_t}+\beta_{i_0,\dots,i_t}p^{m_{t+1}}\right),0\right) + \left(\beta_{i_0,\dots,i_{n-1},j}p^{k},j\right) \, \middle| \, i_t\in\left[p^{m_{t+1}-m_t-l_t}\right], j\in\left[q\right] \right\}.
    \end{align*}
    Recall that we set $r=n$. Let $j=i_r$ and set a new value for $m_{n+1}$ to be equal to $k$.
    Then
    \begin{align*}
        C &= \left\{ \left(\sum_{t\in[n+1]\setminus\{r\}} i_tp^{m_t+l_t}+\sum_{t=0}^{n}\beta_{i_0,\dots,i_t}p^{m_{t+1}},i_r\right) \, \middle| \, i_t\in\left[p^{m_{t+1}-m_t-l_t}\right], \text{ for } t\ne r,\ i_r\in[q] \right\}.
    \end{align*}
    The pair $(\so,C)$ here is in the form of pair (a).

    \smallskip
    \textsf{Subcase 2.2.} $|L_C|=p^{s}$.
    In this case, $G/{L_C}\cong \Z_{p^{k-s}}\times\Z_q$, $|(\so+L_C)/L_C|=|\so|=p^m$, and $(\so+L_C)/L_C$ is periodic. 
    From Case 1, $(\so+L_C)/L_C$ is either in the form of pair (a) or (b) in Theorem \ref{p^k.q}. We calculate $(\so+L_C)/L_C$ in the form of pair (a). The calculation of $(\so+L_C)/L_C$ in the form of pair (b) will follow similarly.

    From Case 1, there are integers $0=m_0<m_1<\cdots<m_{n+1}=k-s$
    and $1\leq l_t\leq m_{t+1}-m_t$ for $t\in[n]$, such that
    \begin{align*}
        (\so+L_C)/L_C
        = & \left\{ \left(\sum_{t=0}^{n} i_tp^{m_t}+ \sum_{t\in[n+1]\setminus\{r\}}\alpha_{i_0,\dots,i_t}p^{m_t+l_t},\, \alpha_{i_0,\dots,i_r}\right) + L_C \, \middle| \, i_t\in\left[p^{l_t}\right]\right\},\ \text{and}\\
        (D+L_C)/L_C
        = & \left\{ \left(\sum_{t\in[n]\setminus\{r\}} i_tp^{m_t+l_t}+ \sum_{t=0}^{n-1}\beta_{i_0,\dots,i_t}p^{m_{t+1}},\, i_r\right) +L_C \, \middle| \, i_t\in\left[p^{m_{t+1}-m_t-l_t}\right] \right. \\
        & \quad \left. \text{ for } t\ne r,\ i_r\in[q] \right\}.
    \end{align*}
    We also have $m_{n}+l_{n}=m_{n+1}=k-s$ from Case 1.
    Since $G/{L_C}\cong \Z_{p^{k-s}}\times\Z_q$, we have $(p^{k-s},0)+L_C=L_C$ and
    \begin{align*}
        (\so+L_C)/L_C
        &= \left\{ \left(\sum_{t=0}^{n-1} i_tp^{m_t}+ \sum_{t\in[n]\setminus\{r\}}\alpha_{i_0,\dots,i_t}p^{m_t+l_t},\, \alpha_{i_0,\dots,i_r}\right) + \left(i_{n}p^{m_{n}},0\right) + L_C \, \middle| \, i_t\in\left[p^{l_t}\right] \right\}.
    \end{align*}
    Since $|L_C|=p^s$, we have $L_C=p^{k-s}\Z_{p^k}=\left\{\left(hp^{k-s},0\right)\, \middle| \, h\in\left[p^s\right]\right\}$.
    Therefore,
    \begin{align*}
        \so
        &= \left\{ \left(\sum_{t=0}^{n-1} i_tp^{m_t}+ \sum_{t\in[n]\setminus\{r\}}\alpha_{i_0,\dots,i_t}p^{m_t+l_t},\, \alpha_{i_0,\dots,i_r}\right) + \left(i_{n}p^{m_{n}} + \alpha_{i_0,\dots,i_{n}}p^{k-s},0\right) \, \middle| \, i_t\in\left[p^{l_t}\right] \right\}\\
        &= \left\{ \left(\sum_{t=0}^{n} i_tp^{m_t}+ \sum_{t\in[n+1]\setminus\{r\}}\alpha_{i_0,\dots,i_t}p^{m_t+l_t},\, \alpha_{i_0,\dots,i_r}\right) \, \middle| \, i_t\in\left[p^{l_t}\right] \right\},\ \text{and}
    \end{align*}
    \begin{align*}
        C = & D\oplus L_C \\
        = & \left\{ \left(\sum_{t\in[n]\setminus\{r\}} i_tp^{m_t+l_t}+ \sum_{t=0}^{n-1}\beta_{i_0,\dots,i_t}p^{m_{t+1}},\, i_r\right) + \left(hp^{k-s},0\right) \, \middle| \, i_t\in\left[p^{m_{t+1}-m_t-l_t}\right] \text{ for } t\ne r,\right.\\
        & \quad \left. i_r\in[q], h\in\left[p^s\right] \right\}\\
        = & \left\{ \left(\sum_{t\in[n]\setminus\{r\}} i_tp^{m_t+l_t}+ \sum_{t=0}^{n-1}\beta_{i_0,\dots,i_t}p^{m_{t+1}},\, i_r\right) + \left(hp^{k-s} + \beta_{i_0,\dots,i_{n-1},h}p^{k},0\right) \, \middle| \, i_t\in\left[p^{m_{t+1}-m_t-l_t}\right]  \right.\\
        & \quad \left. \text{ for } t\ne r,\ i_r\in[q], h\in\left[p^s\right] \right\}.
    \end{align*}
    By setting $i_n=h$ and a new value for $m_{n+1}$ to be equal to $k$, we have
    \begin{align*}
        C &= \left\{ \left(\sum_{t\in[n+1]\setminus\{r\}} i_tp^{m_t+l_t}+ \sum_{t=0}^{n}\beta_{i_0,\dots,i_t}p^{m_{t+1}},\, i_r\right) \, \middle| \, i_t\in\left[p^{m_{t+1}-m_t-l_t}\right] \text{ for } t\ne r,\ i_r\in[q] \right\}.
    \end{align*}
    The pair $(\so,C)$ here is in the form of pair (a).

\smallskip
    \textsf{Subcase 2.3.} $|L_C|=p^{k-m}q$.
    In this case, $C=L_C=\{\left(ip^{m},j\right) \mid i\in\left[p^{k-m}\right], j\in[q]\}=\left<(p^m,1)\right>$ and $G/{L_C}\cong \Z_{p^{m}}$.
    Since $C=L_C$, we can take $D=\{0\}$ and we have $G/{L_C}=(\so+L_C)/L_C$.
    This means $(\so+L_C)/L_C=\{(i,0)+L_C \mid i\in[p^{m}]\}$ and hence $\so=\{(i,0)+\alpha_i(p^{m},1) \mid i\in[p^{k-m}]\}=\{(i+\alpha_ip^{m},\alpha_i) \mid i\in[p^{k-m}]\}$. The pair $(\so,C)$ here is in the form of pair (b).

\smallskip
    \textsf{Subcase 2.4.} $|L_C|=p^{s}q$ for $t<k-m$.
    In this case, $G/{L_C}\cong \Z_{p^{k-s}}$, $|(\so+L_C)/L_C|=|\so|=p^m$, and $(\so+L_C)/L_C$ is periodic. 
    From the case when $\so$ is periodic in Theorem \ref{p^k}, we have
    \begin{align*}
        (\so+L_C)/L_C
        &= \left\{ \left(\sum_{t=0}^{n} \left(i_tp^{m_t}+\alpha_{i_0,\dots,i_t}p^{m_t+l_t}\right),0\right) + L_C \, \middle| \, i_t\in\left[p^{l_t}\right]\right\},\ \text{and}\\
        (D+L_C)/L_C &= \left\{ \left(\sum_{t=0}^{n-1} \left( i_tp^{m_t+l_t}+\beta_{i_0,\dots,i_t}p^{m_{t+1}}\right),0 \right) +L_C \, \middle| \, i_t\in\left[p^{m_{t+1}-m_t-l_t}\right] \right\}.
    \end{align*}
    We also have $m_{n}+l_{n}=m_{n+1}=k-s$ from the same case.
    Since $(p^{k-s},0)+L_C=L_C$,
    \begin{align*}
        (\so+L_C)/L_C
        &= \left\{ \left(\sum_{t=0}^{n-1} \left(i_tp^{m_t}+\alpha_{i_0,\dots,i_t}p^{m_t+l_t}\right),0\right) + \left(i_{n}p^{m_{n}},0\right) + L_C \, \middle| \, i_t\in\left[p^{l_t}\right] \right\}.
    \end{align*}
    Since $|L_C|=p^sq$, $L_C=p^{k-s}\Z_{p^k}\times\Z_q=\{(hp^{k-s},j)\mid h\in\left[p^s\right], j\in[q]\}=\left<(p^{k-s},1)\right>$. Therefore,
    \begin{align*}
        \so
        &= \left\{ \left(\sum_{t=0}^{n-1} \left(i_tp^{m_t}+\alpha_{i_0,\dots,i_t}p^{m_t+l_t}\right),0\right) + \left(i_{n}p^{m_{n}},0\right) + \alpha_{i_0,\dots,i_{n}}\left(p^{k-s},1\right) \, \middle| \, i_t\in\left[p^{l_t}\right] \right\}\\
        &= \left\{ \left(\sum_{t=0}^{n} \left(i_tp^{m_t}+\alpha_{i_0,\dots,i_t}p^{m_t+l_t}\right),\alpha_{i_0,\dots,i_{n}}\right) \, \middle| \, i_t\in\left[p^{l_t}\right] \right\}\ \text{and}
    \end{align*}
    \begin{align*}
        C &= D\oplus L_C \\
        &= \left\{ \left(\sum_{t=0}^{n-1} \left( i_tp^{m_t+l_t}+\beta_{i_0,\dots,i_t}p^{m_{t+1}}\right),0\right) + (hp^{k-s},j) \, \middle| \, i_t\in\left[p^{m_{t+1}-m_t-l_t}\right], h\in\left[p^s\right], j\in\left[q\right] \right\}\\
        &= \left\{ \left(\sum_{t=0}^{n-1} \left( i_tp^{m_t+l_t}+\beta_{i_0,\dots,i_t}p^{m_{t+1}}\right),0\right) + \left(hp^{k-s} + \beta_{i_0,\dots,i_{n-1},h,j}p^{k},j\right) \, \middle| \, i_t\in\left[p^{m_{t+1}-m_t-l_t}\right],\right.\\
        & \quad \left. h\in\left[p^s\right], j\in\left[q\right] \right\}.
    \end{align*}
    By setting $i_n=h$, $r=n$, and a new value for $m_{n+1}$ to be equal to $k$, we have
    \begin{align*}
        C &= \left\{ \left(\sum_{t=0}^{n} i_tp^{m_t+l_t}+\sum_{t=0}^{n-1} \beta_{i_0,\dots,i_t}p^{m_{t+1}}+\sum_{t=r}^{n}\beta_{i_0,\dots,i_t,j}p^{m_{t+1}} ,j\right) \, \middle| \, i_t\in\left[p^{m_{t+1}-m_t-l_t}\right], j\in\left[q\right] \right\}.
    \end{align*}
    The pair $(\so,C)$ here is in the form of pair (b).

\smallskip
    \textsf{Case 3.} $|\so|=p^{m-1}q$ ($m>1$) and $\so$ is periodic.
\smallskip
    
    Let $\so=L_{\so}\oplus D$.
    Then $|L_{\so}|=q$, $p^s$ ($0<s\leq m-1$), or $p^sq$ ($0<s<m-1$). Consider the factorization $G/{L_{\so}}=((D+L_{\so})/L_{\so})\oplus ((C+L_{\so})/L_{\so})$ given in Lemma \ref{quotient}.

\smallskip
    \textsf{Subcase 3.1.} $|\ls|=q$.
    In this case, $G/\ls\cong\Z_{p^{k}}$, $(D+\ls)/\ls$ is aperiodic, and $|(D+\ls)/\ls|=|D|=p^{m-1}$.
    From Theorem \ref{p^k}, there are integers $0=m_0<m_1<\cdots<m_n=k$
    and $1\leq l_t\leq m_{t+1}-m_t$ for $t=0,\dots,n$ such that
    \begin{align*}
        (D+\ls)/\ls &= \left\{ \left(\sum_{t=0}^{n-1}\left(i_tp^{m_t}+\alpha_{i_0,\dots,i_t}p^{m_t+l_t}\right),0\right)+\ls \, \middle| \, i_t\in\left[p^{l_{t}}\right]\right\},\ \text{and}\\
        (C+\ls)/\ls &= \left\{ \left(\sum_{t=0}^{n-1} \left( i_tp^{m_t+l_t}+\beta_{i_0,\dots,i_t}p^{m_{t+1}}\right),0\right)+\ls \, \middle| \, i_t\in\left[p^{m_{t+1}-m_t-l_t}\right] \right\}.
    \end{align*}
    Now $L_{\so}=\{(0,j)\mid j\in[q]\}$. Therefore,
    \begin{align*}
        \so &= L_{\so}\oplus D\\
        &= \left\{ \left(\sum_{t=0}^{n-1}\left(i_tp^{m_t}+\alpha_{i_0,\dots,i_t}p^{m_t+l_t}\right),0\right)+(0,j) \, \middle| \, i_t\in\left[p^{l_{t}}\right],j\in[q]\right\} \\
        &= \left\{ \left(\sum_{t=0}^{n-1}\left(i_tp^{m_t}+\alpha_{i_0,\dots,i_t}p^{m_t+l_t}\right),j\right) \, \middle| \, i_t\in\left[p^{l_{t}}\right],j\in[q]\right\}.
    \end{align*}
    By setting $r=n$, $i_n=j$, and $m_r=m_r+l_r=m_{r+1}=k$, we have
    \begin{align*}
        \so
        &= \left\{ \left(\sum_{t\in[n+1]\setminus\{r\}}i_tp^{m_t}+\sum_{t=0}^{n}\alpha_{i_0,\dots,i_t}p^{m_t+l_t},i_r\right) \, \middle| \, i_t\in\left[p^{l_{t}}\right] \text{ for } t\ne r,i_r\in[q]\right\}.
    \end{align*}
    Now,
    \begin{align*}
        (C+L_{\so})/L_{\so} 
        &= \left\{ \left(\sum_{t=0}^{n-2} \left( i_tp^{m_t+l_t}+\beta_{i_0,\dots,i_t}p^{m_{t+1}}\right)+i_{n-1}p^{m_{n-1}+l_{n-1}},0\right) +\ls \, \middle| \, i_t\in\left[p^{m_{t+1}-m_t-l_t}\right] \right\}.
    \end{align*}
    Therefore,
    \begin{align*}
        C &= \left\{ \left(\sum_{t=0}^{n-2} \left( i_tp^{m_t+l_t}+\beta_{i_0,\dots,i_t}p^{m_{t+1}}\right)+i_{n-1}p^{m_{n-1}+l_{n-1}},0\right) + \beta_{i_0,\dots,i_{n-1}}(0,1) \, \middle| \, i_t\in\left[p^{m_{t+1}-m_t-l_t}\right] \right\}\\
        &= \left\{ \left(\sum_{t=0}^{n-1} i_tp^{m_t+l_t}+\sum_{t=0}^{n-2}\beta_{i_0,\dots,i_t}p^{m_{t+1}},\beta_{i_0,\dots,i_{n-1}}\right) \, \middle| \, i_t\in\left[p^{m_{t+1}-m_t-l_t}\right] \right\}.
    \end{align*}
    Recall that $r=n$, $i_n=j$, and $m_r=m_r+l_r=m_{r+1}=k$.
    So we have
    \begin{align*}
        C
        &= \left\{ \left(\sum_{t=0}^{n} i_tp^{m_t+l_t}+\sum_{t\in[n+1]\setminus\{r-1\}}\beta_{i_0,\dots,i_t}p^{m_{t+1}},\beta_{i_0,\dots,i_{r-1}}\right) \, \middle| \, i_t\in\left[p^{m_{t+1}-m_t-l_t}\right] \right\}.
    \end{align*}
    The pair $(\so,C)$ here is in the form of pair (c).

\smallskip
    \textsf{Subcase 3.2.} $|\ls|=p^s$ ($0<s\leq m-1$).
    In this case, $G/\ls\cong\Z_{p^{k-s}}\times\Z_q$, $(D+\ls)/\ls$ is aperiodic, and $|(D+\ls)/\ls|=|D|=p^{m-t-1}q$.
    From the induction hypothesis, there are integers $0=m_0<m_1<\cdots<m_n=k-s$
    and $1\leq l_t\leq m_{t+1}-m_t$ for $t=0,\dots,n$ such that $(D+L_{\so})/L_{\so}$ and $(C+L_{\so})/L_{\so}$ are one of the four pairs in Theorem \ref{p^k.q}. Since $|(D+\ls)/\ls|=p^{m-t-1}q$, they can only be pair (c) or (d).

    If $(D+L_{\so})/L_{\so}$ and $(C+L_{\so})/L_{\so}$ are pair (c), then
    \begin{align*}
        (D+L_{\so})/L_{\so} &= \left\{ \left(\sum_{t\in[n]\setminus\{r\}} i_tp^{m_t}+ \sum_{t=0}^{n-1}\alpha_{i_0,\dots,i_t}p^{m_t+l_t},\, i_r\right) +\ls \, \middle| \, i_t\in\left[p^{l_t}\right] \text{ for } t\ne r,\, i_r\in[q]\right\},\ \text{and}\\
        (C+L_{\so})/L_{\so} &= \left\{ \left(\sum_{t=0}^{n-1} i_tp^{m_t+l_t}+ \sum_{t\in[n]\setminus\{r-1\}}\beta_{i_0,\dots,i_t}p^{m_{t+1}},\, \beta_{i_0,\dots,i_{r-1}} \right) +\ls \, \middle| \, i_t\in\left[p^{m_{t+1}-m_t-l_t}\right]\right\}.
    \end{align*}
    Now, $L_{\so}=p^{k-s}\Z_{p^k}\times\{0\}=\left\{\left(hp^{k-s},\,0\right)\, \middle| \, h\in[p^s]\right\}$.
    Therefore,
    \begin{align*}
        \so = & L_{\so}\oplus D\\
        = & \left\{ \left(\sum_{t\in[n]\setminus\{r\}} i_tp^{m_t}+ \sum_{t=0}^{n-1}\alpha_{i_0,\dots,i_t}p^{m_t+l_t},\, i_r\right) + \left(hp^{k-s},\,0\right) \, \middle| \, i_t\in\left[p^{l_t}\right] \text{ for } t\ne r,\, i_r\in[q], h\in[p^s] \right\} \\
        = & \left\{ \left(\sum_{t\in[n]\setminus\{r\}} i_tp^{m_t}+ \sum_{t=0}^{n-1}\alpha_{i_0,\dots,i_t}p^{m_t+l_t},\, i_r\right) + \left(hp^{k-s}+ \alpha_{i_0,\dots,i_{n-1},h}\,p^{k},\,0\right) \, \middle| \, i_t\in\left[p^{l_t}\right] \text{ for } t\ne r,\right.\\ 
        & \quad \left.  i_r\in[q], h\in[p^s] \right\} \\
        = & \left\{ \left(\sum_{t\in[n+1]\setminus\{r\}} i_tp^{m_t}+ \sum_{t=0}^{n}\alpha_{i_0,\dots,i_t}p^{m_t+l_t},\, i_r\right) \, \middle| \, i_t\in\left[p^{l_t}\right] \text{ for } t\ne r,\, i_r\in[q]\right\},
    \end{align*}
    where $m_{n}=k-s$, $l_{n}=s$, and $i_{n}=h$.
    Since $G/L_{\so}\cong \Z_{p^{k-s}}\times\Z_q$, $(p^{k-s},0)+L_{\so}=L_{\so}$, so we have
    \begin{align*}
        (C+L_{\so})/L_{\so}
        = &\left\{ \left(\sum_{t=0}^{n-2} i_tp^{m_t+l_t}+ \sum_{t\in[n-1]\setminus\{r-1\}}\beta_{i_0,\dots,i_t}p^{m_{t+1}},\, \beta_{i_0,\dots,i_{r-1}} \right)+\left(i_{n-1}p^{m_{n-1}+l_{n-1}},0\right)+L_{\so} \, \middle| \, \right.\\
        & \quad \left. i_t\in\left[p^{m_{t+1}-m_t-l_t}\right]\right\}.
    \end{align*}
    Therefore,
    \begin{align*}
        C = & \left\{ \left(\sum_{t=0}^{n-2} i_tp^{m_t+l_t}+ \sum_{t\in[n-1]\setminus\{r-1\}}\beta_{i_0,\dots,i_t}p^{m_{t+1}},\, \beta_{i_0,\dots,i_{r-1}} \right) + \left(i_{n-1}p^{m_{n-1}+l_{n-1}}+ \beta_{i_0,\dots,i_{n-1}}p^{k-s},0\right) \right.\\
        & \quad \left.\, \middle| \, i_t\in\left[p^{m_{t+1}-m_t-l_t}\right] \right\}\\
        = & \left\{ \left(\sum_{t=0}^{n-1} i_tp^{m_t+l_t}+ \sum_{t\in[n]\setminus\{r-1\}}\beta_{i_0,\dots,i_t}p^{m_{t+1}},\, \beta_{i_0,\dots,i_{r-1}} \right) \, \middle| \, i_t\in\left[p^{m_{t+1}-m_t-l_t}\right] \right\}.
    \end{align*}
    Note that in $C$ we can also set the sum from $t=0$ to $n$ because $m_n+l_n=m_{n+1}=k$. The pair $(\so,C)$ here is also in the form of pair (c).

    If $(D+L_{\so})/L_{\so}$ and $(C+L_{\so})/L_{\so}$ are pair (d), the same algebraic operation we did for pair (c) is still applicable and we will get $(\so,C)$ to be another pair in the form of pair (d).

\smallskip
    \textsf{Subcase 3.3.} $|\ls|=p^sq$ ($0<s<m-1$).
    In this case, $G/\ls\cong\Z_{p^{k-s}}$, $(D+\ls)/\ls$ is aperiodic, and $|(D+\ls)/\ls|=|D|=p^{m-t-1}$.
    From Theorem \ref{p^k}, there are integers $0=m_0<m_1<\cdots<m_n=\log_p(|G/L_{\so}|)=k-s$
    and $1\leq l_t\leq m_{t+1}-m_t$ for $t\in[n]$ such that
    \begin{align*}
        (D+\ls)/\ls &= \left\{ \left(\sum_{t=0}^{n-1}\left(i_tp^{m_t}+\alpha_{i_0,\dots,i_t}p^{m_t+l_t}\right),0\right)+\ls \, \middle| \, i_t\in\left[p^{l_{t}}\right]\right\},\ \text{and}\\
        (C+\ls)/\ls &= \left\{ \left(\sum_{t=0}^{n-1} \left( i_tp^{m_t+l_t}+\beta_{i_0,\dots,i_t}p^{m_{t+1}}\right),0\right)+\ls \, \middle| \, i_t\in\left[p^{m_{t+1}-m_t-l_t}\right] \right\}.
    \end{align*}
    Now $L_{\so}=\left\{\left(hp^{k-s},j\right)\, \middle| \, h\in[p^s], j\in[q]\right\}=\<{(p^{k-s},1)}$.
    Therefore,
    \begin{align*}
        \so &= L_{\so}\oplus D\\
        &= \left\{ \left(\sum_{t=0}^{n-1}\left(i_tp^{m_t}+\alpha_{i_0,\dots,i_t}p^{m_t+l_t}\right),0\right)+(hp^{k-s},j) \, \middle| \, i_t\in\left[p^{l_{t}}\right], h\in[p^s], j\in[q]\right\} \\
        &= \left\{ \left(\sum_{t=0}^{n-1}\left(i_tp^{m_t}+\alpha_{i_0,\dots,i_t}p^{m_t+l_t}\right),0\right)+(hp^{k-s}+\alpha_{i_0,\dots,i_{n-1},h,j}p^k,j) \, \middle| \, i_t\in\left[p^{l_{t}}\right], h\in[p^s], j\in[q]\right\}.
    \end{align*}
    By setting $i_n=h$, $r=n$, $m_n=k-s$, and $l_r=t$, we have
    \begin{align*}
        \so
        &= \left\{ \left(\sum_{t=0}^{n}i_tp^{m_t}+\sum_{t=0}^{r-1}\alpha_{i_0,\dots,i_t}p^{m_t+l_t}+\sum_{t=r}^{n}\alpha_{i_0,\dots,i_t,j}p^{m_t+l_t},j\right) \, \middle| \, i_t\in\left[p^{l_{t}}\right], j\in[q]\right\}.
    \end{align*}
    Since $G/L_{\so}\cong \Z_{p^{k-s}}$, $(p^{k-s},0)+L_{\so}=L_{\so}$, so we have
    \begin{align*}
        & (C+L_{\so})/L_{\so} \\
        = & \left\{ \left(\sum_{t=0}^{n-2} \left( i_tp^{m_t+l_t}+\beta_{i_0,\dots,i_t}p^{m_{t+1}}\right),0\right)+\left(i_{n-1}p^{m_{n-1}+l_{n-1}},0\right)+\ls \, \middle| \, i_t\in\left[p^{m_{t+1}-m_t-l_t}\right] \right\}.
    \end{align*}
    Therefore,
    \begin{align*}
        C &= \left\{ \left(\sum_{t=0}^{n-2} \left( i_tp^{m_t+l_t}+\beta_{i_0,\dots,i_t}p^{m_{t+1}}\right),0\right) +\left(i_{n-1}p^{m_{n-1}+l_{n-1}},0\right)+\beta_{i_0,\dots,i_{n-1}}\left(p^{k-s},1\right) \, \middle| \,\right.\\ 
        & \quad \left. i_t\in\left[p^{m_{t+1}-m_t-l_t}\right] \right\}\\
        &= \left\{ \left(\sum_{t=0}^{n-1} \left( i_tp^{m_t+l_t}+\beta_{i_0,\dots,i_t}p^{m_{t+1}}\right),\beta_{i_0,\dots,i_{r-1}}\right) \, \middle| \, i_t\in\left[p^{m_{t+1}-m_t-l_t}\right] \right\}.
    \end{align*}
    Note that in $C$ we can also set the sum from $t=0$ to $n$ because $m_n+l_n=m_{n+1}=k$.
    The pair $(\so,C)$ here is in the form of pair (d).

\smallskip
    \textsf{Case 4.} $|\so|=p^{m-1}q$ and $\so$ is aperiodic.
\smallskip
    
    In this case, $C$ is periodic and $|C|=p^{k-m+1}$. Let $C=L_C\oplus D$. Then $L_C=p^s$, for $0<s\leq k-m+1$. Consider the factorization $G/{L_C}=((\so+L_C)/L_C)\oplus ((D+L_C)/L_C)$ given in Lemma \ref{quotient}.

\smallskip
    \textsf{Subcase 4.1.} $|L_C|=p^{k-m+1}$.
    In this case, $C=L_C=\{\left(ip^{m-1},0\right) \mid i\in\left[p^{k-m+1}\right]\}=\left<(p^{m-1},0)\right>$ and $G/{L_C}\cong \Z_{p^{m-1}}\times\Z_q$.
    Since $C=L_C$, we can take $D=\{0\}$ and we have $G/{L_C}=(\so+L_C)/L_C$.
    This means $(\so+L_C)/L_C=\{(i,j)+L_C \mid i\in[p^{m-1}], j\in[q]\}$ and hence $\so=\{(i,j)+\alpha_{ij}(p^{m-1},0) \mid i\in[p^{k-m}],j\in[q]\}=\{(i+\alpha_{ij}p^{m},j) \mid i\in[p^{k-m}],j\in[q]\}$. The pair $(\so,C)$ here is in the form of pair (d).

\smallskip
    \textsf{Subcase 4.2.} $|L_C|=p^{s}$ for $t<k-m+1$.
    In this case, $G/{L_C}\cong \Z_{p^{k-s}}\times\Z_q$, $|(\so+L_C)/L_C|=|\so|=p^{m-1}q$, and $(\so+L_C)/L_C$ is periodic. 
    From Case 3, $(\so+L_C)/L_C$ is either in the form of pair (c) or (d) in Theorem \ref{p^k.q}. We will give the calculation for $(\so+L_C)/L_C$ in the form of pair (c). The calculation when $(\so+L_C)/L_C$ is in the form of pair (d) will follow similarly.

    From Cases 3.1 and 3.2, 
    \begin{align*}
        (\so+L_C)/L_C
        &= \left\{ \left(\sum_{t\in[n+1]\setminus\{r\}} i_tp^{m_t}+ \sum_{t=0}^{n}\alpha_{i_0,\dots,i_t}p^{m_t+l_t},\, i_r\right) +\lc \, \middle| \, i_t\in\left[p^{l_t}\right] \text{ for } t\ne r,\, i_r\in[q]\right\}, \\
        (D+L_C)/L_C 
        &= \left\{ \left(\sum_{t=0}^{n-1} i_tp^{m_t+l_t}+ \sum_{t\in[n]\setminus\{r-1\}}\beta_{i_0,\dots,i_t}p^{m_{t+1}},\, \beta_{i_0,\dots,i_{r-1}} \right) +\lc \, \middle| \, i_t\in\left[p^{m_{t+1}-m_t-l_t}\right] \right\}.
    \end{align*}
    We also have $m_{n}+l_{n}=m_{n+1}=k-s$ from Subcase 3.1 and 3.2.
    Since $G/{L_C}\cong \Z_{p^{k-s}}\times\Z_q$, $(p^{k-s},0)+L_C=L_C$, and
    \begin{align*}
        (\so+L_C)/L_C
        &= \left\{ \left(\sum_{t\in[n]\setminus\{r\}} i_tp^{m_t}+ \sum_{t=0}^{n-1}\alpha_{i_0,\dots,i_t}p^{m_t+l_t},\, i_r\right) + \left(i_{n}p^{m_{n}},0\right) + L_C \, \middle| \, i_t\in\left[p^{l_t}\right] \text{ for } t\ne r,\, i_r\in[q] \right\}.
    \end{align*}
    Since $|L_C|=p^s$, $L_C=p^{k-s}\Z_{p^k}=\left\{\left(hp^{k-s},0\right)\, \middle| \, h\in\left[p^s\right]\right\}$. Therefore,
    \begin{align*}
        \so
        &= \left\{\left(\sum_{t\in[n]\setminus\{r\}} i_tp^{m_t}+ \sum_{t=0}^{n-1}\alpha_{i_0,\dots,i_t}p^{m_t+l_t},\, i_r\right) + \left(i_{n}p^{m_{n}} + \alpha_{i_0,\dots,i_{n}}p^{k-s},0\right) \, \middle| \, i_t\in\left[p^{l_t}\right] \text{ for } t\ne r,\, i_r\in[q] \right\}\\
        &= \left\{ \left(\sum_{t\in[n+1]\setminus\{r\}} i_tp^{m_t}+ \sum_{t=0}^{n}\alpha_{i_0,\dots,i_t}p^{m_t+l_t},\, i_r\right) \, \middle| \, i_t\in\left[p^{l_t}\right] \text{ for } t\ne r,\, i_r\in[q] \right\},\ \text{and}
    \end{align*}
    \begin{align*}
        C &= D\oplus L_C \\
        &= \left\{ \left(\sum_{t=0}^{n-1} i_tp^{m_t+l_t}+ \sum_{t\in[n]\setminus\{r-1\}}\beta_{i_0,\dots,i_t}p^{m_{t+1}},\, \beta_{i_0,\dots,i_{r-1}} \right) + \left(hp^{k-s},0\right) \, \middle| \, i_t\in\left[p^{m_{t+1}-m_t-l_t}\right] \right\}\\
        &= \left\{ \left(\sum_{t=0}^{n-1} i_tp^{m_t+l_t}+ \sum_{t\in[n]\setminus\{r-1\}}\beta_{i_0,\dots,i_t}p^{m_{t+1}},\, \beta_{i_0,\dots,i_{r-1}} \right) + \left(hp^{k-s} + \beta_{i_0,\dots,i_{n-1},h}p^{k},0\right) \, \middle| \, i_t\in\left[p^{m_{t+1}-m_t-l_t}\right] \right\}.
    \end{align*}
    By setting $i_n=h$ and a new value for $m_{n+1}$ to be equal to $k$, we have
    \begin{align*}
        C &= \left\{ \left(\sum_{t=0}^{n} i_tp^{m_t+l_t}+ \sum_{t\in[n+1]\setminus\{r-1\}}\beta_{i_0,\dots,i_t}p^{m_{t+1}},\, \beta_{i_0,\dots,i_{r-1}} \right) \, \middle| \, i_t\in\left[p^{m_{t+1}-m_t-l_t}\right] \right\}.
    \end{align*}
    The pair $(\so,C)$ here is in the form of pair (c).
\end{proof}

\subsection{Proof of Theorem \ref{2^k.2}}

\begin{observation}\label{obs}
    There are three subgroups of $G=\Z_{2^k}\times\Z_2$ of order $2^s$, namely $\<{(2^{k-s},0)}$, $\<{(2^{k-s},1)}$, and $\<{(2^{k-s+1}),(0,1)}$.
    Their quotient groups are $G/\<{(2^{k-s},0)}\cong\Z_{2^{k-s}}\times\Z_2$, $G/\<{(2^{k-s},1)}\cong\Z_{2^{k-s+1}}$, and $G/\<{(2^{k-s+1},0),(0,1)}\cong\Z_{2^{k-s+1}}$.
\end{observation}

\begin{proof}[\textbf{Proof of Theorem \ref{2^k.2}}]

    Suppose $(\so,C)$ is one of the pairs in Theorem \ref{p^k.q}.
    Then $G=\so\oplus C$ by Lemma \ref{G=so+c}.
    Suppose $\Cay(G,S)$ has a perfect code $C$. Then $|\so|=2^m$ for some positive integer $m$ and $G=\Z_{p^k}=\so\oplus C$. We will prove $(\so,C)$ is one of the six pairs in Theorem \ref{p^k.q} using induction on $m$.

    The group $G=\Z_{2^k}\times\Z_2$ is not a cyclic group, and since $\so$ generates $G$, there are at least two non-zero elements in $\so$. So we start the basis step at $m=2$. First, we give an outline of the method to be used. When $\so$ is periodic, we will prove the basis step at $m=2$, and for $m>2$, we will use induction by using Theorem \ref{p^k} and when $\so$ is aperiodic and $|\so|<2^m$. When $\so$ is aperiodic, we will use Theorem \ref{p^k} and the result when $\so$ is periodic and $|\so|=2^m$.

    Suppose $|\so|=4$ and $\so$ is periodic. From Lemma \ref{ls<so}, $|\ls|=2$, so $\ls=\<{(2^{k-1},0)}$, $\<{(2^{k-1},1)}$, or $\<{(0,1)}$.
    Consider the factorization $G/{L_{\so}}=((D+L_{\so})/L_{\so})\oplus ((C+L_{\so})/L_{\so})$ given by Lemma \ref{quotient}.
    If $\ls=\<{(2^{k-1},0)}$, then $G/{L_{\so}}\cong\Z_{2^{k-1}}\times\Z_2$ and $|(D+L_{\so})/L_{\so}|=2$. Since $\Z_{2^{k-1}}\times\Z_2$ is not a cyclic group, we know that no set of order $2$ generates $\Z_{2^{k-1}}\times\Z_2$.

    If $\ls=\<{(2^{k-1},1)}$, then $G/{L_{\so}}\cong\Z_{2^k}$ and $(D+L_{\so})/L_{\so}$ is a aperiodic set of order $2$.
    From Lemma \ref{prime deg} and Theorem \ref{p^k}, $(C+L_{\so})/L_{\so}=\{(2i,0)+\ls \mid i \in [2^{k-1} ] \}$ and $(D+L_{\so})/L_{\so}=\left\{ (i+2\alpha_i,0) + \ls \mid i\in[2] \right\}$.
    This means $\so=D\oplus \ls=\left\{ (i+2\alpha_i,0) + j(2^{k-1},1) \mid i,j\in[2] \right\}=\left\{ (i+2\alpha_i+j2^{k-1},j) \mid i,j\in[2] \right\}$.
    By setting $i_0=i$ and $i_1=j$, we have \linebreak $\so=\left\{ (i_0+2\alpha_{i_0}+i_12^{k-1},i_1) \mid i_0,i_1\in[2] \right\}$.
    Note that $C=\{(2i,0)+\beta_i(2^{k-1},1) \mid i \in [2^{k-1}] \}=\{(2i_0+\beta_{i_0} 2^{k-1},\beta_{i_0}) \mid i \in [2^{k-1}] \}$.
    The pair $(\so,C)$ here is pair (e) in Theorem \ref{2^k.2} where $r=1$, $m_0=0$, $m_1=k$, $m_2=k$, $l_0=1$, and $l_1=0$.

    Similarly, if $\ls=\<{(0,1)}$, then $G/{L_{\so}}\cong\Z_{2^k}$, and hence $(C+L_{\so})/L_{\so}=\{(2i,0)+\ls \mid i \in [2^{k-1} ] \}$ and $(D+L_{\so})/L_{\so}=\left\{ (i+2\alpha_i,0) + \ls \mid i\in[2] \right\}$.
    This means $\so=D\oplus \ls=\left\{ (i_0+2\alpha_{i_0},i_1) \mid i_0,i_1\in[2] \right\}$ and $C=\{(2i_0,\beta_{i_0}) \mid i \in [2^{k-1}] \}$.
    The pair $(\so,C)$ here is pair (d) in Theorem \ref{2^k.2} where $r=1$, $m_0=0$, $m_1=k$, $m_2=k$, $l_0=1$, and $l_1=0$.

    Suppose Theorem \ref{2^k.2} is true for $|\so|<2^m$. We will prove that $(\so,C)$ is one of the six pairs in Theorem \ref{2^k.2} for $|\so|=2^m$.

\smallskip
    \textsf{Case 1.} $\so$ is periodic.
\smallskip
    
    Let $\so=L_{\so}\oplus D$ and $|L_{\so}|=2^s$. Then $0<s<m$ because $\so$ generates $G$, and $\ls=\<{(2^{k-s},0)}$, $\<{(2^{k-s},1)}$, or $\<{(2^{k-s+1}),(0,1)}$. Consider the factorization $G/{L_{\so}}=((D+L_{\so})/L_{\so})\oplus ((C+L_{\so})/L_{\so})$ given by Lemma \ref{quotient}.

\smallskip
    \textsf{Subcase 1.1.} $\ls=\<{(2^{k-s},0)}$.
    In this case, $G/{L_{\so}}\cong\Z_{2^{k-s}}\times\Z_2$ and $|(D+L_{\so})/L_{\so}|=|D|=2^{m-t}$. From the induction hypothesis, there are integers $0=m_0<m_1<\cdots<m_n=k-s$
    and $1\leq l_t\leq m_{t+1}-m_t$ for $t=0\in[n]$, such that $(D+L_{\so})/L_{\so}$ and $(C+L_{\so})/L_{\so}$ are one of the six pairs in Theorem \ref{2^k.2}.

    We will give the calculation for $(D+L_{\so})/L_{\so}$ and $(C+L_{\so})/L_{\so}$ in the form of pair (c). The calculation for the other pairs will follow similarly.
    Since $(D+L_{\so})/L_{\so}$ and $(C+L_{\so})/L_{\so}$ are pair (c), we have
    \begin{align*}
        (D+L_{\so})/L_{\so} = & \left\{\left(\sum_{t=0}^{n-1} \left( i_t2^{m_t}+\alpha_{i_0,\dots,i_t}2^{m_t+l_t}\right),\, \gamma_{i_0,\dots,i_r}\right) +L_{\so} \, \middle| \, i_t\in\left[2^{l_t}\right]\right\},\ \text{and}\\
        (C+L_{\so})/L_{\so} = & \left\{ \left(\sum_{t=0}^{n-1} i_t2^{m_t+l_t} + \sum_{t=0}^{r-1} \beta_{i_0,\dots,i_t}2^{m_{t+1}} + \sum_{t=r}^{n-1} \beta_{i_0,\dots,i_t,j}\,2^{m_{t+1}},\,j\right) +L_{\so} \right.\\ 
        & \quad \left. \, \middle| \, i_t\in\left[2^{m_{t+1}-m_t-l_t}\right],\,j\in[2] \right\}.
    \end{align*}

    \noindent Now, $L_{\so}=2^{k-s}\Z_{2^k}\times\{0\}=\left\{\left(h2^{k-s},\,0\right)\, \middle| \, h\in[2^s]\right\}$. Therefore,
    \begin{align*}
        \so &= L_{\so}\oplus D\\
        &= \left\{ \left(\sum_{t=0}^{n-1} \left(i_t2^{m_t}+\alpha_{i_0,\dots,i_t}2^{m_t+l_t}\right),\, \gamma_{i_0,\dots,i_r}\right) + \left(h2^{k-s},\,0\right) \, \middle| \, i_t\in \left[2^{l_t}\right], h\in[2^s] \right\} \\
        &= \left\{ \left(\sum_{t=0}^{n-1} \left(i_t2^{m_t}+\alpha_{i_0,\dots,i_t}2^{m_t+l_t}\right),\, \gamma_{i_0,\dots,i_r}\right) + \left(h2^{k-s}+ \alpha_{i_0,\dots,i_{n-1},h}\,2^{k},\,0\right) \, \middle| \, i_t\in\left[2^{l_t}\right], h\in[2^s] \right\} \\
        &= \left\{ \left(\sum_{t=0}^{n} \left(i_t2^{m_t}+\alpha_{i_0,\dots,i_t}2^{m_t+l_t}\right),\, \gamma_{i_0,\dots,i_r}\right) \, \middle| \, i_t\in\left[2^{l_t}\right], h\in[2^s]\right\}
    \end{align*}
    where $m_{n}=k-s$, $l_{n}=s$, and $i_{n}=h$. Moreover, $m_{n+1}=k>m_{n}$ and $1\leq l_{n}=m_{n+1}-m_{n}$.
    Since $G/L_{\so}\cong \Z_{2^{k-s}}\times\Z_2$, we have $(2^{k-s},0)+L_{\so}=L_{\so}$ and
    \begin{align*}
        (C+L_{\so})/L_{\so}
        = &\left\{ \left(\sum_{t=0}^{n-2} i_t2^{m_t+l_t} + \sum_{t=0}^{r-1} \beta_{i_0,\dots,i_t}2^{m_{t+1}} + \sum_{t=r}^{n-2} \beta_{i_0,\dots,i_t,j}\,2^{m_{t+1}},\,j\right)+\left(i_{n-1}2^{m_{n-1}+l_{n-1}},0\right)+L_{\so} \right.\\
        & \quad \left.\, \middle| \, i_t\in\left[2^{m_{t+1}-m_t-l_t}\right],\ j\in[2] \right\}.
    \end{align*}
    Therefore,
    \begin{align*}
        C = & \left\{ \left(\sum_{t=0}^{n-2} i_t2^{m_t+l_t} + \sum_{t=0}^{r-1} \beta_{i_0,\dots,i_t}2^{m_{t+1}} + \sum_{t=r}^{n-2} \beta_{i_0,\dots,i_t,j}\,2^{m_{t+1}},\,j\right) + \left(i_{n-1}2^{m_{n-1}+l_{n-1}}+ \beta_{i_0,\dots,i_{n-1},j}2^{k-s},0\right) \right. \\
        &\quad \left.\, \middle| \, i_t\in\left[2^{m_{t+1}-m_t-l_t}\right], j\in[q] \right\}\\
        = & \left\{ \left(\sum_{t=0}^{n-1} i_t2^{m_t+l_t} + \sum_{t=0}^{r-1} \beta_{i_0,\dots,i_t}2^{m_{t+1}} + \sum_{t=r}^{n-1} \beta_{i_0,\dots,i_t,j}\,2^{m_{t+1}},\,j\right) \, \middle| \, i_t\in\left[2^{m_{t+1}-m_t-l_t}\right], j\in[2] \right\}.
    \end{align*}
    Note that in $C$ we can also set the sum from $t=0$ to $n$ because $m_n+l_n=m_{n+1}=k$. The pair $(\so,C)$ here is also in the form of pair (c).

\smallskip
    \textsf{Subcase 1.2.} $\ls=\<{(2^{k-s},1)}$.
    In this case, $G/{L_{\so}}\cong\Z_{2^{k-s+1}}$ and $|(D+L_{\so})/L_{\so}|=|D|=2^{m-t}$.
    
    From Theorem \ref{p^k}, there are integers $0=m_0<m_1<\cdots<m_n=k-s+1$
    and $1\leq l_t\leq m_{t+1}-m_t$ for $t\in[n]$ such that
    \begin{align*}
        (D+\ls)/\ls &= \left\{ \left(\sum_{t=0}^{n-1}\left(i_t2^{m_t}+\alpha_{i_0,\dots,i_t}2^{m_t+l_t}\right),0\right)+\ls \, \middle| \, i_t\in\left[2^{l_{t}}\right]\right\},\ \text{and}\\
        (C+\ls)/\ls &= \left\{ \left(\sum_{t=0}^{n-1} \left( i_t2^{m_t+l_t}+\beta_{i_0,\dots,i_t}2^{m_{t+1}}\right),0\right)+\ls \, \middle| \, i_t\in\left[2^{m_{t+1}-m_t-l_t}\right] \right\}.
    \end{align*}
    Now $L_{\so}=\<{(2^{k-s},1)}=\left\{h\left(2^{k-s},1\right)\, \middle| \, h\in[2^s]\right\}$. Therefore,
    \begin{align*}
        \so &= L_{\so}\oplus D\\
        &= \left\{ \left(\sum_{t=0}^{n-1}\left(i_t2^{m_t}+\alpha_{i_0,\dots,i_t}2^{m_t+l_t}\right),0\right)+(h2^{k-s},h) \, \middle| \, i_t\in\left[2^{l_{t}}\right], h\in[2^s]\right\} \\
        &= \left\{ \left(\sum_{t=0}^{n-1}\left(i_t2^{m_t}+\alpha_{i_0,\dots,i_t}2^{m_t+l_t}\right),0\right)+(h2^{k-s}+\alpha_{i_0,\dots,i_{n-1},h}2^k,h) \, \middle| \, i_t\in\left[2^{l_{t}}\right], h\in[2^s]\right\}.
    \end{align*}
    By setting $i_n=h$, $r=n$, $m_n=k-s+1$, and $l_n=s-1$, and $m_{n+1}=m_n+l_n=k$, we have
    \begin{align*}
        \so
        &= \left\{ \left(\sum_{t\in[n+1]\setminus\{r\}}i_t2^{m_t}+\sum_{t=0}^{n}\alpha_{i_0,\dots,i_t}2^{m_t+l_t}+i_r2^{m_r-1},i_r\right) \, \middle| \, i_t\in\left[2^{l_{t}}\right] \text{ for } t\ne r, i_r\in\left[2^{l_r+1}\right]\right\}.
    \end{align*}
    Since $G/L_{\so}\cong \Z_{2^{k-s+1}}$, we have $(2^{m_n},0)+L_{\so}=L_{\so}$ and hence
    \begin{align*}
        (C+L_{\so})/L_{\so} &= \left\{ \left(\sum_{t=0}^{n-2} \left( i_t2^{m_t+l_t}+\beta_{i_0,\dots,i_t}2^{m_{t+1}}\right),0\right)+\left(i_{n-1}2^{m_{n-1}+l_{n-1}},0\right)+\ls \, \middle| \, i_t\in\left[2^{m_{t+1}-m_t-l_t}\right] \right\}.
    \end{align*}
    Therefore,
    \begin{align*}
        C &= \left\{ \left(\sum_{t=0}^{n-2} \left( i_t2^{m_t+l_t}+\beta_{i_0,\dots,i_t}2^{m_{t+1}}\right),0\right) +\left(i_{n-1}2^{m_{n-1}+l_{n-1}},0\right)+\beta_{i_0,\dots,i_{n-1}}\left(2^{k-s},1\right) \, \middle| \, i_t\in\left[2^{m_{t+1}-m_t-l_t}\right] \right\}\\
        &= \left\{ \left(\sum_{t=0}^{n-1}i_t2^{m_t+l_t}+\sum_{t\in[n]\setminus\{r-1\}}\beta_{i_0,\dots,i_t}2^{m_{t+1}}+\beta_{i_0,\dots,i_{r-1}}2^{m_r-1},\beta_{i_0,\dots,i_{r-1}}\right) \, \middle| \, i_t\in\left[2^{m_{t+1}-m_t-l_t}\right] \right\}.
    \end{align*}
    Note that in $C$ we can also set the sum from $t=0$ to $n$ because $m_n+l_n=m_{n+1}=k$.
    The pair $(\so,C)$ here is in the form of pair (e).

\smallskip
    \textsf{Subcase 1.3.} $\ls=\<{(2^{k-s+1},0),(0,1)}$ for $1<s<m$.
    In this case, $G/{L_{\so}}\cong\Z_{2^{k-s+1}}$ and $|(D+L_{\so})/L_{\so}|=|D|=2^{m-t}$.
    From Theorem \ref{p^k}, there are integers $0=m_0<m_1<\cdots<m_n=k-s+1$
    and $1\leq l_t\leq m_{t+1}-m_t$ for $t\in[n]$ such that
    \begin{align*}
        (D+\ls)/\ls &= \left\{ \left(\sum_{t=0}^{n-1}\left(i_t2^{m_t}+\alpha_{i_0,\dots,i_t}2^{m_t+l_t}\right),0\right)+\ls \, \middle| \, i_t\in\left[2^{l_{t}}\right]\right\},\ \text{and}\\
        (C+\ls)/\ls &= \left\{ \left(\sum_{t=0}^{n-1} \left( i_t2^{m_t+l_t}+\beta_{i_0,\dots,i_t}2^{m_{t+1}}\right),0\right)+\ls \, \middle| \, i_t\in\left[2^{m_{t+1}-m_t-l_t}\right] \right\}.
    \end{align*}
    Now $L_{\so}=\left\{\left(h2^{k-s+1},j\right)\, \middle| \, h\in[2^{s-1}], j\in[2]\right\}$. Therefore,
    \begin{align*}
        \so &= L_{\so}\oplus D\\
        &= \left\{ \left(\sum_{t=0}^{n-1}\left(i_t2^{m_t}+\alpha_{i_0,\dots,i_t}2^{m_t+l_t}\right),0\right)+(h2^{k-s+1},j) \, \middle| \, i_t\in\left[2^{l_{t}}\right], h\in[2^{s-1}], j\in[2]\right\} \\
        &= \left\{ \left(\sum_{t=0}^{n-1}\left(i_t2^{m_t}+\alpha_{i_0,\dots,i_t}2^{m_t+l_t}\right),0\right)+(h2^{k-s+1}+\alpha_{i_0,\dots,i_{n-1},h,j}2^k,j) \, \middle| \, i_t\in\left[2^{l_{t}}\right], h\in[2^{s-1}], j\in[2]\right\}.
    \end{align*}
    By setting $i_n=h$, $r=n$, $m_n=k-s+1$, and $l_n=s-1$, and $m_{n+1}=m_n+l_n=k$, we have
    \begin{align*}
        \so
        &= \left\{ \left(\sum_{t=0}^{n}i_t2^{m_t}+\sum_{t=0}^{r-1}\alpha_{i_0,\dots,i_t}2^{m_t+l_t}+\sum_{t=r}^{n}\alpha_{i_0,\dots,i_t,j}2^{m_t+l_t},j\right) \, \middle| \, i_t\in\left[2^{l_{t}}\right], j\in[2]\right\}.
    \end{align*}
    
    \noindent Since $G/L_{\so}\cong \Z_{2^{k-s+1}}$, we have $(2^{m_n},0)+L_{\so}=L_{\so}$ and hence
    \begin{align*}
        (C+L_{\so})/L_{\so} &= \left\{ \left(\sum_{t=0}^{n-2} \left( i_t2^{m_t+l_t}+\beta_{i_0,\dots,i_t}2^{m_{t+1}}\right),0\right)+\left(i_{n-1}2^{m_{n-1}+l_{n-1}},0\right)+\ls \, \middle| \, i_t\in\left[2^{m_{t+1}-m_t-l_t}\right] \right\}.
    \end{align*}
    Therefore,
    \begin{align*}
        C = & \left\{ \left(\sum_{t=0}^{n-2} \left( i_t2^{m_t+l_t}+\beta_{i_0,\dots,i_t}2^{m_{t+1}}\right),0\right) +\left(i_{n-1}2^{m_{n-1}+l_{n-1}},0\right)+\beta_{i_0,\dots,i_{n-1}}\left(2^{k-s+1},0\right)+\gamma_{i_0,\dots,i_{n-1}}\left(0,1\right) \right.\\
        &\quad \left.\, \middle| \, i_t\in\left[2^{m_{t+1}-m_t-l_t}\right] \right\}\\
        = & \left\{ \left(\sum_{t=0}^{n-1} \left( i_t2^{m_t+l_t}+\beta_{i_0,\dots,i_t}2^{m_{t+1}}\right),\gamma_{i_0,\dots,i_{r-1}}\right) \, \middle| \, i_t\in\left[2^{m_{t+1}-m_t-l_t}\right] \right\}.
    \end{align*}
    Note that in $C$ we can also set the sum from $t=0$ to $n$ because $m_n+l_n=m_{n+1}=k$.
    The pair $(\so,C)$ here is in the form of pair (f).

\smallskip
    \textsf{Subcase 1.4.} $\ls=\<{(2^{k-s+1},0),(0,1)}$ for $t=1$.
    In this case, $\ls=\<{(0,1)}$ and $G/\ls\cong\Z_{2^{k}}$, $(D+\ls)/\ls$ is aperiodic, and $|(D+\ls)/\ls|=|D|=2^{m-1}$.
    From Theorem \ref{p^k}, there are integers $0=m_0<m_1<\cdots<m_n=k$
    and $1\leq l_t\leq m_{t+1}-m_t$ for $t=0,\dots,n$ such that
    \begin{align*}
        (D+\ls)/\ls &= \left\{ \left(\sum_{t=0}^{n-1}\left(i_t2^{m_t}+\alpha_{i_0,\dots,i_t}2^{m_t+l_t}\right),0\right)+\ls \, \middle| \, i_t\in\left[2^{l_{t}}\right]\right\},\ \text{and}\\
        (C+\ls)/\ls &= \left\{ \left(\sum_{t=0}^{n-1} \left( i_t2^{m_t+l_t}+\beta_{i_0,\dots,i_t}2^{m_{t+1}}\right),0\right)+\ls \, \middle| \, i_t\in\left[2^{m_{t+1}-m_t-l_t}\right] \right\}.
    \end{align*}

    \noindent Now $L_{\so}=\{(0,j)\mid j\in[2]\}$. Therefore,
    \begin{align*}
        \so &= L_{\so}\oplus D\\
        &= \left\{ \left(\sum_{t=0}^{n-1}\left(i_t2^{m_t}+\alpha_{i_0,\dots,i_t}2^{m_t+l_t}\right),0\right)+(0,j) \, \middle| \, i_t\in\left[2^{l_{t}}\right],j\in[2]\right\} \\
        &= \left\{ \left(\sum_{t=0}^{n-1}\left(i_t2^{m_t}+\alpha_{i_0,\dots,i_t}2^{m_t+l_t}\right),j\right) \, \middle| \, i_t\in\left[2^{l_{t}}\right],j\in[2]\right\}.
    \end{align*}
    By setting $r=n$, $i_n=j$, and $m_r=m_r+l_r=m_{r+1}=k$, we have
    \begin{align*}
        \so
        &= \left\{ \left(\sum_{t\in[n+1]\setminus\{r\}}i_t2^{m_t}+\sum_{t=0}^{n}\alpha_{i_0,\dots,i_t}2^{m_t+l_t},i_r\right) \, \middle| \, i_t\in\left[2^{l_{t}}\right] \text{ for } t\ne r,i_r\in[2]\right\}.
    \end{align*}
    Now,
    \begin{align*}
        (C+L_{\so})/L_{\so} 
        &= \left\{ \left(\sum_{t=0}^{n-2} \left( i_t2^{m_t+l_t}+\beta_{i_0,\dots,i_t}2^{m_{t+1}}\right)+i_{n-1}2^{m_{n-1}+l_{n-1}},0\right) +\ls \, \middle| \, i_t\in\left[2^{m_{t+1}-m_t-l_t}\right] \right\}.
    \end{align*}
    Therefore,
    \begin{align*}
        C &= \left\{ \left(\sum_{t=0}^{n-2} \left( i_t2^{m_t+l_t}+\beta_{i_0,\dots,i_t}2^{m_{t+1}}\right)+i_{n-1}2^{m_{n-1}+l_{n-1}},0\right) + \beta_{i_0,\dots,i_{n-1}}(0,1) \, \middle| \, i_t\in\left[2^{m_{t+1}-m_t-l_t}\right] \right\}\\
        &= \left\{ \left(\sum_{t=0}^{n-1} i_t2^{m_t+l_t}+\sum_{t=0}^{n-2}\beta_{i_0,\dots,i_t}2^{m_{t+1}},\beta_{i_0,\dots,i_{n-1}}\right) \, \middle| \, i_t\in\left[2^{m_{t+1}-m_t-l_t}\right] \right\}.
    \end{align*}
    Recall that $r=n$, $i_n=j$, and $m_r=m_r+l_r=m_{r+1}=k$. So we have
    \begin{align*}
        C
        &= \left\{ \left(\sum_{t=0}^{n} i_t2^{m_t+l_t}+\sum_{t\in[n+1]\setminus\{r-1\}}\beta_{i_0,\dots,i_t}2^{m_{t+1}},\beta_{i_0,\dots,i_{r-1}}\right) \, \middle| \, i_t\in\left[2^{m_{t+1}-m_t-l_t}\right] \right\}.
    \end{align*}
    The pair $(\so,C)$ here is in the form of pair (d).

\smallskip
    \textsf{Case 2.} $\so$ is aperiodic.
\smallskip
    
    In this case, $C$ is periodic and $|C|=2^{k-m+1}$. Let $C=L_C\oplus D$, $|L_C|=2^s$ for $0<s\leq k-m+1$. Consider the factorization $G/{L_C}=((\so+L_C)/L_C)\oplus ((D+L_C)/L_C)$ given by Lemma \ref{quotient}.

\smallskip
    \textsf{Subcase 2.1.} $|L_C|=2^{k-m+1}$.
    In this case $C=L_C$, and by taking $D=\{0\}$ we have $G/{L_C}=(\so+L_C)/L_C$.

    If $L_C=\<{(2^{m-1},0)}=\left\{\left(i2^{m-1},0\right)\mid i\in\left[2^{k-m+1}\right]\right\}$, then $(\so+L_C)/L_C=\{(i,j)+L_C \mid i\in\left[2^{m-1}\right],j\in[2]\}$.
    This means $\so=\{(i,j)+\alpha_{ij}(2^{m-1},0) \mid i\in\left[2^{m-1}\right],j\in[2]\}=\{(i+\alpha_{ij}2^{m-1},j) \mid i\in\left[2^{m-1}\right],j\in[2]\}$.
    The pair $(\so,C)$ here is in the form of pair (f) where $n=1$ and $r=0$.

    If $L_C=\<{(2^{m-1},1)}=\left\{\left(i2^{m-1},i\right)\mid i\in\left[2^{k-m+1}\right]\right\}$, then $(\so+L_C)/L_C=\{(i,0)+L_C \mid i\in\left[2^{m}\right]\}$.
    This means $\so=\{(i,0)+\alpha_{i}(2^{m-1},1) \mid i\in\left[2^{m}\right]\}=\{(i+\alpha_{i}2^{m-1},\alpha_i) \mid i\in\left[2^{m}\right]\}$.
    The pair $(\so,C)$ here is in the form of pair (b) where $n=1$ and $r=0$.

    If $L_C=\<{(2^{m},0),(0,1)}$ for $m<k$, then $L_C=\left\{\left(i2^{m},j\right)\mid i\in\left[2^{k-m}\right],\,j\in[2]\right\}$ and $(\so+L_C)/L_C=\{(i,0)+L_C \mid i\in\left[2^{m}\right]\}$.
    This means $\so=\{(i,0)+\alpha_{i}(2^{m},0)+\gamma_i(0,1) \mid i\in\left[2^{m}\right]\}=\{(i+\alpha_{i}2^{m},\gamma_i) \mid i\in\left[2^{m}\right]\}$.
    The pair $(\so,C)$ here is in the form of pair (c) where $n=1$ and $r=0$.
    
    If $L_C=\<{(2^{m},0),(0,1)}$ for $m=k$, then $L_C=\<{(0,1)}=\left\{\left(0,i\right)\mid i\in[2]\right\}$ and $(\so+L_C)/L_C=\{(i,0)+L_C \mid i\in\left[2^{k}\right]\}$.
    This means $\so=\{(i,0)+\alpha_{i}(0,1) \mid i\in\left[2^{k}\right]\}=\{(i,\alpha_{i}) \mid i\in\left[2^{k}\right]\}$.
    The pair $(\so,C)$ here is in the form of pair (a) where $n=1$ and $r=0$.

\smallskip
    \textsf{Subcase 2.2.} $L_C=\<{(2^{k-s},0)}$ for $t<k-m+1$.
    In this case, $G/{L_C}\cong \Z_{2^{k-s}}\times\Z_2$, $|(\so+L_C)/L_C|=|\so|=2^m$, and $(\so+L_C)/L_C$ is periodic. 
    From Case 1, $(\so+L_C)/L_C$ is one of the six pairs in Theorem \ref{2^k.2}. We will calculate $(\so+L_C)/L_C$ and $(D+L_C)/L_C$ in the form of pair (a). The calculation for the other forms will follow similarly.

    From Case 1, there are integers $0=m_0<m_1<\cdots<m_{n+1}=k-s$
    and $1\leq l_t\leq m_{t+1}-m_t$ for $t\in[n]$, such that
    \begin{align*}
        (\so+L_C)/L_C
        &= \left\{ \left(\sum_{t=0}^{n} i_t2^{m_t}+ \sum_{t\in[n+1]\setminus\{r\}}\alpha_{i_0,\dots,i_t}2^{m_t+l_t},\, \alpha_{i_0,\dots,i_r}\right) + L_C \, \middle| \, i_t\in\left[2^{l_t}\right]\right\},\ \text{and}\\
        (D+L_C)/L_C &= \left\{ \left(\sum_{t\in[n]\setminus\{r\}} i_t2^{m_t+l_t}+ \sum_{t=0}^{n-1}\beta_{i_0,\dots,i_t}2^{m_{t+1}},\, i_r\right) +L_C \, \middle| \, i_t\in\left[2^{m_{t+1}-m_t-l_t}\right] \text{ for } t\ne r,\ i_r\in[2] \right\}.
    \end{align*}
    We also have $m_{n}+l_{n}=m_{n+1}=k-s$ from Case 1.
    Since $G/{L_C}\cong \Z_{2^{k-s}}\times\Z_2$, $(2^{m_n+l_n},0)+L_C=L_C$, and
    \begin{align*}
        (\so+L_C)/L_C
        &= \left\{ \left(\sum_{t=0}^{n-1} i_t2^{m_t}+ \sum_{t\in[n]\setminus\{r\}}\alpha_{i_0,\dots,i_t}2^{m_t+l_t},\, \alpha_{i_0,\dots,i_r}\right) + \left(i_{n}2^{m_{n}},0\right) + L_C \, \middle| \, i_t\in\left[2^{l_t}\right] \right\}.
    \end{align*}
    Since $L_C=\<{(2^{k-s},0)}=\left\{\left(h2^{k-s},0\right)\, \middle| \, h\in\left[2^s\right]\right\}$, we have
    \begin{align*}
        \so
        &= \left\{ \left(\sum_{t=0}^{n-1} i_t2^{m_t}+ \sum_{t\in[n]\setminus\{r\}}\alpha_{i_0,\dots,i_t}2^{m_t+l_t},\, \alpha_{i_0,\dots,i_r}\right) + \left(i_{n}2^{m_{n}} + \alpha_{i_0,\dots,i_{n}}2^{k-s},0\right) \, \middle| \, i_t\in\left[2^{l_t}\right] \right\}\\
        &= \left\{ \left(\sum_{t=0}^{n} i_t2^{m_t}+ \sum_{t\in[n+1]\setminus\{r\}}\alpha_{i_0,\dots,i_t}2^{m_t+l_t},\, \alpha_{i_0,\dots,i_r}\right) \, \middle| \, i_t\in\left[2^{l_t}\right] \right\}\ \text{and}
    \end{align*}
    \begin{align*}
        C = & D\oplus L_C \\
        = & \left\{ \left(\sum_{t\in[n]\setminus\{r\}} i_t2^{m_t+l_t}+ \sum_{t=0}^{n-1}\beta_{i_0,\dots,i_t}2^{m_{t+1}},\, i_r\right) + \left(h2^{k-s},0\right) \, \middle| \, i_t\in\left[2^{m_{t+1}-m_t-l_t}\right] \text{ for } t\ne r,\right.\\ 
        & \quad \left.  i_r\in[2], h\in\left[2^s\right] \right\}\\
        = & \left\{ \left(\sum_{t\in[n]\setminus\{r\}} i_t2^{m_t+l_t}+ \sum_{t=0}^{n-1}\beta_{i_0,\dots,i_t}2^{m_{t+1}},\, i_r\right) + \left(h2^{k-s} + \beta_{i_0,\dots,i_{n-1},h}2^{k},0\right) \, \middle| \, i_t\in\left[2^{m_{t+1}-m_t-l_t}\right] \text{ for } t\ne r,\right.\\
        & \quad \left.\ i_r\in[2], h\in\left[2^s\right] \right\}.
    \end{align*}
    By setting $i_n=h$ and a new value for $m_{n+1}$ to be equal to $k$, we have
    \begin{align*}
        C &= \left\{ \left(\sum_{t\in[n+1]\setminus\{r\}} i_t2^{m_t+l_t}+ \sum_{t=0}^{n}\beta_{i_0,\dots,i_t}2^{m_{t+1}},\, i_r\right) \, \middle| \, i_t\in\left[2^{m_{t+1}-m_t-l_t}\right] \text{ for } t\ne r,\ i_r\in[2] \right\}.
    \end{align*}
    The pair $(\so,C)$ here is in the form of pair (a).

\smallskip
    \textsf{Subcase 2.3.} $L_C=\<{(2^{k-s},1)}$ for $t<k-m+1$.
    In this case, $G/{L_C}\cong \Z_{2^{k-s+1}}$, $|(\so+L_C)/L_C|=|\so|=2^m$, and $(\so+L_C)/L_C$ is periodic. 
    From the case when $\so$ is periodic in Theorem \ref{p^k}, we have
    \begin{align*}
        (\so+L_C)/L_C
        &= \left\{ \left(\sum_{t=0}^{n} \left(i_t2^{m_t}+\alpha_{i_0,\dots,i_t}2^{m_t+l_t}\right),0\right) + L_C \, \middle| \, i_t\in\left[2^{l_t}\right]\right\},\ \text{and}\\
        (D+L_C)/L_C &= \left\{ \left(\sum_{t=0}^{n-1} \left( i_tp^{m_t+l_t}+\beta_{i_0,\dots,i_t}2^{m_{t+1}}\right),0 \right) +L_C \, \middle| \, i_t\in\left[2^{m_{t+1}-m_t-l_t}\right] \right\}.
    \end{align*}
    We also have $m_{n}+l_{n}=m_{n+1}=k-s+1$ from the same case.
    Since $(2^{m_n+l_n},0)+L_C=L_C$,
    \begin{align*}
        (\so+L_C)/L_C
        &= \left\{ \left(\sum_{t=0}^{n} i_t2^{m_t}+\sum_{t=0}^{n-1}\alpha_{i_0,\dots,i_t}2^{m_t+l_t},0\right)+ L_C \, \middle| \, i_t\in\left[2^{l_t}\right] \right\}.
    \end{align*}
    Since $L_C=\<{(2^{k-s},1)}$, 
    \begin{align*}
        \so
        &= \left\{ \left(\sum_{t=0}^{n} i_t2^{m_t}+\sum_{t=0}^{n-1}\alpha_{i_0,\dots,i_t}2^{m_t+l_t},0\right) + \alpha_{i_0,\dots,i_{n}}\left(2^{k-s},1\right) \, \middle| \, i_t\in\left[2^{l_t}\right] \right\}.
    \end{align*}
    By setting $r=n$, we have
    \begin{align*}
        \so
        &= \left\{ \left(\sum_{t=0}^{n} i_t2^{m_t}+\sum_{t\in[n+1]\setminus\{r\}}\alpha_{i_0,\dots,i_t}2^{m_t+l_t}+\alpha_{i_0,\dots,i_r}2^{m_r+l_r-1},\alpha_{i_0,\dots,i_r}\right) \, \middle| \, i_t\in\left[2^{l_t}\right] \right\}.
    \end{align*}

    \noindent Now $L_C=\{(h2^{k-s},h)\mid h\in\left[2^s\right]\}$. Therefore,
    \begin{align*}
        C = & D\oplus L_C \\
        = & \left\{ \left(\sum_{t=0}^{n-1} \left( i_t2^{m_t+l_t}+\beta_{i_0,\dots,i_t}2^{m_{t+1}}\right),0\right) + (h2^{k-s},h) \, \middle| \, i_t\in\left[2^{m_{t+1}-m_t-l_t}\right], h\in\left[2^s\right] \right\}\\
        = & \left\{ \left(\sum_{t=0}^{n-1} \left( i_t2^{m_t+l_t}+\beta_{i_0,\dots,i_t}2^{m_{t+1}}\right),0\right) + \left(h2^{m_r+l_r-1} + \beta_{i_0,\dots,i_{n-1},h}2^{k},h\right) \, \middle| \,\; i_t\in\left[2^{m_{t+1}-m_t-l_t}\right], \right.\\ 
        & \quad \left. h\in\left[2^{m_{r+1}-m_r-l_r+1}\right] \right\}.
    \end{align*}
    By setting $i_r=h$, and a new value for $m_{n+1}$ to be equal to $k$, we have
    \begin{align*}
        C &= \left\{ \left(\sum_{t\in[n+1]\setminus\{r\}} i_t2^{m_t+l_t}+\sum_{t=0}^{n} \beta_{i_0,\dots,i_t}2^{m_{t+1}}+ i_r2^{m_r+l_r-1},i_r\right) \, \middle| \, i_t\in\left[2^{m_{t+1}-m_t-l_t}\right] \text{ for } t\ne r,\right.\\ 
        & \quad \left. i_r\in\left[2^{m_{r+1}-m_r-l_r+1}\right]\right\}.
    \end{align*}
    The pair $(\so,C)$ here is in the form of pair (b).

\smallskip
    \textsf{Subcase 2.4.} $L_C=\<{(2^{k-s+1},0),(0,1)}$ for $1<s<k-m+1$.
    In this case, $G/{L_C}\cong \Z_{2^{k-s+1}}$ and $(\so+L_C)/L_C$ is periodic. 
    From the case when $\so$ is periodic in Theorem \ref{p^k}, we have
    \begin{align*}
        (\so+L_C)/L_C
        &= \left\{ \left(\sum_{t=0}^{n} \left(i_t2^{m_t}+\alpha_{i_0,\dots,i_t}2^{m_t+l_t}\right),0\right) + L_C \, \middle| \, i_t\in\left[2^{l_t}\right]\right\},\ \text{and}\\
        (D+L_C)/L_C &= \left\{ \left(\sum_{t=0}^{n-1} \left( i_tp^{m_t+l_t}+\beta_{i_0,\dots,i_t}2^{m_{t+1}}\right),0 \right) +L_C \, \middle| \, i_t\in\left[2^{m_{t+1}-m_t-l_t}\right] \right\}.
    \end{align*}
    We also have $m_{n}+l_{n}=m_{n+1}=k-s+1$ from the same case.
    Since $(2^{m_n+l_n},0)+L_C=L_C$,
    \begin{align*}
        (\so+L_C)/L_C
        &= \left\{ \left(\sum_{t=0}^{n} i_t2^{m_t}+\sum_{t=0}^{n-1}\alpha_{i_0,\dots,i_t}2^{m_t+l_t},0\right)+ L_C \, \middle| \, i_t\in\left[2^{l_t}\right] \right\}.
    \end{align*}
    Since $L_C=\<{(2^{k-s+1},0),(0,1)}$, 
    \begin{align*}
        \so
        &= \left\{ \left(\sum_{t=0}^{n} i_t2^{m_t}+\sum_{t=0}^{n-1}\alpha_{i_0,\dots,i_t}2^{m_t+l_t},0\right) + \alpha_{i_0,\dots,i_{n}}\left(2^{k-s+1},0\right) + \gamma_{i_0,\dots,i_{n}}\left(0,1\right) \, \middle| \, i_t\in\left[2^{l_t}\right] \right\}\\
        &= \left\{ \left(\sum_{t=0}^{n} \left(i_t2^{m_t}+\alpha_{i_0,\dots,i_t}2^{m_t+l_t}\right),\gamma_{i_0,\dots,i_n}\right) \, \middle| \, i_t\in\left[2^{l_t}\right] \right\}.
    \end{align*}

    \noindent Now $L_C=\{(h2^{k-s+1},j)\mid h\in\left[2^{s-1}\right],\,j\in[2]\}$. Therefore,
    \begin{align*}
        C = & D\oplus L_C \\
        = & \left\{ \left(\sum_{t=0}^{n-1} \left( i_t2^{m_t+l_t}+\beta_{i_0,\dots,i_t}2^{m_{t+1}}\right),0\right) + (h2^{k-s+1},j) \, \middle| \, i_t\in\left[2^{m_{t+1}-m_t-l_t}\right], h\in\left[2^{s-1}\right], j\in[2] \right\}\\
        = & \left\{ \left(\sum_{t=0}^{n-1} \left( i_t2^{m_t+l_t}+\beta_{i_0,\dots,i_t}2^{m_{t+1}}\right),0\right) + \left(h2^{m_r+l_r} + \beta_{i_0,\dots,i_{n-1},h,j}2^{k},j\right) \, \middle| \, i_t\in\left[2^{m_{t+1}-m_t-l_t}\right], \right.\\ 
        & \quad \left. h\in\left[2^{s-1}\right], ,j\in[2] \right\}.
    \end{align*}
    By setting $r=n$, $i_n=h$, and a new value for $m_{n+1}$ to be equal to $k$, we have
    \begin{align*}
        C &= \left\{ \left(\sum_{t=0}^{n}i_t2^{m_t+l_t}+ \sum_{t=0}^{r-1}\beta_{i_0,\dots,i_t}2^{m_{t+1}}+\sum_{t=r}^{n}\beta_{i_0,\dots,i_t,j}2^{m_{t+1}},j\right) \, \middle| \, i_t\in\left[2^{m_{t+1}-m_t-l_t}\right], j\in[2]\right\}.
    \end{align*}
    The pair $(\so,C)$ here is in the form of pair (c).

\smallskip
    \textsf{Subcase 2.5.} $L_C=\<{(2^{k-s+1},0),(0,1)}$ for $t=1$.
    In this case $L_C=\<{(0,1)}$, $G/{L_C}\cong \Z_{2^k}$, and $(\so+L_C)/L_C$ is periodic. 
    From the case when $\so$ is periodic in Theorem \ref{p^k}, we have
    \begin{align*}
        (\so+L_C)/L_C
        &= \left\{ \left(\sum_{t=0}^{n} \left(i_t2^{m_t}+\alpha_{i_0,\dots,i_t}2^{m_t+l_t}\right),0\right) + L_C \, \middle| \, i_t\in\left[2^{l_t}\right]\right\},\ \text{and}\\
        (D+L_C)/L_C &= \left\{ \left(\sum_{t=0}^{n-1} \left( i_t2^{m_t+l_t}+\beta_{i_0,\dots,i_t}2^{m_{t+1}}\right),0 \right) +L_C \, \middle| \, i_t\in\left[2^{m_{t+1}-m_t-l_t}\right] \right\}.
    \end{align*}
    From the proof of Theorem \ref{p^k}, we also have $m_{n}+l_{n}=m_{n+1}=k$.
    Since $(2^{m_n+l_n},0)+L_C=L_C$,
    \begin{align*}
        (\so+L_C)/L_C
        &= \left\{ \left(\sum_{t=0}^{n-1} \left(i_t2^{m_t}+\alpha_{i_0,\dots,i_t}2^{m_t+l_t}\right),0\right) + \left(i_{n}2^{m_{n}},0\right) + L_C \, \middle| \, i_t\in\left[2^{l_t}\right] \right\}.
    \end{align*}
    Since $L_C=\<{(0,1)}=\{(0,j)\mid j\in[2]\}$,
    \begin{align*}
        \so
        &= \left\{ \left(\sum_{t=0}^{n-1} \left(i_t2^{m_t}+\alpha_{i_0,\dots,i_t}2^{m_t+l_t}\right),0\right) + \left(i_{n}2^{m_{n}},0\right) + \alpha_{i_0,\dots,i_{n}}\left(0,1\right) \, \middle| \, i_t\in\left[2^{l_t}\right] \right\}\\
        &= \left\{ \left(\sum_{t=0}^{n} i_t2^{m_t}+\sum_{t=0}^{n-1}\alpha_{i_0,\dots,i_t}2^{m_t+l_t},\alpha_{i_0,\dots,i_{n}}\right) \, \middle| \, i_t\in\left[2^{l_t}\right] \right\}.
    \end{align*}
    By setting $r=n$, we get
    \begin{align*}
        \so 
        &= \left\{ \left(\sum_{t=0}^{n} i_t2^{m_t}+\sum_{t\in[n+1]\setminus \{r\}}\alpha_{i_0,\dots,i_t}2^{m_t+l_t},\alpha_{i_0,\dots,i_{r}}\right) \, \middle| \, i_t\in\left[2^{l_t}\right] \right\}.
    \end{align*}
    Now, we will prove that $C$ is also in the form of $C$ in pair (a). We have
    \begin{align*}
        C &= D\oplus L_C \\
        &= \left\{ \left(\sum_{t=0}^{n-1} \left( i_t2^{m_t+l_t}+\beta_{i_0,\dots,i_t}2^{m_{t+1}}\right),0\right) + (0,j) \, \middle| \, i_t\in\left[2^{m_{t+1}-m_t-l_t}\right], j\in\left[2\right] \right\}\\
        &= \left\{ \left(\sum_{t=0}^{n-1} \left( i_t2^{m_t+l_t}+\beta_{i_0,\dots,i_t}2^{m_{t+1}}\right),0\right) + \left(\beta_{i_0,\dots,i_{n-1},j}2^{k},j\right) \, \middle| \, i_t\in\left[2^{m_{t+1}-m_t-l_t}\right], j\in\left[2\right] \right\}.
    \end{align*}
    Recall that we set $r=n$. Let $j=i_r$ and set a new value for $m_{n+1}$ to be equal to $k$.
    Then we have
    \begin{align*}
        C &= \left\{ \left(\sum_{t\in[n+1]\setminus\{r\}} i_t2^{m_t+l_t}+\sum_{t=0}^{n}\beta_{i_0,\dots,i_t}2^{m_{t+1}},i_r\right) \, \middle| \, i_t\in\left[2^{m_{t+1}-m_t-l_t}\right], \text{ for } t\ne r,\ i_r\in[2] \right\}.
    \end{align*}
    The pair $(\so,C)$ here is in the form of pair (a).
\end{proof}

\section{Perfect codes in Cayley graphs of $\Z_{p^k} \times \Z_2^l$}
\label{sec:p^k2^l}

In this section, we characterize Cayley graphs $\Cay(G, S)$ of Haj\'os groups of the form $G=\Z_{p^k} \times \Z_2^l$. These Haj\'os groups are $\Z_{p^k}$, $\Z_{p^k} \times \Z_2$, $\Z_{p^3} \times \Z_2^2$, $\Z_{p^2} \times \Z_2^3$, $\Z_{p} \times \Z_2^4$ and all their subgroups.
The first two groups, $\Z_{p^k}$, $\Z_{p^k} \times \Z_2$, and their subgroups are covered in the previous section. We are left to consider $\Z_{p^3} \times \Z_2^2$, $\Z_{p^2} \times \Z_2^3$, $\Z_{p} \times \Z_2^4$ and their subgroups.

The calculation required to prove $G=\so\oplus C$ for given $(\so, C)$ in this section is similar to the calculation in the proof of Theorem \ref{p,q,q}. So from this point onward, to prove that $\Cay(G, S)$ admits a perfect code, we will provide the perfect code $C$ but omit the calculation leading to $G=\so\oplus C$.

\begin{theorem}\label{p^2}
Let $G=\Z_{p^2}$ for a prime $p$ and let $S$ be a proper subset of $G\setminus \{0\}$ that generates $G$. Let $0\in C\subseteq G$. Then $C$ is a perfect code in $\Cay(G,S)$ if and only if $\so=\{i+\alpha_i p \mid i\ttk{p}\}$ for some $\alpha_i$'s in $\Z_{p^2}$ with $\alpha_0=0$ and $C=\{ip \mid i\ttk{p}\}$.
\end{theorem}
\begin{proof}
    Suppose $\Cay(G,S)$ admits a perfect code $C$. Since $|G|=p^2$ and $S$ is a proper subset of $G\setminus \{0\}$, we have $|\so|=p$. From Lemma \ref{prime deg}, $C$ is a subgroup of order $p$, so $C=\lc=p\Z_{p^2}$.
    From the factorization of $G/\lc$ given in Lemma \ref{quotient}, $(\so+\lc)/\lc=G/\lc=\Z_{p^2}/p\Z_{p^2}=\{i+\lc \mid i\ttk{p}\}$, and hence $\so=\{i+\alpha_i p \mid i\ttk{p}\}$, $\alpha_i\in \Z$. Since $0\in \so$, we have $\alpha_0=0$.
\end{proof}

\begin{theorem}\label{2.2}
Let $G=\Z_2\times\Z_2$ and let $S$ be a proper subset of $G\setminus \{0\}$ that generates $G$. Then $\Cay(G,S)$ does not admit a perfect code.
\end{theorem}
\begin{proof}
    Suppose $\Cay(G,S)$ admits a perfect code $C$. From Lemma \ref{ls<so}, $|\so|\geq3$.
    Since $G=\so\oplus C$, $|\so|$ divides $|G|=4$.
    This means $\so=G$, contradicting the assumption that $S$ a proper subset of $G\setminus \{0\}$.
\end{proof}

\begin{theorem}\label{p^3}
Let $G=\Z_{p^3}$ for a prime $p$ and let $S$ be a proper subset of $G\setminus \{0\}$ that generates $G$. Let $0\in C\subseteq G$. Then $C$ is a perfect code in $\Cay(G,S)$ if and only if $(\so, C)$ is one of the pairs in Table \ref{t.p^3} for some $\alpha_i$'s in $G$ with $\alpha_0=0$.
\end{theorem}

\rc\begin{longtable}{|c|l|l|}
\caption{Perfect codes in Cayley graphs of $\Z_{p^3}$}\label{t.p^3}\\
\fl
\tn & $\{i+\alpha_ip \mid i\ttk{p}\}$  & $\{ip^2 \mid i\ttk{p^2}\}$ \\
\tn & $\{i+\alpha_ip+jp^2 \mid i,j\ttk{p}\}$  & $\{ip+\alpha_i p^2 \mid i\ttk{p}\}$ \\
\tn & $\{i+\alpha_ip^2 \mid i\ttk{p^2}\}$  & $\{ip^2 \mid i\ttk{p}\}$ \\
\hline
\end{longtable}

\setcounter{cn}{1}\renewcommand{\Table}{\ref{t.p^3}}
\begin{proof}
    Suppose $\Cay(G,S)$ admits a perfect code $C$.
    Then $G=\so\oplus C$ and so $|\so|\cdot |C|=|G|=p^3$. Since $S$ is a perfect subset of $G\setminus \{0\}$, $\Cay(G,S)$ is not a complete graph. Hence $|C|\geq 2$ and therefore $2\leq |\so|, |C|<p^3$.
    
    If $|\so|=p$, then from Lemma \ref{prime deg}, $\so$ is aperiodic and $C$ is a subgroup of $G$ with order $p^2$, so $C=\lc=p\Z_{p^3}$.
    From the factorization of $G/\lc$ given in Lemma \ref{quotient}, $(\so+\lc)/\lc=G/\lc=\Z_{p^3}/p\Z_{p^3}=\{i+\lc \mid i\ttk{p}\}$, and hence $\so=\{i+\alpha_i p \mid i\ttk{p}\}$, $\alpha_i\in \Z$. \R

    If $|\so|=p^2$ and $\so$ is periodic, then $|\ls|=p$ by Lemma \ref{ls<so}, so $\ls=p^2\Z_{p^3}$.
    Let $\so=D\oplus\ls$ for some subset $D$. From the factorization of $G/\ls$ given in Lemma \ref{quotient}, we have $G/\ls=((D+\ls)/\ls)\oplus ((C+\ls)/\ls)$. Note that $(D+\ls)/\ls$ is aperiodic and $G/\ls\cong \Z_{p^2}$. From Theorem \ref{p^2}, we have $(D+\ls)/\ls=\{i+\alpha_i p +\ls \mid i\ttk{p}\}$ and $(C+\ls)/\ls=\{ip+\ls \mid i\ttk{p}\}$. Therefore, $\so=D\oplus\ls=\{i+\alpha_i p + jp^2 \mid i,j\ttk{p}\}$ and $C=\{ip+\alpha_i p^2 \mid i\ttk{p}\}$ where $\alpha_i\in \Z$. \R

    If $|\so|=p^2$ and $\so$ is aperiodic, then $|C|=p$ and $C$ is periodic. Hence $C=\lc=p^2\Z_{p^3}=\{ip^2 \mid i\ttk{p}\}$. From the factorization of $G/\lc$ given in Lemma \ref{quotient}, $(\so+\lc)/\lc=G/\lc=\Z_{p^3}/p^2\Z_{p^3}=\{i+\lc \mid i\ttk{p^2}\}$, and hence $\so=\{i+\alpha_i p^2 \mid i\ttk{p^2}\}$, $\alpha_i\in \Z$. \R
\end{proof}

\begin{theorem}\label{p^2.2}
Let $G=\Z_{p^2}\x\Z_q$ for distinct primes $p$ and $q$, and let $S$ be a proper subset of $G\setminus \{0\}$ that generates $G$. Let $0\in C\subseteq G$. Then $C$ is a perfect code in $\Cay(G,S)$ if and only if $(\so, C)$ is one of the pairs in Table \ref{t.p^2.2} for some $\alpha_i$'s and $\alpha_{ij}$'s in $\Z$ where $\alpha_0=\alpha_{00}=0$.
\end{theorem}

\rc\begin{longtable}{|c|l|l|}
\caption{Perfect codes in Cayley graphs of $\Z_{p^2}\x\Z_q$}
\label{t.p^2.2}
\\
\fl
 \tn & $\{(i+\alpha_ip,\alpha_i) \mid i\ttk{p}\}$ & $\{(ip,j) \mid i\ttk{p},\,j\ttk{q}\}$ \\
 \tn & $\{(\alpha_i,i) \mid i\ttk{q}\}$ & $\{(i,0) \mid i\ttk{p^2}\}$ \\
 \tn & $\{(i+jp,\alpha_i) \mid i,j\ttk{p}\}$ & $\{(\alpha_i p,i) \mid i\ttk{q}\}$ \\
 \tn & $\{(i,\alpha_i) \mid i\ttk{p^2}\}$ & $\{(0,i) \mid i\ttk{q}\}$ \\
 \tn & $\{(i+\alpha_i p,j) \mid i\ttk{p},\,j\ttk{q}\}$ & $\{(ip,\alpha_i) \mid i\ttk{p}\}$ \\
 \tn & $\{(\alpha_i+jp,i) \mid i\ttk{q}, \,j\ttk{p}\}$ & $\{(i+\alpha_i p,0) \mid i\ttk{p}\}$ \\
 \tn & $\{(i+\alpha_{ij}p,j) \mid i\ttk{p},\,j\ttk{q}\}$ & $\{(ip,0) \mid i\ttk{p}\}$ \\
\hline
\end{longtable}

\begin{proof}
    Suppose $C$ is a perfect code in $\Cay(G,S)$. Then $G=\so\oplus C$ and so $|\so|\cdot |C|=|G|=p^2q$. Since $S$ is a perfect subset of $G\setminus \{0\}$, $\Cay(G,S)$ is not a complete graph. Hence $|C|\geq 2$ and therefore $2\leq |\so|, |C|<p^2q$.

\smallskip
\textsf{Case 1.} $\so$ is aperiodic.
\smallskip

    Since $G$ is Haj\'os and $G=\so\oplus C$, in this case $C$ must be periodic, that is, $\lc$ is nontrivial.
    Write $C=\lc\oplus D$ for some subset $D$ of $G$ whose existence is ensured by Lemma \ref{period}.
    Then $|\lc|\cdot |D|=|C|$ and $2\leq |\lc|\leq |C|$.
    If $|\so|$ is a prime, then $|\lc|=|C|$ by Lemma \ref{prime deg} part (c).
    If $|\so|$ is not a prime, then $|\so|$ must be a multiple of two primes (either $p^2$ or $pq$). This means that $|C|=p^2q/|\so|$ is a prime, and hence $|\lc|= |C|$.
    Set $$ A=(\so+\lc)/\lc\quad\text{and}\quad B=(D+\lc)/\lc.$$
    Since $C$ is periodic, by Lemma \ref{quotient}, we have
    $$G/\lc=A\oplus B.$$
    Since $|\lc|=|C|$, we have $C=\lc$ as $0\in C$.
    This means we can take $D=\{0\}$ and hence $B=\lc/\lc$ and $G/\lc=A$.
    We use $C$ to get $A$ and then use Lemma \ref{d_i+l_i} to obtain $\so$.

    All possible values for $\left(|\so|,|C|,|\lc|\right)$ are given in the first column of Table \ref{tap.p^2.2}. 
    In the second column, we list $C$, $A$, and $\so$.
    The last column indicates which row in Table \ref{t.p^2.2} the result corresponds to.
    
\begin{longtable}{|c|l|c|}
\caption{The case when $G=\Z_{p^2}\x\Z_q$ and $\so$ is aperiodic}
\label{tap.p^2.2}
\\
\hline
\rowcolor{black!10}
\setcounter{cn}{1} $(|\so|,|C|,|\lc|)$ & \multicolumn{1}{c|}{$C$, $A$, and $\so$} & Result \\
\hline
\endhead
& $C=\{(ip,j) \mid i\ttk{p},\,j\ttk{q}\} $ & \\
$(p,pq,pq)$ 
& $A=\{(i,0)+\lc \mid i\ttk{p}\}$ 
& row 1 \\
& $\so=\{(i+\alpha_ip,\alpha_i) \mid i\ttk{p}\} $ & \\ \hline

& $C=\{(i,0) \mid i\ttk{p^2}\} $ & \\
$(q,p^2,p^2)$ 
& $A=\{(0,i)+\lc \mid i\ttk{q}\} $ 
& row 2 \\
& $\so=\{(\alpha_i,i) \mid i\ttk{q}\} $ & \\ \hline

& $C=\{(0,i) \mid i\ttk{q}\} $ & \\
$(p^2,q,q)$ 
& $A=\{(i,0)+\lc \mid i\ttk{p^2}\} $ 
& row 4 \\
& $\so=\{(i,\alpha_i) \mid i\ttk{p^2}\} $ & \\ \hline

& $C=\{(ip,0) \mid i\ttk{p}\} $ & \\
$(pq,p,p)$ 
& $A=\{(i,j)+\lc \mid i\ttk{p},\,j\ttk{q}\} $ 
& row 7 \\
& $\so=\{(i+\alpha_{ij}p,j) \mid i\ttk{p},\,j\ttk{q}\} $ & \\ \hline



\end{longtable}

\smallskip
\textsf{Case 2.} $\so$ is periodic.
\smallskip

    In this case we have $|\ls|\geq 2$ and $\so=D\oplus \ls$ for some subset $D$ of $G$ by Lemma \ref{period}.
    So $|D|\cdot|\ls|=|\so|$ and $1<|\ls|<|\so|$ by Lemma \ref{prime deg} part (b).
    Set $$X=(D+\ls)/\ls\quad\text{and}\quad Y=(C+\ls)/\ls.$$
    Since $\so$ is periodic, by Lemma \ref{quotient}, we have
    $$G/\ls=X\oplus Y$$ and $X$ is aperiodic in $G/\ls$.
    We also know that $X$ contains $\ls$ and $X$ generates $G/\ls$ by Lemma \ref{generate-zero}.
    This factorization enables us to make use of Lemma \ref{d_i+l_i} and previously established theorems to obtain the pair $(\so, C)$ via $(X,Y)$.
    
    All possible values for $\left(|\so|,|C|,|\ls|\right)$ are given in the first column of Table \ref{tp.p^2.2}. 
    In the second column we list $\ls$, $X$, $Y$, $\so$, and $C$.
    The third column shows which previous theorem is used to obtain $(X,Y)$.
    The last column indicates which row in Table \ref{t.p^2.2} the result corresponds to.

\begin{longtable}{|c|l|c|c|}
\caption{The case when $G=\Z_{p^2}\x\Z_q$ and $\so$ is periodic}
\label{tp.p^2.2}
\\
\flp

& $\ls=\<{(p,0)} $ & & \\
& $X=\{(i,\alpha_i)+\ls \mid i\ttk{p}\} $ & & \\
$(p^2,q,p)$ 
& $Y=\{(0,i)+\ls \mid i\ttk{q}\} $ 
& \ref{p.q} \row1 & row 3 \\
& $\so=\{(i+jp,\alpha_i) \mid i,j\ttk{p}\} $ & & \\
& $C=\{(\alpha_i p,i) \mid i\ttk{q}\} $ & & \\ \hline

& $\ls=\<{(0,1)} $ & & \\
& $X=\{(i+\alpha_i p,0)+\ls \mid i\ttk{p}\} $ & & \\
$(pq,p,q)$ 
& $Y=\{(ip,0)+\ls \mid i\ttk{p}\} $ 
& \ref{p^2} & row 5 \\
& $\so=\{(i+\alpha_i p,j) \mid i\ttk{p},\,j\ttk{q}\} $ & & \\
& $C=\{(ip,0)+\alpha_i(0,1) \mid i\ttk{p}\}=\{(ip,\alpha_i) \mid i\ttk{p}\} $ & & \\ \hline

& $\ls=\<{(p,0)} $ & & \\
& $X=\{(\alpha_i,i)+\ls \mid i\ttk{q}\} $ & & \\
$(pq,p,p)$ 
& $Y=\{(i,0)+\ls \mid i\ttk{p}\} $ 
& \ref{p.q} \row2 & row 6 \\
& $\so=\{(\alpha_i+jp,i) \mid i\ttk{q}, \,j\ttk{p}\} $ & & \\
& $C=\{(i+\alpha_i p,0) \mid i\ttk{p}\} $ & & \\ \hline

\end{longtable}
\end{proof}

\begin{theorem}\label{2.2.2}
Let $G=\Z_2\x\Z_2\x\Z_2$ and let $S$ be a proper subset of $G\setminus \{0\}$ that generates $G$. Let $0\in C\subseteq G$. Then $C$ is a perfect code in $\Cay(G,S)$ if and only if $\so=\{\alpha_{ij}\e1+i\e2+j\e3 \mid i,j\ttk{2}\}$ for some $\alpha_{ij}$'s in $\Z$ where $\alpha_{00}=0$ and $C=\{i\e1 \mid i\ttk{2}\}$.
\end{theorem}

\begin{proof}
    \setcounter{cn}{1}\renewcommand{\Table}{\ref{p^3}}
    Suppose $\Cay(G,S)$ admits a perfect code $C$.
    Then $G=\so\oplus C$ and so $|\so|\cdot |C|=|G|=8$.
    From lemma \ref{ls<so}, we have $|\so|>2$ and hence $(|\so|,|C|)=(4,2)$ as $|\so|\cdot|C|=8$ and $|C|\geq 2$.

    If $\so$ is periodic, then $|\ls|=2$ and $\ls=\left<\e1\right>$.    
    Let $\so=D\oplus\ls$ for some subset $D$. From the factorization of $G/\ls$ given in Lemma \ref{quotient}, we have $G/\ls=((D+\ls)/\ls)\oplus ((C+\ls)/\ls)$.
    From Lemma \ref{generate-zero}, we know that $(D+\ls)/\ls)$ is a generator of $G/\ls\cong\Z_2\times\Z_2$ and $\ls\in(D+\ls)/\ls)$.
    This case is not possible by Theorem \ref{2.2}.

    If $\so$ is aperiodic, then $|C|=2$ and $C$ is periodic, and hence $C=\lc=\<{\e1}$.
    From the factorization of $G/\lc$ given in Lemma \ref{quotient}, $(\so+\lc)/\lc=G/\lc=\{i\e2+j\e3+\lc \mid i,j\ttk{2}\}$, and therefore $\so=\{i\e2+j\e3+\alpha_{ij}\e1 \mid i,j\ttk{2}\}$, $\alpha_{ij}\in \Z$.
\end{proof}

\begin{remark}
    In Theorem \ref{2.2.2}, we can change $\e1$, $\e2$, and $\e3$ to any three generators of $\Z_2\x\Z_2\x\Z_2$.
\end{remark}

\begin{theorem}\label{p^3.2} 
Let $G=\Z_{p^3}\x\Z_2$ for an odd prime $p$ and let $S$ be a proper subset of $G\setminus \{0\}$ that generates $G$. Let $0\in C\subseteq G$. Then $C$ is a perfect code in $\Cay(G,S)$ if and only if $(\so, C)$ is one of the pairs in Table \ref{t.p^3.2} for some $\alpha_i$'s, $\alpha_{ij}$'s, and $\alpha_{ik}$'s in $\Z$ where $\alpha_0=\alpha_{00}=0$.
\end{theorem}

\rc\begin{longtable}{|c|l|l|}
\caption{Perfect codes in Cayley graphs of $\Z_{p^3}\x\Z_2$}
\label{t.p^3.2}
\\
\fl
\tn & $\{(i+\alpha_ip,\alpha_i) \mid i\ttk{p}\}$
& $\{(ip,j) \mid i\ttk{p^2},\,j\ttk{2}\}$ \\
\tn & $\{(\alpha_i,i)\mid i\ttk{2}\}$ 
& $\{(i,0) \mid i\ttk{p^3}\}$ \\
\tn & $\{(i+\alpha_ip+jp^2,\alpha_i) \mid i,j\ttk{p}\}$
& $\{(ip+\alpha_{ij}p^2,j) \mid i\ttk{p},\,j\ttk{2}\}$ \\
\tn & $\{(i+\alpha_ip+jp^2,\alpha_{ij}) \mid i,j\ttk{p}\}$
& $\{(ip+\alpha_i p^2,j) \mid i\ttk{p},\,j\ttk{2}\}$ \\
\tn & $\{(i+jp+\alpha_{ij}p^2,\alpha_i) \mid i,j\ttk{p}\}$
& $\{(\alpha_i p+jp^2,i) \mid i\ttk{2}, \,j\ttk{p}\}$ \\
\tn & $\{(i+\alpha_ip^2,\alpha_i) \mid i\ttk{p^2}\}$
& $\{(ip^2,j) \mid i\ttk{p},\,j\ttk{2}\}$ \\
\tn & $\{(i+\alpha_ip,j) \mid i\ttk{p},\,j\ttk{2}\}$
& $\{(ip^2,\alpha_i) \mid i\ttk{p^2}\}$ \\
\tn & $\{(\alpha_i+jp^2,i) \mid i\ttk{2},\,j\ttk{p}\}$
& $\{(i+\alpha_ip^2,0) \mid i\ttk{p^2}\}$ \\
\tn & $\{(i+\alpha_i p+\alpha_{ij}p^2,j) \mid i\ttk{p},\,j\ttk{2}\}$
& $\{(ip+jp^2,\alpha_i) \mid i,j\ttk{p}\}$ \\
\tn & $\{(\alpha_i+jp+\alpha_{ij}p^2,i) \mid i\ttk{2}, \,j\ttk{p}\}$
& $\{(i+\alpha_i p + j p^2,0) \mid i,j\ttk{p}\}$ \\
\tn & $\{(i + \alpha_{ij} p,j) \mid i\ttk{p},\,j\ttk{2}\}$
& $\{(ip,0) \mid i\ttk{p^2}\}$ \\
\tn & $\{(i+jp^2,\alpha_i) \mid i\ttk{p^2},\,j\ttk{p}\}$
& $\{(\alpha_i p^2,i) \mid i\ttk{2}\}$ \\
\tn & $\{(i+jp,\alpha_i) \mid i\ttk{p},\,j\ttk{p^2}\}$
& $\{(\alpha_i p,i) \mid i\ttk{2}\}$ \\
\tn & $\{(i,\alpha_i) \mid i\ttk{p^3}\}$
& $\{(0,i) \mid i\ttk{2}\}$ \\
\tn & $\{(i+\alpha_ip^2,j) \mid i\ttk{p^2},\,j\ttk{2}\}$
& $\{(i p^2,\alpha_i) \mid i\ttk{p}\}$ \\
\tn & $\{(i+\alpha_{ik}p+jp^2,k) \mid i,j\ttk{p},k\ttk{2}\}$
& $\{(ip + \alpha_i p^2,0) \mid i\ttk{p}\}$ \\
\tn & $\{(i+\alpha_ip+jp^2,k) \mid i,j\ttk{p},\,k\ttk{2}\}$
& $\{(ip + \alpha_i p^2,\alpha_i) \mid i\ttk{p}\}$ \\
\tn & $\{(\alpha_i+jp,i) \mid i\ttk{2},\,j\ttk{p^2}\}$
& $\{(i + \alpha_i p,0) \mid i\ttk{p}\}$ \\
\tn & $\{(i+\alpha_{ij}p^2,j) \mid i\ttk{p^2},\,j\ttk{2}\}$
& $\{(ip^2,0) \mid i\ttk{p}\}$ \\
\hline
\end{longtable}

\begin{proof}
    Suppose $C$ is a perfect code in $\Cay(G,S)$. Then $G=\so\oplus C$ and so $|\so|\cdot |C|=|G|=2p^3$. Since $S$ is a perfect subset of $G\setminus \{0\}$, $\Cay(G,S)$ is not a complete graph. Hence $|C|\geq 2$ and therefore $2\leq |\so|, |C|<2p^3$.

\smallskip
\textsf{Case 1.} $\so$ is aperiodic.
\smallskip

    Since $G$ is Haj\'os and $G=\so\oplus C$, in this case $C$ must be periodic, that is, $\lc$ is nontrivial.
    Write $C=\lc\oplus D$ for some subset $D$ of $G$ whose existence is ensured by Lemma \ref{period}.
    Then $|\lc|\cdot |D|=|C|$ and $2\leq |\lc|\leq |C|$.
    If $|\so|$ is a prime, then $|\lc|=|C|$ by Lemma \ref{prime deg} part (c).
    If $|\so|$ is not a prime, then $|\lc|$ could be any divisor of $|C|$.
    Set $$ A=(\so+\lc)/\lc\quad\text{and}\quad B=(D+\lc)/\lc.$$
    Since $C$ is periodic, by Lemma \ref{quotient}, we have
    $$G/\lc=A\oplus B$$ and $B$ is aperiodic in $G/\lc$.
    Since $G/\lc$ is isomorphic to a subgroup of $G$ and subgroups of Haj\'os group are also Haj\'os groups, $A$ is periodic in $G/\lc$.
    We also know that $A$ contains $\lc$ and $A$ generates $G/\lc$ by Lemma \ref{generate-zero}.
    Furthermore, $|A|=|\so|$ and $|B|=|D|$.
    Note that if $|\lc|=|C|$, then $C=\lc$ as $0\in C$.
    This means we can take $D=\{0\}$ and hence $B=\lc/\lc$ and $G/\lc=A$.
    We use $C$ to get $A$ and then use Lemma \ref{d_i+l_i} to obtain $\so$.
    If $|\lc|<|C|$, then this factorization enables us to make use of previously established theorems to obtain the pair $(\so, C)$ via $(A,B)$.

    The first columns of Table \ref{tap.p^3.2} represents all possible values for $\left(|\so|,|C|,|\lc|\right)$. 
    In the second column, we list $C$, $A$, and $\so$ when $|\lc|=|C|$ and we list $L_C$, $A$, $B$, $\so$, and $C$ when $|\lc|<|C|$.
    The third column shows which previous theorem is used to obtain $(A,B)$ when $|\lc|<|C|$.
    The last column indicates which row in Table \ref{t.p^3.2} the result corresponds to.
    
\begin{longtable}{|c|l|c|c|}
\caption{The case when $G=\Z_{p^3}\x\Z_2$ and $\so$ is aperiodic}
\label{tap.p^3.2}
\\
\flap

& $C=\{(ip,j) \mid i\ttk{p^2},\,j\ttk{2}\}$ & & \\
$(p,2p^2,2p^2)$ 
& $A=\{(i,0)+\lc \mid i\ttk{p}\}$ 
& & row 1 \\
& $\so=\{(i+\alpha_ip,\alpha_i) \mid i\ttk{p}\}$ & & \\ \hline

& $C=\{(i,0) \mid i\ttk{p^3}\}$ & & \\
$(2,p^3,p^3)$ 
& $A=\{(0,i)+\lc \mid i\ttk{2}\}$ 
& & row 2 \\
& $\so=\{(\alpha_i,i) \mid i\ttk{2}\}$ & & \\ \hline

& $\lc=\<{(0,1)}$ & & \\
& $A=\{(i+\alpha_ip+jp^2,0) + \lc \mid i,j\ttk{p}\}$ & & \\
$(p^2,2p,2)$ 
& $B=\{(ip+\alpha_i p^2,0) +\lc \mid i\ttk{p}\}$ 
& \ref{p^3} \row2 & row 4 \\
& $\so=\{(i+\alpha_ip+jp^2,\alpha_{ij}) \mid i,j\ttk{p}\}$ & & \\
& $C=\{(ip+\alpha_i p^2,j) \mid i\ttk{p},\,j\ttk{2}\}$ & & \\ \hline

& $\lc=\<{(p^2,0)} $ & & \\
& $A=\{(i+jp,\alpha_i) +\lc \mid i,j\ttk{p}\} $ & & \\
$(p^2,2p,p)$ 
& $B=\{(\alpha_i p,i) +\lc \mid i\ttk{2}\} $ 
& \ref{p^2.2} \row3 & row 5  \\
& $\so=\{(i+jp+\alpha_{ij}p^2,\alpha_i) \mid i,j\ttk{p}\} $ & & \\
& $C=\{(\alpha_i p+jp^2,i) \mid i\ttk{2}, \,j\ttk{p}\} $ & & \\ \hline

& $C=\{(ip^2,j) \mid i\ttk{p},\,j\ttk{2}\} $ & & \\
$(p^2,2p,2p)$ 
& $A=\{(i,0)+\lc \mid i\ttk{p^2}\} $ 
& & row 6 \\
& $\so=\{(i+\alpha_ip^2,\alpha_i) \mid i\ttk{p^2}\} $ & & \\ \hline

& $\lc=\<{(p^2,0)} $ & & \\
& $A=\{(i+\alpha_i p,j)+\lc \mid i\ttk{p},\,j\ttk{2}\} $ & & \\
$(2p,p^2,p)$ 
& $B=\{(ip,\alpha_i)+\lc \mid i\ttk{p}\} $ 
& \ref{p^2.2} \row5 & row 9 \\
& $\so=\{(i+\alpha_i p+\alpha_{ij}p^2,j) \mid i\ttk{p},\,j\ttk{2}\} $ & & \\
& $C=\{(ip+jp^2,\alpha_i) \mid i,j\ttk{p}\} $ & & \\ \hline

& $\lc=\<{(p^2,0)} $ & & \\
& $A=\{(\alpha_i+jp,i) +\lc \mid i\ttk{2}, \,j\ttk{p}\} $ & & \\
$(2p,p^2,p)$ 
& $B=\{(i+\alpha_i p,0)+\lc \mid i\ttk{p}\} $ 
& \ref{p^2.2} \row6 & row 10 \\
& $\so=\{(\alpha_i+jp+\alpha_{ij}p^2,i) \mid i\ttk{2}, \,j\ttk{p}\} $ & & \\
& $C=\{(i+\alpha_i p + j p^2,0) \mid i,j\ttk{p}\} $ & & \\ \hline

& $C=\{(ip,0) \mid i\ttk{p^2}\} $ & & \\
$(2p,p^2,p^2)$ 
& $A=\{(i,j)+\lc \mid i\ttk{p},\,j\ttk{2}\} $
& & row 11 \\
& $\so=\{(i + \alpha_{ij} p,j) \mid i\ttk{p},\,j\ttk{2}\} $ & & \\ \hline

& $C=\{(0,i) \mid i\ttk{2}\} $ & & \\
$(p^3,2,2)$ 
& $A=\{(i,0) + \lc \mid i\ttk{p^3}\} $ 
& & row 14 \\
& $\so=\{(i,\alpha_i) \mid i\ttk{p^3}\} $ & & \\ \hline

& $C=\{(ip^2,0) \mid i\ttk{p}\} $ & & \\
$(2p^2,p,p)$ 
& $A=\{(i,j)+\lc \mid i\ttk{p^2},\,j\ttk{2}\} $ 
& & row 19 \\
& $\so=\{(i+\alpha_{ij}p^2,j) \mid i\ttk{p^2},\,j\ttk{2}\} $ & & \\ \hline



\end{longtable}

\smallskip
\textsf{Case 2.} $\so$ is periodic.
\smallskip

    In this case we have $|\ls|\geq 2$ and $\so=D\oplus \ls$ for some subset $D$ of $G$ by Lemma \ref{period}.
    So $|D|\cdot|\ls|=|\so|$ and $1<|\ls|<|\so|$ by Lemma \ref{prime deg} part (b).
    Set $$X=(D+\ls)/\ls\quad\text{and}\quad Y=(C+\ls)/\ls.$$
    Since $\so$ is periodic, by Lemma \ref{quotient}, we have
    $$G/\ls=X\oplus Y$$ and $X$ is aperiodic in $G/\ls$.
    We also know that $X$ contains $\ls$ and $X$ generates $G/\ls$ by Lemma \ref{generate-zero}.
    This factorization enables us to make use of Lemma \ref{d_i+l_i} and previously established theorems to obtain the pair $(\so, C)$ via $(X,Y)$.
    
    The first columns of Table \ref{tp.p^3.2} represents all possible values for $\left(|\so|,|C|,|\ls|\right)$.
    In the second column we list $\ls$, $X$, $Y$, $\so$, and $C$.
    The third column shows which previous theorem is used to obtain $(X,Y)$.
    The last column indicates which row in Table \ref{t.p^3.2} the result corresponds to.

\begin{longtable}{|c|l|c|c|}
\caption{The case when $G=\Z_{p^3}\x\Z_2$ and $\so$ is periodic}
\label{tp.p^3.2}
\\
\flp

& $\ls=\<{(p^2,0)}$ & & \\
& $X=\{(i+\alpha_ip,\alpha_i)+\ls \mid i\ttk{p}\}$ & & \\
$(p^2,2p,p)$ 
& $Y=\{(ip,j)+\ls \mid i\ttk{p},\,j\ttk{2}\}$ 
& \ref{p^2.2} \row1 & row 3 \\
& $\so=\{(i+\alpha_ip+jp^2,\alpha_i) \mid i,j\ttk{p}\}$ & & \\
& $C=\{(ip+\alpha_{ij} p^2,j) \mid i\ttk{p},\,j\ttk{2}\}$ & & \\ \hline

& $\ls=\<{(0,1)} $ & & \\
& $X=\{(i+\alpha_ip,0)+\ls \mid i\ttk{p}\} $ & & \\
$(2p,p^2,2)$ 
& $Y=\{(ip^2,0)+\ls \mid i\ttk{p}\} $ 
& \ref{p^3} \row1 & row 7 \\
& $\so=\{(i+\alpha_ip,j) \mid i\ttk{p},\,j\ttk{2}\} $ & & \\
& $C=\{(ip^2,\alpha_i) \mid i\ttk{p^2}\} $ & & \\ \hline

& $\ls=\left<(p^2,0)\right> $ & & \\
& $X=\{(\alpha_i,i)+\ls \mid i\ttk{2}\} $ & & \\
$(2p,p^2,p)$ 
& $Y=\{(i,0)+\ls \mid i\ttk{p^2}\} $ 
& \ref{p^2.2} \row2 & row 8 \\
& $\so=\{(\alpha_i+jp^2,i) \mid i\ttk{2},\,j\ttk{p}\} $ & & \\
& $C=\{(i+\alpha_ip^2,0) \mid i\ttk{p^2}\} $ & & \\ \hline

& $\ls=\<{(p^2,0)} $ & & \\
& $X=\{(i,\alpha_i) +\ls \mid i\ttk{p^2}\} $ & & \\
$(p^3,2,p)$ 
& $Y=\{(0,i) +\ls \mid i\ttk{2}\} $ 
& \ref{p^2.2} \row4 & row 12 \\
& $\so=\{(i+jp^2,\alpha_i) \mid i\ttk{p^2},\,j\ttk{p}\} $ & & \\
& $C=\{(\alpha_i p^2,i) \mid i\ttk{2}\} $ & & \\ \hline

& $\ls=\<{(p,0)} $ & & \\
& $X=\{(i,\alpha_i) +\ls \mid i\ttk{p}\} $ & & \\
$(p^3,2,p^2)$ 
& $Y=\{(0,i) +\ls \mid i\ttk{2}\} $ 
& \ref{p.q} \row1 & row 13 \\
& $\so=\{(i+jp,\alpha_i) \mid i\ttk{p},\,j\ttk{p^2}\} $ & & \\
& $C=\{(\alpha_i p,i) \mid i\ttk{2}\} $ & & \\ \hline

& $\ls=\<{(0,1)} $ & & \\
& $X=\{(i+\alpha_ip^2,0) +\ls \mid i\ttk{p^2}\} $ & & \\
$(2p^2,p,2)$ 
& $Y=\{(ip^2,0) +\ls \mid i\ttk{p}\} $ 
& \ref{p^3} \row3 & row 15 \\
& $\so=\{(i+\alpha_ip^2,j) \mid i\ttk{p^2},\,j\ttk{2}\} $ & & \\
& $C=\{(i p^2,\alpha_i) \mid i\ttk{p}\} $ & & \\ \hline

& $\ls=\<{(p^2,0)} $ & & \\
& $X=\{(i+\alpha_{ij}p,j) +\ls \mid i\ttk{p},j\ttk{2}\} $ & & \\
$(2p^2,p,p)$ 
& $Y=\{(ip,0) +\ls \mid i\ttk{p}\} $ 
& \ref{p^2.2} \row7 & row 16 \\
& $\so=\{(i+\alpha_{ik}p+jp^2,k) \mid i,j\ttk{p},k\ttk{2}\} $ & & \\
& $C=\{(ip + \alpha_i p^2,0) \mid i\ttk{p}\} $ & & \\ \hline

& $\ls=\<{(p^2,1)} $ & & \\
& $X=\{(i+\alpha_ip,0) +\ls \mid i\ttk{p}\} $ & & \\
$(2p^2,p,2p)$ 
& $Y=\{(ip,0) +\ls \mid i\ttk{p}\} $ 
& \ref{p^2} & row 17 \\
& $\so=\{(i+\alpha_ip+jp^2,k) \mid i,j\ttk{p},\,k\ttk{2}\} $ & & \\
& $C=\{(ip + \alpha_i p^2,\alpha_i) \mid i\ttk{p}\} $ & & \\ \hline

& $\ls=\<{(p,0)} $ & & \\
& $X=\{(\alpha_i,i) +\ls \mid i\ttk{2}\} $ & & \\
$(2p^2,p,p^2)$ 
& $Y=\{(i,0) +\ls \mid i\ttk{p}\} $ 
& \ref{p.q} \row2 & row 18 \\
& $\so=\{(\alpha_i+jp,i) \mid i\ttk{2},\,j\ttk{p^2}\} $ & & \\
& $C=\{(i + \alpha_i p,0) \mid i\ttk{p}\} $ & & \\ \hline


\end{longtable}
\end{proof}

\begin{theorem}\label{p^2.2.2}
Let $G=\Z_{p^2}\x\Z_2\x\Z_2$ for an odd prime $p$ and let $S$ be a proper subset of $G\setminus \{0\}$ that generates $G$. Let $0\in C\subseteq G$. Then $C$ is a perfect code in $\Cay(G,S)$ if and only if $(\so, C)$ is one of the pairs in Table \ref{t.p^2.2.2} for some $\alpha_i$'s, $\alpha_{ij}$'s, $\alpha_{ijk}$'s, $\beta_i$'s in $\Z$ where $\alpha_0=\alpha_{00}=\alpha_{000}=\beta_0=0$.
\end{theorem}

\rc\begin{longtable}{|c|l|l|}
\caption{Perfect codes in Cayley graphs of $\Z_{p^2}\x\Z_2\x\Z_2$}
\label{t.p^2.2.2}
\\
\fl
\tn & $\{(i+\alpha_i p)\e1+\alpha_i\e2+\beta_i\e3 \mid i\ttk{p}\}$
& $\{ip\e1+j\e2+k\e3 \mid i\ttk{p},\,j,k\ttk{2}\}$ \\
\tn & $\{(i+jp)\e1+\alpha_i\e2+\beta_i\e3 \mid i,j\ttk{p}\}$
& $\{\alpha_{ij}p\e1+i\e2+j\e3 \mid i,j\ttk{2}\}$ \\
\tn & $\{(i+jp)\e1 + \alpha_{ij}\e2 + \alpha_i\e3  \mid i,j\ttk{p}\}$
& $\{\alpha_i p\e1 + j\e2 + i\e3 \mid i,j\ttk{2}\}$ \\
\tn & $\{i\e1+\alpha_i\e2+\beta_i\e3 \mid i\ttk{p^2}\}$
& $\{i\e2+j\e3 \mid i,j\ttk{2}\}$ \\
\tn & $\{i+\alpha_ip\e1 + j\e2 + \alpha_i\e3 \mid i\ttk{p},\,j\ttk{2}\}$
& $\{ip\e1 + \alpha_{ij}\e2 + j\e3 \mid i\ttk{p},\,j\ttk{2}\}$ \\
\tn & $\{i+\alpha_{ij}p\e1 + j\e2 + \alpha_i\e3 \mid i\ttk{p},\,j\ttk{2}\}$
& $\{jp\e1+\alpha_i\e2+i\e3 \mid i\ttk{2},\,j\ttk{p} \}$ \\
\tn & $\{i+\alpha_i p\e1+\alpha_{ij}\e2+j\e3 \mid i\ttk{p},\,j\ttk{2}\}$
& $\{ ip\e1+j\e2+\alpha_i\e3 \mid i\ttk{p},\,j\ttk{2}\}$ \\
\tn & $\{\alpha_i+jp\e1 +\alpha_{ij}\e2 + i\e3 \mid i\ttk{2}, \,j\ttk{p}\}$
& $\{ i+\alpha_i p\e1+j\e2 \mid i\ttk{p},\,j\ttk{2}\}$ \\
\tn & $\{i+\alpha_{ij}p\e1 + \alpha_{ij}\e2 + j\e3 \mid i\ttk{p},\,j\ttk{2}\}$
& $\{ip\e1+j\e2 \mid i\ttk{p},\,j\ttk{2} \}$ \\
\tn & $\{\alpha_i\e1 + j\e2 + i\e3 \mid i,j\ttk{2}\}$
& $\{i\e1+\alpha_i \e2 \mid i\ttk{p^2}\}$ \\
\tn & $\{(\alpha_i+\alpha_{ij}p)\e1 + j\e2 + i\e3 \mid i,j\ttk{2}\}$
& $\{(i+jp)\e1 + \alpha_i\e2 \mid i,j\ttk{p}\}$ \\
\tn & $\{\alpha_{ij}\e1+i\e2+j\e3 \mid i,j\ttk{2}\}$
& $\{i\e1 \mid i\ttk{p^2} \}$ \\
\tn & $\{i\e1 + j\e2 + \alpha_i\e3 \mid i\ttk{p^2},\,j\ttk{2}\}$
& $\{\alpha_i\e2 + i\e3 \mid i\ttk{2}\}$ \\
\tn & $\{(i+kp)\e1 +\alpha_{ij}\e2 + j\e3 \mid i,k\ttk{p},j\ttk{2}\}$
& $\{ \alpha_i p\e1 + i\e2 \mid i\ttk{2}\}$ \\
\tn & $\{(i+jp)\e1 + k\e2 + \alpha_i \e3 \mid i,j\ttk{p},\,k\ttk{2}\}$
& $\{\alpha_i p\e1 + \alpha_i\e2 + i\e3 \mid i\ttk{2} \}$ \\
\tn & $\{i\e1+\alpha_{ij}\e2+j\e3 \mid i\ttk{p^2},\,j\ttk{2}\}$
& $\{ i\e2 \mid i\ttk{2}\}$ \\
\tn & $\{(i+\alpha_{ij}p)\e1 + k\e2 + j\e3 \mid i\ttk{p},\,j,k\ttk{2}\}$
& $\{ip\e1 + \alpha_i\e2 \mid i\ttk{p} \}$ \\
\tn & $\{(\alpha_{ij}+kp)\e1+i\e2+j\e3 \mid i,j\ttk{2},\,k\ttk{p}\}$
& $\{ (i+\alpha_ip)\e1 \mid i\ttk{p} \}$ \\
\tn & $\{(\alpha_i+jp)\e1 + k\e2 + i \e3 \mid i,k\ttk{2},\,j\ttk{p} \}$
& $\{(i+\alpha_ip)\e1 + \alpha_i\e2 \mid i\ttk{p} \}$ \\
\tn & $\{(i+\alpha_ip)\e1+j\e2+k\e3 \mid i\ttk{p},\,j,k\ttk{2} \}$
& $\{ip\e1+\alpha_i\e2+\beta_i\e3  \mid i\ttk{p} \}$ \\
\tn & $\{(i+\alpha_{ijk}p)\e1+j\e2+k\e3 \mid i\ttk{p},\,j,k\ttk{2}\}$
& $\{ip\e1  \mid i\ttk{p} \}$ \\
\hline
\end{longtable}

\begin{proof}
    Suppose $C$ is a perfect code in $\Cay(G,S)$. Then $G=\so\oplus C$ and so $|\so|\cdot |C|=|G|=4p^2$.
    Since $G$ is not a cyclic group, $|\so|>2$ by Lemma \ref{ls<so}.
    Therefore, $2<|\so|<4p^2$.

\smallskip
\textsf{Case 1.} $\so$ is aperiodic.
\smallskip

    Since $G$ is Haj\'os and $G=\so\oplus C$, in this case $C$ must be periodic, that is, $\lc$ is nontrivial.
    Write $C=\lc\oplus D$ for some subset $D$ of $G$ whose existence is ensured by Lemma \ref{period}.
    Then $|\lc|\cdot |D|=|C|$ and $2\leq |\lc|\leq |C|$.
    If $|\so|$ is a prime, then $|\lc|=|C|$ by Lemma \ref{prime deg} part (c).
    If $|\so|$ is not a prime, then $|\lc|$ can be any divisor of $|C|$.
    Set $$A=(\so+\lc)/\lc\quad\text{and}\quad B=(D+\lc)/\lc.$$
    Since $C$ is periodic, by Lemma \ref{quotient}, we have
    $$G/\lc=A\oplus B$$ and $B$ is aperiodic in $G/\lc$.
    Since $G/\lc$ is isomorphic to a subgroup of $G$ and subgroups of Haj\'os group are also Haj\'os groups, $A$ is periodic in $G/\lc$.
    We also know that $A$ contains $\lc$ and $A$ generates $G/\lc$ by Lemma \ref{generate-zero}.
    Furthermore, $|A|=|\so|$ and $|B|=|D|$.
    Note that if $|\lc|=|C|$, then $C=\lc$ as $0\in C$.
    This means we can take $D=\{0\}$ and hence $B=\lc/\lc$ and $G/\lc=A$.
    We use $C$ to get $A$ and then use Lemma \ref{d_i+l_i} to obtain $\so$.
    If $|\lc|<|C|$, then this factorization enables us to make use of previously established theorems to obtain the pair $(\so, C)$ via $(A,B)$.

    The columns of Table \ref{tap.p^2.2.2} represents all possible values for $\left(|\so|,|C|,|\lc|\right)$; $C$, $A$, and $\so$ when $|\lc|=|C|$ or $L_C$, $A$, $B$, $\so$, and $C$ when $|\lc|<|C|$; which previous theorem is used to obtain $(A,B)$ when $|\lc|<|C|$; and which row in Table \ref{t.p^2.2.2} the result corresponds to.
    
\begin{longtable}{|c|l|c|c|}
\caption{The case when $G=\Z_{p^2}\x\Z_2\x\Z_2 $ and $\so$ is aperiodic}
\label{tap.p^2.2.2}\\
\flap

& $C=\<{p\e1+\e2,\e3} $ & & \\
$(p,4p,4p)$ 
& $A=\{i\e1+\lc \mid i\ttk{p}\} $ 
& & row 1 \\
& $\so=\{(i+\alpha_ip)\e1+\alpha_i\e2+\beta_i\e3 \mid i\ttk{p}\} $ & & \\ \hline

& $\lc=\<{\e2} $ & & \\
& $A=\{(i+jp)\e1 + \alpha_i\e3 + \lc \mid i,j\ttk{p}\} $ & & \\
$(p^2,4,2)$ 
& $B=\{\alpha_i p\e1 + i\e3 + \lc \mid i\ttk{2}\} $ 
& \ref{p^2.2} \row3 & row 3 \\
& $\so=\{(i+jp)\e1 + \alpha_{ij}\e2 + \alpha_i\e3 \mid i,j\ttk{p}\} $ & & \\
& $C=\{\alpha_i p\e1 + j\e2 + i\e3  \mid i,j\ttk{2}\} $ & & \\ \hline

& $C=\<{\e2,\e3} $ & & \\
$(p^2,4,4)$ 
& $A=\{i\e1+\lc \mid i\ttk{p^2}\} $ 
& & row 4 \\
& $\so=\{i\e1+\alpha_i\e2+\beta_i\e3 \mid i\ttk{p^2}\} $ & & \\ \hline

& $\lc=\<{p\e1} $ & & \\
& $A=\{i\e1+j\e2+\alpha_i\e3 +\lc \mid i\ttk{p},\,j\ttk{2}\} $ & & \\
$(2p,2p,p)$ 
& $B=\{\alpha_i\e2+i\e3 + \lc \mid i\ttk{2}\} $ 
& \ref{p.2.2} \row4 & row 6 \\
& $\so=\{(i+\alpha_{ij}p)\e1+j\e2+\alpha_i\e3 \mid i\ttk{p},\,j\ttk{2}\} $ & & \\
& $C=\{\alpha_i\e2+i\e3 + jp\e1 \mid i\ttk{2}\,j\ttk{p} \} $ & & \\ \hline

& $\lc=\<{\e2} $ & & \\
& $A=\{(i+\alpha_i p)\e1+j\e3+\lc \mid i\ttk{p},\,j\ttk{2}\} $ & & \\
$(2p,2p,2)$ 
& $B=\{ip\e1+\alpha_i\e3+\lc \mid i\ttk{p}\} $ 
& \ref{p^2.2} \row5 & row 7 \\
& $\so=\{(i+\alpha_i p)\e1+j\e3+\alpha_{ij}\e2 \mid i\ttk{p},\,j\ttk{2}\} $ & & \\
& $C=\{ip\e1+j\e2+\alpha_i\e3 \mid i\ttk{p},\,j\ttk{2}\} $ & & \\ \hline

& $\lc=\<{\e2} $ & & \\
& $A=\{(\alpha_i+jp)\e1 + i\e3 +\lc \mid i\ttk{2}, \,j\ttk{p}\} $ & & \\
$(2p,2p,2)$ 
& $B=\{(i+\alpha_i p)\e1+\lc \mid i\ttk{p}\} $ 
& \ref{p^2.2} \row6 & row 8 \\
& $\so=\{(\alpha_i+jp)\e1 + i\e3 +\alpha_{ij}\e2 \mid i\ttk{2}, \,j\ttk{p}\} $ & & \\
& $C=\{(i+\alpha_i p)\e1+j\e2 \mid i\ttk{p},\,j\ttk{2}\} $ & & \\ \hline

& $C=\<{p\e1+\e2} $ & & \\
$(2p,2p,2p)$ 
& $A=\{i\e1 + j\e3 + \lc \mid i\ttk{p},\,j\ttk{2}\} $ 
& & row 9 \\
& $\so=\{(i+\alpha_{ij}p)\e1 + \alpha_{ij}\e2 + j\e3 \mid i\ttk{p},\,j\ttk{2}\} $ & & \\ \hline

& $\lc=\<{p\e1} $ & & \\
& $A=\{\alpha_i\e1 + j\e2 + i\e3 + \lc \mid i,j\ttk{2}\} $ & & \\
$(4,p^2,p)$ 
& $B=\{i\e1 + \alpha_i\e2 + \lc \mid i\ttk{p}\} $ 
& \ref{p.2.2} \row6 & row 11 \\
& $\so=\{(\alpha_i+\alpha_{ij}p)\e1 + j\e2 + i\e3 \mid i,j\ttk{2}\} $ & & \\
& $C=\{(i+jp)\e1 + \alpha_i\e2 \mid i,j\ttk{p}\} $ & & \\ \hline

& $C=\<{\e1} $ & & \\
$(4,p^2,p^2)$ 
& $A=\{i\e2+j\e3 + \lc \mid i,j\ttk{2}\} $ 
& & row 12 \\
& $\so=\{\alpha_{ij}\e1+i\e2+j\e3 \mid i,j\ttk{2}\} $ & & \\ \hline

& $C=\<{\e2} $ & & \\
$(2p^2,2,2)$ 
& $A=\{i\e1 + j\e3 + \lc \mid i\ttk{p^2},\,j\ttk{2}\} $ 
& & row 16 \\
& $\so=\{i\e1+j\e3+\alpha_{ij}\e2 \mid i\ttk{p^2},\,j\ttk{2}\} $ & & \\ \hline

& $C=\<{p\e1} $ & & \\
$(4p,p,p)$ 
& $A=\{i\e1+j\e2+k\e3 + \lc \mid i\ttk{p},\,j,k\ttk{2}\} $ 
& & row 21 \\
& $\so=\{(i+\alpha_{ijk}p)\e1+j\e2+k\e3 \mid i\ttk{p},\,j,k\ttk{2}\} $ & & \\ \hline



\end{longtable}

\smallskip
    \textsf{Case 2.} $\so$ is periodic.
    \smallskip

    In this case we have $|\ls|\geq 2$ and $\so=D\oplus \ls$ for some subset $D$ of $G$ by Lemma \ref{period}.
    So $|D|\cdot|\ls|=|\so|$ and $1<|\ls|<|\so|$ by Lemma \ref{prime deg} part (b).
    Set $$X=(D+\ls)/\ls\quad\text{and}\quad Y=(C+\ls)/\ls.$$
    Since $\so$ is periodic, by Lemma \ref{quotient}, we have
    $$G/\ls=X\oplus Y$$ and $X$ is aperiodic in $G/\ls$.
    We also know that $X$ contains $\ls$ and $X$ generates $G/\ls$ by Lemma \ref{generate-zero}.
    This factorization enables us to make use of Lemma \ref{d_i+l_i} and previously established theorems to obtain the pair $(\so, C)$ via $(X,Y)$.
    
    The columns of Table \ref{tp.p^2.2.2} represents all possible values of $\left(|\so|,|C|,|\ls|\right)$; $\ls$, $X$, $Y$, $\so$, and $C$; which previous theorem is used to obtain $(X,Y)$; and which row in Table \ref{t.p^2.2.2} the result corresponds to. 
    Note that in some cases, for example when $\left(|\so|,|C|,|\ls|\right)=(2p,2p,p)$, no $S_0$ and $C$ exist since no $X$ exist from the previous theorem given in the third column.

\begin{longtable}{|c|l|c|c|}
\caption{The case when $G=\Z_{p^2}\x\Z_2\x\Z_2 $ and $\so$ is periodic}
\label{tp.p^2.2.2}
\\
\flp

& $\ls=\<{p\e1} $ & & \\
& $X=\{i\e1+\alpha_i\e2+\beta_i\e3+\ls \mid i\ttk{p}\} $ & & \\
$(p^2,4,p)$ 
& $Y=\{i\e1+j\e2+\ls \mid i,j\ttk{2}\} $ 
& \ref{p.2.2} \row1 & row 2 \\
& $\so=\{(i+jp)\e1+\alpha_i\e2+\beta_i\e3 \mid i,j\ttk{p}\} $ & & \\
& $C=\{\alpha_{ij}p\e1+i\e2+j\e3 \mid i,j\ttk{2}\} $ & & \\ \hline

& $\ls=\<{\e2} $ & & \\
& $X=\{(i+\alpha_ip)\e1 + \alpha_i\e3 + \ls \mid i\ttk{p}\} $ & & \\
$(2p,2p,2)$ 
& $Y=\{ip\e1 + j\e3 + \ls \mid i\ttk{p},\,j\ttk{2}\} $ 
& \ref{p^2.2} \row1 & row 5 \\
& $\so=\{(i+\alpha_ip)\e1 + j\e2 + \alpha_i\e3 \mid i\ttk{p},\,j\ttk{2}\} $ & & \\
& $C=\{ip\e1 + \alpha_{ij}\e2 + j\e3 \mid i\ttk{p},\,j\ttk{2}\} $ & & \\ \hline

$(2p,2p,p)$ 
& Not possible since no $X$ exists
& \ref{p.2.2}  & \\ \hline

& $\ls=\<{\e2} $ & & \\
& $X=\{\alpha_i\e1 + i\e3 + \ls \mid i\ttk{2}\} $ & & \\
$(4,p^2,2)$ 
& $Y=\{i\e1+\ls \mid i\ttk{p^2}\} $ 
& \ref{p^2.2} \row2 & row 10 \\
& $\so=\{\alpha_i\e1 + j\e2 + i\e3  \mid i,j\ttk{2}\} $ & & \\
& $C=\{i\e1+\alpha_i \e2 \mid i\ttk{p^2}\} $ & & \\ \hline

& $\ls=\<{\e2} $ & & \\
& $X=\{i\e1 + \alpha_i\e3 +\ls \mid i\ttk{p^2}\} $ & & \\
$(2p^2,2,2)$ 
& $Y=\{ i\e3 + \ls \mid i\ttk{2}\} $ 
& \ref{p^2.2} \row4 & row 13 \\
& $\so=\{i\e1 + j\e2 + \alpha_i\e3 \mid i\ttk{p^2},\,j\ttk{2}\} $ & & \\
& $C=\{\alpha_i\e2 + i\e3 \mid i\ttk{2}\} $ & & \\ \hline

& $\ls=\<{p\e1} $ & & \\
& $X=\{i\e1 +\alpha_{ij}\e2 + j\e3 + \ls \mid i\ttk{p},j\ttk{2}\} $ & & \\
$(2p^2,2,p)$ 
& $Y=\{ i\e2 + \ls \mid i\ttk{2}\} $ 
& \ref{p.2.2} \row5 & row 14 \\
& $\so=\{(i+kp)\e1 +\alpha_{ij}\e2 + j\e3 \mid i,k\ttk{p},j\ttk{2}\} $ & & \\
& $C=\{ i\e2 + \alpha_i p\e1 \mid i\ttk{2}\} $ & & \\ \hline

& $\ls=\<{p\e1+\e2} $ & & \\
& $X=\{i\e1 + \alpha_i \e3 + \ls \mid i\ttk{p}\} $ & & \\
$(2p^2,2,2p)$ 
& $Y=\{i\e3 +\ls \mid i\ttk{2} \} $ 
& \ref{p.q} \row1 & row 15 \\
& $\so=\{i\e1 + \alpha_i \e3 + j p\e1 + k\e2 \mid i,j\ttk{p},\,k\ttk{2}\} $ & & \\
& $C=\{i\e3 + \alpha_i p\e1 +\alpha_i\e2 \mid i\ttk{2} \} $ & & \\ \hline

$(2p^2,2,p^2)$ 
& Not possible since no $X$ exists
& \ref{2.2}  & \\ \hline

& $\ls=\<{\e2} $ & & \\
& $X=\{(i+\alpha_{ij}p)\e1 + j\e3 +\ls \mid i\ttk{p},\,j\ttk{2}\} $ & & \\
$(4p,p,2)$ 
& $Y=\{ip\e1 + \ls \mid i\ttk{p} \} $ 
& \ref{p^2.2} \row7 & row 17 \\
& $\so=\{(i+\alpha_{ij}p)\e1 + k\e2 + j\e3 \mid i\ttk{p},\,j,k\ttk{2}\} $ & & \\
& $C=\{ip\e1 + \alpha_i\e2 \mid i\ttk{p} \} $ & & \\ \hline

& $\ls=\<{p\e1} $ & & \\
& $X=\{\alpha_{ij}\e1+i\e2+j\e3 + \ls \mid i,j\ttk{2}\} $ & & \\
$(4p,p,p )$ 
& $Y=\{ i\e1 + \ls \mid i\ttk{p} \} $ 
& \ref{p.2.2} \row7 & row 18 \\
& $\so=\{(\alpha_{ij}+kp)\e1+i\e2+j\e3 \mid i,j\ttk{2},\,k\ttk{p}\} $ & & \\
& $C=\{ (i+\alpha_ip)\e1 \mid i\ttk{p} \} $ & & \\ \hline

& $\ls=\<{p\e1+\e2} $ & & \\
& $X=\{\alpha_i\e1 + i \e3 + \ls \mid i\ttk{2}\} $ & & \\
$(4p,p,2p)$ 
& $Y=\{i\e1 +\ls \mid i\ttk{p} \} $ 
& \ref{p.q} \row2 & row 19 \\
& $\so=\{(\alpha_i+jp)\e1 + k\e2 + i\e3 \mid i,k\ttk{2},\,j\ttk{p} \} $ & & \\
& $C=\{(i+\alpha_ip)\e1 + \alpha_i\e2 \mid i\ttk{p} \} $ & & \\ \hline

& $\ls=\<{\e2,\e3} $ & & \\
& $X=\{(i+\alpha_ip)\e1 + \ls \mid i\ttk{p}\} $ & & \\
$(4p,p,4)$ 
& $Y=\{ip\e1 +\ls \mid i\ttk{p} \} $ 
& \ref{p^2} & row 20 \\
& $\so=\{(i+\alpha_ip)\e1+j\e2+k\e3 \mid i\ttk{p},\,j,k\ttk{2} \} $ & & \\
& $C=\{ip\e1+\alpha_i\e2+\beta_i\e3  \mid i\ttk{p} \} $ & & \\ \hline


\end{longtable}
\end{proof}

\begin{remark}
    In Theorem \ref{p^2.2.2}, we can change $\e2$ and $\e3$ to any two generators of $\{0\}\x\Z_2\x\Z_2$.
\end{remark}

\begin{theorem}\label{p.2.2.2}
Let $G=\Z_p\x\Z_2\x\Z_2\x\Z_2$ for an odd prime $p$ and let $S$ be a proper subset of $G\setminus \{0\}$ that generates $G$. Let $0\in C\subseteq G$. Then $C$ is a perfect code in $\Cay(G,S)$ if and only if $(\so, C)$ is one of the pairs in Table \ref{t.p.2.2.2} for some $\alpha_i$'s, $\alpha_{ij}$'s, $\alpha_{ijk}$'s, $\beta_i$'s in $\Z$ where $\alpha_0=\alpha_{00}=\alpha_{000}=\beta_0=0$.
\end{theorem}
\rc\begin{longtable}{|c|l|l|}
\caption{Perfect codes in Cayley graphs of $\Z_p\x\Z_2\x\Z_2\x\Z_2$}
\label{t.p.2.2.2}
\\
\fl
\tn & $\{i\e1+\alpha_i\e2+\beta_i\e3+\gamma_i\e4 \mid i\ttk{p}\}$
& $\{i\e2+j\e3+k\e4 \mid i,j,k\ttk{2} \}$ \\
\tn & $\{i\e1 + j\e2 + \alpha_i\e3 + \beta_i\e4 \mid i\ttk{p},j\ttk{2}\}$
& $\{\alpha_{ij}\e2+i\e3 +j\e4 \mid i,j\ttk{2}\}$ \\
\tn & $\{i\e1+\alpha_{ij}\e2+j\e3+\alpha_i\e4 \mid i\ttk{p},\,j\ttk{2}\}$
& $\{+ j\e2+\alpha_i\e3+i\e4  \mid i,j\ttk{2}\}$ \\
\tn & $\{i\e1 + \alpha_{ij}\e2 + \beta_{ij}\e3 + j\e4 \mid i\ttk{p},\,j\ttk{2}\}$
& $\{i\e2+j\e3 \mid i,j\ttk{2} \}$ \\
\tn & $\{\alpha_i\e1 +\alpha_{ij}\e2+j\e3+i\e4 \mid i,j\ttk{2}\}$
& $\{i\e1 + j\e2 +\alpha_i\e3 \mid i\ttk{p},\,j\ttk{2}\}$ \\
\tn & $\{\alpha_{ij}\e1 + \alpha_{ij}\e2 + i\e3 + j\e4 \mid i,j\ttk{2}\}$
& $\{i\e1+ j\e2 \mid i\ttk{p},\,j\ttk{2}\}$ \\
\tn & $\{i\e1 + k\e2 + \alpha_{ij}\e3 + j\e4  \mid i\ttk{p},\,j,k\ttk{2}\}$
& $\{\alpha_i\e2 + i\e3 \mid i\ttk{2} \}$ \\
\tn & $\{k\e1 + \alpha_{ij}\e2 + i\e3 + j\e4  \mid i,j\ttk{2},\,k\ttk{p}\}$
& $\{ \alpha_i\e1 + i\e2 \mid i\ttk{2} \}$ \\
\tn & $\{i\e1 + j\e2 + k\e3 + \alpha_i\e4 \mid i\ttk{p},\,j,k\ttk{2}\}$
& $\{\alpha_i\e2 +\beta_i\e3 +i\e4 \mid i\ttk{2} \}$ \\
\tn & $\{i\e1 + \alpha_{ijk}\e2+ j\e3 + k\e4  \mid i\ttk{p},\,j,k\ttk{2}\}$
& $\{i\e2 \mid i\ttk{2}\}$ \\
\tn & $\{\alpha_{ij}\e1 + k\e2 + i\e3+j\e4 \mid i,j,k\ttk{2}\}$
& $\{i\e1 + \alpha_i \e2 \mid i\ttk{p} \}$ \\
\tn & $\{\alpha_i\e1 + j\e2 + k\e3 + i\e4 \mid i,j,k\ttk{2}\}$
& $\{ i\e1 + \alpha_i\e2 + \beta_i\e3 \mid i\ttk{p} \}$ \\
\tn & $\{\alpha_{ijk}\e1+i\e2+j\e3+k\e4 \mid i,j,k\ttk{2}\}$
& $\{ i\e1 \mid i\ttk{p} \}$ \\
\hline
\end{longtable}

\begin{proof}
    Suppose $C$ is a perfect code in $\Cay(G,S)$. Then $G=\so\oplus C$ and so $|\so|\cdot |C|=|G|=8p$. 
    Since $G$ is not a cyclic group, $|\so|>2$ by Lemma \ref{ls<so}.
    Therefore, $2<|\so|<8p$.

\smallskip
    \textsf{Case 1.} $\so$ is aperiodic.
    \smallskip

    Since $G$ is Haj\'os and $G=\so\oplus C$, in this case $C$ must be periodic, that is, $\lc$ is nontrivial.
    Write $C=\lc\oplus D$ for some subset $D$ of $G$ whose existence is ensured by Lemma \ref{period}.
    Then $|\lc|\cdot |D|=|C|$ and $2\leq |\lc|\leq |C|$.
    If $|\so|$ is a prime, then $|\lc|=|C|$ by Lemma \ref{prime deg} part (c).
    If $|\so|$ is not a prime, then $|\lc|$ can be any divisor of $|C|$.
    Set $$ A=(\so+\lc)/\lc\quad\text{and}\quad B=(D+\lc)/\lc.$$
    Since $C$ is periodic, by Lemma \ref{quotient}, we have
    $$G/\lc=A\oplus B$$ and $B$ is aperiodic in $G/\lc$.
    Since $G/\lc$ is isomorphic to a subgroup of $G$ and subgroups of Haj\'os group are also Haj\'os groups, $A$ is periodic in $G/\lc$.
    We also know that $A$ contains $\lc$ and $A$ generates $G/\lc$ by Lemma \ref{generate-zero}.
    Furthermore, $|A|=|\so|$ and $|B|=|D|$.
    Note that if $|\lc|=|C|$, then $C=\lc$ as $0\in C$.
    This means we can take $D=\{0\}$ and hence $B=\lc/\lc$ and $G/\lc=A$.
    We use $C$ to get $A$ and then use Lemma \ref{d_i+l_i} to obtain $\so$.
    If $|\lc|<|C|$, then this factorization enables us to make use of previously established theorems to obtain the pair $(\so, C)$ via $(A,B)$.

    Similar to Table \ref{tap.p^2.2.2}, the columns of Table \ref{tap.p.2.2.2} represents all possible values for $\left(|\so|,|C|,|\lc|\right)$; $C$, $A$, and $\so$ when $|\lc|=|C|$ or $L_C$, $A$, $B$, $\so$, and $C$ when $|\lc|<|C|$; which previous theorem is used to obtain $(A,B)$ when $|\lc|<|C|$; and which row in Table \ref{t.p.2.2.2} the result corresponds to.
        
\begin{longtable}{|c|l|c|c|}
\caption{The case when $G=\Z_p\x\Z_2\x\Z_2\x\Z_2$ and $\so$ is aperiodic}
\label{tap.p.2.2.2}
\\
\flap

& $C=\<{\e2,\e3,\e4} $ & & \\
$(p,8,8)$ 
& $A=\{i\e1+\lc \mid i\ttk{p}\} $ 
& & row 1 \\
& $\so=\{i\e1+\alpha_i\e2+\beta_i\e3+\gamma_i\e4 \mid i\ttk{p}\} $ & & \\ \hline

& $\lc=\<{\e2} $ & & \\
& $A=\{i\e1+j\e3+\alpha_i\e4 +\lc \mid i\ttk{p},\,j\ttk{2}\} $ & & \\
$(2p,4,2)$ 
& $B=\{\alpha_i\e3 + i\e4 + \lc \mid i\ttk{2}\} $ 
& \ref{p.2.2} \row4 & row 3 \\
& $\so=\{i\e1+j\e3+\alpha_i\e4 +\alpha_{ij}\e2 \mid i\ttk{p},\,j\ttk{2}\} $ & & \\
& $C=\{\alpha_i\e3 + i\e4 + j\e2 \mid i,j\ttk{2}\} $ & & \\ \hline

& $C=\<{\e2,\e3} $ & & \\
$(2p,4,4)$ 
& $A=\{i\e1 + j\e4 + \lc \mid i\ttk{p},\,j\ttk{2}\} $ 
& & row 4 \\
& $\so=\{i\e1 + \alpha_{ij}\e2 + \beta{ij}\e3 + j\e4 \mid i\ttk{p},\,j\ttk{2}\} $ & & \\ \hline

$(4,2p,p)$ 
& Not possible since no $A$ exists
& \ref{2.2.2}  & \\ \hline

& $\lc=\<{\e2} $ & & \\
& $A=\{\alpha_i\e1+j\e3+i\e4 +\lc \mid i,j\ttk{2}\} $ & & \\
$(4,2p,2)$ 
& $B=\{i\e1+\alpha_i\e3 + \lc \mid i\ttk{p}\} $ 
& \ref{p.2.2} \row6 & row 5 \\
& $\so=\{\alpha_i\e1+j\e3+i\e4 +\alpha_{ij}\e2 \mid i,j\ttk{2}\} $ & & \\
& $C=\{i\e1+j\e2 + \alpha_i\e3 \mid i\ttk{p},\,j\ttk{2}\} $ & & \\ \hline

& $C=\<{\e1+\e2} $ & & \\
$(4,2p,2p)$ 
& $A=\{i\e3 + j\e4 + \lc \mid i,j\ttk{2}\} $ 
& & row 6 \\
& $\so=\{\alpha_{ij}\e1+ \alpha_{ij}\e2 + i\e3 + j\e4 \mid i,j\ttk{2}\} $ & & \\ \hline

& $C=\<{\e2} $ & & \\
$(4p,2,2)$ 
& $A=\{i\e1 + j\e3 + k\e4 + \lc \mid i\ttk{p},\,j,k\ttk{2}\} $ 
& & row 10 \\
& $\so=\{i\e1 + \alpha_{ijk}\e2 + j\e3 + k\e4 \mid i\ttk{p},\,j,k\ttk{2}\} $ & & \\ \hline

& $C=\<{\e1} $ & & \\
$(8,p,p)$ 
& $A=\{i\e1+j\e2+k\e3 \mid i,j,k\ttk{2}\} $ 
& & row 13 \\
& $\so=\{\alpha_{ijk}\e1+i\e2+j\e3+k\e4 \mid i,j,k\ttk{2}\} $ & & \\ \hline

\end{longtable}

\smallskip
    \textsf{Case 2.} $\so$ is periodic.
    \smallskip

    In this case we have $|\ls|\geq 2$ and $\so=D\oplus \ls$ for some subset $D$ of $G$ by Lemma \ref{period}.
    So $|D|\cdot|\ls|=|\so|$ and $1<|\ls|<|\so|$ by Lemma \ref{prime deg} part (b).
    Set $$X=(D+\ls)/\ls\quad\text{and}\quad Y=(C+\ls)/\ls.$$
    Since $\so$ is periodic, by Lemma \ref{quotient}, we have
    $$G/\ls=X\oplus Y$$ and $X$ is aperiodic in $G/\ls$.
    We also know that $X$ contains $\ls$ and $X$ generates $G/\ls$ by Lemma \ref{generate-zero}.
    This factorization enables us to make use of Lemma \ref{d_i+l_i} and previously established theorems to obtain the pair $(\so, C)$ via $(X,Y)$.
    
    Similar to Table \ref{tp.p^2.2.2}, the columns of Table \ref{tp.p.2.2.2} represents all possible values for $\left(|\so|,|C|,|\ls|\right)$; $\ls$, $X$, $Y$, $\so$, and $C$; which previous theorem is used to obtain $(X,Y)$; and which row in Table \ref{t.p.2.2.2} the result corresponds to.

\begin{longtable}{|c|l|c|c|}
\caption{The case when $G=\Z_p\x\Z_2\x\Z_2\x\Z_2$ and $\so$ is periodic}
\label{tp.p.2.2.2}
\\
\flp

& $\ls=\<{\e2} $ & & \\
& $X=\{i\e1 + \alpha_i\e3 + \beta_i\e4 + \ls \mid i\ttk{p}\} $ & & \\
$(2p,4,2)$ 
& $Y=\{i\e3 +j\e4 + \ls \mid i,j\ttk{2}\} $ 
& \ref{p.2.2} \row1 & row 2 \\
& $\so=\{i\e1 + j\e2 + \alpha_i\e3 + \beta_i\e4 \mid i\ttk{p},j\ttk{2}\} $ & & \\
& $C=\{i\e3 +j\e4 + \alpha_{ij}\e2 \mid i,j\ttk{2}\} $ & & \\ \hline

$(2p,4,p)$ 
& Not possible since no $X$ exists
& \ref{2.2.2}  & \\ \hline

$(4,2p,2)$ 
& Not possible since no $X$ exists
& \ref{p.2.2}  & \\ \hline

& $\ls=\<{\e2} $ & & \\
& $X=\{i\e1 + \alpha_{ij}\e3+ j\e4  + \ls \mid i\ttk{p},\,j\ttk{2}\} $ & & \\
$(4p,2,2)$ 
& $Y=\{i\e3 + \ls \mid i\ttk{2} \} $ 
& \ref{p.2.2} \row5 & row 7 \\
& $\so=\{i\e1 + \alpha_{ij}\e3 + j\e4  + k\e2 \mid i\ttk{p},\,j,k\ttk{2}\} $ & & \\
& $C=\{\alpha_i\e2 + i\e3 \mid i\ttk{2} \} $ & & \\ \hline

& $\ls=\<{\e1} $ & & \\
& $X=\{\alpha_{ij}\e2+i\e3 + j\e4 + \ls \mid i,j\ttk{2}\} $ & & \\
$(4p,2,p)$ 
& $Y=\{ i\e2 + \ls \mid i\ttk{2} \} $ 
& \ref{2.2.2} & row 8 \\
& $\so=\{k\e1+\alpha_{ij}\e2+ i\e3 + j\e4 \mid i,j\ttk{2},\,k\ttk{p}\} $ & & \\
& $C=\{ \alpha_i\e1 + i\e2 \mid i\ttk{2} \} $ & & \\ \hline

$(4p,2,2p)$ 
& Not possible since no $X$ exists
& \ref{2.2}  & \\ \hline

& $\ls=\<{\e2,\e3} $ & & \\
& $X=\{i\e1 + \alpha_i \e4 + \ls \mid i\ttk{p}\} $ & & \\
$(4p,2,4)$ 
& $Y=\{i\e4 +\ls \mid i\ttk{2} \} $ 
& \ref{p.q} \row1 & row 9 \\
& $\so=\{i\e1 + j\e2 + k\e3 + \alpha_i \e4 \mid i\ttk{p},\,j,k\ttk{2}\} $ & & \\
& $C=\{\alpha_i\e2 +\beta_i\e3 + i\e4 \mid i\ttk{2} \} $ & & \\ \hline

& $\ls=\<{\e2} $ & & \\
& $X=\{\alpha_{ij}\e1 + i\e3+j\e4 + \ls \mid i,j\ttk{2}\} $ & & \\
$(8,p,2)$ 
& $Y=\{i\e1 + \ls \mid i\ttk{p} \} $ 
& \ref{p.2.2} \row7 & row 11 \\
& $\so=\{\alpha_{ij}\e1 + k\e2 + i\e3 + j\e4 \mid i,j,k\ttk{2}\} $ & & \\
& $C=\{i\e1 + \alpha_i \e2 \mid i\ttk{p} \} $ & & \\ \hline

& $\ls=\<{\e2,\e3} $ & & \\
& $X=\{\alpha_i\e1 + i\e4 + \ls \mid i\ttk{2}\} $ & & \\
$(8,p,4 )$ 
& $Y=\{ i\e1 + \ls \mid i\ttk{p} \} $ 
& \ref{p.q} \row2 & row 12 \\
& $\so=\{\alpha_i\e1 + j\e2 + k\e3 + i\e4  \mid i,j,k\ttk{2}\} $ & & \\
& $C=\{ i\e1 + \alpha_i\e2 + \beta_i\e3 \mid i\ttk{p} \} $ & & \\ \hline

\end{longtable}
\end{proof}

\begin{remark}
    In Theorem \ref{p.2.2.2}, we can change $\e2$, $\e3$, and $\e4$ to any three generators of $\{0\}\x\Z_2\x\Z_2\x\Z_2$.
\end{remark}

\begin{theorem}\label{2.2.2.2}
Let $G=\Z_2\x\Z_2\x\Z_2\x\Z_2$ and let $S$ be a proper subset of $G\setminus \{0\}$ that generates $G$. Let $0\in C\subseteq G$. Then $C$ is a perfect code in $\Cay(G,S)$ if and only if $(\so, C)$ is one of the pairs in Table \ref{t.2.2.2.2} for some $\alpha_i$'s, $\alpha_{ij}$'s, $\alpha_{ijk}$'s in $\Z$ where $\alpha_0=\alpha_{00}=\alpha_{000}=0$.
\end{theorem}

\rc\begin{longtable}{|c|l|l|}
\caption{Perfect codes in Cayley graphs of $\Z_2\x\Z_2\x\Z_2\x\Z_2$}
\label{t.2.2.2.2}
\\
\fl
\tn & $\{ k\e1 + \alpha_{ij}\e2 + i\e3 + j\e4 \mid i,j,k\ttk{2}\}$
& $\{ \alpha_i \e1 + i\e2 + \mid i\ttk{2} \}$ \\
\tn & $\{ \alpha_{ijk}\e1 + i\e2 + j\e3 + k\e4 \mid i,j,k\ttk{2}\}$
& $\{ i\e1 \mid i\ttk{2} \}$ \\
\hline
\end{longtable}

\begin{proof}
    \setcounter{cn}{1}\renewcommand{\Table}{\ref{t.2.2.2.2}}
    Suppose $\Cay(G,S)$ admits a perfect code $C$.
    Then $G=\so\oplus C$ and so $|\so|\cdot |C|=|G|=16$.
    Since $\so$ generates $G$, $\so$ contains at least four non-zero elements, so $|\so|\geq 5$.
    Also, $\Cay(G,S)$ is not a complete graph. Hence $|C|\geq 2$ and therefore $(|\so|,|C|)=(8,2)$.

    If $\so$ is periodic, then $|\ls|=2$ or $4$.
    Suppose $|\ls|=4$. Then $\ls=\left<\e1,\e2\right>$.
    Let $\so=D\oplus\ls$ for some subset $D$. From the factorization of $G/\ls$ given in Lemma \ref{quotient}, we have $G/\ls=((D+\ls)/\ls)\oplus ((C+\ls)/\ls)$.
    Note that $(D+\ls)/\ls$ is a set of order $2$ containing $0+\ls$ and generates $G/\ls\cong \Z_{2}\times\Z_2$ by Lemma \ref{generate-zero}. 
    Since $\Z_2\times\Z_2$ is a not a cyclic group, $(D+\ls)/\ls$ cannot generate $G/\ls$, a contradiction.
    Therefore, $|\ls|=2$, and hence $\ls=\<{\e1}$.
    Let $\so=D\oplus\ls$ for some subset $D$. From the factorization of $G/\ls$ given in Lemma \ref{quotient}, we have $G/\ls=((D+\ls)/\ls)\oplus ((C+\ls)/\ls)$.
    Note that $(D+\ls)/\ls$ is an aperiodic set of order $4$ and $G/\ls\cong \Z_{2}\times\Z_2\times\Z_2$.
    From the case of $|\so|=4$ and $\so$ is aperiodic in Theorem \ref{2.2.2}, we have $(D+\ls)/\ls=\{ \alpha_{ij}\e2 + i\e3 + j\e4 +\ls \mid i,j\ttk{2}\}$ and $(C+\ls)/\ls=\{ i\e2 +\ls \mid i\ttk{2} \}$.
    Therefore, $\so=D\oplus\ls=\{ k\e1 + \alpha_{ij}\e2 + i\e3 + j\e4 \mid i,j,k\ttk{2}\}$ and $C=\{ \alpha_i \e1 + i\e2 \mid i\ttk{2} \}$, where $\alpha_i,\beta_i \in \Z$. \R

    If $|\so|=8$ and $\so$ is aperiodic, then $|C|=2$ and $C$ is periodic, and hence $C=\lc=\<{\e1}$.
    From the factorization of $G/\lc$ given in Lemma \ref{quotient}, $(\so+\lc)/\lc=G/\lc=\{ i\e2 + j\e3 + k\e4 + \lc \mid i,j,k\ttk{2}\}$, and hence $\so=\{ \alpha_{ijk}\e1 + i\e2 + j\e3 + k\e4 \mid i,j,k\ttk{2}\}$, $\alpha_{ijk}\in \Z$. \R
\end{proof}

\begin{remark}
    In Theorem \ref{2.2.2.2}, we can change $\e1$, $\e2$, $\e3$, and $\e4$ to any four generators of $\Z_2\x\Z_2\x\Z_2\x\Z_2$.
\end{remark}

\begin{theorem}\label{p^3.2.2}
Let $G=\Z_{p^3}\x\Z_2\x\Z_2$ for an odd prime $p$ and let $S$ be a proper subset of $G\setminus \{0\}$ that generates $G$. Let $0\in C\subseteq G$. Then $C$ is a perfect code in $\Cay(G,S)$ if and only if $(\so, C)$ is one of the pairs in Table \ref{t.p^3.2.2} for some $\alpha_i$'s, $\alpha_{ij}$'s, $\alpha_{ijk}$'s, $\beta_i$'s, $\gamma_i$'s in $\Z$ where $\alpha_0=\alpha_{00}=\alpha_{000}=\beta_0=\gamma_i=0$.
\end{theorem}

\begin{small}
\rc\begin{longtable}{|c|l|l|}
\caption{Perfect codes in Cayley graphs of $\Z_{p^3}\x\Z_2\x\Z_2$}
\label{t.p^3.2.2} \\
\fl
\tn & $\{(i+\alpha_ip)\e1+\alpha_i\e2+\beta_i\e3 \mid i\ttk{p}\}$
& $\{(i+\alpha_i p^2)\e1 + j\e2 \mid i\ttk{p^2},\,j\ttk{2}\}$ \\
\tn & $\{(i+\alpha_i p + jp^2)\e1+\alpha_i\e2+\beta_i\e3 \mid i,j\ttk{p}\}$
& $\{(ip+kp^2)\e1 + \alpha_{ij}\e2 + j\e3 \mid i,k\ttk{p},\,j\ttk{2}\}$ \\
\tn & $\{(i+\alpha_ip+jp^2)\e1 + \alpha_{ij}\e2 + \alpha_i\e3 \mid i,j\ttk{p}\}$
& $\{(ip+jp^2)\e1 + k\e2 + \alpha_i\e3 \mid i,j\ttk{p},\,k\ttk{2}\}$ \\
\tn & $\{(i+jp+\alpha_{ij}p^2)\e1+\alpha_i\e2+\beta_i\e3 \mid i,j\ttk{p}\}$
& $\{(i+\alpha_ip+jp^2)\e1 +k\e2 \mid i,j\ttk{p},\,k\ttk{2}\}$ \\
\tn & $\{(i+\alpha_ip+jp^2)\e1+\alpha_{ij}\e2+\beta_{ij}\e3 \mid i,j\ttk{p}\}$
& $\{jp\e1 + \alpha_i\e2 + i\e3 \mid i\ttk{2},\,j\ttk{p^2}\}$ \\
\tn & $\{(i+jp+\alpha_{ij}p^2)\e1 + \alpha_{ij}\e2 + \alpha_i\e3 \mid i,j\ttk{p}\}$
& $\{ip\e1+j\e2 \mid i\ttk{p^2},\,j\ttk{2}\}$ \\
\tn & $\{(i+\alpha_ip^2)\e1+\beta_{i}\e2+\gamma_i\e3 \mid i\ttk{p^2}\}$
& $\{i\e1+\alpha_i\e2 \mid i\ttk{p^3}\}$ \\
\tn & $\{(i+\alpha_i p)\e1+j\e2+\alpha_i\e3\mid i\ttk{p},\,j\ttk{2}\}$
& $\{(i+jp^2)\e1 + \alpha_i\e2 \mid i\ttk{p^2},j\ttk{p}\}$ \\
\tn & $\{(i+\alpha_i p)\e1 + \alpha_{ij}\e2 + j\e3 \mid i\ttk{p},\,j\ttk{2}\}$
& $\{ip\e1+j\e2+k\e3 \mid i\ttk{p^2},\, j,k\ttk{2}\}$ \\
\tn & $\{(\alpha_i+jp^2)\e1 + \alpha_{ij}\e2 + i\e3 \mid i\ttk{2},\,j\ttk{p}\}$
& $\{(i+\alpha_{ijk}p^2)\e1+j\e2+k\e3 \mid i\ttk{p},\,j,k\ttk{2}\}$ \\
\tn & $\{(i+\alpha_ip+\alpha_{ij}p^2)\e1 + j\e2 + \alpha_i\e3 \mid i\ttk{p},\,j\ttk{2}\}$
& $\{(ip + \alpha_{ij}p^2)\e1 + k\e2 + j\e3 \mid i\ttk{p},\, j,k\ttk{2}\}$ \\
\tn & $\{(i+\alpha_ip+\alpha_{ij}p^2)\e1 + j\e3 \mid i\ttk{p},\,j\ttk{2}\}$
& $\{(\alpha_{ij} p +kp^2)\e1+i\e2+j\e3 \mid i,j\ttk{2},\, k\ttk{p}\}$ \\
\tn & $\{(\alpha_i+jp+\alpha_{ij}p^2)\e1 + i\e3 \mid i\ttk{2},\,j\ttk{p}\}$
& $\{(ip+\alpha_ip^2)\e1+j\e2+k\e3 \mid i\ttk{p},\,j,k\ttk{2}\}$ \\
\tn & $\{(i+\alpha_{ij}p)\e1 + j\e2 + \alpha_i\e3 \mid i\ttk{p},\,j\ttk{2}\}$
& $\{(\alpha_i p + jp^2)\e1 + k\e2 + i\e3  \mid i,k\ttk{2},\, j\ttk{p}\}$ \\
\tn & $\{(i+\alpha_{ij}p)\e1 + \alpha_{ij}\e2 + j\e3  \mid i\ttk{p},\,j\ttk{2}\}$
& $\{ip^2\e1+j\e2+k\e3 \mid i\ttk{p},\,j,k\ttk{2}\}$ \\
\tn & $\{\alpha_i\e1 + j\e2 + i\e3  \mid i\ttk{p},\,j\ttk{2}\}$
& $\{ip\e1+j\e3+\alpha_{ij}\e2 \mid i\ttk{p^2},\,j\ttk{2}\}$ \\
\tn & $\{(\alpha_i + \alpha_{ij}p^2)\e1 + j\e2 + i\e3  \mid i,j\ttk{2}\}$
& $\{ip^2\e1 + j\e2 + \alpha_i\e3 \mid i\ttk{p^2},\,j\ttk{2}\}$ \\
\tn & $\{(\alpha_i + \alpha_{ij}p)\e1 + j\e2 + i\e3  \mid i,j\ttk{2}\}$
& $\{(i+jp)\e1 + \alpha_i\e2 \mid i\ttk{p},j\ttk{p^2}\}$ \\
\tn & $\{\alpha_{ij}\e1+i\e2+j\e3 \mid i,j\ttk{2}\}$
& $\{i\e1 \mid i\ttk{p^3} \}$ \\
\tn & $\{(i+jp+kp^2)\e1 + \alpha_{ij}\e2 + \alpha_i\e3 \mid i,j,k\ttk{p}\}$
& $\{(\alpha_ip+\alpha_{ij}p^2)\e1 + j\e2 + i\e3  \mid i,j\ttk{2}\}$ \\
\tn & $\{(i+jp^2)\e1+\alpha_i\e2+\beta_i\e3 \mid i\ttk{p^2},j\ttk{p}\}$
& $\{\alpha_{ij}p^2\e1+i\e2+j\e3 \mid i,j\ttk{2}\}$ \\
\tn & $\{(i+jp)\e1+\alpha_i\e2+\beta_i\e3 \mid i\ttk{p},j\ttk{p^2}\}$
& $\{\alpha_{ij}p\e1+i\e2+j\e3 \mid i,j\ttk{2}\}$ \\
\tn & $\{(i+jp^2)\e1 + \alpha_{ij}\e2 + \alpha_i\e3 \mid i\ttk{p^2},\,j\ttk{p}\}$
& $\{\alpha_ip^2\e1 + j\e2 + i\e3  \mid i,j\ttk{2}\}$ \\
\tn & $\{(i+jp)\e1 + \alpha_{ij}\e2 + \alpha_i\e3 \mid i\ttk{p},\,j\ttk{p^2}\}$
& $\{\alpha_ip\e1 + j\e2 + i\e3  \mid i,j\ttk{2}\}$ \\
\tn & $\{i\e1+\alpha_i\e2+\beta_i\e3 \mid i\ttk{p^3}\}$
& $\{i\e2+j\e3 \mid i,j\ttk{2}\}$ \\
\tn & $\{(i+\alpha_ip+jp^2)\e1 + k\e2 + \alpha_{ij}\e3 \mid i,j\ttk{p},\,k\ttk{2}\}$
& $\{(ip+\alpha_ip^2)\e1 + \alpha_{ij}\e2 + j\e3 \mid i,j\ttk{2}\}$ \\
\tn & $\{(i+jp+\alpha_{ij}p^2)\e1 + k\e2 + i\e3 \mid i,j\ttk{p},\,k\ttk{2}\}$
& $\{(\alpha_ip+jp^2)\e1 + \alpha_{ij}\e2 + i\e3 \mid i\ttk{2},\,j\ttk{p}\}$ \\
\tn & $\{(i+\alpha_ip^2)\e1 + j\e2 + \alpha_i\e3 \mid i\ttk{p^2},\,j\ttk{2}\}$
& $\{ip^2\e1 + \alpha_{ij}\e2 + j\e3 \mid i\ttk{p},\,j\ttk{2}\}$ \\
\tn & $\{(i+\alpha_ip+kp^2)\e1 + \alpha_{ij}\e2 + j\e3 \mid i,k\ttk{p},\,j\ttk{2}\}$
& $\{(ip + \alpha_{ij}p^2)\e1 + j\e2 + \alpha_i\e3 \mid i\ttk{p},\,j\ttk{2}\}$ \\
\tn & $\{(\alpha_i + jp+kp^2)\e1 + \alpha_{ij}\e2 + i\e3 \mid i\ttk{2},\,j,k\ttk{p}\}$
& $\{(i+\alpha_ip+\alpha_{ij}p^2)\e1 + j\e2 \mid i\ttk{p},\,j\ttk{2}\}$ \\
\tn & $\{(i+\alpha_{ij}p+kp^2)\e1 + j\e2 + \alpha_i\e3 \mid i,k\ttk{p},\,j\ttk{2}\}$
& $\{(jp+\alpha_{ij}p^2)\e1 + \alpha_i\e2 + i\e3 \mid i\ttk{2},\,j\ttk{p}\}$ \\
\tn & $\{(i+\alpha_{ij}p+kp^2)\e1 + \alpha_{ij}\e2 + j\e3 \mid i,k\ttk{p},\,j\ttk{2}\}$
& $\{(ip+\alpha_{ij}p^2)\e1 + j\e2 \mid i\ttk{p},\,j\ttk{2}\}$ \\
\tn & $\{(i+jp^2)\e1 + k\e2 + \alpha_i\e3 \mid i,j\ttk{p},\,k\ttk{2}\}$
& $\{(ip+\alpha_{ij}p^2)\e1 + \alpha_{ij}\e2 + j\e3 \mid i\ttk{p},\,j\ttk{2}\}$ \\
\tn & $\{(i+\alpha_i p^2)\e1 + \alpha_{ij}\e2 + j\e3 \mid i\ttk{p^2},\,j\ttk{2}\}$
& $\{ip^2\e1 + j\e2 + \alpha_i\e3 \mid i\ttk{p},\,j\ttk{2}\}$ \\
\tn & $\{(i+\alpha_{ik}p+jp^2)\e1 + \alpha_{ijk}\e2 + k\e3 \mid i,j\ttk{p},\,k\ttk{2}\}$
& $\{(ip+\alpha_ip^2)\e1 + j\e2 \mid i\ttk{p},\,j\ttk{2}\}$ \\
\tn & $\{(i+\alpha_ip+jp^2)\e1 + \alpha_{ijk}\e2 + k\e3 \mid i,j\ttk{p},\,k\ttk{2}\}$
& $\{(ip+\alpha_ip^2)\e1 + j\e2 + \alpha_i\e3 \mid i\ttk{p},\,j\ttk{2}\}$ \\
\tn & $\{(\alpha_i+jp)\e1 + \alpha_{ij}\e2 + i\e3 \mid i\ttk{2},\,j\ttk{p^2}\}$
& $\{(i+\alpha_i p)\e1 + j\e2 \mid i\ttk{p},\,j\ttk{2}\}$ \\
\tn & $\{(i+\alpha_{ij}p^2)\e1 + j\e2 + \alpha_i\e3 \mid i\ttk{p^2},\,j\ttk{2}\}$
& $\{jp^2\e1 + \alpha_i\e2 + i\e3 + \lc \mid i\ttk{2},\,j\ttk{p}\}$ \\
\tn & $\{(i+kp+\alpha_{ijk}p^2)\e1 + \alpha_{ij}\e2 + j\e3 \mid i,k\ttk{p},\,j\ttk{2}\}$
& $\{(\alpha_ip+jp^2)\e1 + i\e2 + \lc \mid i\ttk{2},\,j\ttk{p}\}$ \\
\tn & $\{(i+jp+\alpha_{ijk}p^2)\e1 + k\e2 + \alpha_i\e3 \mid i,j\ttk{p},\,k\ttk{2}\}$
& $\{(\alpha_ip+jp^2)\e1 + \alpha_i\e2 + i\e3 \mid i\ttk{2},\,j\ttk{p}\}$ \\
\tn & $\{(i+\alpha_{ij}p^2)\e1 +\alpha_{ij}\e2 + j\e3 \mid i\ttk{p^2},\,j\ttk{2}\}$
& $\{ip^2\e1+j\e2 \mid i\ttk{p},\,j\ttk{2}\}$ \\
\tn & $\{(i+\alpha_ip+\alpha_{ij}p^2)\e1 + k\e2 + j\e3 \mid i\ttk{p},\,j,k\ttk{2}\}$
& $\{(ip+jp^2)\e1 + \alpha_{ij}\e2 + \alpha_i\e3 \mid i,j\ttk{p}\}$ \\
\tn & $\{(\alpha_i+jp+\alpha_{ij}p^2)\e1 + k\e2 + i\e3 \mid i,k\ttk{2},\,j\ttk{p}\}$
& $\{(i+\alpha_ip+jp^2)\e1 + \alpha_{ij}\e2 \mid i,j\ttk{p}\}$ \\
\tn & $\{(i+\alpha_{ij}p)\e1 + k\e2 + j\e3 \mid i\ttk{p},\,j,k\ttk{2}\}$
& $\{ip\e1 + \alpha_i\e2 \mid i\ttk{p^2}\}$ \\
\tn & $\{(\alpha_i+\alpha_{ij}p+kp^2)\e1 + j\e2 + i\e3  \mid i,j\ttk{2},\,k\ttk{p}\}$
& $\{(i+jp+\alpha_{ij}p^2)\e1 + \alpha_i\e2 \mid i,j\ttk{p}\}$ \\
\tn & $\{(\alpha_{ij}+kp^2)\e1+i\e2+j\e3 \mid i,j\ttk{2},\,k\ttk{p}\}$
& $\{(i+\alpha_i p^2)\e1 \mid i\ttk{p^2}\}$ \\
\tn & $\{(i+\alpha_ip)\e1+j\e2+k\e3 \mid i\ttk{p},\,j,k\ttk{2}\}$
& $\{ip\e1+\alpha_i\e2+\beta_i\e3 \mid i\ttk{p^2}\}$ \\
\tn & $\{(\alpha_i+jp^2)\e1 + k\e2 + i\e3 \mid i,k\ttk{2},\,j\ttk{p}\}$
& $\{(i+\alpha_i p^2)\e1 + \alpha_i\e2 \mid i\ttk{p^2}\}$ \\
\tn & $\{(i+\alpha_{ij}p+\alpha_{ijk}p^2)\e1 + k\e2 + j\e3 \mid i\ttk{p},\,j,k\ttk{2}\}$
& $\{(ip+jp^2)\e1 + \alpha_i\e2  \mid i,j\ttk{p}\}$ \\
\tn & $\{(\alpha_{ij}+kp+\alpha_{ijk}p^2)\e1+i\e2+j\e3 \mid i,j\ttk{2},\,k\ttk{p}\}$
& $\{ (i+\alpha_ip+jp^2)\e1 \mid i,j\ttk{p} \}$ \\
\tn & $\{(\alpha_i+jp+\alpha_{ijk}p^2)\e1 + k\e2 + i\e3 \mid i,k\ttk{2},\,j\ttk{p} \}$
& $\{(i+\alpha_ip+jp^2)\e1 + \alpha_i\e2 \mid i,j\ttk{p} \}$ \\
\tn & $\{(i+\alpha_ip+\alpha_{ijk}p^2)\e1+j\e2+k\e3 \mid i\ttk{p},\,j,k\ttk{2} \}$
& $\{(ip+jp^2)\e1+\alpha_i\e2+\beta_i\e3 \mid i,j\ttk{p} \}$ \\
\tn & $\{(i+\alpha_{ijk}p)\e1+j\e2+k\e3 \mid i\ttk{p},\,j,k\ttk{2}\}$
& $\{ip\e1 \mid i\ttk{p^2} \}$ \\
\tn & $\{i\e1 + j\e2 + \alpha_i\e3 \mid i\ttk{p^3},\,j\ttk{2}\}$
& $\{\alpha_i\e2 + i\e3 \mid i\ttk{2}\}$ \\
\tn & $\{(i+kp^2)\e1 + \alpha_{ij}\e2 + j\e3 \mid i\ttk{p^2},\,j\ttk{2},\,k\ttk{p}\}$
& $\{\alpha_ip^2\e1 + i\e2 \mid i\ttk{2}\}$ \\
\tn & $\{(i+jp^2)\e1 + k\e2 + \alpha_i\e3 \mid i\ttk{p^2},\,j\ttk{p},\,k\ttk{2}\}$
& $\{\alpha_ip^2\e1 + \alpha_i\e2 + i\e3 \mid i\ttk{2}\}$ \\
\tn & $\{(i+kp)\e1 + \alpha_{ij}\e2 + j\e3 \mid i\ttk{p},\,j\ttk{2},\,k\ttk{p^2}\}$
& $\{\alpha_ip\e1 + i\e2 \mid i\ttk{2}\}$ \\
\tn & $\{(i+jp)\e1 + k\e2 + \alpha_i\e3 \mid i\ttk{p},\,j\ttk{p^2},\,k\ttk{2}\}$
& $\{\alpha_ip\e1 + \alpha_i\e2 + i\e3 \mid i\ttk{2}\}$ \\
\tn & $\{i\e1 + \alpha_{ij}\e2 + j\e3 \mid i\ttk{p^3},\,j\ttk{2}\}$
& $\{i\e2 \mid i\ttk{2}\}$ \\
\tn & $\{(i+\alpha_{ij}p^2)\e1 + k\e2 + j\e3 \mid i\ttk{p^2},\,j,k\ttk{2}\}$
& $\{ip^2\e1 +\alpha_i\e2 \mid i\ttk{p}\}$ \\
\tn & $\{(i+\alpha_{ijk}p+lp^2)\e1+j\e2+k\e3\mid i,l\ttk{p},\,j,k\ttk{2}\}$
& $\{(ip+\alpha_ip^2)\e1 \mid i\ttk{p}\}$ \\
\tn & $\{(i+\alpha_ip^2)\e1+j\e2+k\e3 \mid i\ttk{p^2},\,j,k\ttk{2}\}$
& $\{ip^2\e1+\alpha_i\e2+\beta_i\e3 \mid i\ttk{p}\}$ \\
\tn & $\{(i+\alpha_{ij}p+kp^2)\e1 + l\e2 + j\e3 \mid i,k\ttk{p},\,j,l\ttk{2}\}$
& $\{(ip+\alpha_ip^2)\e1 +\alpha_i\e2 \mid i\ttk{p}\}$ \\
\tn & $\{(\alpha_{ij}+kp)\e1+i\e2+j\e3 \mid i,j\ttk{2}\,k\ttk{p^2}\} $
& $\{(i+\alpha_ip)\e1 \mid i\ttk{p}\} $ \\
\tn & $\{(i+\alpha_ip+jp^2)\e1+k\e2+l\e3 \mid i,j\ttk{p},\,k,l\ttk{2}\}$
& $\{(ip+\alpha_ip^2)\e1+\alpha_i\e2+\beta_i\e3 \mid i\ttk{p}\}$ \\
\tn & $\{(\alpha_i+jp)\e1 + k\e2 + i\e3 \mid i,k\ttk{2},\,j\ttk{p^2}\}$
& $\{(i+\alpha_ip)\e1 +\alpha_i\e2 \mid i\ttk{p}\}$ \\
\tn & $\{(i+\alpha_{ijk}p^2)\e1+j\e2+k\e3 \mid i\ttk{p^2},\,j,k\ttk{2}\}$  
& $\{ip^2\e1 \mid i\ttk{p}\}$ \\
\hline
\end{longtable}
\end{small}

\begin{proof}
    Suppose $C$ is a perfect code in $\Cay(G,S)$. Then $G=\so\oplus C$ and so $|\so|\cdot |C|=|G|=4p^3 $. 
    Since $G$ is not a cyclic group, $|\so|>2$ by Lemma \ref{ls<so}.
    Therefore $2<|\so|<4p^3 $.

\smallskip
    \textsf{Case 1.} $\so$ is aperiodic.
    \smallskip

    Since $G$ is Haj\'os and $G=\so\oplus C$, in this case $C$ must be periodic, that is, $\lc$ is nontrivial.
    Write $C=\lc\oplus D$ for some subset $D$ of $G$ whose existence is ensured by Lemma \ref{period}.
    Then $|\lc|\cdot |D|=|C|$ and $2\leq |\lc|\leq |C|$.
    If $|\so|$ is a prime, then $|\lc|=|C|$ by Lemma \ref{prime deg} part (c).
    If $|\so|$ is not a prime, then $|\lc|$ can be any divisor of $|C|$.
    Set $$ A=(\so+\lc)/\lc\quad\text{and}\quad B=(D+\lc)/\lc.$$
    Since $C$ is periodic, by Lemma \ref{quotient}, we have
    $$G/\lc=A\oplus B$$ and $B$ is aperiodic in $G/\lc$.
    Since $G/\lc$ is isomorphic to a subgroup of $G$ and subgroups of Haj\'os group are also Haj\'os groups, $A$ is periodic in $G/\lc$.
    We also know that $A$ contains $\lc$ and $A$ generates $G/\lc$ by Lemma \ref{generate-zero}.
    Furthermore, $|A|=|\so|$ and $|B|=|D|$.
    Note that if $|\lc|=|C|$, then $C=\lc$ as $0\in C$.
    This means we can take $D=\{0\}$ and hence $B=\lc/\lc$ and $G/\lc=A$.
    We use $C$ to get $A$ and then use Lemma \ref{d_i+l_i} to obtain $\so$.
    If $|\lc|<|C|$, then this factorization enables us to make use of previously established theorems to obtain the pair $(\so, C)$ via $(A,B)$.

    Similar to Table \ref{tap.p^2.2.2}, the columns of Table \ref{tap.p^3.2.2} represents all possible values for $\left(|\so|,|C|,|\lc|\right)$; $C$, $A$, and $\so$ when $|\lc|=|C|$ or $L_C$, $A$, $B$, $\so$, and $C$ when $|\lc|<|C|$; which previous theorem is used to obtain $(A,B)$ when $|\lc|<|C|$; and which row in Table \ref{t.p^3.2.2} the result corresponds to.
    
\begin{longtable}{|c|l|c|c|}
\caption{The case when $G=\Z_{p^3}\x\Z_2\x\Z_2$ and $\so$ is aperiodic}
\label{tap.p^3.2.2}
\\
\flap

& $C=\<{p\e1+\e2,\e3} $ & & \\
$(p,4p^2,4p^2)$ 
& $A=\{i\e1+\lc \mid i\ttk{p}\} $ 
& & row 1 \\
& $\so=\{(i+\alpha_ip)\e1+\alpha_i\e2+\beta_i\e3 \mid i\ttk{p}\} $ & & \\ \hline

& $\lc=\<{\e2} $ & & \\
& $A=\{(i+\alpha_ip+jp^2)\e1 + \alpha_i\e3 + \lc \mid i,j\ttk{p}\} $ & & \\
$(p^2,4p,2)$ 
& $B=\{(ip + \alpha_{ij}p^2)\e1 + j\e3 + \lc \mid i\ttk{p},\, j\ttk{2}\} $ 
& \ref{p^3.2} \row3 & row 3 \\ 
& $\so=\{(i+\alpha_ip+jp^2)\e1 + \alpha_{ij}\e2 + \alpha_i\e3 \mid i,j\ttk{p}\} $ & & \\
& $C=\{(ip + \alpha_{ij}p^2)\e1 + k\e2 + j\e3 \mid i\ttk{p},\, j,k\ttk{2}\} $ & & \\ \hline

& $\lc=\<{p^2\e1} $ & & \\
& $A=\{(i+jp)\e1+\alpha_i\e2+\beta_i\e3 + \lc \mid i,j\ttk{p}\} $ & & \\
$(p^2,4p,p )$ 
& $B=\{(\alpha_{ij} p,i,j) + \lc \mid i,j\ttk{2}\} $ 
& \ref{p^2.2.2} \row2 & row 4 \\ 
& $\so=\{(i+jp+\alpha_{ij}p^2)\e1+\alpha_i\e2+\beta_i\e3 \mid i,j\ttk{p}\} $ & & \\
& $C=\{(\alpha_{ij} p +kp^2)\e1+i\e2+j\e3\mid i,j\ttk{2},\, k\ttk{p}\} $ & & \\ \hline
    
& $\lc=\<{\e2,\e3} $ & & \\
& $A=\{(i+\alpha_ip+jp^2)\e1 + \lc \mid i,j\ttk{p}\} $ & & \\
$(p^2,4p,4)$ 
& $B=\{(ip+\alpha_ip^2)\e1 + \lc \mid i\ttk{p}\} $ 
& \ref{p^3} \row2 & row 5 \\ 
& $\so=\{(i+\alpha_ip+jp^2)\e1+\alpha_{ij}\e2+\beta_{ij}\e3 \mid i,j\ttk{p}\} $ & & \\
& $C=\{(ip+\alpha_ip^2)\e1+j\e2+k\e3 \mid i\ttk{p},\,j,k\ttk{2}\} $ & & \\ \hline

& $\lc=\<{p^2\e1+\e2} $ & & \\
& $A=\{(i+jp)\e1 + \alpha_i\e3 + \lc \mid i,j\ttk{p}\} $ & & \\
$(p^2,4p,2p)$ 
& $B=\{\alpha_i p\e1 + i\e3 + \lc \mid i\ttk{2}\} $ 
& \ref{p^2.2} \row3 & row 6 \\ 
& $\so=\{(i+jp+\alpha_{ij}p^2)\e1 + \alpha_{ij}\e2 + \alpha_i\e3 \mid i,j\ttk{p}\} $ & & \\
& $C=\{(\alpha_i p + jp^2)\e1 + k\e2 + i\e3  \mid i,k\ttk{2},\, j\ttk{p}\} $ & & \\ \hline

& $C=\<{p^2\e1,\e2,\e3} $ & & \\
$(p^2,4p,4p)$ 
& $A=\{i\e1+\lc \mid i\ttk{p^2}\} $ 
& & row 7 \\
& $\so=\{(i+\alpha_ip^2)\e1+\beta_{i}\e2+\gamma_i\e3 \mid i\ttk{p^2}\} $ & & \\ \hline

& $\lc=\<{\e2} $ & & \\
& $A=\{(i+\alpha_i p)\e1 + j\e3 + \lc \mid i\ttk{p},\,j\ttk{2}\} $ & & \\
$(2p,2p^2,2)$ 
& $B=\{ip^2\e1 + \alpha_i\e3 + \lc \mid i\ttk{p^2}\} $ 
& \ref{p^3.2} \row7 & row 9 \\ 
& $\so=\{(i+\alpha_i p)\e1 + \alpha_{ij}\e2 + j\e3 \mid i\ttk{p},\,j\ttk{2}\} $ & & \\
& $C=\{ip^2\e1 + j\e2 + \alpha_i\e3 \mid i\ttk{p^2},\,j\ttk{2}\} $ & & \\ \hline

& $\lc=\<{\e2} $ & & \\
& $A=\{(\alpha_i+jp^2)\e1 + i\e3 + \lc \mid i\ttk{2},\,j\ttk{p}\} $ & & \\
$(2p,2p^2,2 )$ 
& $B=\{(i+\alpha_i p^2)\e1 + \lc \mid i\ttk{p^2}\} $ 
& \ref{p^3.2} \row8 & row 10 \\ 
& $\so=\{(\alpha_i+jp^2)\e1 + \alpha_{ij}\e2 + i\e3 \mid i\ttk{2},\,j\ttk{p}\} $ & & \\
& $C=\{(i+\alpha_i p^2)\e1 + j\e2 \mid i\ttk{p^2},\,j\ttk{2}\} $ & & \\ \hline

& $\lc=\<{p^2\e1} $ & & \\
& $A=\{(i+\alpha_i p)\e1 + j\e2 + \alpha_i\e3 + \lc \mid i\ttk{p},\,j\ttk{2}\} $ & & \\
$(2p,2p^2,p)$ 
& $B=\{ip\e1 + \alpha_{ij}\e2 + j\e3 + \lc \mid i\ttk{p},\,j\ttk{2}\} $ 
& \ref{p^2.2.2} \row5 & row 11 \\ 
& $\so=\{(i+\alpha_ip+\alpha_{ij}p^2)\e1 + j\e2 + \alpha_i\e3 \mid i\ttk{p},\,j\ttk{2}\} $ & & \\
& $C=\{(ip+kp^2)\e1 + \alpha_{ij}\e2 + j\e3 \mid i,k\ttk{p},\,j\ttk{2}\} $ & & \\ \hline

& $\lc=\<{p^2\e1+\e2} $ & & \\
& $A=\{(i+\alpha_i p)\e1 + j\e3 + \lc \mid i\ttk{p},\,j\ttk{2}\} $ & & \\
$(2p,2p^2,2p )$ 
& $B=\{ip\e1 + \alpha_i\e3 + \lc \mid i\ttk{p}\} $ 
& \ref{p^2.2} \row5 & row 12 \\ 
& $\so=\{(i+\alpha_ip+\alpha_{ij}p^2)\e1 + j\e3 \mid i\ttk{p},\,j\ttk{2}\} $ & & \\
& $C=\{(ip+jp^2)\e1 + k\e2 + \alpha_i\e3 \mid i,j\ttk{p},\,k\ttk{2}\} $ & & \\ \hline

& $\lc=\<{p^2\e1+\e2} $ & & \\
& $A=\{(\alpha_i+jp)\e1 + i\e3 + \lc \mid i\ttk{2},\,j\ttk{p}\} $ & & \\
$(2p,2p^2,2p )$ 
& $B=\{(i+\alpha_i p)\e1 + \lc \mid i\ttk{p}\} $ 
& \ref{p^2.2} \row6 & row 13 \\ 
& $\so=\{(\alpha_i+jp+\alpha_{ij}p^2)\e1 + i\e3 \mid i\ttk{2},\,j\ttk{p}\} $ & & \\
& $C=\{(i+\alpha_ip+jp^2)\e1 +k\e2 \mid i,j\ttk{p},\,k\ttk{2}\} $ & & \\ \hline

& $\lc=\<{p\e1} $ & & \\
& $A=\{i\e1 + j\e2 + \alpha_i\e3 + \lc \mid i\ttk{p},\,j\ttk{2}\} $ & & \\
$(2p,2p^2,p^2)$ 
& $B=\{\alpha_i\e2 + i\e3 + \lc \mid i\ttk{2}\} $ 
& \ref{p.2.2} \row4 & row 14 \\ 
& $\so=\{(i+\alpha_{ij}p)\e1 + j\e2 + \alpha_i\e3 \mid i\ttk{p},\,j\ttk{2}\} $ & & \\
& $C=\{jp\e1 + \alpha_i\e2 + i\e3 + \lc \mid i\ttk{2},\,j\ttk{p^2}\} $ & & \\ \hline
    
& $C=\<{p\e1+\e2} $ & & \\
$(2p,2p^2,2p^2)$ 
& $A=\{i\e1+j\e3 + \lc \mid i\ttk{p},\,j\ttk{2}\} $ 
& & row 15 \\
& $\so=\{(i+\alpha_{ij}p)\e1 + \alpha_{ij}\e2 + j\e3  \mid i\ttk{p},\,j\ttk{2}\} $ & & \\ \hline

& $\lc=\<{p^2\e1} $ & & \\
& $A=\{\alpha_i\e1 + j\e2 + i\e3  + \lc \mid i,j\ttk{2}\} $ & & \\
$(4,p^3,p)$ 
& $B=\{i\e1 + \alpha_i\e2 + \lc \mid i\ttk{p^2}\} $ 
& \ref{p^2.2.2} \row{10} & row 17 \\ 
& $\so=\{(\alpha_i + \alpha_{ij}p^2)\e1 + j\e2 + i\e3  \mid i,j\ttk{2}\} $ & & \\
& $C=\{(i+jp^2)\e1 + \alpha_i\e2 \mid i\ttk{p^2},j\ttk{p}\} $ & & \\ \hline

& $\lc=\<{p\e1} $ & & \\
& $A=\{\alpha_i\e1 + j\e2 + i\e3  + \lc \mid i,j\ttk{2}\} $ & & \\
$(4,p^3,p^2 )$ 
& $B=\{i\e1 + \alpha_i\e2 + \lc \mid i\ttk{p}\} $ 
& \ref{p.2.2} \row6 & row 18 \\ 
& $\so=\{(\alpha_i + \alpha_{ij}p)\e1 + j\e2 + i\e3  \mid i,j\ttk{2}\} $ & & \\
& $C=\{(i+jp)\e1 + \alpha_i\e2 \mid i\ttk{p},j\ttk{p^2}\} $ & & \\ \hline

& $C=\<{\e1} $ & & \\
$(4,p^3,p^3 )$ 
& $A=\{i\e2+j\e3 + \lc \mid i,j\ttk{2}\} $ 
& & row 19 \\
& $\so=\{\alpha_{ij}\e1+i\e2+j\e3 \mid i,j\ttk{2}\} $ & & \\ \hline

& $\lc=\<{\e2} $ & & \\
& $A=\{(i+jp^2)\e1 + \alpha_i\e3 + \lc \mid i\ttk{p^2},\,j\ttk{p}\} $ & & \\
$(p^3,4,2)$ 
& $B=\{\alpha_ip^2\e1 + i\e3 + \lc \mid i\ttk{2}\} $ 
& \ref{p^3.2} \row{12} & row 23 \\ 
& $\so=\{(i+jp^2)\e1 + \alpha_{ij}\e2 + \alpha_i\e3 \mid i\ttk{p^2},\,j\ttk{p}\} $ & & \\
& $C=\{\alpha_ip^2\e1 + j\e2 + i\e3  \mid i,j\ttk{2}\} $ & & \\ \hline

& $\lc=\<{\e2} $ & & \\
& $A=\{(i+jp)\e1 + \alpha_i\e3 + \lc \mid i\ttk{p},\,j\ttk{p^2}\} $ & & \\
$(p^3,4,2)$ 
& $B=\{\alpha_ip\e1 + i\e3 + \lc \mid i\ttk{2}\} $ 
& \ref{p^3.2} \row{13} & row 24 \\ 
& $\so=\{(i+jp)\e1 + \alpha_{ij}\e2 + \alpha_i\e3 \mid i\ttk{p},\,j\ttk{p^2}\} $ & & \\
& $C=\{\alpha_ip\e1 + j\e2 + i\e3  \mid i,j\ttk{2}\} $ & & \\ \hline

& $C=\<{\e2,\e3} $ & & \\
$(p^3,4,4)$ 
& $A=\{i\e1 + \lc \mid i\ttk{p^3}\} $ 
& & row 25 \\
& $\so=\{i\e1+\alpha_i\e2+\beta_i\e3 \mid i\ttk{p^3}\} $ & & \\ \hline

& $\lc=\<{\e2} $ & & \\
& $A=\{(i+\alpha_i p^2)\e1 + j\e3 + \lc \mid i\ttk{p^2},\,j\ttk{2}\} $ & & \\
$(2p^2,2p,2)$ 
& $B=\{ip^2\e1 + \alpha_i\e3 + \lc \mid i\ttk{p}\} $ 
& \ref{p^3.2} \row{15} & row 34 \\ 
& $\so=\{(i+\alpha_i p^2)\e1 + \alpha_{ij}\e2 + j\e3 \mid i\ttk{p^2},\,j\ttk{2}\} $ & & \\
& $C=\{ip^2\e1 + j\e2 + \alpha_i\e3 \mid i\ttk{p},\,j\ttk{2}\} $ & & \\ \hline

& $\lc=\<{\e2} $ & & \\
& $A=\{(i+\alpha_{ik}p+jp^2)\e1 + k\e3 + \lc \mid i,j\ttk{p},\,k\ttk{2}\} $ & & \\
$(2p^2,2p,2 )$ 
& $B=\{(ip+\alpha_ip^2)\e1 + \lc \mid i\ttk{p}\} $ 
& \ref{p^3.2} \row{16} & row 35 \\ 
& $\so=\{(i+\alpha_{ik}p+jp^2)\e1 + \alpha_{ijk}\e2 + k\e3 \mid i,j\ttk{p},\,k\ttk{2}\} $ & & \\
& $C=\{(ip+\alpha_ip^2)\e1 + j\e2 \mid i\ttk{p},\,j\ttk{2}\} $ & & \\ \hline

& $\lc=\<{\e2} $ & & \\
& $A=\{(i+\alpha_ip+jp^2)\e1 + k\e3 + \lc \mid i,j\ttk{p},\,k\ttk{2}\} $ & & \\
$(2p^2,2p,2 )$ 
& $B=\{(ip+\alpha_ip^2)\e1 + \alpha_i\e3 + \lc \mid i\ttk{p}\} $ 
& \ref{p^3.2} \row{17} & row 36 \\ 
& $\so=\{(i+\alpha_ip+jp^2)\e1 + \alpha_{ijk}\e2 + k\e3 \mid i,j\ttk{p},\,k\ttk{2}\} $ & & \\
& $C=\{(ip+\alpha_ip^2)\e1 + j\e2 + \alpha_i\e3 \mid i\ttk{p},\,j\ttk{2}\} $ & & \\ \hline

& $\lc=\<{\e2} $ & & \\
& $A=\{(\alpha_i+jp)\e1 + i\e3 + \lc \mid i\ttk{2},\,j\ttk{p^2}\} $ & & \\
$(2p^2,2p,2 )$ 
& $B=\{(i+\alpha_i p)\e1 + \lc \mid i\ttk{p}\} $ 
& \ref{p^3.2} \row{18} & row 37 \\ 
& $\so=\{(\alpha_i+jp)\e1 + \alpha_{ij}\e2 + i\e3 \mid i\ttk{2},\,j\ttk{p^2}\} $ & & \\
& $C=\{(i+\alpha_i p)\e1 + j\e2 \mid i\ttk{p},\,j\ttk{2}\} $ & & \\ \hline

& $\lc=\<{p^2\e1} $ & & \\
& $A=\{i\e1 + j\e2 + \alpha_i\e3 + \lc \mid i\ttk{p^2},\,j\ttk{2}\} $ & & \\
$(2p^2,2p,p)$ 
& $B=\{\alpha_i\e2 + i\e3 + \lc \mid i\ttk{2}\} $ 
& \ref{p^2.2.2} \row{13} & row 38 \\ 
& $\so=\{(i+\alpha_{ij}p^2)\e1 + j\e2 + \alpha_i\e3 \mid i\ttk{p^2},\,j\ttk{2}\} $ & & \\
& $C=\{jp^2\e1 + \alpha_i\e2 + i\e3 + \lc \mid i\ttk{2},\,j\ttk{p}\} $ & & \\ \hline

& $\lc=\<{p^2\e1} $ & & \\
& $A=\{(i+kp)\e1 + \alpha_{ij}\e2 + j\e3 + \lc \mid i,k\ttk{p},\,j\ttk{2}\} $ & & \\
$(2p^2,2p,p)$ 
& $B=\{\alpha_ip\e1 + i\e2 + \lc \mid i\ttk{2}\} $ 
& \ref{p^2.2.2} \row{14} & row 39 \\ 
& $\so=\{(i+kp+\alpha_{ijk}p^2)\e1 + \alpha_{ij}\e2 + j\e3 \mid i,k\ttk{p},\,j\ttk{2}\} $ & & \\
& $C=\{(\alpha_ip+jp^2)\e1 + i\e2 + \lc \mid i\ttk{2},\,j\ttk{p}\} $ & & \\ \hline

& $\lc=\<{p^2\e1} $ & & \\
& $A=\{(i+jp)\e1 + k\e2 + \alpha_i\e3 + \lc \mid i,j\ttk{p},\,k\ttk{2}\} $ & & \\
$(2p^2,2p,p )$ 
& $B=\{\alpha_ip\e1 + \alpha_i\e2 + i\e3 + \lc \mid i\ttk{2}\} $ 
& \ref{p^2.2.2} \row{15} & row 40 \\ 
& $\so=\{(i+jp+\alpha_{ijk}p^2)\e1 + k\e2 + \alpha_i\e3 \mid i,j\ttk{p},\,k\ttk{2}\} $ & & \\
& $C=\{(\alpha_ip+jp^2)\e1 + \alpha_i\e2 + i\e3 \mid i\ttk{2},\,j\ttk{p}\} $ & & \\ \hline

& $C=\<{p^2\e1+\e2} $ & & \\
$(2p^2,2p,2p)$ 
& $A=\{i\e1 + j\e3 + \lc \mid i\ttk{p^2},\,j\ttk{2}\} $ 
& & row 41 \\
& $\so=\{(i+\alpha_{ij}p^2)\e1 +\alpha_{ij}\e2 + j\e3 \mid i\ttk{p^2},\,j\ttk{2}\} $ & & \\ \hline

& $\lc=\<{p^2\e1} $ & & \\
& $A=\{(i+\alpha_{ij}p)\e1 + k\e2 + j\e3 \mid i\ttk{p},\,j,k\ttk{2}\} $ & & \\
$(4p,p^2,p)$ 
& $B=\{ip\e1 + \alpha_i\e2 + \lc \mid i\ttk{p} \} $ 
& \ref{p^2.2.2} \row{17} & row 49 \\ 
& $\so=\{(i+\alpha_{ij}p+\alpha_{ijk}p^2)\e1 + k\e2 + j\e3 \mid i\ttk{p},\,j,k\ttk{2}\} $ & & \\
& $C=\{(ip+jp^2)\e1 + \alpha_i\e2  \mid i,j\ttk{p} \} $ & & \\ \hline

& $\lc=\<{p^2\e1} $ & & \\
& $A=\{(\alpha_{ij}+kp)\e1+i\e2+j\e3 + \lc \mid i,j\ttk{2},\,k\ttk{p}\} $ & & \\
$(4p,p^2,p )$ 
& $B=\{ (i+\alpha_i p)\e1 + \lc \mid i\ttk{p} \} $ 
& \ref{p^2.2.2} \row{18} & row 50 \\ 
& $\so=\{(\alpha_{ij}+kp+\alpha_{ijk}p^2)\e1+i\e2+j\e3 \mid i,j\ttk{2},\,k\ttk{p}\} $ & & \\
& $C=\{ (i+\alpha_ip+jp^2)\e1 \mid i,j\ttk{p} \} $ & & \\ \hline

& $\lc=\<{p^2\e1} $ & & \\
& $A=\{(\alpha_i+jp)\e1 + k\e2 + i\e3 + \lc \mid i,k\ttk{2},\,j\ttk{p} \} $ & & \\
$(4p,p^2,p )$ 
& $B=\{(i+\alpha_i p)\e1 + \alpha_i\e2 + \lc \mid i\ttk{p} \} $ 
& \ref{p^2.2.2} \row{19} & row 51 \\ 
& $\so=\{(\alpha_i+jp+\alpha_{ijk}p^2)\e1 + k\e2 + i\e3 \mid i,k\ttk{2},\,j\ttk{p} \} $ & & \\
& $C=\{(i+\alpha_ip+jp^2)\e1 + \alpha_i\e2 \mid i,j\ttk{p} \} $ & & \\ \hline

& $\lc=\<{p^2\e1} $ & & \\
& $A=\{(i+\alpha_ip,j,k) + \lc \mid i\ttk{p},\,j,k\ttk{2} \} $ & & \\
$(4p,p^2,p )$ 
& $B=\{ip\e1+\alpha_i\e2+\beta_i\e3  + \lc \mid i\ttk{p} \} $ 
& \ref{p^2.2.2} \row{20} & row 52 \\ 
& $\so=\{(i+\alpha_ip+\alpha_{ijk}p^2)\e1+j\e2+k\e3 \mid i\ttk{p},\,j,k\ttk{2} \} $ & & \\
& $C=\{(ip+jp^2)\e1+\alpha_i\e2+\beta_i\e3 \mid i,j\ttk{p} \} $ & & \\ \hline

& $C=\<{p\e1} $ & & \\
$(4p,p^2,p^2 )$ 
& $A=\{i\e1+j\e2+k\e3 + \lc \mid i\ttk{p},\,j,k\ttk{2}\} $ 
& & row 53 \\
& $\so=\{(i+\alpha_{ijk}p)\e1+j\e2+k\e3 \mid i\ttk{p},\,j,k\ttk{2}\} $ & & \\ \hline

& $C=\<{\e2} $ & & \\
$(2p^3,2,2 )$ 
& $A=\{i\e1 + j\e3 +\lc \mid i\ttk{p^3},\,j\ttk{2}\} $ 
& & row 59 \\
& $\so=\{i\e1 + \alpha_{ij}\e2 + j\e3 \mid i\ttk{p^3},\,j\ttk{2}\} $ & & \\ \hline

& $C=\<{p^2\e1} $ & & \\
$(4p^2,p,p)$ 
& $A=\{i\e1+j\e2+k\e3 +\lc \mid i\ttk{p^2},\,j,k\ttk{2}\} $ 
& & row 67 \\
& $\so=\{(i+\alpha_{ijk}p^2)\e1+j\e2+k\e3 \mid i\ttk{p^2},\,j,k\ttk{2}\} $ & & \\ \hline

\end{longtable}

\smallskip
    \textsf{Case 2.} $\so$ is periodic.
    \smallskip

    In this case we have $|\ls|\geq 2$ and $\so=D\oplus \ls$ for some subset $D$ of $G$ by Lemma \ref{period}.
    So $|D|\cdot|\ls|=|\so|$ and $1<|\ls|<|\so|$ by Lemma \ref{prime deg} part (b).
    Set $$X=(D+\ls)/\ls\quad\text{and}\quad Y=(C+\ls)/\ls.$$
    Since $\so$ is periodic, by Lemma \ref{quotient}, we have
    $$G/\lc=X\oplus Y$$ and $X$ is aperiodic in $G/\ls$.
    We also know that $X$ contains $\ls$ and $X$ generates $G/\ls$ by Lemma \ref{generate-zero}.
    This factorization enables us to make use of Lemma \ref{d_i+l_i} and previously established theorems to obtain the pair $(\so, C)$ via $(X,Y)$.
    
    Similar to Table \ref{tp.p^2.2.2}, the columns of Table \ref{tp.p^3.2.2} represents all possible values for $\left(|\so|,|C|,|\ls|\right)$; $\ls$, $X$, $Y$, $\so$, and $C$; which previous theorem is used to obtain $(X,Y)$; and which row in Table \ref{t.p^3.2.2} the result corresponds to.

\begin{longtable}{|c|l|c|c|}
\caption{The case when $G= $ and $\so$ is periodic}
\label{tp.p^3.2.2}
\\
\flp

& $\ls=\<{p^2\e1} $ & & \\
& $X=\{(i+\alpha_ip)\e1+\alpha_i\e2+\beta_i\e3+\ls \mid i\ttk{p}\} $ & & \\
$(p^2,4p,p)$ 
& $Y=\{i\e1+j\e2+k\e3+\ls \mid i\ttk{p},\,j,k\ttk{2}\} $ 
& \ref{p^2.2.2} \row1 & row 2 \\ 
& $\so=\{(i+\alpha_i p + jp^2)\e1+\alpha_i\e2+\beta_i\e3 \mid i,j\ttk{p}\} $ & & \\
& $C=\{(i+\alpha_{ijk}p^2)\e1+j\e2+k\e3 \mid i\ttk{p},\,j,k\ttk{2}\} $ & & \\ \hline

& $\ls=\<{\e2} $ & & \\
& $X=\{(i+\alpha_i p)\e1+\alpha_i\e3+\ls \mid i\ttk{p}\} $ & & \\
$(2p,2p^2,2)$ 
& $Y=\{ip\e1+j\e3+\ls \mid i\ttk{p^2},\,j\ttk{2}\} $ 
& \ref{p^3.2} \row1 & row 8 \\ 
& $\so=\{(i+\alpha_i p)\e1+j\e2+\alpha_i\e3\mid i\ttk{p},\,j\ttk{2}\} $ & & \\
& $C=\{ip\e1+j\e3+\alpha_{ij}\e2 \mid i\ttk{p^2},\,j\ttk{2}\} $ & & \\ \hline

$(2p,2p^2,p)$ 
& Not possible since no $X$ exists
& \ref{p^2.2.2}  & \\ \hline

& $\ls=\<{\e2} $ & & \\
& $X=\{\alpha_i\e1 + i\e3 + \ls \mid i\ttk{p}\} $ & & \\
$(4,p^3,2)$ 
& $Y=\{i\e1+\ls \mid i\ttk{p^3}\} $ 
& \ref{p^3.2} \row2 & row 16 \\ 
& $\so=\{\alpha_i\e1 + j\e2 + i\e3  \mid i\ttk{p},\,j\ttk{2}\} $ & & \\
& $C=\{i\e1+\alpha_i\e2 \mid i\ttk{p^3}\} $ & & \\ \hline

& $\ls=\<{p^2\e1} $ & & \\
& $X=\{(i+jp)\e1 + \alpha_{ij}\e2 + \alpha_i\e3 + \ls \mid i,j\ttk{p}\} $ & & \\
$(p^3,4,p)$ 
& $Y=\{\alpha_ip\e1 + j\e2 + i\e3  + \ls \mid i,j\ttk{2}\} $ 
& \ref{p^2.2.2} \row3 & row 20 \\ 
& $\so=\{(i+jp+kp^2)\e1 + \alpha_{ij}\e2 + \alpha_i\e3 \mid i,j,k\ttk{p}\} $ & & \\
& $C=\{\alpha_ip\e1 + j\e2 + i\e3  + \ls \mid i,j\ttk{2}\} $ & & \\ \hline

& $\ls=\<{p^2\e1} $ & & \\
& $X=\{i\e1+\alpha_i\e2+\beta_i\e3 + \ls \mid i\ttk{p^2}\} $ & & \\
$(p^3,4,p )$ 
& $Y=\{i\e2+j\e3 + \ls \mid i,j\ttk{2}\} $ 
& \ref{p^2.2.2} \row4 & row 21 \\ 
& $\so=\{(i+jp^2)\e1+\alpha_i\e2+\beta_i\e3 \mid i\ttk{p^2},j\ttk{p}\} $ & & \\
& $C=\{\alpha_{ij}p^2\e1+i\e2+j\e3 \mid i,j\ttk{2}\} $ & & \\ \hline

& $\ls=\<{p\e1} $ & & \\
& $X=\{i\e1+\alpha_i\e2+\beta_i\e3 + \ls \mid i\ttk{p}\} $ & & \\
$(p^3,4,p^2 )$ 
& $Y=\{i\e2+j\e3 + \ls \mid i,j\ttk{2}\} $ 
& \ref{p.2.2} \row1 & row 22 \\ 
& $\so=\{(i+jp)\e1+\alpha_i\e2+\beta_i\e3 \mid i\ttk{p},j\ttk{p^2}\} $ & & \\
& $C=\{\alpha_{ij}p\e1+i\e2+j\e3 \mid i,j\ttk{2}\} $ & & \\ \hline

& $\ls=\<{\e2} $ & & \\
& $X=\{(i+\alpha_ip+jp^2)\e1 + \alpha_{ij}\e3 + \ls \mid i,j\ttk{p}\} $ & & \\
$(2p^2,2p,2 )$ 
& $Y=\{(ip+\alpha_ip^2)\e1 + j\e3 + \ls \mid i\ttk{p},\,j\ttk{2}\} $ 
& \ref{p^3.2} \row4 & row 26 \\ 
& $\so=\{(i+\alpha_ip+jp^2)\e1 + k\e2 + \alpha_{ij}\e3 \mid i,j\ttk{p},\,k\ttk{2}\} $ & & \\
& $C=\{(ip+\alpha_ip^2)\e1 + \alpha_{ij}\e2 + j\e3 \mid i,j\ttk{2}\} $ & & \\ \hline

& $\ls=\<{\e2} $ & & \\
& $X=\{(i+jp+\alpha_{ij}p^2)\e1 + i\e3 + \ls \mid i,j\ttk{p}\} $ & & \\
$(2p^2,2p,2 )$ 
& $Y=\{(\alpha_ip+jp^2)\e1 + i\e3 + \ls \mid i\ttk{2},\,j\ttk{p}\} $ 
& \ref{p^3.2} \row5 & row 27 \\ 
& $\so=\{(i+jp+\alpha_{ij}p^2)\e1 + k\e2 + i\e3 \mid i,j\ttk{p},\,k\ttk{2}\} $ & & \\
& $C=\{(\alpha_ip+jp^2)\e1 + \alpha_{ij}\e2 + i\e3 \mid i\ttk{2},\,j\ttk{p}\} $ & & \\ \hline

& $\ls=\<{\e2} $ & & \\
& $X=\{(i+\alpha_ip^2)\e1 + \alpha_i\e3 + \ls \mid i\ttk{p^2}\} $ & & \\
$(2p^2,2p,2 )$ 
& $Y=\{ip^2\e1 + j\e3 + \ls \mid i\ttk{p},\,j\ttk{2}\} $ 
& \ref{p^3.2} \row6 & row 28 \\ 
& $\so=\{(i+\alpha_ip^2)\e1 + j\e2 + \alpha_i\e3 \mid i\ttk{p^2},\,j\ttk{2}\} $ & & \\
& $C=\{ip^2\e1 + \alpha_{ij}\e2 + j\e3 \mid i\ttk{p},\,j\ttk{2}\} $ & & \\ \hline

& $\ls=\<{p^2\e1} $ & & \\
& $X=\{(i+\alpha_ip)\e1 + \alpha_{ij}\e2 + j\e3 + \ls \mid i\ttk{p},\,j\ttk{2}\} $ & & \\
$(2p^2,2p,p )$ 
& $Y=\{ip\e1 + j\e2 + \alpha_i\e3 + \ls \mid i\ttk{p},\,j\ttk{2}\} $ 
& \ref{p^2.2.2} \row6 & row 29 \\ 
& $\so=\{(i+\alpha_ip+kp^2)\e1 + \alpha_{ij}\e2 + j\e3 \mid i,k\ttk{p},\,j\ttk{2}\} $ & & \\
& $C=\{(ip + \alpha_{ij}p^2)\e1 + j\e2 + \alpha_i\e3 \mid i\ttk{p},\,j\ttk{2}\} $ & & \\ \hline

& $\ls=\<{p^2\e1} $ & & \\
& $X=\{(\alpha_i + jp)\e1 + \alpha_{ij}\e2 + i\e3 + \ls \mid i\ttk{2},\,j\ttk{p}\} $ & & \\
$(2p^2,2p,p )$ 
& $Y=\{(i+\alpha_i p)\e1 + j\e2 + \ls \mid i\ttk{p},\,j\ttk{2}\} $ 
& \ref{p^2.2.2} \row7 & row 30 \\ 
& $\so=\{(\alpha_i + jp+kp^2)\e1 + \alpha_{ij}\e2 + i\e3 \mid i\ttk{2},\,j,k\ttk{p}\} $ & & \\
& $C=\{(i+\alpha_ip+\alpha_{ij}p^2)\e1 + j\e2 \mid i\ttk{p},\,j\ttk{2}\} $ & & \\ \hline

& $\ls=\<{p^2\e1} $ & & \\
& $X=\{(i+\alpha_{ij}p)\e1 + j\e2 + \alpha_i\e3 + \ls \mid i\ttk{p},j\ttk{2}\} $ & & \\
$(2p^2,2p,p )$ 
& $Y=\{jp\e1 + \alpha_i\e2 + i\e3 + \ls \mid i\ttk{2},\,j\ttk{p}\} $ 
& \ref{p^2.2.2} \row8 & row 31 \\ 
& $\so=\{(i+\alpha_{ij}p+kp^2)\e1 + j\e2 + \alpha_i\e3 \mid i,k\ttk{p},\,j\ttk{2}\} $ & & \\
& $C=\{(jp+\alpha_{ij}p^2)\e1 + \alpha_i\e2 + i\e3 \mid i\ttk{2},\,j\ttk{p}\} $ & & \\ \hline

& $\ls=\<{p^2\e1} $ & & \\
& $X=\{(i+\alpha_{ij}p)\e1 + \alpha_{ij}\e2 + j\e3 + \ls \mid i\ttk{p},\,j\ttk{2}\} $ & & \\
$(2p^2,2p,p )$ 
& $Y=\{ip\e1 + j\e2 + \ls \mid i\ttk{p},\,j\ttk{2}\} $ 
& \ref{p^2.2.2} \row9 & row 32 \\ 
& $\so=\{(i+\alpha_{ij}p+kp^2)\e1 + \alpha_{ij}\e2 + j\e3 \mid i,k\ttk{p},\,j\ttk{2}\} $ & & \\
& $C=\{(ip+\alpha_{ij}p^2)\e1 + j\e2 \mid i\ttk{p},\,j\ttk{2}\} $ & & \\ \hline

& $\ls=\<{p^2\e1+\e2} $ & & \\
& $X=\{i\e1 + \alpha_i\e3 + \ls \mid i\ttk{p}\} $ & & \\
$(2p^2,2p,2p)$ 
& $Y=\{ip\e1 + j\e3 + \ls \mid i\ttk{p},\,j\ttk{2}\} $ 
& \ref{p^2.2} \row1 & row 33 \\ 
& $\so=\{(i+jp^2)\e1 + k\e2 + \alpha_i\e3 \mid i,j\ttk{p},\,k\ttk{2}\} $ & & \\
& $C=\{(ip+\alpha_{ij}p^2)\e1 + \alpha_{ij}\e2 + j\e3 \mid i\ttk{p},\,j\ttk{2}\} $ & & \\ \hline

$(2p^2,2p,p^2)$ 
& Not possible since no $X$ exists
& \ref{p.2.2}  & \\ \hline

& $\ls=\<{\e2} $ & & \\
& $X=\{(i+\alpha_ip+\alpha_{ij}p^2)\e1 + j\e3 + \ls \mid i\ttk{p},\,j\ttk{2}\} $ & & \\
$(4p,p^2,2 )$ 
& $Y=\{(ip+jp^2)\e1 + \alpha_i\e3 + \ls \mid i,j\ttk{p}\} $ 
& \ref{p^3.2} \row{9} & row 42 \\ 
& $\so=\{(i+\alpha_ip+\alpha_{ij}p^2)\e1 + k\e2 + j\e3 \mid i\ttk{p},\,j,k\ttk{2}\} $ & & \\
& $C=\{(ip+jp^2)\e1 + \alpha_{ij}\e2 + \alpha_i\e3 \mid i,j\ttk{p}\} $ & & \\ \hline

& $\ls=\<{\e2} $ & & \\
& $X=\{(\alpha_i+jp+\alpha_{ij}p^2)\e1 + i\e3 + \ls \mid i\ttk{2},\,j\ttk{p}\} $ & & \\
$(4p,p^2,2 )$ 
& $Y=\{(i+\alpha_ip+jp^2)\e1 + \ls \mid i,j\ttk{p}\} $ 
& \ref{p^3.2} \row{10} & row 43 \\ 
& $\so==\{(\alpha_i+jp+\alpha_{ij}p^2)\e1 + k\e2 + i\e3 \mid i,k\ttk{2},\,j\ttk{p}\} $ & & \\
& $C=\{(i+\alpha_ip+jp^2)\e1 + \alpha_{ij}\e2 \mid i,j\ttk{p}\} $ & & \\ \hline

& $\ls=\<{\e2} $ & & \\
& $X=\{(i+\alpha_{ij}p)\e1 + j\e3 + \ls \mid i\ttk{p},\,j\ttk{2}\} $ & & \\
$(4p,p^2,2 )$ 
& $Y=\{ip\e1 + \ls \mid i\ttk{p^2}\} $ 
& \ref{p^3.2} \row{11} & row 44 \\ 
& $\so=\{(i+\alpha_{ij}p)\e1 + k\e2 + j\e3 \mid i\ttk{p},\,j,k\ttk{2}\} $ & & \\
& $C=\{ip\e1 + \alpha_i\e2 \mid i\ttk{p^2}\} $ & & \\ \hline

& $\ls=\<{p^2\e1} $ & & \\
& $X=\{(\alpha_i + \alpha_{ij}p)\e1 + j\e2 + i\e3  + \ls \mid i,j\ttk{2}\} $ & & \\
$(4p,p^2,p )$ 
& $Y=\{(i+jp)\e1 + \alpha_i\e2 + \ls \mid i,j\ttk{p}\} $ 
& \ref{p^2.2.2} \row{11} & row 45 \\ 
& $\so=\{(\alpha_i+\alpha_{ij}p+kp^2)\e1 + j\e2 + i\e3  \mid i,j\ttk{2},\,k\ttk{p}\} $ & & \\
& $C=\{(i+jp+\alpha_{ij}p^2)\e1 + \alpha_i\e2 \mid i,j\ttk{p}\} $ & & \\ \hline

& $\ls=\<{p^2\e1} $ & & \\
& $X=\{\alpha_{ij}\e1+i\e2+j\e3 + \ls \mid i,j\ttk{2}\} $ & & \\
$(4p,p^2,p )$ 
& $Y=\{i\e1 + \ls \mid i\ttk{p^2}\} $ 
& \ref{p^2.2.2} \row{12} & row 46 \\ 
& $\so=\{(\alpha_{ij}+kp^2)\e1+i\e2+j\e3 \mid i,j\ttk{2},\,k\ttk{p}\} $ & & \\
& $C=\{(i+\alpha_i p^2)\e1 \mid i\ttk{p^2}\} $ & & \\ \hline

& $\ls=\ls=\<{\e2,\e3} $ & & \\
& $X=\{(i+\alpha_ip)\e1 + \ls \mid i\ttk{p}\} $ & & \\
$(4p,p^2,4 )$ 
& $Y=\{ip\e1 + \ls \mid i\ttk{p^2}\} $ 
& \ref{p^3} \row1 & row 47 \\ 
& $\so=\{(i+\alpha_ip)\e1+j\e2+k\e3 \mid i\ttk{p},\,j,k\ttk{2}\} $ & & \\
& $C=\{ip\e1+\alpha_i\e2+\beta_i\e3 \mid i\ttk{p^2}\} $ & & \\ \hline

& $\ls=\<{p^2\e1+\e2} $ & & \\
& $X=\{\alpha_i\e1 + i\e3 + \ls \mid i\ttk{2}\} $ & & \\
$(4p,p^2,2p )$ 
& $Y=\{i\e1 + \ls \mid i\ttk{p^2}\} $ 
& \ref{p^2.2} \row2 & row 48 \\ 
& $\so=\{(\alpha_i+jp^2)\e1 + k\e2 + i\e3 \mid i,k\ttk{2},\,j\ttk{p}\} $ & & \\
& $C=\{(i+\alpha_i p^2)\e1 + \alpha_i\e2 \mid i\ttk{p^2}\} $ & & \\ \hline

& $\ls=\<{\e2} $ & & \\
& $X=\{i\e1 + \alpha_i\e3 + \ls \mid i\ttk{p^3}\} $ & & \\
$(2p^3,2,2)$ 
& $Y=\{i\e3 + \ls \mid i\ttk{2}\} $ 
& \ref{p^3.2} \row{14} & row 54 \\ 
& $\so=\{i\e1 + j\e2 + \alpha_i\e3 \mid i\ttk{p^3},\,j\ttk{2}\} $ & & \\
& $C=\{\alpha_i\e2 + i\e3 \mid i\ttk{2}\} $ & & \\ \hline

& $\ls=\<{p^2\e1} $ & & \\
& $X=\{i\e1 + \alpha_{ij}\e2 + j\e3 + \ls \mid i\ttk{p^2},\,j\ttk{2}\} $ & & \\
$(2p^3,2,p )$ 
& $Y=\{i\e2 + \ls \mid i\ttk{2}\} $ 
& \ref{p^2.2.2} \row{16} & row 55 \\ 
& $\so=\{(i+kp^2)\e1 + \alpha_{ij}\e2 + j\e3 \mid i\ttk{p^2},\,j\ttk{2},\,k\ttk{p}\} $ & & \\
& $C=\{\alpha_ip^2\e1 + i\e2 \mid i\ttk{2}\} $ & & \\ \hline

& $\ls=\<{p^2\e1+\e2} $ & & \\
& $X=\{i\e1 + \alpha_i\e3 + \ls \mid i\ttk{p^2}\} $ & & \\
$(2p^3,2,2p )$ 
& $Y=\{i\e3 + \ls \mid i\ttk{2}\} $ 
& \ref{p^2.2} \row4 & row 56 \\ 
& $\so=\{(i+jp^2)\e1 + k\e2 + \alpha_i\e3 \mid i\ttk{p^2},\,j\ttk{p},\,k\ttk{2}\} $ & & \\
& $C=\{\alpha_ip^2\e1 + \alpha_i\e2 + i\e3 \mid i\ttk{2}\} $ & & \\ \hline

& $\ls=\<{p\e1} $ & & \\
& $X=\{i\e1 + \alpha_{ij}\e2 + j\e3 + \ls \mid i\ttk{p},\,j\ttk{2}\} $ & & \\
$(2p^3,2,p^2 )$ 
& $Y=\{i\e2 + \ls \mid i\ttk{2}\} $ 
& \ref{p.2.2} \row5 & row 57 \\ 
& $\so=\{(i+kp)\e1 + \alpha_{ij}\e2 + j\e3 \mid i\ttk{p},\,j\ttk{2},\,k\ttk{p^2}\} $ & & \\
& $C=\{\alpha_ip\e1 + i\e2 \mid i\ttk{2}\} $ & & \\ \hline

& $\ls=\<{p\e1+\e2} $ & & \\
& $X=\{i\e1 + \alpha_i\e3 + \ls \mid i\ttk{p}\} $ & & \\
$(2p^3,2,2p^2)$ 
& $Y=\{i\e3 + \ls \mid i\ttk{2}\} $ 
& \ref{p.q} \row1 & row 58 \\ 
& $\so=\{(i+jp)\e1 + k\e2 + \alpha_i\e3 \mid i\ttk{p},\,j\ttk{p^2},\,k\ttk{2}\} $ & & \\
& $C=\{\alpha_ip\e1 + \alpha_i\e2 + i\e3 \mid i\ttk{2}\} $ & & \\ \hline

$(2p^3,2,p^3)$ 
& Not possible since no $X$ exists
& \ref{2.2}  & \\ \hline

& $\ls=\<{\e2} $ & & \\
& $X=\{(i+\alpha_{ij}p^2)\e1 + j\e3 + \ls \mid i\ttk{p^2},\,j\ttk{2}\} $ & & \\
$(4p^2,p,2 )$ 
& $Y=\{ip^2\e1 + \ls  \mid i\ttk{p}\} $ 
& \ref{p^3.2} \row{19} & row 60 \\ 
& $\so=\{(i+\alpha_{ij}p^2)\e1 + k\e2 + j\e3 \mid i\ttk{p^2},\,j,k\ttk{2}\} $ & & \\
& $C=\{ip^2\e1 +\alpha_i\e2 \mid i\ttk{p}\} $ & & \\ \hline

& $\ls=\<{p^2\e1} $ & & \\
& $X=\{(i+\alpha_{ijk}p)\e1+j\e2+k\e3 + \ls \mid i\ttk{p},\,j,k\ttk{2}\} $ & & \\
$(4p^2,p,p )$ 
& $Y=\{ip\e1 + \ls \mid i\ttk{p}\} $ 
& \ref{p^2.2.2} \row{21} & row 61 \\ 
& $\so=\{(i+\alpha_{ijk}p+lp^2)\e1+j\e2+k\e3 + \ls \mid i,l\ttk{p},\,j,k\ttk{2}\} $ & & \\
& $C=\{(ip+\alpha_ip^2)\e1 \mid i\ttk{p}\} $ & & \\ \hline

& $\ls=\<{\e2,\e3} $ & & \\
& $X=\{(i+\alpha_ip^2)\e1 + \ls \mid i\ttk{p^2}\} $ & & \\
$(4p^2,p,4 )$ 
& $Y=\{ip^2\e1 + \ls  \mid i\ttk{p}\} $ 
& \ref{p^3} \row3 & row 62 \\ 
& $\so=\{(i+\alpha_ip^2)\e1+j\e2+k\e3 \mid i\ttk{p^2},\,j,k\ttk{2}\} $ & & \\
& $C=\{ip^2\e1+\alpha_i\e2+\beta_i\e3 \mid i\ttk{p}\} $ & & \\ \hline

& $\ls=\<{p^2\e1+\e2} $ & & \\
& $X=\{(i+\alpha_{ij}p)\e1 + j\e3 + \ls \mid i\ttk{p},\,j\ttk{2}\} $ & & \\
$(4p^2,p,2p )$ 
& $Y=\{ip\e1 + \ls  \mid i\ttk{p}\} $ 
& \ref{p^2.2} \row7 & row 63 \\ 
& $\so=\{(i+\alpha_{ij}p+kp^2)\e1 + l\e2 + j\e3 \mid i,k\ttk{p},\,j,l\ttk{2}\} $ & & \\
& $C=\{(ip+\alpha_ip^2)\e1 +\alpha_i\e2 \mid i\ttk{p}\} $ & & \\ \hline

& $\ls=\<{p\e1} $ & & \\
& $X=\{\alpha_{ij}\e1+i\e2+j\e3+\ls \mid i,j\ttk{2}\} $ & & \\
$(4p^2,p,p^2)$ 
& $Y=\{i\e1+\ls \mid i\ttk{p}\} $ 
& \ref{p.2.2} \row7 & row 64 \\ 
& $\so=\{(\alpha_{ij}+kp)\e1+i\e2+j\e3 \mid i,j\ttk{2}\,k\ttk{p^2}\} $ & & \\
& $C=\{(i+\alpha_ip)\e1 \mid i\ttk{p}\} $ & & \\ \hline

& $\ls=\{p^2\e1+\e2,\e3\} $ & & \\
& $X=\{(i+\alpha_ip)\e1 + \ls \mid i\ttk{p}\} $ & & \\
$(4p^2,p,4p )$ 
& $Y=\{ip\e1 + \ls  \mid i\ttk{p}\} $ 
& \ref{p^2} & row 65 \\ 
& $\so=\{(i+\alpha_ip+jp^2)\e1+k\e2+l\e3 \mid i,j\ttk{p},\,k,l\ttk{2}\} $ & & \\
& $C=\{(ip+\alpha_ip^2)\e1+\alpha_i\e2+\beta_i\e3 \mid i\ttk{p}\} $ & & \\ \hline

& $\ls=\<{p\e1+\e2} $ & & \\
& $X=\{\alpha_i\e1 + i\e3 + \ls \mid i\ttk{2}\} $ & & \\
$(4p^2,p,2p^2 )$ 
& $Y=\{i\e1 + \ls  \mid i\ttk{p}\} $ 
& \ref{p.q} \row2 & row 66 \\ 
& $\so=\{(\alpha_i+jp)\e1 + k\e2 + i\e3 \mid i,k\ttk{2},\,j\ttk{p^2}\} $ & & \\
& $C=\{(i+\alpha_ip)\e1 +\alpha_i\e2 \mid i\ttk{p}\} $ & & \\ \hline

\end{longtable}
\end{proof}

\begin{remark}
    In Theorem \ref{p^3.2.2}, we can change $\e2$ and $\e3$ to any two generators of $\{0\}\x\Z_2\x\Z_2$.
\end{remark}

\begin{theorem}\label{p^2.2.2.2}
Let $G=\Z_{p^2}\x\Z_2\x\Z_2\x\Z_2$ for an odd prime $p$ and let $S$ be a proper subset of $G\setminus \{0\}$ that generates $G$. Let $0\in C\subseteq G$. Then $C$ is a perfect code in $\Cay(G,S)$ if and only if $(\so, C)$ is one of the pairs in Table \ref{t.p^2.2.2.2} for some $\alpha_i$'s, $\alpha_{ij}$'s, $\alpha_{ijk}$'s, $\beta_i$'s, and $\gamma_i$'s in $\Z$ where $\alpha_0=\alpha_{00}=\alpha_{000}=\beta_0=\gamma_0=0$.
\end{theorem}

\begin{footnotesize}
\rc\begin{longtable}{|c|l|l|}
\caption{Perfect codes in Cayley graphs of $\Z_{p^2}\x\Z_2\x\Z_2\x\Z_2$}
\label{t.p^2.2.2.2}
\\
\fl
\tn & $\{(i+\alpha_ip)\e1+\alpha_i\e2+\beta_i\e3+\gamma_i\e4 \mid i\ttk{p}\}$
& $\{ip\e1+j\e2+k\e3+l\e4\mid i\ttk{p},\,j,k,l\ttk{2}\}$ \\
\tn & $\{(i+jp)\e1+\alpha_i\e2+\beta_i\e3+\gamma_i\e4 \mid i,j\ttk{p}\}$
& $\{\alpha_{ijk}p\e1+i\e2+j\e3+k\e4 \mid i,j,k\ttk{2}\}$ \\
\tn & $\{(i+jp)\e1 + \alpha_{ij}\e2 + \alpha_i\e3 + \beta_i\e4 \mid i,j\ttk{p}\}$
& $\{\alpha_{ij}p\e1 + k\e2 + i\e3 + j\e4 \mid i,j,k\ttk{2}\}$ \\
\tn & $\{(i+jp)\e1 + \alpha_{ij}\e2 + \beta_{ij}\e3 + \alpha_i\e4  \mid i,j\ttk{p}\}$
& $\{\alpha_ip\e1 + j\e2 + k\e3 + i\e4  \mid i,j,k\ttk{2}\}$ \\
\tn & $\{i\e1+\alpha_i\e2+\beta_i\e3+\gamma_i\e4 \mid i\ttk{p^2}\}$
& $\{\e1+i\e2+j\e3+k\e4 \mid i,j,k\ttk{2}\}$ \\
\tn & $\{(i+\alpha_ip)\e1 + k\e2 + \alpha_i\e3 + \beta_i\e4 \mid i\ttk{p},\,k\ttk{2}\}$
& $\{ip\e1 + \alpha_{ijk}\e2 + j\e3 + k\e4 \mid i\ttk{p},\,j,k\ttk{2}\}$ \\
\tn & $\{(i+\alpha_ip)\e1 + \alpha_{ij}\e2 + j\e3 + \alpha_i\e4 \mid i\ttk{p},\,j\ttk{2}\}$
& $\{ip\e1 + k\e2 + \alpha_{ij}\e3 + j\e4 \mid i\ttk{p},\,j,k\ttk{2}\}$ \\
\tn & $\{(i+\alpha_{ij}p)\e1 + j\e2 + \alpha_i\e3 + \beta_i\e4 \mid i\ttk{p},\,j\ttk{2}\}$
& $\{ kp\e1 + \alpha_{ij}\e2 + i\e3 + j\e4 \mid i,j\ttk{2},\,k\ttk{p}\}$ \\
\tn & $\{(i+\alpha_ip)\e1 + \alpha_{ij}\e2 + \beta_{ij}\e3 + j\e4 \mid i\ttk{p},\,j\ttk{2}\}$
& $\{ip\e1 + \alpha_j\e2 + k\e3 + i\e4  \mid i\ttk{p},\,j,k\ttk{2}\}$ \\
\tn & $\{(\alpha_i+jp)\e1 + \alpha_{ij}\e2 + \beta_{ij}\e3 + i\e4 \mid i\ttk{2},\,j\ttk{p}\}$
& $\{(i+\alpha_ip)\e1 + j\e2 + k\e3 \mid i\ttk{p},\,j,k\ttk{2}\}$ \\
\tn & $\{(i+\alpha_{ij}p)\e1 + \alpha_{ij}\e2 + j\e3 + \alpha_i\e4 \mid i\ttk{p},\,j\ttk{2}\}$
& $\{jp\e1 + k\e2 + \alpha_i\e3 + i\e4 \mid j\ttk{p},\,i,k\ttk{2}\}$ \\
\tn & $\{(i+\alpha_{ij}p)\e1 + \alpha_{ij}\e2 + \beta_{ij}\e3 + j\e4 \mid i\ttk{p},\,j\ttk{2}\}$
& $\{ip\e1+j\e2+k\e3 \mid i\ttk{p}, j,k\ttk{2}\}$ \\
\tn & $\{\alpha_i\e1 + \alpha_{ij}\e2 + j\e3 +i\e4 \mid i,j\ttk{2}\}$
& $\{i\e1 + \alpha_i\e3 + j\e2 \mid i\ttk{p^2},\,j\ttk{2}\}$ \\
\tn & $\{\alpha_i\e1 + \alpha_{ij}\e2 + j\e3 +i\e4 \mid i,j\ttk{2}\}$
& $\{(i+jp)\e1 + k\e2 + \alpha_i \e3 \mid i,j\ttk{p},\,k\ttk{2}\}$ \\
\tn & $\{\alpha_{ij}\e1 + \alpha_{ij}\e2 + i\e3 + j\e4 \mid i,j\ttk{2}\}$
& $\{i\e1+j\e2 \mid i\ttk{p^2}, j\ttk{2}\}$ \\
\tn & $\{(i+jp)\e1 + k\e2 + \alpha_{ij}\e3 + \alpha_i\e4 \mid i,j\ttk{p},\,k\ttk{2}\}$
& $\{\alpha_i p\e1 + \alpha_{ij}\e2 + j\e3 +i\e4 \mid i,j\ttk{2}\}$ \\
\tn & $\{i\e1 + j\e2 + \alpha_i\e3 + \beta_i\e4 \mid i\ttk{p^2},\,j\ttk{2}\}$
& $\{\alpha_{ij}\e2 + i\e3 + j\e4 \mid i,j\ttk{2}\}$ \\
\tn & $\{(i+kp)\e1+\alpha_{ij}\e2 + j\e3 + \alpha_i\e4 \mid i,k\ttk{p},\,j\ttk{2}\}$
& $\{\alpha_{ij}\e1 + j\e2 + \alpha_i\e3 + i\e4 \mid i,j\ttk{2}\}$ \\
\tn & $\{(i+kp)\e1 + \alpha_{ij}\e2 + \beta{ij}\e3 + j\e4 \mid i,k\ttk{p},\,j\ttk{2}\}$
& $\{\alpha_{ij}\e1 + i\e2+j\e3 \mid i,j\ttk{2} \}$ \\
\tn & $\{(i+jp)\e1 + k\e2 + \alpha_i\e3 + \beta_i\e4 \mid i,j\ttk{p},\,k\ttk{2}\}$
& $\{\alpha_{ij}p + \alpha_{ij}\e2 + i\e3 + j\e4 \mid i,j\ttk{2}\}$ \\
\tn & $\{i\e1 + \alpha_{ij}\e2 + j\e3 + \alpha_i\e4 \mid i\ttk{p^2},\,j\ttk{2}\}$
& $\{j\e2 + \alpha_i\e3 + i\e4 \mid i,j\ttk{2}\}$ \\
\tn & $\{(i+kp)\e1 + \alpha_{ijk}\e2 +\alpha_{ij}\e3 + j\e4 \mid i,k\ttk{p},j\ttk{2}\}$
& $\{\alpha_i p\e1 + j\e2 + i\e3 \mid i,j\ttk{2}\}$ \\
\tn & $\{(i+jp)\e1 + \alpha_{ijk}\e2 + k\e3 + \alpha_i\e4 \mid i,j\ttk{p},\,k\ttk{2}\}$
& $\{\alpha_i p\e1 + j\e2 + \alpha_i\e3 + i\e4 \mid i,j\ttk{2}\}$ \\
\tn & $\{i\e1 + \alpha_{ij}\e2 + \beta_{ij}\e3 + j\e4 \mid i\ttk{p^2},\,j\ttk{2}\}$
& $\{i\e2 + j\e3 \mid \}$ \\
\tn & $\{(i+\alpha_{ij}p)\e1+k\e2 + j\e3 + \alpha_i\e4 \mid i\ttk{p},\,j,k\ttk{2}\}$
& $\{jp\e1+ \alpha_{ij}\e2 + \alpha_i\e3 + i\e4 \mid i\ttk{2},\,j\ttk{p}\}$ \\
\tn & $\{(i+\alpha_i p)\e1+k\e2+\alpha_{ij}\e3 +j\e4 \mid i\ttk{p},\,j,k\ttk{2}\}$
& $\{ip\e1+\alpha_{ij}\e2 + j\e3 + \alpha_i\e4 \mid i\ttk{p},\,j\ttk{2}\}$ \\
\tn & $\{(\alpha_i+jp)\e1 + k\e2 +\alpha_{ij}\e3 + i\e4 \mid i\ttk{2}, \,j,k\ttk{p}\}$
& $\{(i+\alpha_i p)\e1+\alpha_{ij}\e2 + j\e3 \mid i\ttk{p},\,j\ttk{2}\}$ \\
\tn & $\{(i+\alpha_{ij}p)\e1 + k\e2 + \alpha_{ij}\e3 + j\e4 \mid i\ttk{p},\,j,k\ttk{2}\}$
& $\{ip\e1+\alpha_{ij}\e2 + j\e3 \mid i\ttk{p},\,j\ttk{2} \}$ \\
\tn & $\{(\alpha_i+kp)\e1+\alpha_{ij}\e2 + j\e3 + i\e4 \mid i,j\ttk{2},\,k\ttk{p}\}$
& $\{(i+\alpha_{ij}p)\e1+j\e2 + \alpha_i\e3 \mid i\ttk{p},\,j\ttk{2}\}$ \\
\tn & $\{(\alpha_{ij}+kp)\e1+ \alpha_{ij}\e2 + i\e3 + j\e4 \mid i,j\ttk{2},\,k\ttk{p}\}$
& $\{(i+\alpha_{ij}p)\e1+ j\e2 \mid i\ttk{p},\,j\ttk{2}\}$ \\
\tn & $\{(i+\alpha_ip)\e1+\alpha_j\e2 + k\e3 + i\e4  \mid i\ttk{p},\,j,k\ttk{2}\}$
& $\{ip\e1+ \alpha_{ij}\e2 + \beta_{ij}\e3 + j\e4 \mid i\ttk{p},\,j\ttk{2}\}$ \\
\tn & $\{(i+\alpha_{ij}p)\e1 + \alpha_{ijk}\e2 + k\e3 + j\e4 \mid i\ttk{p},\,j,k\ttk{2}\}$
& $\{ip\e1 + \alpha_j\e2 + i\e3 \mid i\ttk{p},\,j\ttk{2}\}$ \\
\tn & $\{(\alpha_{ij}+kp)\e1 + \alpha_{ijk}\e2 + i\e3 + j\e4 \mid i,j\ttk{2},\,k\ttk{p}\}$
& $\{(i+\alpha_ip)\e1 +j\e2 \mid i\ttk{p},\,j\ttk{2}\}$ \\
\tn & $\{(\alpha_i+jp)\e1 + \alpha_{ijk}\e2 + k\e3 + i\e4  \mid i,k\ttk{2},\,j\ttk{p} \}$
& $\{(i+\alpha_ip)\e1 + \alpha_j\e2 + i\e3 \mid i\ttk{p},\,j\ttk{2} \}$ \\
\tn & $\{(i+\alpha_ip)\e1 + \alpha_{ijk}\e2 + j\e3 + k\e4 \mid i\ttk{p},\,j,k\ttk{2} \}$
& $\{ip\e1 + j\e2 + \alpha_i\e3 + \beta_i\e4 \mid i\ttk{p},\,j\ttk{2} \}$ \\
\tn & $\{(i+\alpha_{ijk}p)\e1 + k\e2 + \alpha_{ij}\e3 + j\e4 \mid i\ttk{p},\,j,k\ttk{2}\}$
& $\{jp\e1 + \alpha_i\e2 + i\e3 \mid i\ttk{2},\,j\ttk{p}\}$ \\
\tn & $\{(k+\alpha_{ijk}p)\e1+ \alpha_{ij}\e2 + i\e3 + j\e4 \mid i,j\ttk{2},\,k\ttk{p}\}$
& $\{(\alpha_i+jp)\e1 + i\e2 \mid i\ttk{2},\,j\ttk{p}\}$ \\
\tn & $\{(i+\alpha_{ijk}p)\e1 + j\e2 + k\e3 + \alpha_i \e4 \mid i\ttk{p},\,j,k\ttk{2}\}$
& $\{jp\e1 + \alpha_i\e2 +\beta_i\e3 + i\e4 \mid i\ttk{2},\,j\ttk{p}\}$ \\
\tn & $\{(i+\alpha_{ijk}p)\e1 + \alpha_{ijk}\e2 + j\e3 + k\e4 \mid i\ttk{p},\,j,k\ttk{2}\}$
& $\{ip\e1+j\e2 \mid i\ttk{p},\,j\ttk{2} \}$ \\
\tn & $\{(\alpha_i+\alpha_{ij}p)\e1 + k\e2 + j\e3 + i\e4 \mid i,j,k\ttk{2}\}$
& $\{(i+jp)\e1 + \alpha_{ij}\e2 + \alpha_i\e3 \mid i,j\ttk{p}\}$ \\
\tn & $\{\alpha_{ij}\e1 + k\e2 + i\e3 + j\e4 \mid i,j,k\ttk{2}\}$
& $\{i\e1 +\alpha_i\e2 \mid i\ttk{p^2} \}$ \\
\tn & $\{\alpha_i\e1 + j\e2 + k\e3 + i\e4  \mid i,j,k\ttk{2}\}$
& $\{ip\e1 + \alpha_{ij}\e2 + \beta_{ij}\e3 + j\e4 \mid i,j\ttk{p}\}$ \\
\tn & $\{(\alpha_{ij}+\alpha_{ijk}p)\e1 + k\e2 + i\e3 + j\e4 \mid i,j,k\ttk{2}\}$
& $\{(i+jp)\e1 + \alpha_i \e2 \mid i,j\ttk{p} \}$ \\
\tn & $\{(\alpha_i+\alpha_{ijk}p)\e1 + j\e2 + k\e3 + i\e4  \mid i,j,k\ttk{2}\}$
& $\{(i+jp)\e1 + \alpha_i\e2 + \beta_i\e3 \mid i,j\ttk{p} \}$ \\
\tn & $\{(\alpha_{ijk})\e1 + i\e2 + j\e3 + k\e4 \mid i,j,k\ttk{2}\}$
& $\{i\e1 \mid i\ttk{p^2}\}$ \\
\tn & $\{i\e1+k\e2+\alpha_{ij}\e3 +j\e4 \mid i\ttk{p^2},\,j,k\ttk{2}\}$
& $\{\alpha_i\e2 + i\e3 \mid i\ttk{2}\}$ \\    
\tn & $\{(i+lp)\e1 + \alpha_{ijk}\e2 + j\e3 + k\e4 \mid i,l\ttk{p},\,j,k\ttk{2}\}$
& $\{\alpha_ip\e1 + i\e2 \mid i\ttk{2}\}$ \\
\tn & $\{i\e1 + j\e2 + k\e3 + \alpha_i\e4 \mid i\ttk{p^2},\,j,k\ttk{2}\}$
& $\{\alpha_i \e2 + \beta_i \e3 + i\e4 \mid i\ttk{2}\}$ \\
\tn & $\{(i+kp)\e1 + l\e2 + \alpha_{ij}\e3 + j\e4 \mid i,k\ttk{p},\,j,l\ttk{2}\}$
& $\{\alpha_ip\e1 + \alpha_i\e2 + i\e3 \mid i\ttk{2}\}$ \\
\tn & $\{k\e1 + \alpha_{ij}\e2 + i\e3 + j\e4 \mid i,j\ttk{2},\,k\ttk{p^2}\}$
& $\{\alpha_i\e1 + i\e2 \mid i\ttk{2}\}$ \\
\tn & $\{(i+jp)\e1 + k\e2 + l\e3 + \alpha_i\e4 \mid i,j\ttk{p},\,k,l\ttk{2}\}$
& $\{\alpha_ip\e1 + \alpha_i\e2 + \beta_i\e3 + i\e4 \mid i\ttk{2}\}$ \\
\tn & $\{i\e1 + \alpha_{ijk}\e2 + j\e3 + k\e4 \mid i\ttk{p^2},\,j,k\ttk{2}\}$
& $\{i\e2 \mid i\ttk{2}\}$ \\
\tn & $\{(i+\alpha_{ijk}p)\e1 + l\e2 + j\e3 + k\e4 \mid i\ttk{p},\,j,k,l\ttk{2}\}$
& $\{ip\e1 + \alpha_i\e2 \mid i\ttk{p}\}$ \\
\tn & $\{(\alpha_{ijk}+lp)\e1+i\e2+j\e3+k\e4 \mid i,j,k\ttk{2},\,l\ttk{p}\}$
& $\{(i+\alpha_ip)\e1 \mid i\ttk{p}\}\}$ \\
\tn & $\{(i+\alpha_{ij}p)\e1 + k\e2 + l\e3 + j\e4 \mid i\ttk{p},\,j,k,l\ttk{2}\}$
& $\{ ip\e1 + \alpha_i\e2 + \beta_i\e3 \mid i\ttk{p}\}$ \\
\tn & $\{(\alpha_{ij}+kp)\e1 + l\e2 + i\e3 + j\e4 \mid i,j,l\ttk{2},\,k\ttk{p}\}$
& $\{\alpha_ip\e1 + \alpha_i\e2 + i\e3 \mid i\ttk{2}\}$ \\
\tn & $\{(i+\alpha_ip)\e1+j\e2+k\e3+l\e4 \mid i\ttk{p},\,j,k,l\ttk{2}\}$
& $\{ip\e1+\alpha_i\e2+\beta_i\e3+\gamma_i\e4 \mid i\ttk{p}\}$ \\
\tn & $\{(\alpha_i+jp)\e1 + k\e2 +l\e3 + i\e4 \mid i,k,l\ttk{2},\,j\ttk{p}\}$
& $\{(i+\alpha_ip)\e1 + \alpha_i\e2 + \beta_i\e3 \mid i\ttk{p}\}$ \\
\tn & $\{(i+\alpha_{ijkl}p)\e1+j\e2+k\e3+l\e4 \mid i\ttk{p},\,j,k,l\ttk{2}\}$
& $\{ip\e1 \mid i\ttk{p}\}$ \\
\hline
\end{longtable}
\end{footnotesize}

\begin{proof}
    Suppose $C$ is a perfect code in $\Cay(G,S)$. Then $G=\so\oplus C$ and so $|\so|\cdot |C|=|G|=8p^2  $. 
    Since $G$ is not a cyclic group, $|\so|>2$ by Lemma \ref{ls<so}.
    Therefore, $2<|\so|<8p^2$.

    \smallskip
    \textsf{Case 1.} $\so$ is aperiodic.
    \smallskip

    Since $G$ is Haj\'os and $G=\so\oplus C$, in this case $C$ must be periodic, that is, $\lc$ is nontrivial.
    Write $C=\lc\oplus D$ for some subset $D$ of $G$ whose existence is ensured by Lemma \ref{period}.
    Then $|\lc|\cdot |D|=|C|$ and $2\leq |\lc|\leq |C|$.
    If $|\so|$ is a prime, then $|\lc|=|C|$ by Lemma \ref{prime deg} part (c).
    If $|\so|$ is not a prime, then $|\lc|$ could be any divisor of $|C|$.
    Set $$ A=(\so+\lc)/\lc\quad\text{and}\quad B=(D+\lc)/\lc.$$
    Since $C$ is periodic, by Lemma \ref{quotient}, we have
    $$G/\lc=A\oplus B$$ and $B$ is aperiodic in $G/\lc$.
    Since $G/\lc$ is isomorphic to a subgroup of $G$ and subgroups of Haj\'os group are also Haj\'os groups, $A$ is periodic in $G/\lc$.
    We also know that $A$ contains $\lc$ and $A$ generates $G/\lc$ by Lemma \ref{generate-zero}.
    Furthermore, $|A|=|\so|$ and $|B|=|D|$.
    Note that if $|\lc|=|C|$, then $C=\lc$ as $0\in C$.
    This means we can take $D=\{0\}$ and hence $B=\lc/\lc$ and $G/\lc=A$.
    We use $C$ to get $A$ and then use Lemma \ref{d_i+l_i} to obtain $\so$.
    If $|\lc|<|C|$, then this factorization enables us to make use of previously established theorems to obtain the pair $(\so, C)$ via $(A,B)$.

    Table \ref{tap.p^2.2.2.2} represents all possible values for $\left(|\so|,|C|,|\lc|\right)$; $C$, $A$, and $\so$ when $|\lc|=|C|$ or $L_C$, $A$, $B$, $\so$, and $C$ when $|\lc|<|C|$; which previous theorem is used to obtain $(A,B)$ when $|\lc|<|C|$; and which row in Table \ref{t.p^2.2.2.2} the result corresponds to.
    
\begin{longtable}{|c|l|c|c|}
\caption{The case when $G=\Z_{p^2}\x\Z_2\x\Z_2\x\Z_2 $ and $\so$ is aperiodic}
\label{tap.p^2.2.2.2}
\\
\flap

& $C=\<{p\e1+\e2,\e3,\e4} $ & & \\
$(p,8p,8p)$ 
& $A=\{i\e1+\lc \mid i\ttk{p}\} $ 
& & row 1 \\
& $\so=\{(i+\alpha_ip)\e1+\alpha_i\e2+\beta_i\e3+\gamma_i\e4 \mid i\ttk{p}\} $ & & \\ \hline

& $\lc=\<{\e2} $ & & \\
& $A=\{(i+jp)\e1 + \alpha_i\e3 + \beta_i\e4 + \lc \mid i,j\ttk{p}\} $ & & \\
$(p^2,8,2 )$ 
& $B=\{\alpha_{ij}p\e1 + i\e3 + j\e4 + \lc \mid i,j\ttk{2}\} $ 
& \ref{p^2.2.2} \row2 & row 3 \\
& $\so=\{(i+jp)\e1 + \alpha_{ij}\e2 + \alpha_i\e3 + \beta_i\e4 \mid i,j\ttk{p}\} $ & & \\
& $C=\{\alpha_{ij}p\e1 + k\e2 + i\e3 + j\e4 \mid i,j,k\ttk{2}\} $ & & \\ \hline

& $\lc=\<{\e2,\e3} $ & & \\
& $A=\{(i+jp)\e1 + \alpha_i\e4 + \lc \mid i,j\ttk{p}\} $ & & \\
$(p^2,8,4 )$ 
& $B=\{\alpha_ip\e1 + i\e4 + \lc \mid i\ttk{2}\} $ 
& \ref{p^2.2} \row3 & row 4 \\
& $\so=\{(i+jp)\e1 + \alpha_{ij}\e2 + \beta_{ij}\e3 + \alpha_i\e4  \mid i,j\ttk{p}\} $ & & \\
& $C=\{\alpha_ip\e1 + j\e2 + k\e3 + i\e4  \mid i,j,k\ttk{2}\} $ & & \\ \hline

& $C=\<{\e2,\e3,\e4} $ & & \\
$(p^2,8,8 )$ 
& $A=\{i\e1 + \lc \mid i\ttk{p^2}\} $ 
& & row 5 \\
& $\so=\{i\e1+\alpha_i\e2+\beta_i\e3+\gamma_i\e4 \mid i\ttk{p^2}\} $ & & \\ \hline

& $\lc=\<{\e2} $ & & \\
& $A=\{(i+\alpha_ip)\e1 + j\e3 + \alpha_i\e4 + \lc \mid i\ttk{p},\,j\ttk{2}\} $ & & \\
$(2p,4p,2 )$ 
& $B=\{ip\e1 + j\e4 + \alpha_{ij}\e3 + \lc \mid i\ttk{p},\,j\ttk{2}\} $ 
& \ref{p^2.2.2} \row5 & row 7 \\
& $\so=\{(i+\alpha_ip)\e1 + \alpha_{ij}\e2 + j\e3 + \alpha_i\e4 \mid i\ttk{p},\,j\ttk{2}\} $ & & \\
& $C=\{ip\e1 + k\e2 + \alpha_{ij}\e3 + j\e4 \mid i\ttk{p},\,j,k\ttk{2}\} $ & & \\ \hline

& $\lc=\<{p\e1} $ & & \\
& $A=\{i\e1 + j\e2 + \alpha_i\e3 + \beta_i\e4 + \lc \mid i\ttk{p},\,j\ttk{2}\} $ & & \\
$(2p,4p,p )$ 
& $B=\{ \alpha_{ij}\e2 + i\e3 + j\e4 + \lc \mid i,j\ttk{2}\} $ 
& \ref{p.2.2.2} \row2 & row 8 \\
& $\so=\{(i+\alpha_{ij}p)\e1 + j\e2 + \alpha_i\e3 + \beta_i\e4 \mid i\ttk{p},\,j\ttk{2}\} $ & & \\
& $C=\{ kp\e1 + \alpha_{ij}\e2 + i\e3 + j\e4 \mid i,j\ttk{2},\,k\ttk{p}\} $ & & \\ \hline

& $\lc=\<{\e2,\e3} $ & & \\
& $A=\{(i+\alpha_ip)\e1 + j\e4 + \lc \mid i\ttk{p},\,j\ttk{2}\} $ & & \\
$(2p,4p,4 )$ 
& $B=\{ip\e1 + \alpha_i\e4 + \lc \mid i\ttk{p}\} $ 
& \ref{p^2.2} \row5 & row 9 \\
& $\so=\{(i+\alpha_ip)\e1 + \alpha_{ij}\e2 + \beta_{ij}\e3 + j\e4 \mid i\ttk{p},\,j\ttk{2}\} $ & & \\
& $C=\{ip\e1 + \alpha_j\e2 + k\e3 + i\e4  \mid i\ttk{p},\,j,k\ttk{2}\} $ & & \\ \hline

& $\lc=\<{\e2,\e3} $ & & \\
& $A=\{(\alpha_i+jp)\e1 + i\e4 + \lc \mid i\ttk{2},\,j\ttk{p}\} $ & & \\
$(2p,4p,4 )$ 
& $B=\{(i+\alpha_ip)\e1 + \lc \mid i\ttk{p}\} $ 
& \ref{p^2.2} \row6 & row 10 \\
& $\so=\{(\alpha_i+jp)\e1 + \alpha_{ij}\e2 + \beta_{ij}\e3 + i\e4 \mid i\ttk{2},\,j\ttk{p}\} $ & & \\
& $C=\{(i+\alpha_ip)\e1 + j\e2 + k\e3 \mid i\ttk{p},\,j,k\ttk{2}\} $ & & \\ \hline

& $\lc=\<{p\e1+\e2} $ & & \\
& $A=\{i\e1 + j\e3 + \alpha_i\e4 + \lc \mid i\ttk{p},\,j\ttk{2}\} $ & & \\
$(2p,4p,2p )$ 
& $B=\{\alpha_i\e3 + i\e4 + \lc \mid i,j\ttk{2}\} $ 
& \ref{p.2.2} \row4 & row 11 \\
& $\so=\{(i+\alpha_{ij}p)\e1 + \alpha_{ij}\e2 + j\e3 + \alpha_i\e4 \mid i\ttk{p},\,j\ttk{2}\} $ & & \\
& $C=\{jp\e1 + k\e2 + \alpha_i\e3 + i\e4 \mid j\ttk{p},\,i,k\ttk{2}\} $ & & \\ \hline

& $C=\<{p\e1+\e2,\e3} $ & & \\
$(2p,4p,4p )$ 
& $A=\{i\e1 + j\e4 + \lc \mid i\ttk{p},\,j\ttk{2}\} $ 
& & row 12 \\
& $\so=\{(i+\alpha_{ij}p)\e1 + \alpha_{ij}\e2 + \beta_{ij}\e3 + j\e4 \mid i\ttk{p},\,j\ttk{2}\} $ & & \\ \hline

& $\lc=\<{\e2} $ & & \\
& $A=\{\alpha_i\e1 + j\e3 + i\e4 + \lc \mid i,j\ttk{2}\} $ & & \\
$(4,2p^2,2 )$ 
& $B=\{i\e1 + \alpha_i \e3 + \lc \mid i\ttk{p^2}\} $ 
& \ref{p^2.2.2} \row{10} & row 13 \\
& $\so=\{\alpha_i\e1 + \alpha_{ij}\e2 + j\e3 +i\e4 \mid i,j\ttk{2}\} $ & & \\
& $C=\{i\e1 + j\e2 + \alpha_i \e3  \mid i\ttk{p^2},\,j\ttk{2}\} $ & & \\ \hline

$(4,2p^2,p )$ 
& Not possible since no $A$ exists
& \ref{p.2.2.2}  & \\ \hline

& $\lc=\<{p\e1+\e2} $ & & \\
& $A=\{\alpha_i\e1 + j\e3 + i\e4 + \lc \mid i,j\ttk{2}\} $ & & \\
$(4,2p^2,2p )$ 
& $B=\{i\e1 + \alpha_i \e3 + \lc \mid i\ttk{p}\} $ 
& \ref{p.2.2} \row6 & row 14 \\
& $\so=\{\alpha_i\e1 + \alpha_{ij}\e2 + j\e3 +i\e4 \mid i,j\ttk{2}\} $ & & \\
& $C=\{(i+jp)\e1 + k\e2 + \alpha_i \e3 \mid i,j\ttk{p},\,k\ttk{2}\} $ & & \\ \hline

$(4,2p^2,p^2 )$ 
& Not possible since no $A$ exists
& \ref{2.2.2}  & \\ \hline

& $C=\<{\e1+\e2} $ & & \\
$(4,2p^2,2p^2 )$ 
& $A=\{i\e3 + j\e4 + \lc \mid i,j\ttk{2}\} $ 
& & row 15 \\
& $\so=\{\alpha_{ij}\e1 + \alpha_{ij}\e2 + i\e3 + j\e4 +\mid i,j\ttk{2}\} $ & & \\ \hline

& $\lc=\<{\e2} $ & & \\
& $A=\{i\e1 + j\e3 + \alpha_i\e4 +\lc \mid i\ttk{p^2},\,j\ttk{2}\} $ & & \\
$(2p^2,4,2)$ 
& $B=\{\alpha_i\e3 + i\e4 + \lc \mid i\ttk{2}\} $ 
& \ref{p^2.2.2} \row{13} & row 21 \\
& $\so=\{i\e1 + \alpha_{ij}\e2 + j\e3 + \alpha_i\e4 \mid i\ttk{p^2},\,j\ttk{2}\} $ & & \\
& $C=\{j\e2 + \alpha_i\e3 + i\e4 \mid i,j\ttk{2}\}$ & & \\ \hline

& $\lc=\<{\e2} $ & & \\
& $A=\{(i+kp)\e1 + \alpha_{ij}\e3 + j\e4  +\lc \mid i,k\ttk{p},j\ttk{2}\} $ & & \\
$(2p^2,4,2)$ 
& $B=\{\alpha_i p\e1 + i\e3 +\lc \mid i\ttk{2}\} $ 
& \ref{p^2.2.2} \row{14} & row 22 \\
& $\so=\{(i+kp)\e1 + \alpha_{ijk}\e2 +\alpha_{ij}\e3 + j\e4 \mid i,k\ttk{p},j\ttk{2}\} $ & & \\
& $C=\{\alpha_i p\e1 + j\e2 + i\e3 \mid i,j\ttk{2}\} $ & & \\ \hline
    
& $\lc=\<{\e2} $ & & \\
& $A=\{(i+jp)\e1 + k\e3 + \alpha_i\e4 + \lc \mid i,j\ttk{p},\,k\ttk{2}\} $ & & \\
$(2p^2,4,2)$ 
& $B=\{\alpha_i p\e1 + \alpha_i\e3 + i\e4 +\lc \mid i\ttk{2}\} $ 
& \ref{p^2.2.2} \row{15} & row 23 \\
& $\so=\{(i+jp)\e1 + \alpha_{ijk}\e2 + k\e3 + \alpha_i\e4 \mid i,j\ttk{p},\,k\ttk{2}\} $ & & \\
& $C=\{\alpha_i p\e1 + j\e2 + \alpha_i\e3 + i\e4 \mid i,j\ttk{2}\} $ & & \\ \hline

& $C=\<{\e2,\e3} $ & & \\
$(2p^2,4,4)$
& $A=\{i\e1 + j\e4 + \lc \mid i\ttk{p^2},\,j\ttk{2}\} $ 
& & row 24 \\
& $\so=\{i\e1 + \alpha_{ij}\e2 + \beta_{ij}\e3 + j\e4 \mid i\ttk{p^2},\,j\ttk{2}\} $ & & \\ \hline

& $\lc=\<{\e2} $ & & \\
& $A=\{(i+\alpha_{ij}p)\e1 + k\e3 + j\e4  +\lc \mid i\ttk{p},\,j,k\ttk{2}\} $ & & \\
$(4p,2p,2 )$ 
& $B=\{ip\e1 + \alpha_i\e3 +\lc \mid i\ttk{p}\} $ 
& \ref{p^2.2.2} \row{17} & row 32 \\
& $\so=\{(i+\alpha_{ij}p)\e1 + \alpha_{ijk}\e2 + k\e3 + j\e4 \mid i\ttk{p},\,j,k\ttk{2}\} $ & & \\
& $C=\{ip\e1 + \alpha_j\e2 + i\e3 \mid i\ttk{p},\,j\ttk{2}\} $ & & \\ \hline

& $\lc=\<{\e2} $ & & \\
& $A=\{(\alpha_{ij}+kp)\e1 + i\e3 + j\e4 +\lc \mid i,j\ttk{2},\,k\ttk{p}\} $ & & \\
$(4p,2p,2 )$ 
& $B=\{ (i+\alpha_ip)\e1 +\lc \mid i\ttk{p}\} $ 
& \ref{p^2.2.2} \row{18} & row 33 \\
& $\so=\{(\alpha_{ij}+kp)\e1 + \alpha_{ijk}\e2 + i\e3 + j\e4 \mid i,j\ttk{2},\,k\ttk{p}\} $ & & \\
& $C=\{(i+\alpha_ip)\e1 +j\e2 \mid i\ttk{p},\,j\ttk{2}\} $ & & \\ \hline

& $\lc=\<{\e2} $ & & \\
& $A=\{(\alpha_i+jp)\e1 + k\e3 + i \e4 +\lc \mid i,k\ttk{2},\,j\ttk{p} \} $ & & \\
$(4p,2p,2 )$ 
& $B=\{(i+\alpha_ip)\e1 + \alpha_i\e3 +\lc \mid i\ttk{p} \} $ 
& \ref{p^2.2.2} \row{19} & row 34 \\
& $\so=\{(\alpha_i+jp)\e1 + \alpha_{ijk}\e2 + k\e3 + i\e4  \mid i,k\ttk{2},\,j\ttk{p} \} $ & & \\
& $C=\{(i+\alpha_ip)\e1 + \alpha_j\e2 + i\e3 \mid i\ttk{p},\,j\ttk{2} \} $ & & \\ \hline

& $\lc=\<{\e2} $ & & \\
& $A=\{(i+\alpha_ip)\e1 + j\e3 + k\e4 +\lc \mid i\ttk{p},\,j,k\ttk{2} \} $ & & \\
$(4p,2p,2 )$ 
& $B=\{ip\e1 + \alpha_i\e3 + \beta_i\e4 +\lc \mid i\ttk{p} \} $ 
& \ref{p^2.2.2} \row{20} & row 35 \\
& $\so=\{(i+\alpha_ip)\e1 + \alpha_{ijk}\e2 + j\e3 + k\e4 \mid i\ttk{p},\,j,k\ttk{2} \} $ & & \\
& $C=\{ip\e1 + j\e2 + \alpha_i\e3 + \beta_i\e4 \mid i\ttk{p},\,j\ttk{2} \} $ & & \\ \hline

& $\lc=\<{p\e1} $ & & \\
& $A=\{i\e1 + k\e2 + \alpha_{ij}\e3 + j\e4 +\lc \mid i\ttk{p},\,j,k\ttk{2}\} $ & & \\
$(4p,2p,p )$ 
& $B=\{\alpha_i\e2 + i\e3 +\lc \mid i\ttk{2} \} $ 
& \ref{p.2.2.2} \row7 & row 36 \\
& $\so=\{(i+\alpha_{ijk}p)\e1 + k\e2 + \alpha_{ij}\e3 + j\e4 \mid i\ttk{p},\,j,k\ttk{2}\} $ & & \\
& $C=\{jp\e1 + \alpha_i\e2 + i\e3 \mid i\ttk{2},\,j\ttk{p}\} $ & & \\ \hline

& $\lc=\<{p\e1} $ & & \\
& $A=\{k\e1+ \alpha_{ij}\e2 + i\e3 + j\e4 +\lc \mid i,j\ttk{2},\,k\ttk{p}\} $ & & \\
$(4p,2p,p )$ 
& $B=\{ \alpha_i\e1 + i\e2 +\lc \mid i\ttk{2} \} $ 
& \ref{p.2.2.2} \row8 & row 37 \\
& $\so=\{(k+\alpha_{ijk}p)\e1+ \alpha_{ij}\e2 + i\e3 + j\e4 \mid i,j\ttk{2},\,k\ttk{p}\} $ & & \\
& $C=\{(\alpha_i+jp)\e1 + i\e2 \mid i\ttk{2},\,j\ttk{p}\} $ & & \\ \hline

& $\lc=\<{p\e1} $ & & \\
& $A=\{i\e1 + j\e2 + k\e3 + \alpha_i \e4 +\lc \mid i\ttk{p},\,j,k\ttk{2}\} $ & & \\
$(4p,2p,p )$ 
& $B=\{\alpha_i\e2 +\beta_i\e3 + i\e4 +\lc \mid i\ttk{2} \} $ 
& \ref{p.2.2.2} \row9 & row 38 \\
& $\so=\{(i+\alpha_{ijk}p)\e1 + j\e2 + k\e3 + \alpha_i \e4 \mid i\ttk{p},\,j,k\ttk{2}\} $ & & \\
& $C=\{jp\e1 + \alpha_i\e2 +\beta_i\e3 + i\e4 \mid i\ttk{2},\,j\ttk{p}\} $ & & \\ \hline

& $C=\<{p\e1+\e2} $ & & \\
$(4p,2p,2p )$ 
& $A=\{i\e1 + j\e3 +k\e4 + \lc \mid i\ttk{p},\,j,k\ttk{2}\} $ 
& & row 39 \\
& $\so=\{(i+\alpha_{ijk}p)\e1 + \alpha_{ijk}\e2 + j\e3 + k\e4 \mid i\ttk{p},\,j,k\ttk{2}\} $ & & \\ \hline

& $\lc=\<{p\e1} $ & & \\
& $A=\{\alpha_{ij}\e1 + k\e2 + i\e3 + j\e4 +\lc \mid i,j,k\ttk{2}\} $ & & \\
$(8,p^2,p )$ 
& $B=\{i\e1 + \alpha_i \e2 +\lc \mid i\ttk{p} \} $ 
& \ref{p.2.2.2} \row{11} & row 43 \\
& $\so=\{(\alpha_{ij}+\alpha_{ijk}p)\e1 + k\e2 + i\e3 + j\e4 \mid i,j,k\ttk{2}\} $ & & \\
& $C=\{(i+jp)\e1 + \alpha_i \e2 \mid i,j\ttk{p} \} $ & & \\ \hline

& $\lc=\<{p\e1} $ & & \\
& $A=\{\alpha_i\e1 + j\e2 + k\e3 + i\e4  +\lc \mid i,j,k\ttk{2}\} $ & & \\
$(8,p^2,p )$ 
& $B=\{i\e1 + \alpha_i\e2 + \beta_i\e3 +\lc \mid i\ttk{p} \} $ 
& \ref{p.2.2.2} \row{12} & row 44 \\
& $\so=\{(\alpha_i+\alpha_{ijk}p)\e1 + j\e2 + k\e3 + i\e4  \mid i,j,k\ttk{2}\} $ & & \\
& $C=\{(i+jp)\e1 + \alpha_i\e2 + \beta_i\e3 \mid i,j\ttk{p} \} $ & & \\ \hline

& $C=\<{\e1} $ & & \\
$(8,p^2,p^2 )$ 
& $A=\{i\e2 + j\e3 +k\e4 + \lc \mid i,j,k\ttk{2}\} $ 
& & row 45 \\
& $\so=\{(\alpha_{ijk})\e1 + i\e2 + j\e3 + k\e4 \mid i,j,k\ttk{2}\} $ & & \\ \hline

& $C=\<{\e2} $ & & \\
$(4p^2,2,2)$ 
& $A=\{i\e1 + j\e3 +k\e4 + \lc \mid i\ttk{p^2},\,j,k\ttk{2}\} $ 
& & row 52 \\
& $\so=\{i\e1 + \alpha_{ijk}\e2 + j\e3 + k\e4 \mid i\ttk{p^2},\,j,k\ttk{2}\} $ & & \\ \hline

& $C=\<{p\e1} $ & & \\
$(8p,p,p)$ 
& $A=\{i\e1+j\e2+k\e3+l\e4+\lc \mid i\ttk{p},\,j,k,l\ttk{2}\} $ 
& & row 59 \\
& $\so=\{(i+\alpha_{ijkl}p)\e1+j\e2+k\e3+l\e4 \mid i\ttk{p},\,j,k,l\ttk{2}\} $ & & \\ \hline

\end{longtable}

\smallskip
    \textsf{Case 2.} $\so$ is periodic.
    \smallskip

    In this case we have $|\ls|\geq 2$ and $\so=D\oplus \ls$ for some subset $D$ of $G$ by Lemma \ref{period}.
    So $|D|\cdot|\ls|=|\so|$ and $1<|\ls|<|\so|$ by Lemma \ref{prime deg} part (b).
    Set $$X=(D+\ls)/\ls\quad\text{and}\quad Y=(C+\ls)/\ls.$$
    Since $\so$ is periodic, by Lemma \ref{quotient}, we have
    $$G/\lc=X\oplus Y$$ and $X$ is aperiodic in $G/\ls$.
    We also know that $X$ contains $\ls$ and $X$ generates $G/\ls$ by Lemma \ref{generate-zero}.
    This factorization enables us to make use of Lemma \ref{d_i+l_i} and previously established theorems to obtain the pair $(\so, C)$ via $(X,Y)$.
    
    Table \ref{tp.p^2.2.2.2} represents all possible values for $\left(|\so|,|C|,|\ls|\right)$; $\ls$, $X$, $Y$, $\so$, and $C$; which previous theorem is used to obtain $(X,Y)$; and which row in Table \ref{t.p^2.2.2.2} the result corresponds to.

\begin{longtable}{|c|l|c|c|}
\caption{The case when $G=\Z_{p^2}\x\Z_2\x\Z_2\x\Z_2 $ and $\so$ is periodic}
\label{tp.p^2.2.2.2}
\\
\flp

& $\ls=\<{p\e1} $ & & \\
& $X=\{i\e1+\alpha_i\e2+\beta_i\e3+\gamma_i\e4+\ls \mid i\ttk{p}\} $ & & \\
$(p^2,8,p )$ 
& $Y=\{\e1+i\e2+j\e3+k\e4+\ls \mid i,j,k\ttk{2}\} $ 
& \ref{p.2.2.2} \row1 & row 2 \\
& $\so=\{(i+jp)\e1+\alpha_i\e2+\beta_i\e3+\gamma_i\e4 \mid i,j\ttk{p}\} $ & & \\
& $C=\{\alpha_{ijk}p\e1+i\e2+j\e3+k\e4 \mid i,j,k\ttk{2}\} $ & & \\ \hline

& $\ls=\<{\e2} $ & & \\
& $X=\{(i+\alpha_ip)\e1 + \alpha_i\e3 + \beta_i\e4 + \ls \mid i\ttk{p}\} $ & & \\
$(2p,4p,2)$ 
& $Y=\{ip\e1 + j\e3 + k\e4 + \ls \mid i\ttk{p},\,j,k\ttk{2}\} $ 
& \ref{p^2.2.2} \row1 & row 6 \\
& $\so=\{(i+\alpha_ip)\e1 + k\e2 + \alpha_i\e3 + \beta_i\e4 \mid i\ttk{p},\,k\ttk{2}\} $ & & \\
& $C=\{ip\e1 + \alpha_{ijk}\e2 + j\e3 + k\e4 \mid i\ttk{p},\,j,k\ttk{2}\} $ & & \\ \hline

$(2p,4p,p)$ 
& Not possible since no $X$ exists
& \ref{p.2.2.2}  & \\ \hline

$(4,2p^2,2)$ 
& Not possible since no $X$ exists
& \ref{p^2.2.2}  & \\ \hline

& $\ls=\<{\e2} $ & & \\
& $X=\{(i+jp)\e1 + \alpha_{ij}\e3 + \alpha_i\e4 +\ls \mid i,j\ttk{p}\} $ & & \\
$(2p^2,2p,2 )$ 
& $Y=\{\alpha_i p\e1 + j\e3 + i\e4 +\ls \mid i,j\ttk{2}\} $ 
& \ref{p^2.2.2} \row3 & row 16 \\
& $\so=\{(i+jp)\e1 + k\e2 + \alpha_{ij}\e3 + \alpha_i\e4 \mid i,j\ttk{p},\,k\ttk{2}\} $ & & \\
& $C=\{\alpha_i p\e1 + \alpha_{ij}\e2 + j\e3 +i\e4 \mid i,j\ttk{2}\} $ & & \\ \hline

& $\ls=\<{\e2} $ & & \\
& $X=\{i\e1 + \alpha_i\e3 + \beta_i\e4 + \ls \mid i\ttk{p^2}\} $ & & \\
$(2p^2,2p,2 )$ 
& $Y=\{i\e3 + j\e4 +\ls \mid i,j\ttk{2}\} $ 
& \ref{p^2.2.2} \row4 & row 17 \\
& $\so=\{i\e1 + j\e2 + \alpha_i\e3 + \beta_i\e4 \mid i\ttk{p^2} $ & & \\
& $C=\{\alpha_{ij}\e2 + i\e3 + j\e4 \mid i,j\ttk{2}\} $ & & \\ \hline

& $\ls=\<{p\e1} $ & & \\
& $X=\{i\e1+\alpha_{ij}\e2 + j\e3 + \alpha_i\e4 +\ls \mid i\ttk{p},\,j\ttk{2}\} $ & & \\
$(2p^2,2p,p )$ 
& $Y=\{j\e2 + \alpha_i\e3 + i\e4 \mid i,j\ttk{2}\} $ 
& \ref{p.2.2.2} \row3 & row 18 \\
& $\so=\{(i+jp)\e1 + k\e2 + \alpha_{ij}\e3 + \alpha_i\e4 \mid i,j\ttk{p},\,k\ttk{2}\} $ & & \\
& $C=\{\alpha_{ij}\e1 + i\e4 + j\e2 + \alpha_i\e3 \mid i,j\ttk{2} \} $ & & \\ \hline

& $\ls=\<{p\e1} $ & & \\
& $X=\{i\e1 + \alpha_{ij}\e2 + \beta{ij}\e3 + j\e4 + \ls \mid i\ttk{p},\,j\ttk{2}\} $ & & \\
$(2p^2,2p,p )$ 
& $Y=\{i\e2+j\e3 +\ls \mid i,j\ttk{2} \} $ 
& \ref{p.2.2.2} \row4 & row 19 \\
& $\so=\{i\e1 + j\e2 + \alpha_i\e3 + \beta_i\e4 \mid i\ttk{p^2},\,j\ttk{2}\} $ & & \\
& $C=\{\alpha_{ij}\e1 + i\e2+j\e3 \mid i,j\ttk{2} \} $ & & \\ \hline

& $\ls=\<{p\e1+\e2} $ & & \\
& $X=\{i\e1 + \alpha_i\e3 + \beta_i\e4 +\ls \mid i\ttk{p}\} $ & & \\
$(2p^2,2p,2p )$ 
& $Y=\{j\e2 + \alpha_i\e3 + i\e4 \mid i,j\ttk{2}\} $ 
& \ref{p.2.2} \row1 & row 20 \\
& $\so=\{(i+jp)\e1 + k\e2 + \alpha_i\e3 + \beta_i\e4 \mid i,j\ttk{p},\,k\ttk{2}\} $ & & \\
& $C=\{\alpha_{ij}p + \alpha_{ij}\e2 + i\e3 + j\e4 \mid i,j\ttk{2}\} $ & & \\ \hline

$(2p^2,2p,p^2)$ 
& Not possible since no $X$ exists
& \ref{2.2.2}  & \\ \hline

& $\ls=\<{\e2} $ & & \\
& $X=\{(i+\alpha_{ij}p)\e1+j\e3+\alpha_i\e4 +\ls \mid i\ttk{p},\,j\ttk{2}\} $ & & \\
$(4p,2p,2)$ 
& $Y=\{jp\e1+ \alpha_i\e3 + i\e4  +\ls \mid i\ttk{2},\,j\ttk{p}\} $ 
& \ref{p^2.2.2} \row6 & row 25 \\
& $\so=\{(i+\alpha_{ij}p)\e1+k\e2 + j\e3 + \alpha_i\e4 \mid i\ttk{p},\,j,k\ttk{2}\} $ & & \\
& $C=\{jp\e1+ \alpha_{ij}\e2 + \alpha_i\e3 + i\e4 \mid i\ttk{2},\,j\ttk{p}\} $ & & \\ \hline

& $\ls=\<{\e2} $ & & \\
& $X=\{(i+\alpha_i p)\e1+\alpha_{ij}\e3+j\e4 +\ls \mid i\ttk{p},\,j\ttk{2}\} $ & & \\
$(4p,2p,2 )$ 
& $Y=\{ip\e1 +j\e3+\alpha_i\e4 +\ls \mid i\ttk{p},\,j\ttk{2}\} $ 
& \ref{p^2.2.2} \row7 & row 26 \\
& $\so=\{(i+\alpha_i p)\e1+k\e2+\alpha_{ij}\e3 +j\e4 \mid i\ttk{p},\,j,k\ttk{2}\} $ & & \\
& $C=\{ip\e1+\alpha_{ij}\e2 + j\e3 + \alpha_i\e4 \mid i\ttk{p},\,j\ttk{2}\} $ & & \\ \hline

& $\ls=\<{\e2} $ & & \\
& $X=\{(\alpha_i+jp)\e1 +\alpha_{ij}\e3+ i\e4  +\ls \mid i\ttk{2}, \,j\ttk{p}\} $ & & \\
$(4p,2p,2 )$ 
& $Y=\{(i+\alpha_i p)\e1+j\e3 +\ls \mid i\ttk{p},\,j\ttk{2}\} $ 
& \ref{p^2.2.2} \row8 & row 27 \\
& $\so=\{(\alpha_i+jp)\e1 + k\e2 +\alpha_{ij}\e3 + i\e4 \mid i\ttk{2}, \,j,k\ttk{p}\} $ & & \\
& $C=\{(i+\alpha_i p)\e1+\alpha_{ij}\e2 + j\e3 \mid i\ttk{p},\,j\ttk{2}\} $ & & \\ \hline

& $\ls=\<{\e2} $ & & \\
& $X=\{(i+\alpha_{ij}p)\e1 + \alpha_{ij}\e3 + j\e4 +\ls \mid i\ttk{p},\,j\ttk{2}\} $ & & \\
$(4p,2p,2 )$ 
& $Y=\{ip\e1+j\e3 +\ls \mid i\ttk{p},\,j\ttk{2} \} $ 
& \ref{p^2.2.2} \row9 & row 28 \\
& $\so=\{(i+\alpha_{ij}p)\e1 + k\e2 + \alpha_{ij}\e3 + j\e4 \mid i\ttk{p},\,j,k\ttk{2}\} $ & & \\
& $C=\{ip\e1+\alpha_{ij}\e2 + j\e3 \mid i\ttk{p},\,j\ttk{2}\} $ & & \\ \hline

& $\ls=\<{p\e1} $ & & \\
& $X=\{\alpha_i\e1+\alpha_{ij}\e2 + j\e3 + i\e4 +\ls \mid i,j\ttk{2}\} $ & & \\
$(4p,2p,p )$ 
& $Y=\{i\e1+j\e2 + \alpha_i\e3 +\ls \mid i\ttk{p},\,j\ttk{2}\} $ 
& \ref{p.2.2.2} \row5 & row 29 \\
& $\so=\{(\alpha_i+kp)\e1+\alpha_{ij}\e2 + j\e3 + i\e4 \mid i,j\ttk{2},\,k\ttk{p}\} $ & & \\
& $C=\{(i+\alpha_{ij}p)\e1+j\e2 + \alpha_i\e3 \mid i\ttk{p},\,j\ttk{2}\} $ & & \\ \hline

& $\ls=\<{p\e1} $ & & \\
& $X=\{\alpha_{ij}\e1+ \alpha_{ij}\e2 + i\e3 + j\e4 +\ls \mid i,j\ttk{2}\} $ & & \\
$(4p,2p,p )$ 
& $Y=\{i\e1+ j\e2 +\ls \mid i\ttk{p},\,j\ttk{2}\} $ 
& \ref{p.2.2.2} \row6 & row 30 \\
& $\so=\{(\alpha_{ij}+kp)\e1+ \alpha_{ij}\e2 + i\e3 + j\e4 \mid i,j\ttk{2},\,k\ttk{p}\} $ & & \\
& $C=\{(i+\alpha_{ij}p)\e1+ j\e2 \mid i\ttk{p},\,j\ttk{2}\} $ & & \\ \hline

& $\ls=\<{\e2,\e3} $ & & \\
& $X=\{(i+\alpha_ip)\e1+\alpha_i\e4 +\ls \mid i\ttk{p}\} $ & & \\
$(4p,2p,4 )$ 
& $Y=\{ip\e1+ j\e4 +\ls \mid i\ttk{p},\,j\ttk{2}\} $ 
& \ref{p^2.2} \row1 & row 31 \\
& $\so=\{(i+\alpha_ip)\e1+\alpha_j\e2 + k\e3 + i\e4  \mid i\ttk{p},\,j,k\ttk{2}\} $ & & \\
& $C=\{ip\e1+ \alpha_{ij}\e2 + \beta_{ij}\e3 + j\e4 \mid i\ttk{p},\,j\ttk{2}\} $ & & \\ \hline

$(4p,2p,2p)$
& Not possible since no $X$ exists
& \ref{p.2.2}  & \\ \hline

& $\ls=\<{\e2} $ & & \\
& $X=\{(\alpha_i+\alpha_{ij}p)\e1 + j\e3 + i\e4 +\ls \mid i,j\ttk{2}\} $ & & \\
$(8,p^2,2 )$ 
& $Y=\{(i+jp)\e1 + \alpha_i\e3 +\ls \mid i,j\ttk{p}\} $ 
& \ref{p^2.2.2} \row{11} & row 40 \\
& $\so=\{(\alpha_i+\alpha_{ij}p)\e1 + k\e2 + j\e3 + i\e4 \mid i,j,k\ttk{2}\} $ & & \\
& $C=\{(i+jp)\e1 + \alpha_{ij}\e2 + \alpha_i\e3 \mid i,j\ttk{p}\} $ & & \\ \hline

& $\ls=\<{\e2} $ & & \\
& $X=\{\alpha_{ij}\e1 + i\e3 + j\e4 +\ls \mid i,j\ttk{2}\} $ & & \\
$(8,p^2,2 )$ 
& $Y=\{i\e1 +\ls \mid i\ttk{p^2}\} $ 
& \ref{p^2.2.2} \row{12} & row 41 \\
& $\so=\{\alpha_{ij}\e1 + k\e2 + i\e3 + j\e4 \mid i,j,k\ttk{2}\} $ & & \\
& $C=\{i\e1 +\alpha_i\e2 \mid i\ttk{p^2} \} $ & & \\ \hline

& $\ls=\<{\e2,\e3} $ & & \\
& $X=\{\alpha_i\e1 + i\e4 + \ls \mid i\ttk{2}\} $ & & \\
$(8,p^2,4 )$ 
& $Y=\{ip\e1 + j\e4 +\ls \mid i\ttk{p},\,j\ttk{2}\} $ 
& \ref{p^2.2} \row2 & row 42 \\
& $\so=\{\alpha_i\e1 + j\e2 + k\e3 + i\e4  \mid i,j,k\ttk{2}\} $ & & \\
& $C=\{ip\e1 + \alpha_{ij}\e2 + \beta_{ij}\e3 + j\e4 \mid i,j\ttk{p}\} $ & & \\ \hline

& $\ls=\<{\e2} $ & & \\
& $X=\{i\e1+\alpha_{ij}\e3+j\e4 +\ls \mid i\ttk{p^2},\,j\ttk{2}\} $ & & \\
$(4p^2,2,2 )$ 
& $Y=\{ i\e3 +\ls \mid i\ttk{2}\} $ 
& \ref{p^2.2.2} \row{16} & row 46 \\
& $\so=\{i\e1+k\e2+\alpha_{ij}\e3 +j\e4 \mid i\ttk{p^2},\,j,k\ttk{2}\} $ & & \\
& $C=\{\alpha_i\e2 + i\e3 \mid i\ttk{2}\} $ & & \\ \hline

& $\ls=\<{p\e1} $ & & \\
& $X=\{i\e1 + \alpha_{ijk}\e2 + j\e3 + k\e4 +\ls \mid i\ttk{p},\,j,k\ttk{2}\} $ & & \\
$(4p^2,2,p )$ 
& $Y=\{ i\e2 +\ls \mid i\ttk{2}\} $ 
& \ref{p.2.2.2} \row{10} & row 47 \\
& $\so=\{(i+lp)\e1 + \alpha_{ijk}\e2 + j\e3 + k\e4 \mid i,l\ttk{p},\,j,k\ttk{2}\} $ & & \\
& $C=\{\alpha_ip\e1 + i\e2 \mid i\ttk{2}\} $ & & \\ \hline

& $\ls=\<{\e2,\e3} $ & & \\
& $X=\{i\e1 + \alpha_i\e4 + \ls \mid i\ttk{p^2}\} $ & & \\
$(4p^2,2,4 )$ 
& $Y=\{ i\e4 + \ls \mid i\ttk{2}\} $ 
& \ref{p^2.2} \row4 & row 48 \\
& $\so=\{i\e1 + j\e2 + k\e3 + \alpha_i\e4 \mid i\ttk{p^2},\,j,k\ttk{2}\} $ & & \\
& $C=\{\alpha_i \e2 + \beta_i \e3 + i\e4 \mid i\ttk{2}\} $ & & \\ \hline

& $\ls=\<{p\e1+\e2} $ & & \\
& $X=\{i\e1 + \alpha_{ij}\e3 + j\e4 +\ls \mid i\ttk{p},\,j\ttk{2}\} $ & & \\
$(4p^2,2,2p )$ 
& $Y=\{ i\e3 +\ls \mid i\ttk{2}\} $ 
& \ref{p.2.2} \row5 & row 49 \\
& $\so=\{(i+kp)\e1 + l\e2 + \alpha_{ij}\e3 + j\e4 \mid i,k\ttk{p},\,j,l\ttk{2}\} $ & & \\
& $C=\{\alpha_ip\e1 + \alpha_i\e2 + i\e3 \mid i\ttk{2}\} $ & & \\ \hline

& $\ls=\<{\e1} $ & & \\
& $X=\{\alpha_{ij}\e2 + i\e3 + j\e4 + \ls \mid i,j\ttk{2}\} $ & & \\
$(4p^2,2,p^2 )$ 
& $Y=\{ i\e2 +\ls \mid i\ttk{2}\} $ 
& \ref{2.2.2} & row 50 \\
& $\so=\{k\e1 + \alpha_{ij}\e2 + i\e3 + j\e4 \mid i,j\ttk{2},\,k\ttk{p^2}\} $ & & \\
& $C=\{\alpha_i\e1 + i\e2 \mid i\ttk{2}\} $ & & \\ \hline

& $\ls=\<{p\e1+\e2,\e3} $ & & \\
& $X=\{i\e1 + \alpha_i\e4+\ls \mid i\ttk{p}\} $ & & \\
$(4p^2,2,4p )$ 
& $Y=\{ i\e4 +\ls \mid i\ttk{2}\} $ 
& \ref{p.q} \row1 & row 51 \\
& $\so=\{(i+jp)\e1 +k\e2 + l\e3 + \alpha_i\e4 \mid i,j\ttk{p},\,k,l\ttk{2}\} $ & & \\
& $C=\{\alpha_ip\e1 + \alpha_i\e2 + \beta_i\e3 + i\e4 \mid i\ttk{2}\} $ & & \\ \hline

$(4p^2,2,2p^2)$ 
& Not possible since no $X$ exists
& \ref{2.2}  & \\ \hline

& $\ls=\<{\e2} $ & & \\
& $X=\{(i+\alpha_{ijk}p)\e1 + j\e3 + k\e4 + \ls \mid i\ttk{p},\,j,k\ttk{2}\} $ & & \\
$(8p,p,2)$ 
& $Y=\{ ip\e1 +\ls \mid i\ttk{p}\} $ 
& \ref{p^2.2.2} \row{21} & row 53 \\
& $\so=\{(i+\alpha_{ijk}p)\e1 + l\e2 + j\e3 + k\e4 \mid i\ttk{p},\,j,k,l\ttk{2}\} $ & & \\
& $C=\{ ip\e1 + \alpha_i\e2 \mid i\ttk{p}\} $ & & \\ \hline

& $\ls=\<{p\e1} $ & & \\
& $X=\{(\alpha_{ijk})\e1+i\e2+j\e3+k\e4 +\ls \mid i,j,k\ttk{2}\} $ & & \\
$(8p,p,p )$ 
& $Y=\{ i\e1 +\ls \mid i\ttk{p}\}\} $ 
& \ref{p.2.2.2} \row{13} & row 54 \\
& $\so=\{(\alpha_{ijk}+lp)\e1+i\e2+j\e3+k\e4 \mid i,j,k\ttk{2},\,l\ttk{p}\} $ & & \\
& $C=\{ (i+\alpha_ip)\e1 \mid i\ttk{p}\} $ & & \\ \hline

& $\ls=\<{\e2,\e3} $ & & \\
& $X=\{(i+\alpha_{ij}p)\e1 + j\e4 + \ls \mid i\ttk{p},\,j\ttk{2}\} $ & & \\
$(8p,p,4 )$ 
& $Y=\{ ip\e1 + \ls \mid i\ttk{p}\} $ 
& \ref{p^2.2} \row7 & row 55 \\
& $\so= \{(i+\alpha_{ij}p)\e1 + k\e2 + l\e3 + j\e4 \mid i\ttk{p},\,j,k,l\ttk{2}\}$ & & \\
& $C=\{ ip\e1 + \alpha_i\e2 + \beta_i\e3 \mid i\ttk{p}\} $ & & \\ \hline

& $\ls=\<{p\e1+\e2} $ & & \\
& $X=\{\alpha_{ij}\e1 + i\e3 + j\e4 +\ls \mid i,j\ttk{2}\} $ & & \\
$(8p,p,2p )$ 
& $Y=\{ i\e1 +\ls \mid i\ttk{p}\} $ 
& \ref{p.2.2} \row7 & row 56 \\
& $\so=\{(\alpha_{ij}+kp)\e1 + l\e2 + i\e3 + j\e4 \mid i,j,l\ttk{2},\,k\ttk{p}\} $ & & \\
& $C=\{\alpha_ip\e1 + \alpha_i\e2 + i\e3 \mid i\ttk{2}\} $ & & \\ \hline

& $\ls=\<{\e2,\e3,\e4} $ & & \\
& $X=\{(i+\alpha_ip)\e1 + \ls \mid i\ttk{p}\} $ & & \\
$(8p,p,8 )$ 
& $Y=\{ ip\e1 + \ls \mid i\ttk{p}\} $ 
& \ref{p^2} & row 57 \\
& $\so=\{(i+\alpha_ip)\e1+j\e2+k\e3+l\e4 \mid i\ttk{p},\,j,k,l\ttk{2}\} $ & & \\
& $C=\{ip\e1+\alpha_i\e2+\beta_i\e3+\gamma_i\e4 \mid i\ttk{p}\} $ & & \\ \hline

& $\ls=\<{p\e1+\e2,\e3} $ & & \\
& $X=\{\alpha_i\e1 + i\e4 +\ls \mid i\ttk{2}\} $ & & \\
$(8p,p,4p )$ 
& $Y=\{ i\e1 +\ls \mid i\ttk{p}\} $ 
& \ref{p.q} \row2 & row 58 \\
& $\so=\{(\alpha_i+jp)\e1 + k\e2 +l\e3 + i\e4 \mid i,k,l\ttk{2},\,j\ttk{p}\} $ & & \\
& $C=\{(i+\alpha_ip)\e1 + \alpha_i\e2 + \beta_i\e3 \mid i\ttk{p}\} $ & & \\ \hline

\end{longtable}
\end{proof}

\begin{remark}
    In Theorem \ref{p^2.2.2.2}, we can change $\e2$, $\e3$, and $\e4$ to any three generators of $\{0\}\x\Z_2\x\Z_2\x\Z_2$.
\end{remark}

\begin{theorem}\label{p.2.2.2.2}
Let $G=\Z_{p}\x\Z_2\x\Z_2\x\Z_2\x\Z_2$ for an odd prime $p$ and let $S$ be a proper subset of $G\setminus \{0\}$ that generates $G$. Let $0\in C\subseteq G$. Then $C$ is a perfect code in $\Cay(G,S)$ if and only if $(\so, C)$ is one of the pairs in Table \ref{t.p.2.2.2.2} for some $\alpha_i$'s, $\alpha_{ij}$'s, $\alpha_{ijk}$'s, $\beta_i$'s, $\beta_{ij}$, $\beta_{ijk}$, and $\gamma_i$ in $\Z$ where $\alpha_0=\alpha_{00}=\alpha_{000}=\beta_0=\beta_{00}=\beta_{0000}=\gamma_0=0$.
\end{theorem}

\begin{small}
\rc\begin{longtable}{|c|l|l|}
\caption{Perfect codes in Cayley graphs of $\Z_{p}\x\Z_2\x\Z_2\x\Z_2\x\Z_2$}
\label{t.p.2.2.2.2}
\\
\fl
\tn & $\{i\e1+\alpha_i\e2+\beta_i\e3+\gamma_i\e4+\delta_i\e5 \mid i\ttk{p}\}$
& $\{i\e2 + j\e3 + k\e4 + l\e5 \mid i,j,k,l\ttk{2}\}$ \\
\tn & $\{i\e1 + j\e2 + \alpha_i\e3 + \beta_i\e4 + \gamma_i\e5  \mid i\ttk{p},\,j\ttk{2}\}$
& $\{\alpha_{ijk}\e2 + i\e3 + j\e4 + k\e5 \mid i,j,k\ttk{2}\}$ \\
\tn & $\{i\e1 + \alpha_{ij}\e2 +j\e3 + \alpha_i\e4 + \alpha_i\e5   \mid i\ttk{p},\,j\ttk{2}\}$
& $\{k\e2 + \alpha_{ij}\e3 + i\e4 + j\e5  \mid i,j,k\ttk{2}\}$ \\
\tn & $\{i\e1 + \alpha_{ij}\e2 +\beta_{ij}\e3 +j\e4+ \alpha_i\e5   \mid i\ttk{p},\,j\ttk{2}\}$
& $\{j\e2 +k\e3 + \alpha_i\e4 + i\e5 \mid i,j,k\ttk{2}\}$ \\
\tn & $\{i\e1 + \alpha_{ij}\e2 + \beta_{ij}\e3 + \gamma_{ij}\e4 + j\e5 \mid i\ttk{p},\,j\ttk{2}\}$
& $\{i\e2 + j\e3 + k\e4 \mid i,j,k\ttk{2}\}$ \\
\tn & $\{i\e1 + k\e2 +\alpha_{ij}\e3 + j\e4 +\alpha_i\e5 \mid i\ttk{p},\,j,k\ttk{2}\}$
& $\{\alpha_{ij}\e2+ j\e3+\alpha_i\e4  + i\e5 \mid i,j\ttk{2}\}$ \\
\tn & $\{i\e1 +k\e2+ \alpha_{ij}\e3 + \beta{ij}\e4 + j\e5  \mid i\ttk{p},\,j,k\ttk{2}\}$
& $\{\alpha_{ij}\e2+i\e3+j\e4 \mid i,j\ttk{2} \}$ \\
\tn & $\{i\e1+j\e2 +k\e3+\alpha_i\e4+\beta_i\e5 \mid i\ttk{p},\,j,k\ttk{2}\}$
& $\{\alpha_{ij}\e2 +\beta_{ij}\e3+i\e4+ j\e5 \mid i,j\ttk{2}\}$ \\
\tn & $\{i\e1 +\alpha_{ijk}\e2 + k\e3 + \alpha_{ij}\e4 + j\e5    \mid i\ttk{p},\,j,k\ttk{2}\}$
& $\{ j\e2+ \alpha_i\e3 +i\e4 \mid i,j\ttk{2} \}$ \\
\tn & $\{k\e1 + \alpha_{ijk}\e2 + \alpha_{ij}\e3 + i\e4 + j\e5  \mid i,j\ttk{2},\,k\ttk{p}\}$
& $\{\alpha_i\e1 +j\e2+ i\e3  \mid i,j\ttk{2} \}$ \\
\tn & $\{i\e1 +\alpha_{ijk}\e2 + j\e3 + k\e4 + \alpha_i \e5 \mid i\ttk{p},\,j,k\ttk{2}\}$
& $\{j\e2 + \alpha_i\e3 +\beta_i\e4 + i\e5 \mid i,j\ttk{2} \}$ \\
\tn & $\{i\e1 + \alpha_{ijk}\e2 + \beta_{ijk}\e3 + j\e4 + k\e5  \mid i\ttk{p},\,j,k\ttk{2}\}$
& $\{i\e2 + j\e3 \mid i,j\ttk{2}\}$ \\
\tn & $\{\alpha_i\e1 + k\e2 + \alpha_{ij}\e3 + j\e4 + i\e5 \mid i,j,k\ttk{2}\}$
& $\{i\e1 + \alpha_{ij}\e2 + j\e3 + \alpha_i\e4 \mid i\ttk{p},\,j\ttk{2}\}$ \\
\tn & $\{\alpha_{ij}\e1 + k\e2 + \alpha_{ij}\e3 + i\e4 + j\e5   \mid i,j,k\ttk{2}\}$
& $\{i\e1 + \alpha_{ij}\e2 + j\e3  \mid i\ttk{p},\,j\ttk{2}\}$ \\
\tn & $\{\alpha_{ij}\e1 + \alpha_{ijk}\e2 + k\e3 + i\e4 + j\e5   \mid i,j,k\ttk{2}\}$
& $\{i\e1 +j\e2 + \alpha_i\e3  \mid i\ttk{p},\,j\ttk{2} \}$ \\
\tn & $\{\alpha_i\e1 + \alpha_{ijk}\e2 + j\e3 + k\e4 + i\e5 \mid i,j,k\ttk{2}\}$
& $\{i\e1 +j\e2 + \alpha_i\e3 + \beta_i\e4  \mid i\ttk{p},\,j\ttk{2} \}$ \\
\tn & $\{\alpha_{ijk}\e1 + k\e2 + \alpha_{ij}\e3 + i\e4 + j\e5   \mid i,j,k\ttk{2}\}$
& $\{j\e1+ \alpha_i \e2 + i\e3  \mid i\ttk{2} \}$ \\
\tn & $\{\alpha_{ijk}\e1 + \alpha_{ijk}\e2 + i\e3 + j\e4 + k\e5  \mid i,j,k\ttk{2}\}$
& $\{i\e1 + j\e2 \mid i\ttk{p},\,j\ttk{2}\}$ \\
\tn & $\{i\e1+l\e2+\alpha_{ijk}\e3+j\e4+k\e5 \mid i\ttk{p},\,j,k,l\ttk{2}\}$
& $\{i\e3 + \alpha_i \e2 \mid i\ttk{2}\}$ \\
\tn & $\{l\e1+\alpha_{ijk}\e2+i\e3+j\e4+k\e5 \mid i,j,k\ttk{2},\,l\ttk{p}\}$
& $\{\alpha_i\e1+ i\e2 \mid i\ttk{2}\}$ \\
\tn & $\{i\e1 +k\e2 +l\e3+\alpha_{ij}\e4+j\e5 \mid i\ttk{p},\,j,k,l\ttk{2}\}$
& $\{\alpha_i\e2 + \beta_i\e3 + i\e4 \mid i\ttk{2}\}$ \\
\tn & $\{k\e1+ l\e2 +\alpha_{ij}\e3+i\e4+j\e5  \mid i,j,l\ttk{2},\,k\ttk{p}\}$
& $\{\alpha_i\e1 +\alpha_i \e2 + i\e3  \mid i\ttk{2}\}$ \\
\tn & $\{i\e1 +j\e2 +k\e3 +l\e4 +\alpha_i\e5 \mid i\ttk{p},\,j,k,l\ttk{2}\}$
& $\{\alpha_i\e2 +\beta_i\e3 +\gamma_i\e4 + i\e5 \mid i\ttk{2}\}$ \\
\tn & $\{i\e1 + \alpha_{ijkl}\e2 + j\e3 + k\e4 +l\e5 \mid i\ttk{p},\,j,k,l\ttk{2}\}$
& $\{i\e2 \mid i\ttk{2}\}$ \\
\tn & $\{\alpha_{ijk}\e1+l\e2+i\e3+j\e4+k\e5  \mid i,j,k,l\ttk{2}\}$
& $\{i\e1+\alpha_i\e2 \mid i\ttk{p}\}$ \\
\tn & $\{\alpha_{ij}\e1+k\e2 +l\e3+i\e4+j\e5  \mid i,j,k,l\ttk{2}\}$
& $\{i\e1+\alpha_i\e2 +\beta_i\e3 \mid i\ttk{p}\}$ \\
\tn & $\{\alpha_i\e1 +j\e2 +k\e3 +l\e4+i\e5 \mid i,j,k,l\ttk{2}\}$
& $\{i\e1+\alpha_i\e2 +\beta_i\e3 +\gamma_i\e4 \mid i\ttk{p}\}$ \\
\tn & $\{\alpha_{ijkl}\e1+i\e2+j\e3+k\e4+l\e5 \mid i,j,k,l\ttk{2}\}$
& $\{i\e1 \mid i\ttk{p}\}$ \\
\hline
\end{longtable}
\end{small}

\begin{proof}
    Suppose $C$ is a perfect code in $\Cay(G,S)$. Then $G=\so\oplus C$ and so $|\so|\cdot |C|=|G|=16p  $. 
    Since $\so$ generates $G$, $\so$ contains at least four non-zero elements, so $|\so|\geq 5$.
    Therefore, $4<|\so|<16p$.

\smallskip
    \textsf{Case 1.} $\so$ is aperiodic.
    \smallskip

    Since $G$ is Haj\'os and $G=\so\oplus C$, in this case $C$ must be periodic, that is, $\lc$ is nontrivial.
    Write $C=\lc\oplus D$ for some subset $D$ of $G$ whose existence is ensured by Lemma \ref{period}.
    Then $|\lc|\cdot |D|=|C|$ and $2\leq |\lc|\leq |C|$.
    If $|\so|$ is a prime, then $|\lc|=|C|$ by Lemma \ref{prime deg} part (c).
    If $|\so|$ is not a prime, then $|\lc|$ can be any divisor of $|C|$.
    Set $$ A=(\so+\lc)/\lc\quad\text{and}\quad B=(D+\lc)/\lc.$$
    Since $C$ is periodic, by Lemma \ref{quotient}, we have
    $$G/\lc=A\oplus B$$ and $B$ is aperiodic in $G/\lc$.
    Since $G/\lc$ is isomorphic to a subgroup of $G$ and subgroups of Haj\'os group are also Haj\'os groups, $A$ is periodic in $G/\lc$.
    We also know that $A$ contains $\lc$ and $A$ generates $G/\lc$ by Lemma \ref{generate-zero}.
    Furthermore, $|A|=|\so|$ and $|B|=|D|$.
    Note that if $|\lc|=|C|$, then $C=\lc$ as $0\in C$.
    This means we can take $D=\{0\}$ and hence $B=\lc/\lc$ and $G/\lc=A$.
    We use $C$ to get $A$ and then use Lemma \ref{d_i+l_i} to obtain $\so$.
    If $|\lc|<|C|$, then this factorization enables us to make use of previously established theorems to obtain the pair $(\so, C)$ via $(A,B)$.

    Table \ref{tap.p.2.2.2.2} represents all possible values for $\left(|\so|,|C|,|\lc|\right)$; $C$, $A$, and $\so$ when $|\lc|=|C|$ or $L_C$, $A$, $B$, $\so$, and $C$ when $|\lc|<|C|$; which previous theorem is used to obtain $(A,B)$ when $|\lc|<|C|$; and which row in Table \ref{t.p.2.2.2.2} the result corresponds to.
        
\begin{longtable}{|c|l|c|c|}
\caption{The case when $G=\Z_{p}\x\Z_2\x\Z_2\x\Z_2\x\Z_2 $ and $\so$ is aperiodic}
\label{tap.p.2.2.2.2}
\\
\flap

& $C=\<{\e2,\e3,\e4,\e5} $ & & \\
$(p,16,16)$ 
& $A=\{i\e1+\lc \mid i\ttk{p}\} $ 
& & row 1 \\
& $\so=\{i\e1+\alpha_i\e2+\beta_i\e3+\gamma_i\e4+\delta_i\e5 \mid i\ttk{p}\} $ & & \\ \hline

& $\lc=\<{\e2} $ & & \\
& $A=\{i\e1 +j\e3 + \alpha_i\e4 + \alpha_i\e5 + \lc \mid i\ttk{p},\,j\ttk{2}\} $ & & \\
$(2p,8,2 )$ 
& $B=\{\alpha_{ij}\e3 +  i\e4 + j\e5 + \lc \mid i,j\ttk{2}\} $ 
& \ref{p.2.2.2} \row2 & row 3 \\
& $\so=\{i\e1 + \alpha_i\e4 + \alpha_i\e5 +\alpha_{ij}\e2 + j\e3 \mid i\ttk{p},\,j\ttk{2}\} $ & & \\
& $C=\{\alpha_{ij}\e3 +  i\e4 + j\e5 + k\e2 \mid i,j,k\ttk{2}\} $ & & \\ \hline

& $\lc=\<{\e2,\e3} $ & & \\
& $A=\{i\e1 +j\e4 + \alpha_i\e5  + \lc \mid i\ttk{p},\,j\ttk{2}\} $ & & \\
$(2p,8,4 )$ 
& $B=\{ \alpha_i\e4 + i\e5 + \lc \mid i\ttk{2}\} $ 
& \ref{p.2.2} \row4 & row 4 \\
& $\so=\{i\e1 +j\e4 + \alpha_i\e5  + \alpha_{ij}\e2 +\beta_{ij}\e3 \mid i\ttk{p},\,j\ttk{2}\} $ & & \\
& $C=\{ \alpha_i\e4 +i\e5+j\e2 +k\e3 \mid i,j,k\ttk{2}\} $ & & \\ \hline

& $C=\<{\e2,\e3,\e4} $ & & \\
$(2p,8,8 )$ 
& $A=\{i\e1 + j\e5 + \lc \mid i\ttk{p},\,j\ttk{2}\} $ 
& & row 5 \\
& $\so=\{i\e1 + j\e5 + \alpha_{ij}\e2 + \beta_{ij}\e3 + \gamma_{ij}\e4 \mid i\ttk{p},\,j\ttk{2}\} $ & & \\ \hline

& $\lc=\<{\e2} $ & & \\
& $A=\{i\e1 + k\e3 + \alpha_{ij}\e4 + j\e5 +\lc \mid i\ttk{p},\,j,k\ttk{2}\} $ & & \\
$(4p,4,2 )$ 
& $B=\{\alpha_i\e3 + i\e4 +\lc \mid i\ttk{2} \} $ 
& \ref{p.2.2.2} \row7 & row 9 \\
& $\so=\{i\e1 + k\e3 + \alpha_{ij}\e4 + j\e5 +\alpha_{ijk}\e2 \mid i\ttk{p},\,j,k\ttk{2}\} $ & & \\
& $C=\{\alpha_i\e3 + i\e4 +j\e2 \mid i,j\ttk{2} \} $ & & \\ \hline

& $\lc=\<{\e2} $ & & \\
& $A=\{k\e1+ \alpha_{ij}\e3 +  i\e4 + j\e5 +\lc \mid i,j\ttk{2},\,k\ttk{p}\} $ & & \\
$(4p,4,2 )$ 
& $B=\{\alpha_i\e1 + i\e3 +\lc \mid i\ttk{2} \} $ 
& \ref{p.2.2.2} \row8 & row 10 \\
& $\so=\{k\e1+ \alpha_{ij}\e3 +  i\e4 + j\e5 +\alpha_{ijk}\e2 \mid i,j\ttk{2},\,k\ttk{p}\} $ & & \\
& $C=\{\alpha_i\e1 + j\e2 + i\e3 \mid i,j\ttk{2} \} $ & & \\ \hline

& $\lc=\<{\e2} $ & & \\
& $A=\{i\e1 + j\e3 + k\e4 + \alpha_i \e5 +\lc \mid i\ttk{p},\,j,k\ttk{2}\} $ & & \\
$(4p,4,2 )$ 
& $B=\{\alpha_i\e3 +\beta_i\e4 + i\e5 +\lc \mid i\ttk{2} \} $ 
& \ref{p.2.2.2} \row9 & row 11 \\
& $\so=\{i\e1 + j\e3 + k\e4 + \alpha_i \e5 +\alpha_{ijk}\e2 \mid i\ttk{p},\,j,k\ttk{2}\} $ & & \\
& $C=\{\alpha_i\e3 +\beta_i\e4 + i\e5 +j\e2 \mid i,j\ttk{2} \} $ & & \\ \hline

& $C=\<{\e2,\e3} $ & & \\
$(4p,4,4 )$ 
& $A=\{i\e1 + j\e4 + k\e5 + \lc \mid i\ttk{p},\,j,k\ttk{2}\} $ 
& & row 12 \\
& $\so=\{i\e1 + j\e4 + k\e5 + \alpha_{ijk}\e2 + \beta_{ijk}\e3 \mid i\ttk{p},\,j,k\ttk{2}\} $ & & \\ \hline

& $\lc=\<{\e2} $ & & \\
& $A=\{\alpha_{ij}\e1 +  k\e3+i\e4+j\e5 +\lc \mid i,j,k\ttk{2}\} $ & & \\
$(8,2p,2 )$ 
& $B=\{i\e1 + \alpha_i \e3 +\lc \mid i\ttk{p} \} $ 
& \ref{p.2.2.2} \row{11} & row 15 \\
& $\so=\{\alpha_{ij}\e1 +  k\e3+i\e4+j\e5 +\alpha_{ijk}\e2 \mid i,j,k\ttk{2}\} $ & & \\
& $C=\{i\e1 + \alpha_j\e2 + i\e3 \mid i\ttk{p},\,j\ttk{2} \} $ & & \\ \hline

& $\lc=\<{\e2} $ & & \\
& $A=\{\alpha_i\e1 + j\e3 + k\e4 + i\e5 +\lc \mid i,j,k\ttk{2}\} $ & & \\
$(8,2p,2 )$ 
& $B=\{i\e1 + \alpha_i\e3 + \beta_i\e4 +\lc \mid i\ttk{p} \} $ 
& \ref{p.2.2.2} \row{12} & row 16 \\
& $\so=\{\alpha_i\e1 + j\e3 + k\e4 + i\e5 +\alpha_{ijk}\e2 \mid i,j,k\ttk{2}\} $ & & \\
& $C=\{i\e1 + j\e2 + \alpha_i\e3 + \beta_i\e4 \mid i\ttk{p},\,j\ttk{2} \} $ & & \\ \hline

& $\lc=\<{\e1} $ & & \\
& $A=\{k\e2 + \alpha_{ij}\e3 +  i\e4 + j\e5  +\lc \mid i,j,k\ttk{2}\} $ & & \\
$(8,2p,p )$ 
& $B=\{\alpha_i \e2+ i\e3  +lc \mid i\ttk{2} \} $ 
& \ref{2.2.2.2} \row 1 & row 17 \\
& $\so=\{\alpha_{ij}\e1 +  k\e3+i\e4+j\e5 +\alpha_{ijk}\e2 \mid i,j,k\ttk{2}\} $ & & \\
& $C=\{i\e1 + \alpha_j\e2 + i\e3 \mid i\ttk{p},\,j\ttk{2} \} $ & & \\ \hline

& $C=\<{\e1+\e2} $ & & \\
$(8,2p,2p )$ 
& $A=\{i\e3 + j\e4 + k\e5 + \lc \mid i,j,k\ttk{2}\} $ 
& & row 18 \\
& $\so=\{\alpha_{ijk}\e1 + i\e3 + j\e4 + k\e5 + \alpha_{ijk}\e2 \mid i,j,k\ttk{2}\} $ & & \\ \hline

& $C=\<{\e2} $ & & \\
$(8p,2,2 )$ 
& $A=\{i\e1 + j\e3 + k\e4 +l\e5 + \lc \mid i\ttk{p},\,j,k,l\ttk{2}\} $ 
& & row 24 \\
& $\so=\{i\e1 + j\e3 + k\e4 +l\e5 + \alpha_{ijkl}\e2 \mid i\ttk{p},\,j,k,l\ttk{2}\} $ & & \\ \hline

& $C=\<{\e1} $ & & \\
$(16,p,p)$ 
& $A=\{i\e2+j\e3+k\e4+l\e5 + \lc \mid i,j,k,l\ttk{2}\} $ 
& & row 28 \\
& $\so=\{\alpha_{ijkl}\e1+i\e2+j\e3+k\e4+l\e5 \mid i,j,k,l\ttk{2}\} $ & & \\ \hline

\end{longtable}

\smallskip
    \textsf{Case 2.} $\so$ is periodic.
    \smallskip

    In this case we have $|\ls|\geq 2$ and $\so=D\oplus \ls$ for some subset $D$ of $G$ by Lemma \ref{period}.
    So $|D|\cdot|\ls|=|\so|$ and $1<|\ls|<|\so|$ by Lemma \ref{prime deg} part (b).
    Set $$X=(D+\ls)/\ls\quad\text{and}\quad Y=(C+\ls)/\ls.$$
    Since $\so$ is periodic, by Lemma \ref{quotient}, we have
    $$G/\lc=X\oplus Y$$ and $X$ is aperiodic in $G/\ls$.
    We also know that $X$ contains $\ls$ and $X$ generates $G/\ls$ by Lemma \ref{generate-zero}.
    This factorization enables us to make use of Lemma \ref{d_i+l_i} and previously established theorems to obtain the pair $(\so, C)$ via $(X,Y)$.
    
    Table \ref{tp.p.2.2.2.2} represents all possible values for $\left(|\so|,|C|,|\ls|\right)$; $\ls$, $X$, $Y$, $\so$, and $C$; which previous theorem is used to obtain $(X,Y)$; and which row in Table \ref{t.p.2.2.2.2} the result corresponds to.

\begin{longtable}{|c|l|c|c|}
\caption{The case when $G=\Z_{p}\x\Z_2\x\Z_2\x\Z_2\x\Z_2 $ and $\so$ is periodic}
\label{tp.p.2.2.2.2}
\\
\flp

& $\ls=\<{\e2} $ & & \\
& $X=\{i\e1 + \alpha_i\e3 + \beta_i\e4 + \gamma_i\e5 + \ls \mid i\ttk{p}\} $ & & \\
$(2p,8,2 )$ 
& $Y=\{i\e3 + j\e4 + k\e5 + \ls \mid i,j,k\ttk{2}\} $ 
& \ref{p.2.2.2} \row1 & row 2 \\
& $\so=\{i\e1 + \alpha_i\e3 + \beta_i\e4 + \gamma_i\e5 + j\e2 \mid i\ttk{p},\,j\ttk{2}\} $ & & \\
& $C=\{i\e3 + j\e4 + k\e5 + \alpha_{ijk}\e2 \mid i,j,k\ttk{2}\} $ & & \\ \hline

$(2p,8,p )$ 
& Not possible since no $X$ exists
& \ref{2.2.2.2}  & \\ \hline

& $\ls=\<{\e2} $ & & \\
& $X=\{i\e1+\alpha_{ij}\e3+j\e4 +\alpha_i\e5 +\ls \mid i\ttk{p},\,j\ttk{2}\} $ & & \\
$(4p,4,2)$ 
& $Y=\{j\e3+\alpha_i\e4 + i\e5 +\ls \mid i,j\ttk{2}\} $ 
& \ref{p.2.2.2} \row3 & row 6 \\
& $\so=\{i\e1+\alpha_{ij}\e3+j\e4 +\alpha_i\e5 +k\e2 \mid i\ttk{p},\,j,k\ttk{2}\} $ & & \\
& $C=\{j\e3+\alpha_i\e4 + i\e5 +\alpha_{ij}\e2 \mid i,j\ttk{2}\} $ & & \\ \hline

& $\ls=\<{\e2} $ & & \\
& $X=\{i\e1 + \alpha_{ij}\e3 + \beta{ij}\e4 + j\e5 +\ls \mid i\ttk{p},\,j\ttk{2}\} $ & & \\
$(4p,4,2 )$ 
& $Y=\{i\e3+j\e4 +\ls \mid i,j\ttk{2} \} $ 
& \ref{p.2.2.2} \row4 & row 7 \\
& $\so=\{i\e1 + \alpha_{ij}\e3 + \beta{ij}\e4 + j\e5 +k\e2 \mid i\ttk{p},\,j,k\ttk{2}\} $ & & \\
& $C=\{i\e3+j\e4 +\alpha_{ij}\e2 \mid i,j\ttk{2} \} $ & & \\ \hline

$(4p,4,p )$ 
& Not possible since no $X$ exists
& \ref{2.2.2.2}  & \\ \hline

& $\ls=\<{\e2,\e3} $ & & \\
& $X=\{i\e1+\alpha_i\e4+\beta_i\e5 +\ls \mid i\ttk{p}\} $ & & \\
$(4p,4,4 )$ 
& $Y=\{i\e4+ j\e5 +\ls \mid i,j\ttk{2}\} $ 
& \ref{p.2.2} \row1 & row 8 \\
& $\so=\{i\e1+\alpha_i\e4+\beta_i\e5 +j\e2 +k\e3 \mid i\ttk{p},\,j,k\ttk{2}\} $ & & \\
& $C=\{i\e4+ j\e5 + \alpha_{ij}\e2 +\beta_{ij}\e3 \mid i,j\ttk{2}\} $ & & \\ \hline

$(4p,4,2p )$ 
& Not possible since no $X$ exists
& \ref{2.2.2}  & \\ \hline

& $\ls=\<{\e2} $ & & \\
& $X=\{\alpha_i\e1+\alpha_{ij}\e3+j\e4 +i\e5 +\ls \mid i,j\ttk{2}\} $ & & \\
$(8,2p,2 )$ 
& $Y=\{i\e1+ j\e3 + \alpha_i\e4 \mid i\ttk{p},\,j\ttk{2}\} $ 
& \ref{p.2.2.2} \row5 & row 13 \\
& $\so=\{\alpha_i\e1+\alpha_{ij}\e3+j\e4 +i\e5 +k\e2 \mid i,j,k\ttk{2}\} $ & & \\
& $C=\{i\e1+ j\e3 + \alpha_i\e4 +\alpha_{ij}\e2 \mid i\ttk{p},\,j\ttk{2}\} $ & & \\ \hline

& $\ls=\<{\e2} $ & & \\
& $X=\{\alpha_{ij}\e1+ \alpha_{ij}\e3 +  i\e4 + j\e5 +\ls \mid i,j\ttk{2}\} $ & & \\
$(8,2p,2 )$ 
& $Y=\{i\e1+ j\e3 \mid i\ttk{p},\,j\ttk{2}\} $ 
& \ref{p.2.2.2} \row6 & row 14 \\
& $\so=\{\alpha_{ij}\e1+ \alpha_{ij}\e3 +  i\e4 + j\e5 +k\e2 \mid i,j,k\ttk{2}\} $ & & \\
& $C=\{i\e1+ \alpha_{ij}\e2 + j\e3 \mid i\ttk{p},\,j\ttk{2}\} $ & & \\ \hline

$(8,2p,4 )$ 
& Not possible since no $X$ exists
& \ref{p.2.2}  & \\ \hline

& $\ls=\<{\e2} $ & & \\
& $X=\{i\e1+\alpha_{ijk}\e3+j\e4+k\e5 +\ls \mid i\ttk{p},\,j,k\ttk{2}\} $ & & \\
$(8p,2,2 )$ 
& $Y=\{i\e3+\ls \mid i\ttk{2}\} $ 
& \ref{p.2.2.2} \row{10} & row 19 \\
& $\so=\{i\e1+\alpha_{ijk}\e3+j\e4+k\e5 +l\e2 \mid i\ttk{p},\,j,k,l\ttk{2}\} $ & & \\
& $C=\{i\e3 + \alpha_i \e2 \mid i\ttk{2}\} $ & & \\ \hline

& $\ls=\<{\e1} $ & & \\
& $X=\{\alpha_{ijk}\e2+i\e3+j\e4+k\e5 +\ls \mid i,j,k\ttk{2}\} $ & & \\
$(8p,2,p )$ 
& $Y=\{i\e2+\ls \mid i\ttk{2}\} $ 
& \ref{2.2.2.2} \row2 & row 20 \\
& $\so=\{l\e1+i\e3+j\e4+k\e5 +\alpha_{ijk}\e2 \mid i,j,k\ttk{2},\,l\ttk{p}\} $ & & \\
& $C=\{\alpha_i\e1+ i\e2 \mid i\ttk{2}\} $ & & \\ \hline

& $\ls=\<{\e2,\e3} $ & & \\
& $X=\{i\e1 +\alpha_{ij}\e4 + j\e5 +\ls \mid i\ttk{p},\,j\ttk{2}\} $ & & \\
$(8p,2,4 )$ 
& $Y=\{i\e4+\ls \mid i\ttk{2}\} $ 
& \ref{p.2.2} \row5 & row 21 \\
& $\so=\{i\e1 +\alpha_{ij}\e4 + j\e5 +k\e2 +l\e3 \mid i\ttk{p},\,j,k,l\ttk{2}\} $ & & \\
& $C=\{\alpha_i\e2 +\beta_i\e3 + i\e4 \mid i\ttk{2}\} $ & & \\ \hline

& $\ls=\<{\e1+\e2} $ & & \\
& $X=\{\alpha_{ij}\e3+i\e4+j\e5 +\ls \mid i,j\ttk{2}\} $ & & \\
$(8p,2,2p )$ 
& $Y\{i\e3+\ls \mid i\ttk{2}\}= $ 
& \ref{2.2.2} & row 22 \\
& $\so=\{k\e1+\alpha_{ij}\e3+i\e4+j\e5 + l\e2 \mid i,j,l\ttk{2},\,k\ttk{p}\} $ & & \\
& $C=\{\alpha_i\e1+ i\e3 +\alpha_i \e2 \mid i\ttk{2}\} $ & & \\ \hline

& $\ls=\<{\e2,\e3,\e4} $ & & \\
& $X=\{i\e1 +\alpha_i\e5 +\ls \mid i\ttk{p}\} $ & & \\
$(8p,2,8 )$ 
& $Y=\{i\e5+\ls \mid i\ttk{2}\} $ 
& \ref{p.q} \row1 & row 23 \\
& $\so=\{i\e1 +\alpha_i\e5 +j\e2 +k\e3 +l\e4 \mid i\ttk{p},\,j,k,l\ttk{2}\} $ & & \\
& $C=\{i\e5 + \alpha_i\e2 +\beta_i\e3 +\gamma_i\e4 \mid i\ttk{2}\} $ & & \\ \hline

$(8p,2,4p )$ 
& Not possible since no $X$ exists
& \ref{2.2}  & \\ \hline

& $\ls=\<{\e2} $ & & \\
& $X=\{\alpha_{ijk}\e1+i\e3+j\e4+k\e5 +\ls \mid i,j,k\ttk{2}\} $ & & \\
$(16,p,2 )$ 
& $Y=\{i\e1+\ls \mid i\ttk{p}\} $ 
& \ref{p.2.2.2} \row{13} & row 25 \\
& $\so=\{\alpha_{ijk}\e1+i\e3+j\e4+k\e5 +l\e2 \mid i,j,k,l\ttk{2}\} $ & & \\
& $C=\{i\e1+\alpha_i\e2 \mid i\ttk{p}\} $ & & \\ \hline

& $\ls=\<{\e2,\e3} $ & & \\
& $X=\{\alpha_{ij}\e1+i\e4+j\e5 +\ls \mid i,j\ttk{2}\} $ & & \\
$(16,p,4 )$ 
& $Y=\{i\e1+\ls \mid i\ttk{p}\} $ 
& \ref{p.2.2} \row7 & row 26 \\
& $\so=\{\alpha_{ij}\e1+i\e4+j\e5 +k\e2 +l\e3 \mid i,j,k,l\ttk{2}\} $ & & \\
& $C=\{i\e1+\alpha_i\e2 +\beta_i\e3 \mid i\ttk{p}\} $ & & \\ \hline

& $\ls=\<{\e2,\e3,\e4} $ & & \\
& $X=\{\alpha_i\e1+i\e5 +\ls \mid i\ttk{2}\} $ & & \\
$(16,p,8 )$ 
& $Y=\{i\e1+\ls \mid i\ttk{p}\} $ 
& \ref{p.q} \row2 & row 27 \\
& $\so=\{\alpha_i\e1+i\e5 +j\e2 +k\e3 +l\e4 \mid i,j,k,l\ttk{2}\} $ & & \\
& $C=\{i\e1+\alpha_i\e2 +\beta_i\e3 +\gamma_i\e4 \mid i\ttk{p}\} $ & & \\ \hline

\end{longtable}
\end{proof}

\begin{remark}
    In Theorem \ref{p.2.2.2.2}, we can change $\e2$, $\e3$, $\e4$, and $\e5$ to any four generators of $\{0\}\x\Z_2\x\Z_2\x\Z_2\x\Z_2$.
\end{remark}
\section{Other Haj\'os groups}
\label{sec:pq22}

In this section, we characterize perfect codes in Cayley graphs of the remaining Haj\'os groups. These groups are subgroups of $\Z_p \times \Z_q \times \Z_2 \times \Z_2$, $\Z_p \times \Z_4 \times \Z_2$, $\Z_9 \x \Z_3$, $\Z_{p^2}\x\Z_{q^2}$, and $\Z_{p^2}\x\Z_q\x\Z_r$.

\subsection{Subgroups of $\Z_p \times \Z_q \times \Z_2 \times \Z_2$}

\begin{theorem}\label{p.q.2}
Let $G=\Z_{p}\x\Z_q\x\Z_r$ for distinct odd primes $p$, $q$ and $r$, and let $S$ be a proper subset of $G\setminus \{0\}$ that generates $G$. Let $0\in C\subseteq G$. Then $C$ is a perfect code in $\Cay(G,S)$ if and only if $(\so,C)$ is one of the pairs in Table \ref{t.p.q.2} for some $\alpha_i$'s, $\alpha_{ij}$'s, and $\beta_i$'s in $\Z$ where $\alpha_0=\alpha_{00}=\alpha_{000}=\beta_0=0$.
\end{theorem}

\rc\begin{longtable}{|c|l|l|}
\caption{Perfect codes in Cayley graphs of $\Z_{p}\x\Z_q\x\Z_r$}
\label{t.p.q.2}
\\
\fl
\tn & $\{(i,\alpha_i,\beta_i) \mid i\ttk{p}\}$
& $\{(0,i,j) \mid i\ttk{q},\,j\ttk{r}\}$ \\
\tn & $\{(\alpha_i,i,\beta_i) \mid i\ttk{q}\}$
& $\{(i,0,j) \mid i\ttk{p},\,j\ttk{r}\}$ \\
\tn & $\{(\alpha_i,\beta_i,i) \mid i\ttk{r}\}$
& $\{(i,j,0) \mid i\ttk{p},\,j\ttk{q}\}$ \\
\tn & $\{(j,i,\alpha_i) \mid i\ttk{q},\,j\ttk{p}\}$
& $\{(\alpha_i,0,i) \mid i\ttk{r}\}$ \\
\tn & $\{(i,j,\alpha_i) \mid i\ttk{p},\,j\ttk{q}\}$
& $\{(0,\alpha_i,i) \mid i\ttk{r}\}$ \\
\tn & $\{(j,\alpha_i,i) \mid i\ttk{r},\,j\ttk{p}\}$
& $\{(\alpha_i,i,0) \mid i\ttk{q}\}$ \\
\tn & $\{(i,\alpha_i,j) \mid i\ttk{p},\,j\ttk{r}\}$
& $\{(0,i,\alpha_i) \mid i\ttk{q}\}$ \\
\tn & $\{(\alpha_i,j,i) \mid i\ttk{r},\,j\ttk{q}\}$
& $\{(i,\alpha_i,0) \mid i\ttk{p}\}$ \\
\tn & $\{(\alpha_i,i,j) \mid i\ttk{q},\,j\ttk{r}\}$
& $\{(i,0,\alpha_i) \mid i\ttk{p}\}$ \\
\tn & $\{(i,j,\alpha_{ij}) \mid i\ttk{p},\,j\ttk{q}\}$
& $\{(0,0,i) \mid i\ttk{r}\}$ \\
\tn & $\{(i,\alpha_{ij},j) \mid i\ttk{p},\,j\ttk{r}\}$
& $\{(0,i,0) \mid i\ttk{q}\}$ \\
\tn & $\{(\alpha_{ij},i,j) \mid i\ttk{q},\,j\ttk{r}\}$
& $\{(i,0,0) \mid i\ttk{p}\}$ \\
\hline
\end{longtable}

\begin{proof}
    \setcounter{cn}{1}\renewcommand{\Table}{\ref{t.p.q.2}}
    Suppose $C$ is a perfect code in $\Cay(G,S)$.

    If $|\so|=p$, then from Lemma \ref{prime deg}, $\so$ is aperiodic and $C$ is a subgroup of order $qr$, so $C=\lc=\{0\}\times\Z_q\times\Z_r$.
    From the factorization of $G/\lc$ given in Lemma \ref{quotient}, $(\so+\lc)/\lc=G/\lc=\{(i,0,0)+\lc \mid i\ttk{p}\}$, and hence $\so=\{(i,\alpha_i,\beta_i) \mid i\ttk{p}\}$, where $\alpha_i,\beta_i\in \Z$. \R 
    The same argument can be applied when $|\so|=q$ and $|\so|=r$ to get rows 2 and 3 in Table \ref{t.p.q.2}.

    If $|\so|=pq$ and $\so$ periodic, then $|\ls|=p$ or $q$.
    Let $\so=D\oplus\ls$ for some subset $D$. From the factorization of $G/\ls$ given in Lemma \ref{quotient}, we have $G/\ls=(D+\ls)/\ls\oplus (C+\ls)/\ls$.
    If $|\ls|=p$, then $\ls=\<{(1,0,0)}$.
    Note that $(D+\ls)/\ls$ is a aperiodic set of order $q$ and $G/\ls\cong \Z_{q}\times\Z_r$.
    From Theorem \ref{p.q} \row2, we get $(D+\ls)/\ls=\{(0,i,\alpha_i)+\ls \mid i\ttk{q}\}$ and $(C+\ls)/\ls=\{(0,0,i)+\ls \mid i\ttk{r}\}$.
    Therefore, $\so=D\oplus\ls=\{(j,i,\alpha_i) \mid i\ttk{q},\,j\ttk{p}\}$ and $C=\{(\alpha_i,0,i) \mid i\ttk{r}\}$ where $\alpha_i\in \Z$. This is row 4 in Table \ref{t.p.q.2}.
    By using the same arguments, we can get the case of $|\ls|=q$. This is row 5 in Table \ref{t.p.q.2}.
    We can get rows 6 to 9 by changing $|\so|$ to $pr$ and $qr$, and also set $|\ls|$ to a prime divisor of $|\so|$.

    If $|\so|=pq$ and $\so$ aperiodic, then $|C|=r$ and $C$ is periodic, 
    which means $C=\lc=\<{(0,0,1)}$.
    From the factorization of $G/\lc$, $(\so+\lc)/\lc=G/\lc=\{(i,j,0) + \lc \mid i\ttk{p},\,j\ttk{q}\}$, and hence $\so=\{(i,j,\alpha_{ij}) \mid i\ttk{p},\,j\ttk{q}\}$, $\alpha_{ij}\in \Z$. This is row 10 in Table \ref{t.p.q.2}. 
    Rows 11 and 12 are the cases $|\so|$ equal to $pr$ and $qr$ with the same arguments.
\end{proof}

\begin{theorem}\label{p,q,2,2}
Let $G=\Z_{p}\x\Z_q\x\Z_2\x\Z_2$ for distinct odd primes $p$ and $q$, and let $S$ be a proper subset of $G\setminus \{0\}$ that generates $G$. Let $0\in C\subseteq G$. Then $C$ is a perfect code in $\Cay(G,S)$ if and only if $(\so,C)$ is one of the pairs in Table \ref{t.p.q.2.2} for some $\alpha_i$'s, $\alpha_{ij}$'s, $\alpha_{ijk}$'s, and $\beta_i$'s, and $\gamma_i$'s in $\Z$ where $\alpha_0=\alpha_{00}=\alpha_{000}=\beta_0=\gamma_0=0$.
\end{theorem}

\rc\begin{longtable}{|c|l|l|}
\caption{Perfect codes in Cayley graphs of $\Z_{p}\x\Z_q\x\Z_2\x\Z_2$}
\label{t.p.q.2.2}
\\
\fl
\tn & $\{i\e1+\alpha_i\e2+\beta_i\e3+\gamma_i\e4 \mid i\ttk{p}\}$
& $\{i\e2+j\e3+k\e4 \mid i\ttk{q},\,j,k\ttk{2}\}$ \\
\tn & $\{\alpha_i\e1+i\e2+\beta_i\e3+\gamma_i\e4 \mid i\ttk{q}\}$
& $\{i\e1+j\e3+k\e4 \mid i\ttk{p},\,j,k\ttk{2}\}$ \\
\tn & $\{j\e1+i\e2+\alpha_i\e3+\beta_i\e4 \mid i\ttk{q},\,j\ttk{p}\}$
& $\{\alpha_{ij}\e1+i\e3+j\e4 \mid i,j\ttk{2}\}$ \\
\tn & $\{i\e1+j\e2+\alpha_i\e3+\beta_i\e4 \mid i\ttk{p},\,j\ttk{q}\}$
& $\{\alpha_{ij}\e2+i\e3+j\e4 \mid i,j\ttk{2}\}$ \\
\tn & $\{j\e1+i\e2 +\alpha_{ij}\e3 +\alpha_i\e4 \mid i\ttk{q},\,j\ttk{p}\}$
& $\{\alpha_i\e1+j\e3 +i\e4  \mid i,j\ttk{2}\}$ \\
\tn & $\{i\e1+j\e2 +\alpha_{ij}\e3 +\alpha_i\e4  \mid i\ttk{p},\,j\ttk{q}\}$
& $\{\alpha_i\e2 + j\e3 + i\e4  \mid i,j\ttk{2}\}$ \\
\tn & $\{i\e1+j\e2+\alpha_{ij}\e3+\beta_{ij}\e4 \mid i\ttk{p},\,j\ttk{q}\}$
& $\{i\e3+j\e4 \mid i,j\ttk{2}\}$ \\
\tn & $\{i\e1+\alpha_i\e2+j\e3+\beta_i\e4 \mid i\ttk{p},\,j\ttk{2}\}$
& $\{i\e2+\alpha_{ij}\e3+j\e4 \mid i\ttk{q},\,j\ttk{2}\}$ \\
\tn & $\{j\e1+\alpha_i\e2+\alpha_{ij}\e3 +i\e4  \mid i\ttk{2},\,j\ttk{p}\}$
& $\{\alpha_i\e1+i\e2 + j\e3 \mid i\ttk{q},\,j\ttk{2}\}$ \\
\tn & $\{i\e1+\alpha_i\e2+\alpha_{ij}\e3 +j\e4  \mid i\ttk{p},\,j\ttk{2}\}$
& $\{i\e2+ j\e3 + \alpha_i\e4  \mid i\ttk{q},\,j\ttk{2}\}$ \\
\tn & $\{i\e1+\alpha_{ij}\e2 + j\e3+\alpha_i\e4  \mid i\ttk{p},\,j\ttk{2}\}$
& $\{j\e2+\alpha_i\e3+i\e4 \mid i\ttk{2},\,j\ttk{q}\}$ \\
\tn & $\{i\e1+\alpha_{ij}\e2 + \alpha_{ij}\e3+ j\e4  \mid i\ttk{p},\,j\ttk{2}\}$
& $\{i\e2 + j\e3 \mid i\ttk{q},\,j\ttk{2}\}$ \\
\tn & $\{\alpha_i\e1+i\e2+j\e3+\beta_i\e4 \mid i\ttk{q},\,j\ttk{2}\}$
& $\{i\e1+\alpha_{ij}\e3+j\e4 \mid i\ttk{p},\,j\ttk{2}\}$ \\
\tn & $\{\alpha_i\e1+j\e2+\alpha_{ij}\e3 +i\e4  \mid i\ttk{2},\,j\ttk{q}\}$
& $\{i\e1+\alpha_i\e2 + j\e3 \mid i\ttk{p},\,j\ttk{2}\}$ \\
\tn & $\{\alpha_i\e1+i\e2+\alpha_{ij}\e3 +j\e4  \mid i\ttk{q},\,j\ttk{2}\}$
& $\{i\e1+ j\e3 + \alpha_i\e4  \mid i\ttk{p},\,j\ttk{2}\}$ \\
\tn & $\{\alpha_{ij}\e1+i\e2 + j\e3+\alpha_i\e4  \mid i\ttk{q},\,j\ttk{2}\}$
& $\{j\e1+\alpha_i\e3+i\e4 \mid i\ttk{2},\,j\ttk{p}\}$ \\
\tn & $\{\alpha_{ij}\e1+i\e2 + \alpha_{ij}\e3 + j\e4 \mid i\ttk{q},\,j\ttk{2}\}$
& $\{i\e1 + j\e3 \mid i\ttk{p},\,j\ttk{2}\}$ \\
\tn & $\{\alpha_1\e1+\beta_i\e2+j\e3+i\e4 \mid i,j\ttk{2}\}$
& $\{i\e1+j\e2+\alpha_{ij}\e3 \mid i\ttk{p},\,j\ttk{q}\}$ \\
\tn & $\{\alpha_{ij}\e1+\alpha_i\e2 + j\e3+i\e4  \mid i,j\ttk{2}\}$
& $\{j\e1+i\e2 + \alpha_i\e3 \mid i\ttk{q},\,j\ttk{p}\}$ \\
\tn & $\{\alpha_i\e1+\alpha_{ij}\e2+ j\e3 +i\e4  \mid i,j\ttk{2}\}$
& $\{i\e1+j\e2 + \alpha_i\e3 \mid i\ttk{p},\,j\ttk{q}\}$ \\
\tn & $\{\alpha_{ij}\e1+\alpha_{ij}\e2+i\e3+j\e4 \mid i,j\ttk{2}\}$
& $\{i\e1+j\e2 \mid i\ttk{p},\,j\ttk{q}\}$ \\
\tn & $\{i\e1+j\e2+k\e3+\alpha_{ij}\e4 \mid i\ttk{p},\,j\ttk{q},\,k\ttk{2}\}$
& $\{\alpha_i\e3+i\e4 \mid i\ttk{2}\}$ \\
\tn & $\{k\e1+i\e2+\alpha_{ij}\e3+j\e4 \mid i\ttk{q},\,j\ttk{2},\,k\ttk{p}\}$
& $\{\alpha_i\e1+i\e3 \mid i\ttk{2}\}$ \\
\tn & $\{i\e1+k\e2+\alpha_{ij}\e3+j\e4 \mid i\ttk{p},\,j\ttk{2},\,k\ttk{q}\}$
& $\{\alpha_i\e2+i\e3 \mid i\ttk{2}\}$ \\
\tn & $\{j\e1+i\e2+k\e3+\alpha_i\e4 \mid i\ttk{q},\,j\ttk{p},\,k\ttk{2}\}$
& $\{\alpha_i\e1+\alpha_i\e3+i\e4 \mid i\ttk{2}\}$ \\
\tn & $\{i\e1+j\e2+k\e3+\alpha_i\e4 \mid i\ttk{p},\,j\ttk{q},\,k\ttk{2}\}$
& $\{\alpha_i\e2+\alpha_i\e3+i\e4 \mid i\ttk{2}\}$ \\
\tn & $\{i\e1+j\e2 + \alpha_{ijk}\e3+ k\e4  \mid i\ttk{p},\,j\ttk{q},\,k\ttk{2}\}$
& $\{i\e3 \mid i\ttk{2}\}$ \\
\tn & $\{i\e1+\alpha_{ij}\e2+k\e3 +j\e4\mid i\ttk{p},\,j,k\ttk{2}\}$
& $\{i\e2+\alpha_i\e3 \mid i\ttk{q}\}$ \\
\tn & $\{k\e1+\alpha_{ij}\e2+i\e3+j\e4 \mid i,j\ttk{2},\,k\ttk{p}\}$
& $\{\alpha_i\e1+i\e2 \mid i\ttk{q}\}$ \\
\tn & $\{i\e1+\alpha_i\e2+j\e3+k\e4 \mid i\ttk{p},\,j,k\ttk{2}\}$
& $\{i\e2+\alpha_1\e3+\beta_i\e4 \mid i\ttk{q}\}$ \\
\tn & $\{j\e1+\alpha_i\e2+k\e3+i\e4 \mid i,k\ttk{2},\,j\ttk{p}\}$
& $\{\alpha_i\e1+i\e2 + \alpha_i\e3 \mid i\ttk{q}\}$ \\
\tn & $\{i\e1+\alpha_{ijk}\e2+j\e3+k\e4 \mid i\ttk{p},\,j,k\ttk{2}\}$
& $\{i\e2 \mid i\ttk{q}\}$ \\
\tn & $\{\alpha_{ij}\e1+i\e2+k\e3+j\e4 \mid i\ttk{q},\,j,k\ttk{2}\}$
& $\{i\e1+\alpha_i\e3 \mid i\ttk{p}\}$ \\
\tn & $\{\alpha_{ij}\e1+k\e2+i\e3+j\e4 \mid i,j\ttk{2},\,k\ttk{q}\}$
& $\{i\e1+\alpha_i\e2 \mid i\ttk{p}\}$ \\
\tn & $\{\alpha_i\e1+i\e2+j\e3+k\e4 \mid i\ttk{q},\,j,k\ttk{2}\}$
& $\{i\e1+\alpha_i\e3+\beta_i\e4 \mid i\ttk{p}\}$ \\
\tn & $\{\alpha_i\e1+j\e2+k\e3+i\e4 \mid i,k\ttk{2},\,j\ttk{q}\}$
& $\{i\e1+\alpha_i\e2 + \alpha_i\e3 \mid i\ttk{p}\}$ \\
\tn & $\{\alpha_{ijk}\e1+i\e2+j\e3+k\e4 \mid i\ttk{q},\,j,k\ttk{2}\}$
& $\{i\e1 \mid i\ttk{p}\}$ \\
\hline
\end{longtable}

\begin{proof}
    Suppose $C$ is a perfect code in $\Cay(G,S)$. Then $G=\so\oplus C$ and so $|\so|\cdot |C|=|G|=4pq  $. 
    Since $G$ is not a cyclic group, $|\so|>2$ by Lemma \ref{ls<so}.
    Therefore, $2<|\so|<4pq  $.

\smallskip
    \textsf{Case 1.} $\so$ is aperiodic.
    \smallskip

    Since $G$ is Haj\'os and $G=\so\oplus C$, in this case $C$ must be periodic, that is, $\lc$ is nontrivial.
    Write $C=\lc\oplus D$ for some subset $D$ of $G$ whose existence is ensured by Lemma \ref{period}.
    Then $|\lc|\cdot |D|=|C|$ and $2\leq |\lc|\leq |C|$.
    If $|\so|$ is a prime, then $|\lc|=|C|$ by Lemma \ref{prime deg} part (c).
    If $|\so|$ is not a prime, then $|\lc|$ can be any divisor of $|C|$.
    Set $$ A=(\so+\lc)/\lc\quad\text{and}\quad B=(D+\lc)/\lc.$$
    Since $C$ is periodic, by Lemma \ref{quotient}, we have
    $$G/\lc=A\oplus B$$ and $B$ is aperiodic in $G/\lc$.
    Since $G/\lc$ is isomorphic to a subgroup of $G$ and subgroups of Haj\'os groups are also Haj\'os groups, $A$ is periodic in $G/\lc$.
    We also know that $A$ contains $\lc$ and $A$ generates $G/\lc$ by Lemma \ref{generate-zero}.
    Furthermore, $|A|=|\so|$ and $|B|=|D|$.
    Note that if $|\lc|=|C|$, then $C=\lc$ as $0\in C$.
    This means we can take $D=\{0\}$ and hence $B=\lc/\lc$ and $G/\lc=A$.
    We use $C$ to get $A$ and then use Lemma \ref{d_i+l_i} to obtain $\so$.
    If $|\lc|<|C|$, then this factorization enables us to make use of previously established theorems to obtain the pair $(\so, C)$ via $(A,B)$.

    All possible values for $\left(|\so|,|C|,|\lc|\right)$ are given in the first column of Table \ref{tap.p.q.2.2}.
    In the second column, we list $C$, $A$, and $\so$ when $|\lc|=|C|$ and we list $L_C$, $A$, $B$, $\so$, and $C$ when $|\lc|<|C|$.
    The third column shows which previous theorem is used to obtain $(A,B)$ when $|\lc|<|C|$.
    The last column indicates which row in Table \ref{t.p.q.2.2}
    the result corresponds to.
    The result in the cases when $|\so|=q$, $2q$, and $4q$ can be derived by swapping $\e1$ and $\e2$ from the cases when $|\so|=p$, $2p$, and $4p$, respectively. The resulting form of $\so$ and $C$ is in rows 2, 14 to 17, and 37 in Table \ref{t.p.q.2.2}. 
    
\begin{longtable}{|c|l|c|c|}
\caption{The case when $G=\Z_{p}\x\Z_q\x\Z_2\x\Z_2 $ and $\so$ is aperiodic}
\label{tap.p.q.2.2}
\\
\flap

& $C=\<{\e2,\e3,\e4} $ & & \\
$(p,4q,4q)$ 
& $A=\{i\e1+\lc \mid i\ttk{p}\} $ 
& & row 1 \\
& $\so=\{i\e1+\alpha_i\e2+\beta_i\e3+\gamma_i\e4 \mid i\ttk{p}\} $ & & \\ \hline


& $\lc=\<{\e3} $ & & \\
& $A=\{j\e1+i\e2 +\alpha_i\e4 +\lc \mid i\ttk{q},\,j\ttk{p}\} $ & & \\
$(pq,4,2 )$ 
& $B=\{\alpha_i\e1 +i\e4 +\lc \mid i\ttk{2}\} $ 
& \ref{p.q.2} \row4 & row 5 \\
& $\so=\{j\e1+i\e2 +\alpha_i\e4 +\alpha_{ij}\e3 \mid i\ttk{q},\,j\ttk{p}\} $ & & \\
& $C=\{\alpha_i\e1 +i\e4 +j\e3 \mid i,j\ttk{2}\} $ & & \\ \hline

& $\lc=\<{\e3} $ & & \\
& $A=\{i\e1+j\e2 +\alpha_i\e4 +\lc \mid i\ttk{p},\,j\ttk{q}\} $ & & \\
$(pq,4,2 )$ 
& $B=\{\alpha_i\e2 +i\e4 +\lc \mid i\ttk{2}\} $ 
& \ref{p.q.2} \row5 & row 6 \\
& $\so=\{i\e1+j\e2 +\alpha_i\e4 +\alpha_{ij}\e3 \mid i\ttk{p},\,j\ttk{q}\} $ & & \\
& $C=\{\alpha_i\e2 +i\e4 +j\e3 \mid i,j\ttk{2}\} $ & & \\ \hline

& $C=\<{\e3,\e4} $ & & \\
$(pq,4,4 )$ 
& $A=\{i\e1+j\e2 + \lc \mid i\ttk{p},\,j\ttk{q}\} $ 
& & row 7 \\
& $\so=\{i\e1+j\e2+\alpha_{ij}\e3+\beta_{ij}\e4 \mid i\ttk{p},\,j\ttk{p}\} $ & & \\ \hline

& $\lc=\<{\e3} $ & & \\
& $A=\{j\e1+\alpha_i\e2 +i\e4 +\lc \mid i\ttk{2},\,j\ttk{p}\} $ & & \\
$(2p,2q,2)$ 
& $B=\{\alpha_i\e1+i\e2 +\lc \mid i\ttk{q}\} $ 
& \ref{p.q.2} \row6 & row 9 \\
& $\so=\{j\e1+\alpha_i\e2 +i\e4 +\alpha_{ij}\e3 \mid i\ttk{2},\,j\ttk{p}\} $ & & \\
& $C=\{\alpha_i\e1+i\e2 + j\e3 \mid i\ttk{q},\,j\ttk{2}\} $ & & \\ \hline

& $\lc=\<{\e3} $ & & \\
& $A=\{i\e1+\alpha_i\e2 +j\e4 +\lc \mid i\ttk{p},\,j\ttk{2}\} $ & & \\
$(2p,2q,2 )$ 
& $B=\{i\e2 + \alpha_i\e4 +\lc \mid i\ttk{q}\} $ 
& \ref{p.q.2} \row7 & row 10 \\
& $\so=\{i\e1+\alpha_i\e2 +j\e4 +\alpha_{ij}\e3 \mid i\ttk{p},\,j\ttk{2}\} $ & & \\
& $C=\{i\e2 + \alpha_i\e4 + j\e3 \mid i\ttk{q},\,j\ttk{2}\} $ & & \\ \hline

& $\lc=\<{\e2} $ & & \\
& $A=\{i\e1 + j\e3+\alpha_i\e4  + \lc \mid i\ttk{p},\,j\ttk{2}\} $ & & \\
$(2p,2q,q )$ 
& $B=\{(\alpha_i\e3+i\e4 +\lc \mid i\ttk{2}\} $ 
& \ref{p.2.2} \row4 & row 11 \\
& $\so=\{i\e1+\alpha_{ij}\e2+ j\e3 +\alpha_i\e4  \mid i\ttk{p},\,j\ttk{2}\} $ & & \\
& $C=\{j\e2+\alpha_i\e3+i\e4 \mid i\ttk{2}\} $ & & \\ \hline

& $C=\<{\e2+\e3} $ & & \\
$(2p,2q,2q)$ 
& $A=\{i\e1 + j\e4 + \lc \mid i\ttk{p},\,j\ttk{2}\} $ 
& & row 12 \\
& $\so=\{i\e1+\alpha_{ij}\e2 + j\e4 + \alpha_{ij}\e3 \mid i\ttk{p},\,j\ttk{2}\} $ & & \\ \hline

& $\lc=\<{\e1} $ & & \\
& $A=\{\alpha_i\e2 + j\e3+i\e4  +\lc \mid i,j\ttk{2}\} $ & & \\
$(4,pq,p )$ 
& $B=\{i\e2 + \alpha_i\e3 + \lc \mid i\ttk{q}\} $ 
& \ref{p.2.2} \row6 & row 19 \\
& $\so=\{\alpha_{ij}\e1+\alpha_i\e2 + j\e3+i\e4  \mid i,j\ttk{2}\} $ & & \\
& $C=\{j\e1+i\e2 + \alpha_i\e3 \mid i\ttk{q},\,j\ttk{p}\} $ & & \\ \hline

& $\lc=\<{\e2} $ & & \\
& $A=\{\alpha_i\e1 + j\e3+i\e4  +\lc \mid i,j\ttk{2}\} $ & & \\
$(4,pq,q )$ 
& $B=\{i\e1 + \alpha_i\e3 + \lc \mid i\ttk{p}\} $ 
& \ref{p.2.2} \row6 & row 20 \\
& $\so=\{\alpha_i\e1 +\alpha_{ij}\e2+ j\e3+i\e4  \mid i,j\ttk{2}\} $ & & \\
& $C=\{i\e1 +j\e2+ \alpha_i\e3 \mid i\ttk{p},\,j\ttk{q}\} $ & & \\ \hline

& $C=\<{\e1+\e2} $ & & \\
$(4,pq,pq)$ 
& $A=\{i\e3+j\e4 + \lc \mid i,j\ttk{2}\} $ 
& & row 21 \\
& $\so=\{\alpha_{ij}\e1+\alpha_{ij}\e2+i\e3+j\e4 \mid i,j\ttk{2}\} $ & & \\ \hline

& $C=\<{\e3} $ & & \\
$(2pq,2,2 )$ 
& $A=\{i\e1+j\e2 + k\e4 + \lc \mid i\ttk{p},\,j\ttk{q},\,k\ttk{2}\} $ 
& & row 27 \\
& $\so=\{i\e1+j\e2 + k\e4 + \alpha_{ijk}\e3 \mid i\ttk{p},\,j\ttk{q},\,k\ttk{2}\} $ & & \\ \hline

& $C=\<{\e2} $ & & \\
$(4p,q,q)$ 
& $A=\{i\e1+j\e3+k\e4 + \lc \mid i\ttk{p},\,j,k\ttk{2}\} $ 
& & row 32 \\
& $\so=\{i\e1+\alpha_{ijk}\e2+j\e3+k\e4 \mid i\ttk{p},\,j,k\ttk{2}\} $ & & \\ \hline

\end{longtable}

\smallskip
    \textsf{Case 2.} $\so$ is periodic.
    \smallskip

    In this case we have $|\ls|\geq 2$ and $\so=D\oplus \ls$ for some subset $D$ of $G$ by Lemma \ref{period}.
    So $|D|\cdot|\ls|=|\so|$ and $1<|\ls|<|\so|$ by Lemma \ref{prime deg} part (b).
    Set $$X=(D+\ls)/\ls\quad\text{and}\quad Y=(C+\ls)/\ls.$$
    Since $\so$ is periodic, by Lemma \ref{quotient}, we have
    $$G/\lc=X\oplus Y$$ and $X$ is aperiodic in $G/\ls$.
    We also know that $X$ contains $\ls$ and $X$ generates $G/\ls$ by Lemma \ref{generate-zero}.
    This factorization enables us to make use of Lemma \ref{d_i+l_i} and previously established theorems to obtain the pair $(\so, C)$ via $(X,Y)$.
    
    All possible values for $\left(|\so|,|C|,|\ls|\right)$ are given in the first column of Table \ref{tp.p.q.2.2}.
    In the second column we list $\ls$, $X$, $Y$, $\so$, and $C$.
    The third column shows which previous theorem is used to obtain $(X,Y)$.
    The last column indicates which row in Table \ref{t.p.q.2.2} the result corresponds to.
    The case for $|\so|=2q$ and $4q$ 
    is derived from the case of $|\so|=2p$ and $4p$, respectively. The resulting form of $\so$ and $C$ is in rows 13, and 33 to 36 in Table \ref{t.p.q.2.2}.

\begin{longtable}{|c|l|c|c|}
\caption{The case when $G=\Z_{p}\x\Z_q\x\Z_2\x\Z_2 $ and $\so$ is periodic}
\label{tp.p.q.2.2}
\\
\flp

& $\ls=\<{\e1} $ & & \\
& $X=\{i\e2+\alpha_1\e3+\beta_i\e4+\ls \mid i\ttk{q}\} $ & & \\
$(pq,4,p )$ 
& $Y=\{i\e3+j\e4+\ls \mid i,j\ttk{2}\} $ 
& \ref{p.2.2} \row1 & row 3 \\
& $\so=\{j\e1+i\e2+\alpha_i\e3+\beta_i\e4 \mid i\ttk{q},\,j\ttk{p}\} $ & & \\
& $C=\{\alpha_{ij}\e1+i\e3+j\e4 \mid i,j\ttk{2}\} $ & & \\ \hline

& $\ls=\<{\e2} $ & & \\
& $X=\{i\e1+\alpha_1\e3+\beta_i\e4+\ls \mid i\ttk{p}\} $ & & \\
$(pq,4,p )$ 
& $Y=\{i\e3+j\e4+\ls \mid i,j\ttk{2}\} $ 
& \ref{p.2.2} \row1 & row 4 \\
& $\so=\{i\e1+j\e2+\alpha_i\e3+\beta_i\e4 \mid i\ttk{p},\,j\ttk{q}\} $ & & \\
& $C=\{\alpha_{ij}\e2+i\e3+j\e4 \mid i,j\ttk{2}\} $ & & \\ \hline

& $\ls=\<{\e3} $ & & \\
& $X=\{i\e1+\alpha_i\e2+\beta_i\e4+\ls \mid i\ttk{p}\} $ & & \\
$(2p,2q,2 )$ 
& $Y=\{i\e2+j\e4+\ls \mid i\ttk{q},\,j\ttk{2}\} $ 
& \ref{p.q.2} \row{1} & row 8 \\
& $\so=\{i\e1+\alpha_i\e2+\beta_i\e4+j\e3 \mid i\ttk{p},\,j\ttk{2}\} $ & & \\
& $C=\{i\e2+j\e4+\alpha_{ij}\e3 \mid i\ttk{q},\,j\ttk{2}\} $ & & \\ \hline
    
$(2p,2q,p )$ 
& Not possible since no $X$ exists
& \ref{p.2.2}  & \\ \hline

& $\ls=\<{\e3} $ & & \\
& $X=\{\alpha_1\e1+\beta_i\e2+i\e4+\ls \mid i\ttk{2}\} $ & & \\
$(4,pq,2 )$ 
& $Y=\{i\e1+j\e2+\ls \mid i\ttk{p},\,j\ttk{q}\} $ 
& \ref{p.q.2} \row{3} & row 18 \\
& $\so=\{\alpha_1\e1+\beta_i\e2+i\e4+j\e3 \mid i,j\ttk{2}\} $ & & \\
& $C= $ & & \\ \hline

& $\ls=\<{\e3} $ & & \\
& $X=\{i\e1+j\e2+\alpha_{ij}\e4+\ls \mid i\ttk{p},\,j\ttk{q}\} $ & & \\
$(2pq,2,2 )$ 
& $Y=\{i\e4+\ls \mid i\ttk{2}\} $ 
& \ref{p.q.2} \row{10} & row 22 \\
& $\so=\{i\e1+j\e2+\alpha_{ij}\e4+k\e3 \mid i\ttk{p},\,j\ttk{q},\,k\ttk{2}\} $ & & \\
& $C=\{i\e4+\alpha_i\e3 \mid i\ttk{2}\} $ & & \\ \hline

& $\ls=\<{\e1} $ & & \\
& $X=\{i\e2+\alpha_{ij}\e3+j\e4+\ls \mid i\ttk{q},\,j\ttk{2}\} $ & & \\
$(2pq,2,p )$ 
& $Y=\{i\e3+\ls \mid i\ttk{2}\} $ 
& \ref{p.2.2} \row5 & row 23 \\
& $\so=\{k\e1+i\e2+\alpha_{ij}\e3+j\e4 \mid i\ttk{q},\,j\ttk{2},\,k\ttk{p}\} $ & & \\
& $C=\{\alpha_i\e1+i\e3 \mid i\ttk{2}\} $ & & \\ \hline

& $\ls=\<{\e2} $ & & \\
& $X=\{i\e1+\alpha_{ij}\e3+j\e4+\ls \mid i\ttk{p},\,j\ttk{2}\} $ & & \\
$(2pq,2,q )$ 
& $Y=\{i\e3+\ls \mid i\ttk{2}\} $ 
& \ref{p.2.2} \row5 & row 24 \\
& $\so=\{i\e1+k\e2+\alpha_{ij}\e3+j\e4 \mid i\ttk{p},\,j\ttk{2},\,k\ttk{q}\} $ & & \\
& $C=\{\alpha_i\e2+i\e3 \mid i\ttk{2}\} $ & & \\ \hline

& $\ls=\<{\e1+\e3} $ & & \\
& $X=\{i\e2+\alpha_i\e4+\ls \mid i\ttk{q}\} $ & & \\
$(2pq,2,2p )$ 
& $Y=\{i\e4+\ls \mid i\ttk{2}\} $ 
& \ref{p.q} \row1 & row 25 \\
& $\so=\{j\e1+i\e2+k\e3+\alpha_i\e4 \mid i\ttk{q},\,j\ttk{p},\,k\ttk{2}\} $ & & \\
& $C=\{\alpha_i\e1+\alpha_i\e3+i\e4 \mid i\ttk{2}\} $ & & \\ \hline

& $\ls=\<{\e2+\e3} $ & & \\
& $X=\{i\e1+\alpha_i\e4+\ls \mid i\ttk{p}\} $ & & \\
$(2pq,2,2q )$ 
& $Y=\{i\e4+\ls \mid i\ttk{2}\} $ 
& \ref{p.q} \row1 & row 26 \\
& $\so=\{i\e1+j\e2+\alpha_i\e4+k\e3 \mid i\ttk{q},\,j\ttk{p},\,k\ttk{2}\}$ & & \\
& $C=\{\alpha_i\e2+\alpha_i\e3+i\e4 \mid i\ttk{2}\} $ & & \\ \hline

$(2pq,2,pq )$ 
& Not possible since no $X$ exists
& \ref{2.2}  & \\ \hline

& $\ls=\<{\e3} $ & & \\
& $X=\{i\e1+\alpha_{ij}\e2+j\e4+\ls \mid i\ttk{p},\,j\ttk{2}\} $ & & \\
$(4p,q,2 )$ 
& $Y=\{i\e2+\ls \mid i\ttk{q}\} $ 
& \ref{p.q.2} \row{11} & row 28 \\
& $\so=\{i\e1+\alpha_{ij}\e2+j\e4+k\e3 \mid i\ttk{p},\,j,k\ttk{2}\} $ & & \\
& $C=\{i\e2+\alpha_i\e3 \mid i\ttk{q}\} $ & & \\ \hline

& $\ls=\<{\e1} $ & & \\
& $X=\{\alpha_{ij}\e2+i\e3+j\e4+\ls \mid i,j\ttk{2}\} $ & & \\
$(4p,q,p )$ 
& $Y=\{i\e2+\ls \mid i\ttk{q}\} $ 
& \ref{p.2.2} \row7 & row 29 \\
& $\so=\{k\e1+\alpha_{ij}\e2+i\e3+j\e4 \mid i,j\ttk{2},\,k\ttk{p}\} $ & & \\
& $C=\{\alpha_i\e1+i\e2 \mid i\ttk{q}\} $ & & \\ \hline

& $\ls=\<{\e3,\e4} $ & & \\
& $X=\{i\e1+\alpha_i\e2+\ls \mid i\ttk{p}\} $ & & \\
$(4p,q,4 )$ 
& $Y=\{i\e2+\ls \mid i\ttk{q}\} $ 
& \ref{p.q} \row1 & row 30 \\
& $\so=\{i\e1+\alpha_i\e2+j\e3+k\e4+ \mid i\ttk{p},\,j,k\ttk{2}\} $ & & \\
& $C=\{i\e2+\alpha_1\e3+\beta_i\e4 \mid i\ttk{q}\} $ & & \\ \hline

& $\ls=\<{\e1+\e3} $ & & \\
& $X=\{\alpha_i\e2+i\e4+\ls \mid i\ttk{2}\} $ & & \\
$(4p,q,2p )$ 
& $Y=\{i\e2+\ls \mid i\ttk{q}\} $ 
& \ref{p.q} \row2 & row 31 \\
& $\so=\{j\e1+\alpha_i\e2+i\e4+k\e3 \mid i,k\ttk{2},\,j\ttk{p}\} $ & & \\
& $C=\{\alpha_i\e1+i\e2 + \alpha_i\e3 \mid i\ttk{q}\} $ & & \\ \hline

\end{longtable}
\end{proof}

\begin{remark}
    In Theorem \ref{p,q,2,2}, we can change $\e3$ and $\e4$ to any two generators of $\{0\}\x\{0\}\x\Z_2\x\Z_2$.
\end{remark} 
\subsection{Subgroups of $\Z_p \times \Z_{q^2} \times \Z_q$}

\begin{theorem}\label{4,2}\label{p^2.p}
Let $G=\Z_p^2\x\Z_p$ for $p=2$ or $3$, and let $S$ be a proper subset of $G\setminus \{0\}$ that generates $G$. Let $0\in C\subseteq G$. Then $C$ is a perfect code in $\Cay(G,S)$ if and only if $(\so,C)$ is one of the pairs in Table \ref{t.p^2.p} for some $\alpha_i$'s and $\alpha_{ij}$'s in $\Z$ where $\alpha_0=\alpha_{00}=0$.
\end{theorem}

\rc\begin{longtable}{|c|l|l|}
\caption{Perfect codes in Cayley graphs of $\Z_{p^2}\x\Z_p$}
\label{t.4.2}\label{t.p^2.p}
\\
\fl
\tn & $\{(\alpha_i,i) \mid i\ttk{3}\}$
& $\{(i,0) \mid i\ttk{9}\}$ \\
\tn & $\{(i+3\alpha_i,\beta_i) \mid i\ttk{3}\}$
& $\{(3i,j) \mid i,j\ttk{3}\}$ \\
\tn & $\{(\alpha_i+3j,i) \mid i,j\ttk{3}\}$
& $\{(i+\alpha_i j,0) \mid i\ttk{3}\}$ \\
\tn & $\{(\alpha_i+3j,i+j)\mid i,j\ttk{3}\}$
& $\{(i+3\alpha_i, \alpha_i) \mid i\ttk{3}\}$ \\
\tn & $\{(i+p\alpha_i,j) \mid i,j\ttk{p}\}$
& $\{(pi,\alpha_i) \mid i\ttk{p}\}$ \\
\tn & $\{(i+p\alpha_{ij},j) \mid i,j\ttk{p}\}$
& $\{(pi,0) \mid i\ttk{p}\}$ \\
\tn & $\{(i+p\alpha_{ij},j+\alpha_{ij}) \mid i,j\ttk{p}\}$
& $\{(pi,i) \mid i\ttk{p}\}$ \\
\tn & $\{(i,\alpha_i) \mid i\ttk{p^2}\}$
& $\{(0,i) \mid i\ttk{p}\}$ \\
\hline
\end{longtable}

\begin{proof}
    \setcounter{cn}{1}\renewcommand{\Table}{\ref{t.p^2.p}}
    Suppose $C$ is a perfect code in $\Cay(G,S)$. From lemma \ref{ls<so}, $|\so|\geq 3$.
    If $|\so|=p$, then $p=3$.
    From Lemma \ref{prime deg}, $\so$ is aperiodic and $C$ is a subgroup of order $9$, so $C=\lc=\<{(1,0)}$ or $\<{(3,0),(0,1)}$.
    From the factorization of $G/\lc$ given in Lemma \ref{quotient}, $G/\lc=(\so+\lc)/\lc$. 
    If $\lc=\<{(1,0)}$, then $(\so+\lc)/\lc=\{(0,i)+\lc \mid i\ttk{3}\}$, and hence $\so=\{(\alpha_i,i) \mid i\ttk{3}\}$, for some $\alpha_i$'s in $\Z$. \R
    If $\lc=\<{(3,0),(0,1)}$, then $(\so+\lc)/\lc=\{(i,0)+\lc \mid i\ttk{3}\}$, and hence $\so=\{(i,0)+\alpha_i(3,0)+\beta_i(0,1)\mid i\ttk{3}\}=\{(i+3\alpha_i,\beta_i) \mid i\ttk{3}\}$, for some $\alpha_i$'s and $\beta_i$'s in $\Z$. \R

    If $|\so|=p^2$ and $\so$ is periodic, then $|\ls|=p$ and $\ls=\<{(p,0)}$, $\<{(p,1)}$, or $\<{(0,1)}$.
    Let $\so=D\oplus\ls$ for some subset $D$. From the factorization of $G/\ls$ given in Lemma \ref{quotient}, we have $G/\ls=(D+\ls)/\ls\oplus (C+\ls)/\ls$.
    If $\ls=\<{(p,0)}$ or $\<{(p,1)}$, then $G/\ls\cong \Z_p\times\Z_p$.
    If $p=2$, then from Theorem \ref{2.2} we know this case is not possible. So $p=3$. 
    From the case of $|\so|=p$ in Theorem \ref{p,p}, we get $(D+\ls)/\ls=\{(\alpha_i,i)+\ls \mid i\ttk{3}\}$ and $(C+\ls)/\ls=\{(i,0)+\ls \mid i\ttk{3}\}$ where $\alpha_i\in \Z$.
    If $\ls=\<{(3,0)}$, then $\so=D\oplus\ls=\{(\alpha_i+3j,i) \mid i,j\ttk{3}\}$ and $C=\{(i+\alpha_i j,0) \mid i\ttk{3}\}$ where $\alpha_i\in \Z$. \R
    If $\ls=\<{(3,1)}$, then $\so=D\oplus\ls=\{(\alpha_i+3j,i+j)\mid i,j\ttk{3}\}$ and $C=\{(i+3\alpha_i, \alpha_i) \mid i\ttk{3}\}$ where $\alpha_i\in \Z$. \R
    If $\ls=\<{(0,1)}$, then $G/\ls\cong \Z_{p^2}$.
    From Theorem \ref{p^2}, we get $(D+\ls)/\ls=\{(i+p\alpha_i,0)+\ls \mid i\ttk{p}\}$ and $(C+\ls)/\ls=\{(pi,0)+\ls \mid i\ttk{p}\}$.
    Therefore, $\so=D\oplus\ls=\{(i+p\alpha_i,j) \mid i,j\ttk{p}\}$ and $C=\{(pi,\alpha_i) \mid i\ttk{p}\}$ where $\alpha_i\in \Z$. \R

    If $|\so|=p^2$ and $\so$ is aperiodic, then $|C|=p$ and periodic, so $C=\lc=\<{(p,0)}$, $\<{(p,1)}$, or $\<{(0,1)}$.
    From the factorization of $G/\lc$ given in Lemma \ref{quotient}, we have $G/\lc=(\so+\lc)/\lc\oplus (D+\lc)/\lc$.
    If $C=\lc=\<{(p,0)}$, from the factorization of $G/\lc$, $(\so+\lc)/\lc=G/\lc=\{(i,j) + \lc \mid i,j\ttk{p}\}$, and hence $\so=\{(i+p\alpha_{ij},j) \mid i,j\ttk{p}\}$, $\alpha_{ij}\in \Z$. \R
    If $C=\lc=\<{(p,1)}$, from the factorization of $G/\lc$, $(\so+\lc)/\lc=G/\lc=\{(i,j) + \lc \mid i,j\ttk{p}\}$, and hence $\so=\{(i+p\alpha_{ij},j+\alpha_{ij}) \mid i,j\ttk{p}\}$, $\alpha_{ij}\in \Z$. \R
    If $C=\lc=\<{(0,1)}$, from the factorization of $G/\lc$, $(\so+\lc)/\lc=G/\lc=\{(i,0) + \lc \mid i\ttk{p^2}\}$, and hence $\so=\{(i,\alpha_i) \mid i\ttk{p^2}\}$, $\alpha_{i}\in \Z$. \R
\end{proof}

\begin{theorem}\label{p,4}
Let $G=\Z_p\x\Z_4$ for an odd primes $p$, and let $S$ be a proper subset of $G\setminus \{0\}$ that generates $G$. Let $0\in C\subseteq G$. Then $C$ is a perfect code in $\Cay(G,S)$ if and only if $(\so,C)$ is one of the pairs in Table \ref{t.p.4} for some $\alpha_i$'s and $\alpha_{ij}$'s in $\Z$ where $\alpha_0=\alpha_{00}=0$.
\end{theorem}

\rc\begin{longtable}{|c|l|l|}
\caption{Perfect codes in Cayley graphs of $\Z_{p}\x\Z_4$}
\label{t.p.4}
\\
\fl
\tn & $\{(\alpha_i,i+2\alpha_i) \mid i\ttk{2}\}$
& $\{(i,2j) \mid i\ttk{p},\,j\ttk{2}\}$ \\
\tn & $\{(i,\alpha_i) \mid i\ttk{p}\}$
& $\{(0,i) \mid i\ttk{4}\}$ \\
\tn & $\{(i,\alpha_i+2j) \mid i\ttk{p},\,j\ttk{2}\}$
& $\{(0,i+2\alpha_i) \mid i\ttk{2}\}$ \\
\tn & $\{(j,i+2\alpha_i) \mid i\ttk{2},\,j\ttk{p}\}$
& $\{(\alpha_i,2i) \mid i\ttk{2}\}$ \\
\tn & $\{(i,j+2\alpha_{ij}) \mid i\ttk{p},\,j\ttk{2}\}$
& $\{(0,2i) \mid i\ttk{2}\}$ \\
\tn & $\{(\alpha_i,i+2j) \mid i,j\ttk{2}\}$
& $\{(i,2\alpha_i) \mid i\ttk{p}\}$ \\
\tn & $\{(\alpha_i,i) \mid i\ttk{4}\}$
& $\{(i,0) \mid i\ttk{p}\}$ \\
\hline
\end{longtable}

\begin{proof}
    \setcounter{cn}{1}\renewcommand{\Table}{\ref{t.p.4}}
    Suppose $C$ is a perfect code in $\Cay(G,S)$.

    If $|\so|=2$, then from Lemma \ref{prime deg}, $\so$ is aperiodic and $C$ is a subgroup of order $2p$, so $C=\lc=\<{(1,2)}$.
    From the factorization of $G/\lc$ given in Lemma \ref{quotient}, $(\so+\lc)/\lc=G/\lc=\{(0,i)+\lc \mid i\ttk{2}\}$, and hence $\so=\{(\alpha_i,i+2\alpha_i) \mid i\ttk{2}\}$, where $\alpha_i\in\Z$. \R

    If $|\so|=p$, then from Lemma \ref{prime deg}, $\so$ is aperiodic and $C$ is a subgroup of order $4$, so $C=\lc=\<{(0,1)}$.
    From the factorization of $G/\lc$ given in Lemma \ref{quotient}, $(\so+\lc)/\lc=G/\lc=\{(i,0)+\lc \mid i\ttk{p}\}$, and hence $\so=\{(i,\alpha_i) \mid i\ttk{p}\}$, where $\alpha_i\in\Z$. \R

    If $|\so|=2p$ and $\so$ is periodic, then $|\ls|=2$ or $p$.
    Let $\so=D\oplus\ls$ for some subset $D$. From the factorization of $G/\ls$ given in Lemma \ref{quotient}, we have $G/\ls=(D+\ls)/\ls\oplus (C+\ls)/\ls$.
    If $|\ls|=2$, then $\ls=\<{(0,2)}$.
    Note that $(D+\ls)/\ls$ is a aperiodic set of order $p$ and $G/\ls\cong \Z_p\times\Z_2$.
    From the case of $|\so|=p$ in Theorem \ref{p.q} \row1, we get $(D+\ls)/\ls=\{(i,\alpha_i)+\ls \mid i\ttk{p}\}$ and $(C+\ls)/\ls=\{(0,i)+\ls \mid i\ttk{2}\}$.
    Therefore, $\so=D\oplus\ls=\{(i,\alpha_i+2j)+\ls \mid i\ttk{p},\,j\ttk{2}\}$ and $C=\{(0,i+2\alpha_i) \mid i\ttk{2}\}$ where $\alpha_i\in \Z$. \R

    If $|\so|=2p$, $\so$ is periodic, and $|\ls|=p$, then $\ls=\<{(1,0)}$.
    Note that $(D+\ls)/\ls$ is an aperiodic set of cardinality $2$ and $G/\ls\cong \Z_4$.
    From Theorem \ref{p^2}, we get $(D+\ls)/\ls=\{(0,i+2\alpha_i)+\ls \mid i\ttk{2}\}$ and $(C+\ls)/\ls=\{(0,2i)+\ls \mid i\ttk{2}\}$.
    Therefore, $\so=D\oplus\ls=\{(j,i+2\alpha_i) \mid i\ttk{2},\,j\ttk{p}\}$ and $C=\{(\alpha_i,2i) \mid i\ttk{2}\}$ where $\alpha_i\in \Z$. \R

    If $|\so|=2p$ and $\so$ is aperiodic, then $|C|=2$ and $C$ is periodic,
    which means $C=\lc=\<{(0,2)}$.
    From the factorization of $G/\lc$, $(\so+\lc)/\lc=G/\lc=\{(i,j) + \lc \mid i\ttk{p},\,j\ttk{2}\}$, and hence $\so=\{(i,j+2\alpha_{ij}) \mid i\ttk{p},\,j\ttk{2}\}$, $\alpha_{ij}\in \Z$. \R

    If $|\so|=4$ and $\so$ is periodic, then $|\ls|=2$ and hence $\ls=\<{(0,2)}$.
    Let $\so=D\oplus\ls$ for some subset $D$. From the factorization of $G/\ls$ given in Lemma \ref{quotient}, we have $G/\ls=(D+\ls)/\ls\oplus (C+\ls)/\ls$.
    Note that $(D+\ls)/\ls$ is an aperiodic set of cardinality $2$ and $G/\ls\cong \Z_p\times\Z_2$.
    From the case of $|\so|=2$ in Theorem \ref{p.q} \row2, we get $(D+\ls)/\ls=\{(\alpha_i,i)+\ls \mid i\ttk{2}\}$ and $(C+\ls)/\ls=\{(i,0)+\ls \mid i\ttk{p}\}$.
    Therefore, $\so=D\oplus\ls=\{(\alpha_i,i+2j) \mid i,j\ttk{2}\}$ and $C=\{(i,2\alpha_i) \mid i\ttk{p}\}$ where $\alpha_i\in \Z$. \R

    If $|\so|=4$ and $\so$ is aperiodic, then $|C|=p$ and periodic, which means $C=\lc=\<{(1,0)}$.
    From the factorization of $G/\lc$, $(\so+\lc)/\lc=G/\lc=\{(0,i) + \lc \mid i\ttk{4}\}$, and hence $\so=\{(\alpha_i,i) \mid i\ttk{4}\}$, $\alpha_{i}\in \Z$. \R
\end{proof}

\begin{theorem}\label{p,4,2}
Let $G=\Z_p\x\Z_4\x\Z_2$ for an odd primes $p$, and let $S$ be a proper subset of $G\setminus \{0\}$ that generates $G$. Let $0\in C\subseteq G$. Then $C$ is a perfect code in $\Cay(G,S)$ if and only if $(\so,C)$ is one of the pairs in Table \ref{t.p.4.2} for some $\alpha_i$'s and $\alpha_{ij}$'s in $\Z$ where $\alpha_0=\alpha_{00}=0$.
\end{theorem}

\rc\begin{longtable}{|c|l|l|}
\caption{Perfect codes in Cayley graphs of $\Z_{p}\x\Z_4\x\Z_2$}
\label{t.p.4.2}
\\
\fl
\tn & $\{(i,\alpha_i,\beta_i) \mid i\ttk{p}\}$
& $\{(0,i,j) \mid i\ttk{4},\,j\ttk{2}\}$ \\
\tn & $\{(i,\alpha_i+2j,\beta_i) \mid i\ttk{p},\,j\ttk{2}\}$
& $\{(0,i+2\alpha_{ij},j) \mid i,j\ttk{2}\}$ \\
\tn & $\{(i,\alpha_i+2j,\beta_i+j) \mid i\ttk{p},\,j\ttk{2}\}$
& $\{(0,i+2\alpha_{ij},j+\alpha_{ij}) \mid i,j\ttk{2}\}$ \\
\tn & $\{(i,\alpha_i,j) \mid i\ttk{p},\,j\ttk{2}\}$
& $\{(0,i,\alpha_i) \mid i\ttk{4}\}$ \\
\tn & $\{(i,j+2\alpha_{ij},\alpha_i) \mid i\ttk{p},\,j\ttk{2}\}$
& $\{(0,\alpha_i+2j,i) \mid i,j\ttk{2}\}$ \\
\tn & $\{(i,j+2\alpha_{ij},\alpha_i+\alpha_{ij}) \mid i\ttk{p},\,j\ttk{2}\}$
& $\{(0,\alpha_i+2j,i+j) \mid i,j\ttk{2}\}$ \\
\tn & $\{(i,\alpha_i+2j,\alpha_{ij}) \mid i\ttk{p},\,j\ttk{2}\}$
& $\{(0,i+2\alpha_i,j) \mid i,j\ttk{2}\}$ \\
\tn & $\{(j,i+2\alpha_i,\alpha_{ij}) \mid i\ttk{2},\,j\ttk{p}\}$
& $\{(\alpha_i,2i,j)  \mid i,j\ttk{2}\}$ \\
\tn & $\{(i,\alpha_{ij},j) \mid i\ttk{p},\,j\ttk{2}\}$
& $\{(0,i,0) \mid i\ttk{4}\}$ \\
\tn & $\{(i,\alpha_{ij},j+\alpha_{ij}) \mid i\ttk{p},\,j\ttk{2}\}$
& $\{(0,i,i) \mid i\ttk{4}\}$ \\
\tn & $\{(i,j+2\alpha_{ij},j+\beta_{ij}) \mid i\ttk{p},\,j\ttk{2}\}$
& $\{(0,2i,j) \mid i,j\ttk{2}\}$ \\
\tn & $\{(\alpha_i,i+2\alpha_i,j) \mid i,j\ttk{2}\}$
& $\{(i,2j,\alpha_{ij}) \mid i\ttk{p},\,j\ttk{2}\}$ \\
\tn & $\{(\alpha_i,j+2\alpha_{ij},i) \mid i,j\ttk{2}\}$
& $\{(i,\alpha_i+2j,0) \mid i\ttk{p},\,j\ttk{2}\}$ \\
\tn & $\{(\alpha_i,j+2\alpha_{ij},i+\alpha_{ij}) \mid i,j\ttk{2}\}$
& $\{(i,\alpha_i+2j,j) \mid i\ttk{p},\,j\ttk{2}\}$ \\
\tn & $\{(\alpha_i,i+2j,\alpha_{ij}) \mid i,j\ttk{2}\}$
& $\{(i,2\alpha_i,j) \mid i\ttk{p},\,j\ttk{2}\}$ \\
\tn & $\{(\alpha_{ij},i+2\alpha_i,j) \mid i,j\ttk{2}\}$
& $\{(j,2i,\alpha_i) \mid i\ttk{2},\,j\ttk{p}\}$ \\
\tn & $\{(\alpha_{ij},i+2\alpha_{ij},j) \mid i,j\ttk{2}\}$
& $\{(i,2j,0) \mid i\ttk{p},\,j\ttk{2}\}$ \\
\tn & $\{(\alpha_{ij},i+2\alpha_{ij},j+\alpha_{ij}) \mid i,j\ttk{2}\}$
& $\{(i,2j,j) \mid i\ttk{p},\,j\ttk{2}\}$ \\
\tn & $\{(\alpha_{i},i,\alpha_i) \mid i\ttk{4}\}$
& $\{(i,0,j) \mid i\ttk{p},\,j\ttk{2}\}$ \\
\tn & $\{(i,\alpha_{ij}+2k,j) \mid i\ttk{p},\,j,k\ttk{2}\}$
& $\{(0,i+2\alpha_i,0) \mid i\ttk{2}\}$ \\
\tn & $\{(i,\alpha_{ij}+2k,j+k) \mid i\ttk{p},\,j,k\ttk{2}\}$
& $\{(0,i+2\alpha_i,\alpha_i) \mid i\ttk{2}\}$ \\
\tn & $\{(i,j+2\alpha_{ij},k) \mid i\ttk{p},\,j,k\ttk{2}\}$
& $\{(0,2i,\alpha_i) \mid i\ttk{2}\}$ \\
\tn & $\{(k,i+2\alpha_{ij},j) \mid i,j\ttk{2},\,k\ttk{p}\}$
& $\{(\alpha_i,2i,0) \mid i\ttk{2}\}$ \\
\tn & $\{(k,i+2\alpha_{ij},j+\alpha_{ij}) \mid i,j\ttk{2},\,k\ttk{p}\}$
& $\{(\alpha_i,2i,i) \mid i\ttk{2}\}$ \\
\tn & $\{(j,i,\alpha_{i}) \mid i\ttk{4},\,j\ttk{p}\}$
& $\{(\alpha_i,0,i) \mid i\ttk{2}\}$ \\
\tn & $\{(i,j,\alpha_i) \mid i\ttk{p},\,j\ttk{4}\}$
& $\{(0,\alpha_i,i) \mid i\ttk{2}\}$ \\
\tn & $\{(i,j,\alpha_i+j) \mid i\ttk{p},\,j\ttk{4}\}$
& $\{(0,\alpha_i,i+\alpha_i) \mid i\ttk{2}\}$ \\
\tn & $\{(i,\alpha_i+2j,k) \mid i\ttk{p},\,j,k\ttk{2}\}$
& $\{(0,i+2\alpha_i,\beta_i) \mid i\ttk{2}\}$ \\
\tn & $\{(j,i+2\alpha_i,k) \mid i,k\ttk{2},\,j\ttk{p}\}$
& $\{(\alpha_i,2i,\alpha_i) \mid i\ttk{2}\}$ \\
\tn & $\{(i,j+2\alpha_{ijk},k) \mid i\ttk{p},\,j,k\ttk{2}\}$
& $\{(0,2i,0) \mid i\ttk{2}\}$ \\
\tn & $\{(i,j+2\alpha_{ijk},k+\alpha_{ijk}) \mid i\ttk{p},\,j,k\ttk{2}\}$
& $\{(0,2i,i) \mid i\ttk{2}\}$ \\
\tn & $\{(i,j,\alpha_{ij}) \mid i\ttk{p},\,j\ttk{4}\}$
& $\{(0,0,i) \mid i\ttk{2}\}$ \\
\tn & $\{(\alpha_{ij},i+2k,j) \mid i,j,k\ttk{2}\}$
& $\{(i,2\alpha_i,0) \mid i\ttk{p}\}$ \\
\tn & $\{(\alpha_{ij},i+2k,j+k) \mid i,j,k\ttk{2}\}$
& $\{(i,2\alpha_i,\alpha_i) \mid i\ttk{p}\}$ \\
\tn & $\{(\alpha_{i},i,j) \mid i\ttk{4},\,j\ttk{2}\}$
& $\{(i,0,\alpha_i) \mid i\ttk{p}\}$ \\
\tn & $\{(\alpha_i,j,i)\mid i\ttk{2},\,j\ttk{4}\}$
& $\{(i,\alpha_i,0) \mid i\ttk{p}\}$ \\
\tn & $\{(\alpha_i,j,i+j)\mid i\ttk{2},\,j\ttk{4}\}$
& $\{(i,\alpha_i,\alpha_i) \mid i\ttk{p}\}$ \\
\tn & $\{(\alpha_i,i+2j,k) \mid i,j,k\ttk{2}\}$
& $\{(i,2\alpha_i,\beta_i) \mid i\ttk{p}\}$ \\
\tn & $\{(\alpha_{ij},i,j) \mid i\ttk{4},\,j\ttk{2}\}$
& $\{(i,0,0) \mid i\ttk{p}\}$ \\
\hline
\end{longtable}

\begin{proof}
    Suppose $C$ is a perfect code in $\Cay(G,S)$. Then $G=\so\oplus C$ and so $|\so|\cdot |C|=|G|=8p$. 
    Since $G$ is not a cyclic group, $|\so|>2$ by Lemma \ref{ls<so}.
    Therefore, $2<|\so|<8p$.

\smallskip
    \textsf{Case 1.} $\so$ is aperiodic.
    \smallskip

    Since $G$ is Haj\'os and $G=\so\oplus C$, in this case $C$ must be periodic, that is, $\lc$ is nontrivial.
    Write $C=\lc\oplus D$ for some subset $D$ of $G$ whose existence is ensured by Lemma \ref{period}.
    Then $|\lc|\cdot |D|=|C|$ and $2\leq |\lc|\leq |C|$.
    If $|\so|$ is a prime, then $|\lc|=|C|$ by Lemma \ref{prime deg} part (c).
    If $|\so|$ is not a prime, then $|\lc|$ could be any divisor of $|C|$.
    Set $$ A=(\so+\lc)/\lc\quad\text{and}\quad B=(D+\lc)/\lc.$$
    Since $C$ is periodic, by Lemma \ref{quotient}, we have
    $$G/\lc=A\oplus B$$ and $B$ is aperiodic in $G/\lc$.
    Since $G/\lc$ is isomorphic to a subgroup of $G$ and subgroups of Haj\'os groups are also Haj\'os groups, $A$ is periodic in $G/\lc$.
    We also know that $A$ contains $\lc$ and $A$ generates $G/\lc$ by Lemma \ref{generate-zero}.
    Furthermore, $|A|=|\so|$ and $|B|=|D|$.
    Note that if $|\lc|=|C|$, then $C=\lc$ as $0\in C$.
    This means we can take $D=\{0\}$ and hence $B=\lc/\lc$ and $G/\lc=A$.
    We use $C$ to get $A$ and then use Lemma \ref{d_i+l_i} to obtain $\so$.
    If $|\lc|<|C|$, then this factorization enables us to make use of previously established theorems to obtain the pair $(\so, C)$ via $(A,B)$.
    Note that for some triples $\left(|\so|,|C|,|\lc|\right)$, there are more than one possible $\lc$, see Observation \ref{obs}.

    The columns of Table \ref{tap,p,4,2} represents all possible values of $\left(|\so|,|C|,|\lc|\right)$; $C$, $A$, and $\so$ when $|\lc|=|C|$ or $L_C$, $A$, $B$, $\so$, and $C$ when $|\lc|<|C|$; which previous theorem is used to obtain $(A,B)$ when $|\lc|<|C|$; and which row in Table \ref{t.p.4.2} the result corresponds to.
    
\begin{longtable}{|c|l|c|c|}
\caption{The case when $G=\Z_{p}\x\Z_4\x\Z_2 $ and $\so$ is aperiodic}
\label{tap,p,4,2}
\\
\flap

& $C=\<{(0,1,0),(0,0,1)} $ & & \\
$(p,8,8 )$ 
& $A=\{(i,0,0)+\lc \mid i\ttk{p}\}$
& & row 1 \\
& $\so=\{(i,\alpha_i,\beta_i) \mid i\ttk{p}\} $ & & \\ \hline

& $\lc=\<{(0,2,0)} $ & & \\
& $A=\{(i,0,0) +\alpha_i\e3 +j\e2 + \lc \mid i\ttk{p},\,j\ttk{2}\} $ & & \\
$(2p,4,2)$ 
& $B=\{i\e3+\alpha_i\e2 +\lc \mid i\ttk{2}\} $ 
& \ref{p.2.2} \row4 & row 5 \\
& $\so=\{(i,j+2\alpha_{ij},\alpha_i) \mid i\ttk{p},\,j\ttk{2}\} $ & & \\
& $C=\{(0,2j+\alpha_i,i) \mid i,j\ttk{2}\} $ & & \\ \hline
    
& $\lc=\<{(0,2,1)} $ & & \\
& $A=\{(i,0,0) +\alpha_i\e3 +j\e2 + \lc \mid i\ttk{p},\,j\ttk{2}\} $ & & \\
$(2p,4,2 )$ 
& $B=\{i\e3+\alpha_i\e2 +\lc \mid i\ttk{2}\} $ 
& \ref{p.2.2} \row4 & row 6 \\
& $\so=\{(i,j+2\alpha_{ij},\alpha_i+\alpha_{ij}) \mid i\ttk{p},\,j\ttk{2}\} $ & & \\
& $C=\{(0,\alpha_i+2j,j+1) \mid i,j\ttk{2}\} $ & & \\ \hline

& $\lc=\<{(0,0,1)} $ & & \\
& $A=\{(j,i+2\alpha_i,0) +\lc \mid i\ttk{2},\,j\ttk{p}\} $ & & \\
$(2p,4,2 )$ 
& $B=\{(\alpha_i,2i,0) +\lc \mid i\ttk{2}\} $ 
& \ref{p,4} \row3 & row 7 \\
& $\so=\{(j,i+2\alpha_i,\alpha_{ij}) \mid i\ttk{2},\,j\ttk{p}\} $ & & \\
& $C=\{(\alpha_i,2i,j)  \mid i,j\ttk{2}\} $ & & \\ 
\hhline{~---}
& $\lc=\<{(0,0,1)} $ & & \\
& $A=\{(j,i+2\alpha_i,0) +\lc \mid i\ttk{2},\,j\ttk{p}\} $ & & \\
$(2p,4,2 )$ 
& $B=\{(\alpha_i,2i,0) +\lc \mid i\ttk{2}\} $ 
& \ref{p,4} \row4 & row 8 \\
& $\so=\{(j,i+2\alpha_i,\alpha_{ij}) \mid i\ttk{2},\,j\ttk{p}\} $ & & \\
& $C=\{(\alpha_i,2i,j)  \mid i,j\ttk{2}\} $ & &  \\ \hline

\multirow{1}{*}{ }
& $C=\<{(0,1,0)} $ & & \\
$(2p,4,4 )$ 
& $A=\{(i,0,j) + \lc \mid i\ttk{p},\,j\ttk{2}\} $ 
& & row 9 \\
& $\so=\{(i,\alpha_{ij},j) \mid i\ttk{p},\,j\ttk{2}\} $ & & \\ \hline

& $C=\<{(0,1,1)} $ & & \\
$(2p,4,4 )$ 
& $A=\{(i,0,j) + \lc \mid i\ttk{p},\,j\ttk{2}\} $ 
& & row 10 \\
& $\so=\{(i,\alpha_{ij},j+\alpha_{ij}) \mid i\ttk{p},\,j\ttk{2}\} $ & & \\ \hline

& $C=\<{(0,2,0),(0,0,1)} $ & & \\
$(2p,4,4 )$ 
& $A=\{(i,j,0) + \lc \mid i\ttk{p},\,j\ttk{2}\} $ 
& & row 11 \\
& $\so=\{(i,j+2\alpha_{ij},j+\beta_{ij}) \mid i\ttk{p},\,j\ttk{2}\} $ & & \\ \hline

& $\lc=\<{(0,2,0)} $ & & \\
& $A=\{(\alpha_i,j,i) +i\e3 +j\e2 + \lc \mid i,j\ttk{2}\} $ & & \\
$(4,2p,2 )$ 
& $B=\{(i,0,0)+\alpha_i\e2 +\lc \mid i\ttk{p}\} $ 
& \ref{p.2.2} \row6 & row 13 \\
& $\so=\{(\alpha_i,j+2\alpha_{ij},i) \mid i,j\ttk{2}\} $ & & \\
& $C=\{(i,\alpha_i+2j,0) \mid i\ttk{p},\,j\ttk{2}\} $ & & \\ \hline

& $\lc=\<{(0,2,1)} $ & & \\
& $A=\{(\alpha_i,0,0) +i\e3 +j\e2 + \lc \mid i,j\ttk{2}\} $ & & \\
$(4,2p,2 )$ 
& $B=\{(i,0,0)+\alpha_i\e2 +\lc \mid i\ttk{p}\} $ 
& \ref{p.2.2} \row6 & row 14 \\
& $\so=\{(\alpha_i,j+2\alpha_{ij},i+\alpha_{ij}) \mid i,j\ttk{2}\} $ & & \\
& $C=\{(i,\alpha_i+2j,j) \mid i\ttk{p},\,j\ttk{2}\} $ & & \\ \hline

& $\lc=\<{(0,0,1)} $ & & \\
& $A=\{(\alpha_i,i+2j,0) + \lc \mid i,j\ttk{2}\} $ & & \\
$(4,2p,2 )$ 
& $B=\{(i,2\alpha_i,0)+\lc \mid i\ttk{p}\} $ 
& \ref{p,4} \row{6} & row 15 \\
& $\so=\{(\alpha_i,i+2j,\alpha_{ij}) \mid i,j\ttk{2}\} $ & & \\
& $C=\{(i,2\alpha_i,j) \mid i\ttk{p},\,j\ttk{2}\} $ & & \\ \hline

& $\lc=\<{(1,0,0)} $ & & \\
& $A=\{(0,i+2\alpha_i,j) + \lc \mid i,j\ttk{2}\} $ & & \\
$(4,2p,p )$ 
& $B=\{(0,2i,\alpha_i)+\lc \mid i\ttk{2}\} $ 
& \ref{p,4} \row{6} & row 16 \\
& $\so=\{(\alpha_{ij},i+2\alpha_i,j) \mid i,j\ttk{2}\} $ & & \\
& $C=\{(j,2i,\alpha_i) \mid i\ttk{2},\,j\ttk{p}\} $ & & \\ \hline

& $C=\<{(1,2,0)} $ & & \\
$(4,2p,2p )$ 
& $A=\{(0,i,j) + \lc \mid i,j\ttk{2}\} $ 
& & row 17 \\
& $\so=\{(\alpha_{ij},i+2\alpha_{ij},j) \mid i,j\ttk{2}\} $ & & \\ \hline

& $C=\<{(1,2,1)} $ & & \\
$(4,2p,2p )$ 
& $A=\{(0,i,j) + \lc \mid i,j\ttk{2}\} $ 
& & row 18 \\
& $\so=\{(\alpha_{ij},i+2\alpha_{ij},j+\alpha_{ij}) \mid i,j\ttk{2}\} $ & & \\ \hline

& $C=\<{(1,0,1)} $ & & \\
$(4,2p,2p )$ 
& $A=\{(0,i,0) + \lc \mid i\ttk{4}\} $ 
& & row 19 \\
& $\so=\{(\alpha_{i},i,\alpha_i) \mid i\ttk{4}\} $ & & \\ \hline

& $C=\<{(0,2,0)} $ & & \\
$(4p,2,2)$ 
& $A=\{(i,j,k) + \lc \mid i\ttk{p},\,j,k\ttk{2}\} $ 
& & row 30 \\
& $\so=\{(i,j+2\alpha_{ijk},k) \mid i\ttk{p},\,j,k\ttk{2}\} $ & & \\ \hline

& $C=\<{(0,2,1)} $ & & \\
$(4p,2,2)$ 
& $A=\{(i,j,k) + \lc \mid i\ttk{p},\,j,k\ttk{2}\} $ 
& & row 31 \\
& $\so=\{(i,j+2\alpha_{ijk},k+\alpha_{ijk}) \mid i\ttk{p},\,j,k\ttk{2}\} $ & & \\ \hline

& $C=\<{(0,0,1)} $ & & \\
$(4p,2,2)$ 
& $A=\{(i,j,0) + \lc \mid i\ttk{p},\,j\ttk{4}\} $ 
& & row 32 \\
& $\so=\{(i,j,\alpha_{ij}) \mid i\ttk{p},\,j\ttk{4}\} $ & & \\ \hline

& $C=\<{(1,0,0)} $ & & \\
$(8,p,p)$ 
& $A=\{(0,i,j) + \lc \mid i\ttk{4},\,j\ttk{2}\} $ 
& & row 39 \\
& $\so=\{(\alpha_{ij},i,j) \mid i\ttk{4},\,j\ttk{2}\} $ & & \\ \hline

\end{longtable}

\smallskip
    \textsf{Case 2.} $\so$ is periodic.
    \smallskip

    In this case we have $|\ls|\geq 2$ and $\so=D\oplus \ls$ for some subset $D$ of $G$ by Lemma \ref{period}.
    So $|D|\cdot|\ls|=|\so|$ and $1<|\ls|<|\so|$ by Lemma \ref{prime deg} part (b).
    Set $$X=(D+\ls)/\ls\quad\text{and}\quad Y=(C+\ls)/\ls.$$
    Since $\so$ is periodic, by Lemma \ref{quotient}, we have
    $$G/\lc=X\oplus Y$$ and $X$ is aperiodic in $G/\ls$.
    We also know that $X$ contains $\ls$ and $X$ generates $G/\ls$ by Lemma \ref{generate-zero}.
    This factorization enables us to make use of Lemma \ref{d_i+l_i} and previously established theorems to obtain the pair $(\so, C)$ via $(X,Y)$.
    Note that for some triples $\left(|\so|,|C|,|\ls|\right)$, there are more than one possible $\ls$, see Observation \ref{obs}.
    
    The columns of Table \ref{tp.p.4.2} represents all possible values for $\left(|\so|,|C|,|\ls|\right)$; $\ls$, $X$, $Y$, $\so$, and $C$; which previous theorem is used to obtain $(X,Y)$; and which row in Table \ref{t.p.4.2} the result corresponds to.

\begin{longtable}{|c|l|c|c|}
\caption{The case when $G=\Z_p\x\Z_4\x\Z_2 $ and $\so$ is periodic}
\label{tp.p.4.2}
\\
\flp

& $\ls=\<{(0,2,0)} $ & & \\
& $X=\{(i,\alpha_i,\beta_i)+\ls \mid i\ttk{p}\} $ & & \\
$(2p,4,2 )$ 
& $Y=\{(0,i,j)+\ls \mid i,j\ttk{2}\} $ 
& \ref{p.2.2} \row1 & row 2 \\
& $\so=\{(i,\alpha_i+2j,\beta_i) \mid i\ttk{p},\,j\ttk{2}\} $ & & \\
& $C=\{(0,i+2\alpha_{ij},j) \mid i,j\ttk{2}\} $ & & \\ \hline

& $\ls=\<{(0,2,1)} $ & & \\
& $X=\{(i,\alpha_i,\beta_i)+\ls \mid i\ttk{p}\} $ & & \\
$(2p,4,2 )$ 
& $Y=\{(0,i,j)+\ls \mid i,j\ttk{2}\} $ 
& \ref{p.2.2} \row1 & row 3 \\
& $\so=\{(i,\alpha_i+2j,\beta_i+j) \mid i\ttk{p},\,j\ttk{2}\} $ & & \\
& $C=\{(0,i+2\alpha_{ij},j+\alpha_{ij}) \mid i,j\ttk{2}\} $ & & \\ \hline

& $\ls=\<{(0,0,1)} $ & & \\
& $X=\{(i,\alpha_i,0)+\ls \mid i\ttk{p}\} $ & & \\
$(2p,4,2 )$ 
& $Y=\{(0,i,0)+\ls \mid i\ttk{4}\} $ 
& \ref{p,4} \row{2} & row 4 \\
& $\so=\{(i,\alpha_i,j) \mid i\ttk{p},\,j\ttk{2}\} $ & & \\
& $C=\{(0,i,\alpha_i) \mid i\ttk{4}\} $ & & \\ \hline

$(2p,4,p)$ 
& Not possible since no $X$ exists
& \ref{4,2}  & \\ \hline

$(4,2p,2 )$ 
& Not possible for $\ls=\<{(0,2,0)}$ and $\<{(0,2,1)}$
& \ref{p.2.2}  & \\ \hline

& $\ls=\<{(0,0,1)} $ & & \\
& $X=\{(\alpha_i,i+2\alpha_i,0) +\ls \mid i\ttk{2}\} $ & & \\
$(4,2p,2 )$ 
& $Y=\{(i,2j,0) +\ls \mid i\ttk{p},\,j\ttk{2}\} $ 
& \ref{p,4} \row{1} & row 12 \\
& $\so=\{(\alpha_i,i+2\alpha_i,j) \mid i,j\ttk{2}\} $ & & \\
& $C=\{(i,2j,\alpha_{ij}) \mid i\ttk{p},\,j\ttk{2}\} $ & & \\ \hline

& $\ls=\<{(0,2,0)} $ & & \\
& $X=\{(i,0,0)+j\e3+\alpha_{ij}\e2+\ls \mid i\ttk{p},\,j\ttk{2}\} $ & & \\
$(4p,2,2 )$ 
& $Y=\{i\e2+\ls \mid i\ttk{2}\} $ 
& \ref{p.2.2} \row5 & row 20 \\
& $\so=\{(i,\alpha_{ij}+2k,j) \mid i\ttk{p},\,j,k\ttk{2}\} $ & & \\
& $C=\{(0,i+2\alpha_i,0) \mid i\ttk{2}\} $ & & \\ \hline

& $\ls=\<{(0,2,1)} $ & & \\
& $X=\{(i,0,0)+j\e3+\alpha_{ij}\e2+\ls \mid i\ttk{p},\,j\ttk{2}\} $ & & \\
$(4p,2,2 )$ 
& $Y=\{i\e2+\ls \mid i\ttk{2}\} $ 
& \ref{p.2.2} \row5 & row 21 \\
& $\so=\{(i,\alpha_{ij}+2k,j+k) \mid i\ttk{p},\,j,k\ttk{2}\} $ & & \\
& $C=\{(0,i+2\alpha_i,\alpha_i) \mid i\ttk{2}\} $ & & \\ \hline

& $\ls=\<{(0,0,1)} $ & & \\
& $X=\{(i,j+2\alpha_{ij},0)+\ls \mid i\ttk{p},\,j\ttk{2}\} $ & & \\
$(4p,2,2 )$ 
& $Y=\{(0,2i,0)+\ls \mid i\ttk{2}\} $ 
& \ref{p,4} \row{5} & row 22 \\
& $\so=\{(i,j+2\alpha_{ij},k) \mid i\ttk{p},\,j,k\ttk{2}\} $ & & \\
& $C=\{(0,2i,\alpha_i) \mid i\ttk{2}\} $ & & \\ \hline

& $\ls=\<{(1,0,0)} $ & & \\
& $X=\{(0,i+2\alpha_{ij},j) +\ls \mid i,j\ttk{2}\} $ & & \\
$(4p,2,p )$ 
& $Y=\{(0,2i,0) +\ls \mid i\ttk{2}\} $ 
& \ref{4,2} \row2 & row 23 \\
& $\so=\{(k,i+2\alpha_{ij},j) \mid i,j\ttk{2},\,k\ttk{p}\} $ & & \\
& $C=\{(\alpha_i,2i,0) \mid i\ttk{2}\} $ & & \\ \hline

& $\ls=\<{(1,0,0)} $ & & \\
& $X=\{(0,i+2\alpha_{ij},j+\alpha_{ij}) +\ls \mid i,j\ttk{2}\} $ & & \\
$(4p,2,p )$ 
& $Y=\{(0,2i,i) +\ls \mid i\ttk{2}\} $ 
& \ref{4,2} \row3 & row 24 \\
& $\so=\{(k,i+2\alpha_{ij},j+\alpha_{ij}) \mid i,j\ttk{2},\,k\ttk{p}\} $ & & \\
& $C=\{(\alpha_i,2i,i) \mid i\ttk{2}\} $ & & \\ \hline

& $\ls=\<{(1,0,0)} $ & & \\
& $X=\{(0,i,\alpha_i) +\ls \mid i\ttk{4}\} $ & & \\
$(4p,2,p )$ 
& $Y=\{(0,0,i) +\ls \mid i\ttk{2}\} $ 
& \ref{4,2} \row4 & row 25 \\
& $\so=\{(j,i,\alpha_{i}) \mid i\ttk{4},\,j\ttk{p}\} $ & & \\
& $C=\{(\alpha_i,0,i) \mid i\ttk{2}\} $ & & \\ \hline

& $\ls=\<{(0,1,0)} $ & & \\
& $X=\{(i,0,\alpha_i)+\ls \mid i\ttk{p}\} $ & & \\
$(4p,2,4 )$ 
& $Y=\{(0,0,i)+\ls \mid i\ttk{2}\} $ 
& \ref{p.q} \row1 & row 26 \\
& $\so=\{(i,j,\alpha_i) \mid i\ttk{p},\,j\ttk{4}\} $ & & \\
& $C=\{(0,\alpha_i,i) \mid i\ttk{2}\} $ & & \\ \hline

& $\ls=\<{(0,1,1)} $ & & \\
& $X=\{(i,0,\alpha_i)+\ls \mid i\ttk{p}\} $ & & \\
$(4p,2,4 )$ 
& $Y=\{(0,0,i)+\ls \mid i\ttk{2}\} $ 
& \ref{p.q} \row1 & row 27 \\
& $\so=\{(i,j,\alpha_i+j) \mid i\ttk{p},\,j\ttk{4}\} $ & & \\
& $C=\{(0,\alpha_i,i+\alpha_i) \mid i\ttk{2}\} $ & & \\ \hline

& $\ls=\<{(0,2,0),(0,0,1)} $ & & \\
& $X=\{(i,\alpha_i,0)+\ls \mid i\ttk{p}\} $ & & \\
$(4p,2,4 )$ 
& $Y=\{(0,i,0)+\ls \mid i\ttk{2}\} $ 
& \ref{p.q} \row1 & row 28 \\
& $\so=\{(i,\alpha_i+2j,k) \mid i\ttk{p},\,j,k\ttk{2}\} $ & & \\
& $C=\{(0,i+2\alpha_i,\beta_i) \mid i\ttk{2}\} $ & & \\ \hline

$(4p,2,2p )$ 
& Not possible for $\ls=\<{(1,2,0)}$ and $\<{(1,2,1)}$
& \ref{2.2}  & \\ \hline

& $\ls=\<{(1,0,1)} $ & & \\
& $X=\{(0,i+2\alpha_i,0)+\ls \mid i\ttk{2}\} $ & & \\
$(4p,2,2p )$ 
& $Y=\{(0,2i,0)+\ls \mid i\ttk{2}\} $ 
& \ref{p^2} & row 29 \\
& $\so=\{(j,i+2\alpha_i,k) \mid i,k\ttk{2},\,j\ttk{p}\} $ & & \\
& $C=\{(\alpha_i,2i,\alpha_i) \mid i\ttk{2}\} $ & & \\ \hline

& $\ls=\<{(0,2,0)} $ & & \\
& $X=\{(\alpha_{ij},i,j)+\ls \mid i,j\ttk{2}\}$ & & \\
$(8,p,2)$ 
& $Y=\{(i,0,0)+\ls \mid i\ttk{p}\} $ 
& \ref{p.2.2} \row7 & row 33 \\
& $\so=\{(\alpha_{ij},i+2k,j) \mid i,j,k\ttk{2}\} $ & & \\
& $C=\{(i,2\alpha_i,0) \mid i\ttk{p}\} $ & & \\ \hline

& $\ls=\<{(0,2,1)} $ & & \\
& $X=\{(\alpha_{ij},i,j)+\ls \mid i,j\ttk{2}\} $ & & \\
$(8,p,2 )$ 
& $Y=\{(i,0,0)+\ls \mid i\ttk{p}\} $ 
& \ref{p.2.2} \row7 & row 34 \\
& $\so=\{(\alpha_{ij},i+2k,j+k) \mid i,j,k\ttk{2}\} $ & & \\
& $C=\{(i,2\alpha_i,\alpha_i) \mid i\ttk{p}\} $ & & \\ \hline

& $\ls=\<{(0,0,1)} $ & & \\
& $X=\{(\alpha_{i},i,0)+\ls \mid i\ttk{4}\} $ & & \\
$(8,p,2 )$ 
& $Y=\{(i,0,0)+\ls \mid i\ttk{p}\} $ 
& \ref{p,4} \row{7} & row 35 \\
& $\so=\{(\alpha_{i},i,j) \mid i\ttk{4},\,j\ttk{2}\} $ & & \\
& $C=\{(i,0,\alpha_i) \mid i\ttk{p}\} $ & & \\ \hline

& $\ls=\<{(0,1,0)} $ & & \\
& $X=\{(\alpha_i,0,i)+\ls \mid i\ttk{2}\} $ & & \\
$(8,p,4 )$ 
& $Y=\{(i,0,0)+\ls \mid i\ttk{p}\} $ 
& \ref{p.q} \row2 & row 36 \\
& $\so=\{(\alpha_i,j,i)\mid i\ttk{2},\,j\ttk{4}\} $ & & \\
& $C=\{(i,\alpha_i,0) \mid i\ttk{p}\} $ & & \\ \hline

& $\ls=\<{(0,1,1)} $ & & \\
& $X=\{(\alpha_i,0,i)+\ls \mid i\ttk{2}\} $ & & \\
$(8,p,4 )$ 
& $Y=\{(i,0,0)+\ls \mid i\ttk{p}\} $ 
& \ref{p.q} \row2 & row 37 \\
& $\so=\{(\alpha_i,j,i+j)\mid i\ttk{2},\,j\ttk{4}\} $ & & \\
& $C=\{(i,\alpha_i,\alpha_i) \mid i\ttk{p}\} $ & & \\ \hline

& $\ls=\<{(0,2,0),(0,0,1)} $ & & \\
& $X=\{(\alpha_i,i,0)+\ls \mid i\ttk{2}\} $ & & \\
$(8,p,4 )$ 
& $Y=\{(i,0,0)+\ls \mid i\ttk{p}\} $ 
& \ref{p.q} \row2 & row 38 \\
& $\so=\{(\alpha_i,i+2j,k) \mid i,j,k\ttk{2}\} $ & & \\
& $C=\{(i,2\alpha_i,\beta_i) \mid i\ttk{p}\} $ & & \\ \hline

\end{longtable}
\end{proof}

\subsection{Circulant graphs}

As mentioned in the introduction, cyclic Haj\'os groups are subgroups of $\Z_{p^k}\x\Z_q$, $\Z_{p^2}\x\Z_{q^2}$, $\Z_{p^2}\x\Z_q\x\Z_r$, and $\Z_{p}\x\Z_q\x\Z_r\x\Z_s$. 
Recall that circulant graphs having perfect codes with prime size have been classified in \cite{eds2}. Circulant graphs on Haj\'os groups having perfect codes with size relative prime to the degree plus one have been classified in \cite{eds3}.

If a circulant graph on a subgroup of $\Z_p\x\Z_q\x\Z_r\x\Z_s$ has a perfect code, then the size of the perfect code will be relatively prime to the degree plus one.
This means that circulant graphs on subgroups of $\Z_p\x\Z_q\x\Z_r\x\Z_s$ have been characterized.
Subgroups of $\Z_{p^k}\x\Z_q$ are covered in Theorems \ref{p.q}, \ref{p^k}, and \ref{p^k.q}.
As proper subgroups of either $\Z_{p^2}\x\Z_{q^2}$ and $\Z_{p^2}\x\Z_q\x\Z_r$ are also subgroups of either $\Z_{p}\x\Z_q\x\Z_r\x\Z_s$ or $\Z_{p^k}\x\Z_q$, we are left to characterize circulant graphs on $\Z_{p^2}\x\Z_{q^2}$ and $\Z_{p^2}\x\Z_q\x\Z_r$ having a perfect code $C$ where $|C|$ is not relatively prime to $|G|/|C|$.


\begin{theorem}\label{p^2.q^2}
Let $G=\Z_{p^2}\x\Z_{q^2}$ for distinct primes $p$ and $q$, and let $S$ be a proper subset of $G\setminus \{0\}$ that generates $G$. Let $0\in C\subseteq G$ and let $|C|$ be not relatively prime to $|S_0|$. Then $C$ is a perfect code in $\Cay(G,S)$ if and only if $(\so,C)$ is one of the pairs in Table \ref{t.p^2.q^2} where ${\alpha_i}$'s, ${\alpha_{ij}}$'s, and ${\beta_i}$'s are in $\Z$ and $\alpha_0=\alpha_{00}=\beta_{0}=0$.
\end{theorem} 

\rc\begin{longtable}{|c|l|l|}
\caption{Perfect codes in Cayley graphs of $\Z_{p^2}\x\Z_{q^2}$}\label{t.p^2.q^2}\\
\fl
\tn & $\{(i+\alpha_i p,\beta_i) \mid i\ttk{p}\}$
& $\{(ip,j) \mid i\ttk{p}, j\ttk{q^2}\}$  \\
\tn & $\{(\alpha_i,i+\beta_i q) \mid i\ttk{q}\}$
& $\{(i,jq) \mid i\ttk{p^2}, j\ttk{q}\}$  \\
\tn & $\{(i+\alpha_{i}p,\alpha_i+jq) \mid i\ttk{p},\,j\ttk{q} \}$
& $\{(ip,j+\alpha_{ij}q) \mid i\ttk{p},\,j\ttk{q} \}$  \\
\tn & $\{(\alpha_i+jp,i+\alpha_{i}q) \mid i\ttk{q},\,j\ttk{p} \}$
& $\{(i+\alpha_{ij}p,jq) \mid i\ttk{p},\,j\ttk{q} \}$  \\
\tn & $\{(j+\alpha_{ij}p,i+\alpha_i q) \mid i\ttk{q},\,j\ttk{p}\}$
& $\{(\alpha_i+jp,iq) \mid i\ttk{q},\,j\ttk{p}\}$  \\
\tn & $\{(i+\alpha_ip,\alpha_i + jq) \mid i\ttk{p},\,j\ttk{q}\}$
& $\{(jp,i+\alpha_iq) \mid i\ttk{q},\,j\ttk{p}\}$  \\
\tn & $\{(i+\alpha_i p, j+\alpha_{ij}q) \mid i\ttk{p},\,j\ttk{q}\}$
& $\{(ip,\alpha_i+jq) \mid i\ttk{p},\,j\ttk{q}\}$  \\
\tn & $\{(\alpha_i + jp, i+\alpha_{ij}q) \mid i\ttk{q},\,j\ttk{p}\}$
& $\{(i+\alpha_ip,jq) \mid i\ttk{p},\,j\ttk{q}\}$  \\
\tn & $\{(i+\alpha_{ij}p,j+\alpha_{ij}q) \mid i\ttk{p},\,j\ttk{q} \}$
& $\{(ip,jq) \mid i\ttk{p},\,j\ttk{q} \}$  \\
\tn & $\{(j+kp,i+\alpha_{ij}q) \mid i\ttk{q},\,j,k\ttk{p} \}$
& $\{(p\alpha_i,iq)\mid i\ttk{q} \}$  \\
\tn & $\{(i,\alpha_i+jq) \mid i\ttk{p^2},\,j\ttk{q}\}$
& $\{(0,i+\alpha_iq) \mid i\ttk{q}\}$  \\
\tn & $\{(j,i+\alpha_iq) \mid i\ttk{q},\,j\ttk{p^2}\}$
& $\{ (\alpha_i,iq) \mid i\ttk{q}\}$  \\
\tn & $\{(i+jp,\alpha_i+kq) \mid i,j\ttk{p},\,k\ttk{q} \}$
& $\{(p\alpha_i,i+q\alpha_i) \mid i\ttk{q} \}$  \\
\tn & $\{(i,j+\alpha_{ij}q) \mid i\ttk{p^2},\,j\ttk{q}\}$
& $\{(0,iq) \mid i\ttk{q} \}$  \\
\tn & $\{(i+\alpha_{ij}p,j+kq) \mid i\ttk{p},\,j,k\ttk{q} \}$
& $\{(ip,\alpha_iq)\mid i\ttk{p} \}$  \\
\tn & $\{(\alpha_i+jp,i) \mid i\ttk{q^2},\,j\ttk{p}\}$
& $\{(i+\alpha_ip,0) \mid i\ttk{p}\}$  \\
\tn & $\{(i+\alpha_ip,j) \mid i\ttk{p},\,j\ttk{q^2}\}$
& $\{ (ip,\alpha_i) \mid i\ttk{p}\}$  \\
\tn & $\{(\alpha_i+jp,i+kq) \mid i,k\ttk{q},\,j\ttk{p} \}$
& $\{(i+p\alpha_i,q\alpha_i) \mid i\ttk{p} \}$  \\
\tn & $\{(i+\alpha_{ij}p,j) \mid i\ttk{p},\,j\ttk{q^2}\}$
& $\{(ip,0) \mid i\ttk{p} \}$  \\
\hline
\end{longtable}

\begin{proof}
    Suppose $C$ is a perfect code in $\Cay(G,S)$. Then $G=\so\oplus C$ and so $|\so|\cdot |C|=|G|= p^2q^2$. 
    Since $G$ is a connected non-complete graph, $1<|\so|<p^2q^2$.
    
    \smallskip
    \textsf{Case 1.} $\so$ is aperiodic.
    \smallskip

    Since $G$ is Haj\'os and $G=\so\oplus C$, in this case $C$ must be periodic, that is, $\lc$ is nontrivial.
    Write $C=\lc\oplus D$ for some subset $D$ of $G$ whose existence is ensured by Lemma \ref{period}.
    Then $|\lc|\cdot |D|=|C|$ and $2\leq |\lc|\leq |C|$.
    If $|\so|$ is a prime, then $|\lc|=|C|$ by Lemma \ref{prime deg} part (c).
    If $|\so|$ is not a prime, then $|\lc|$ can be any divisor of $|C|$.
    Set $$ A=(\so+\lc)/\lc\quad\text{and}\quad B=(D+\lc)/\lc.$$
    Since $C$ is periodic, by Lemma \ref{quotient}, we have
    $$G/\lc=A\oplus B$$ and $B$ is aperiodic in $G/\lc$.
    Since $G/\lc$ is isomorphic to a subgroup of $G$ and subgroups of Haj\'os group are also Haj\'os groups, $A$ is periodic in $G/\lc$.
    We also know that $A$ contains $\lc$ and $A$ generates $G/\lc$ by Lemma \ref{generate-zero}.
    Furthermore, $|A|=|\so|$ and $|B|=|D|$.
    Note that if $|\lc|=|C|$, then $C=\lc$ as $0\in C$.
    This means we can take $D=\{0\}$ and hence $B=\lc/\lc$ and $G/\lc=A$.
    We use $C$ to get $A$ and then use Lemma \ref{d_i+l_i} to obtain $\so$.
    If $|\lc|<|C|$, then this factorization enables us to make use of previously established theorems to obtain the pair $(\so, C)$ via $(A,B)$.

    All possible values for $\left(|\so|,|C|,|\lc|\right)$ when $|\so|$ and $|C|$ are not relatively prime are given in the first column of Table \ref{tap.p^2.q^2}.
    In the second column, we list $C$, $A$, and $\so$ when $|\lc|=|C|$ and we list $L_C$, $A$, $B$, $\so$, and $C$ when $|\lc|<|C|$.
    The third column shows which previous theorem is used to obtain $(A,B)$ when $|\lc|<|C|$.
    The last column indicates which row in Table \ref{t.p^2.q^2}
    the result corresponds to.
    
\begin{longtable}{|c|l|c|c|}
\caption{The case when $G=\Z_{p^2}\x\Z_{q^2} $ and $\so$ is aperiodic}
\label{tap.p^2.q^2}
\\
\flap

& $C=\{(ip,j) \mid i\ttk{p}, j\ttk{q^2}\} $ & & \\
$(p,pq^2,pq^2 )$ 
& $A=\{(i,0)+\lc \mid i\ttk{p}\} $ 
& & row 1 \\
& $\so=\{(i+\alpha_i p,\beta_i) \mid i\ttk{p}\} $ & & \\ \hline

& $C=\{(i,jq) \mid i\ttk{p^2}, j\ttk{q}\} $ & & \\
$(q,p^2q,p^2q)$
& $A=\{(0,i)+\lc \mid i\ttk{q}\} $ 
& & row 2 \\
& $\so=\{(\alpha_i,i+\beta_i q) \mid i\ttk{q}\} $ & & \\ \hline

& $\lc=\<{p\e1} $ & & \\
& $A=\{(j,i+\alpha_i q) + \lc \mid i\ttk{q},\,j\ttk{p}\} $ & & \\
$(pq,pq,p)$
& $B=\{(\alpha_i,iq) + \lc \mid i\ttk{q}\} $ 
& \ref{p^2.2} \row5 & row 5 \\
& $\so=\{(j+\alpha_{ij}p,i+\alpha_i q) \mid i\ttk{q},\,j\ttk{p}\} $ & & \\
& $C=\{(\alpha_i+jp,iq) \mid i\ttk{q},\,j\ttk{p}\} $ & & \\ \hline

& $\lc=\<{p\e1} $ & & \\
& $A=\{(i,\alpha_i + jq) + \lc \mid i\ttk{p},\,j\ttk{q}\} $ & & \\
$(pq,pq,p)$
& $B=\{(0,i+\alpha_iq) + \lc \mid i\ttk{q}\} $ 
& \ref{p^2.2} \row6 & row 6 \\
& $\so=\{(i+\alpha_ip,\alpha_i + jq) \mid i\ttk{p},\,j\ttk{q}\} $ & & \\
& $C=\{(jp,i+\alpha_iq) \mid i\ttk{q},\,j\ttk{p}\} $ & & \\ \hline

& $\lc=\<{q\e2} $ & & \\
& $A=\{(i+\alpha_i p, j) + \lc \mid i\ttk{p},\,j\ttk{q}\} $ & & \\
$(pq,pq,q)$
& $B=\{(ip,\alpha_i) + \lc \mid i\ttk{p}\} $ 
& \ref{p^2.2} \row5 & row 7 \\
& $\so=\{(i+\alpha_i p, j+\alpha_{ij}q) \mid i\ttk{p},\,j\ttk{q}\} $ & & \\
& $C=\{(ip,\alpha_i+jq) \mid i\ttk{p},\,j\ttk{q}\} $ & & \\ \hline

& $\lc=\<{q\e2} $ & & \\
& $A=\{(\alpha_i + jp, i) + \lc \mid i\ttk{q},\,j\ttk{p}\} $ & & \\
$(pq,pq,q)$
& $B=\{(i+\alpha_ip,0) + \lc \mid i\ttk{p}\} $ 
& \ref{p^2.2} \row6 & row 8 \\
& $\so=\{(\alpha_i + jp, i+\alpha_{ij}q) \mid i\ttk{q},\,j\ttk{p}\} $ & & \\
& $C=\{(i+\alpha_ip,jq) \mid i\ttk{p},\,j\ttk{q}\} $ & & \\ \hline

& $C=\{(ip,jq) \mid i\ttk{p},\,j\ttk{q} \} $ & & \\
$(pq,pq,pq)$ 
& $A=\{(i,j) + \lc \mid i\ttk{p},\,j\ttk{q} \} $ 
& & row 9 \\
& $\so=\{(i+\alpha_{ij}p,j+\alpha_{ij}q) \mid i\ttk{p},\,j\ttk{q} \} $ & & \\ \hline

& $C=\{(0,iq) \mid i\ttk{q} \} $ & & \\
$(p^2q,q,q)$ 
& $A=\{(i,j)+\lc \mid i\ttk{p^2},\,j\ttk{q}\} $ 
& & row 14 \\
& $\so=\{(i,j+\alpha_{ij}q) \mid i\ttk{p^2},\,j\ttk{q}\} $ & & \\ \hline

& $C=\{(ip,0) \mid i\ttk{p} \} $ & & \\
$(pq^2,p,p)$ 
& $A=\{(i,j)+\lc \mid i\ttk{p},\,j\ttk{q^2}\} $ 
& & row 19 \\
& $\so=\{(i+\alpha_{ij}p,j) \mid i\ttk{p},\,j\ttk{q^2}\} $ & & \\ \hline

\end{longtable}

    \smallskip
    \textsf{Case 2.} $\so$ is periodic.
    \smallskip

    In this case we have $|\ls|\geq 2$ and $\so=D\oplus \ls$ for some subset $D$ of $G$ by Lemma \ref{period}.
    So $|D|\cdot|\ls|=|\so|$ and $1<|\ls|<|\so|$ by Lemma \ref{prime deg} part (b).
    Set $$X=(D+\ls)/\ls\quad\text{and}\quad Y=(C+\ls)/\ls.$$
    Since $\so$ is periodic, by Lemma \ref{quotient}, we have
    $$G/\lc=X\oplus Y$$ and $X$ is aperiodic in $G/\ls$.
    We also know that $X$ contains $\ls$ and $X$ generates $G/\ls$ by Lemma \ref{generate-zero}.
    This factorization enables us to make use of Lemma \ref{d_i+l_i} and previously established theorems to obtain the pair $(\so, C)$ via $(X,Y)$.
    
    All possible values for $\left(|\so|,|C|,|\ls|\right)$ when $|\so|$ and $|C|$ are not relatively prime are given in the first column of Table \ref{tp.p^2.q^2}.
    In the second column we list $\ls$, $X$, $Y$, $\so$, and $C$.
    The third column shows which previous theorem is used to obtain $(X,Y)$.
    The last column indicates which row in Table \ref{t.p^2.q^2} the result corresponds to.

\begin{longtable}{|c|l|c|c|}
\caption{The case when $G=\Z_{p^2}\x\Z_{q^2} $ and $\so$ is periodic}
\label{tp.p^2.q^2}
\\
\flp

& $\ls=\<{q\e2} $ & & \\
& $X=\{(i+\alpha_{i}p,\alpha_i) +\ls \mid i\ttk{p} \} $ & & \\
$(pq,pq,q)$
& $Y=\{(ip,j) +\ls \mid i\ttk{p},\,j\ttk{q} \} $ 
& \ref{p^2.2} \row1 & row 3 \\
& $\so=\{(i+\alpha_{i}p,\alpha_i+jq) \mid i\ttk{p},\,j\ttk{q} \} $ & & \\
& $C=\{(ip,j+\alpha_{ij}q) \mid i\ttk{p},\,j\ttk{q} \} $ & & \\ \hline

& $\ls=\<{p\e1} $ & & \\
& $X=\{(\alpha_i,i+\alpha_{i}q) +\ls \mid i\ttk{q} \} $ & & \\
$(pq,pq,p)$
& $Y=\{(i,jq) +\ls \mid i\ttk{p},\,j\ttk{q} \} $ 
& \ref{p^2.2} \row1 & row 4 \\
& $\so=\{(\alpha_i+jp,i+\alpha_{i}q) \mid i\ttk{q},\,j\ttk{p} \} $ & & \\
& $C=\{(i+\alpha_{ij}p,jq) \mid i\ttk{p},\,j\ttk{q} \} $ & & \\ \hline

& $\ls=\<{p\e1} $ & & \\
& $X=\{(j,i+\alpha_{ij}q) +\ls \mid i\ttk{q},\,j\ttk{p} \} $ & & \\
$(p^2q,q,p)$
& $Y=\{(0,iq) +\ls \mid i\ttk{q} \} $ 
& \ref{p^2.2} \row7 & row 10 \\
& $\so=\{(j+kp,i+\alpha_{ij}q) \mid i\ttk{q},\,j,k\ttk{p} \} $ & & \\
& $C=\{(p\alpha_i,iq)\mid i\ttk{q} \} $ & & \\ \hline

& $\ls=\<{q\e2} $ & & \\
& $X=\{(i,\alpha_i) +\ls \mid i\ttk{p^2}\} $ & & \\
$(p^2q,q,q)$ 
& $Y=\{(0,i)+\ls \mid i\ttk{q}\} $ 
& \ref{p^2.2} \row4 & row 11 \\
& $\so=\{(i,\alpha_i+jq) \mid i\ttk{p^2},\,j\ttk{q}\} $ & & \\
& $C=\{(0,i+\alpha_iq) \mid i\ttk{q}\} $ & & \\ \hline

& $\ls=\<{\e1} $ & & \\
& $X=\{(0,i+\alpha_iq)+\ls \mid i\ttk{q}\} $ & & \\
$(p^2q,q,p^2)$ 
& $Y=\{ (0,iq)+\ls \mid i\ttk{q}\} $ 
& \ref{p^2} & row 12 \\
& $\so=\{(j,i+\alpha_iq) \mid i\ttk{q},\,j\ttk{p^2}\} $ & & \\
& $C=\{ (\alpha_i,iq) \mid i\ttk{q}\} $ & & \\ \hline

& $\ls=\<{p\e1+q\e2} $ & & \\
& $X=\{(i,\alpha_i)+\ls \mid i\ttk{p} \} $ & & \\
$(p^2q,q,pq)$ 
& $Y=\{(0,i)+\ls \mid i\ttk{q} \} $ 
& \ref{p.q} \row1 & row 13 \\
& $\so=\{(i+jp,\alpha_i+kq) \mid i,j\ttk{p},\,k\ttk{q} \} $ & & \\
& $C=\{(p\alpha_i,i+q\alpha_i) \mid i\ttk{q} \} $ & & \\ \hline

& $\ls=\<{q\e2} $ & & \\
& $X=\{(i+\alpha_{ij}p,j) +\ls \mid i\ttk{p},\,j\ttk{q} \} $ & & \\
$(pq^2,p,q)$
& $Y=\{(ip,0) +\ls \mid i\ttk{p} \} $ 
& \ref{p^2.2} \row7 & row 15 \\
& $\so=\{(i+\alpha_{ij}p,j+kq) \mid i\ttk{p},\,j,k\ttk{q} \} $ & & \\
& $C=\{(ip,\alpha_iq)\mid i\ttk{p} \} $ & & \\ \hline

& $\ls=\<{p\e1} $ & & \\
& $X=\{(\alpha_i,i) +\ls \mid i\ttk{q^2}\} $ & & \\
$(pq^2,p,p)$ 
& $Y=\{(i,0)+\ls \mid i\ttk{p}\} $ 
& \ref{p^2.2} \row4 & row 16 \\
& $\so=\{(\alpha_i+jp,i) \mid i\ttk{q^2},\,j\ttk{p}\} $ & & \\
& $C=\{(i+\alpha_ip,0) \mid i\ttk{p}\} $ & & \\ \hline

& $\ls=\<{\e2} $ & & \\
& $X=\{(i+\alpha_iq,0)+\ls \mid i\ttk{p}\} $ & & \\
$(pq^2,p,q^2)$ 
& $Y=\{ (ip,0)+\ls \mid i\ttk{p}\} $ 
& \ref{p^2} & row 17 \\
& $\so=\{(i+\alpha_ip,j) \mid i\ttk{p},\,j\ttk{q^2}\} $ & & \\
& $C=\{ (ip,\alpha_i) \mid i\ttk{p}\} $ & & \\ \hline

& $\ls=\<{p\e1+q\e2} $ & & \\
& $X=\{(\alpha_i,i)+\ls \mid i\ttk{q} \} $ & & \\
$(pq^2,p,pq)$ 
& $Y=\{(i,0)+\ls \mid i\ttk{p} \} $ 
& \ref{p.q} \row2 & row 18 \\
& $\so=\{(\alpha_i+jp,i+kq) \mid i,k\ttk{q},\,j\ttk{p} \} $ & & \\
& $C=\{(i+p\alpha_i,q\alpha_i) \mid i\ttk{p} \} $ & & \\ \hline

\end{longtable}
\end{proof}

\begin{theorem}\label{p^2.q.r}
Let $G=\Z_{p^2}\x\Z_{q}\x\Z_r$ for distinct primes $p$, $q$ and $r$, and let $S$ be a proper subset of $G\setminus \{0\}$ that generates $G$. Let $0\in C\subseteq G$ and let $|C|$ be not relatively prime to $|S_0|$. Then $C$ is a perfect code in $\Cay(G,S)$ if and only if $(\so,C)$ is one of the pairs in Table \ref{t.p^2.q.r} where ${\alpha_i}$'s, ${\alpha_{ij}}$'s, ${\alpha_{ijk}}$'s, and ${\beta_i}$'s are in $\Z$ and $\alpha_0=\alpha_{00}=\alpha_{000}=\beta_{0}=0$.
\end{theorem} 

\rc\begin{longtable}{|c|l|l|}
\caption{Perfect codes in Cayley graphs of $\Z_{p^2}\x\Z_{q}\x\Z_r$}\label{t.p^2.q.r}\\
\fl
\tn & $\{(i+\alpha_ip,\alpha_i,\alpha_i) \mid i\ttk{p} \}$
& $\{(ip,j,k) \mid i\ttk{p},\,j\ttk{q},\,k\ttk{r}\}$  \\
\tn & $\{(\alpha_i+jp,i,\beta_i) \mid i\ttk{q},\,j\ttk{p}\}$
& $\{(i+\alpha_{ij}p,0,j) \mid i\ttk{p},\,j\ttk{r}\}$  \\
\tn & $\{(i+\alpha_ip,j,\alpha_i) \mid i\ttk{p},\,j\ttk{q}\}$
& $\{(ip,\alpha_{ij},j) \mid i\ttk{p},\,j\ttk{r}\}$  \\
\tn & $\{(i+\alpha_{ij}p,j,\alpha_i) \mid i\ttk{p},\,j\ttk{q}\}$
& $\{(\alpha_i+jp,0,i) \mid i\ttk{r},\,j\ttk{p} \}$  \\
\tn & $\{(i+\alpha_{ij}p,j,\alpha_i) \mid i\ttk{p},\,j\ttk{q}\}$
& $\{(jp,\alpha_i,i) \mid i\ttk{r},\,j\ttk{p} \}$  \\
\tn & $\{(i+\alpha_ip,j,\alpha_{ij}) \mid i\ttk{p},\,j\ttk{q}\}$
& $\{(ip,\alpha_i,j) \mid i\ttk{p},\,j\ttk{r} \}$  \\
\tn & $\{(\alpha_i+jp,i,\alpha_{ij}) \mid i\ttk{p},\,j\ttk{q}\}$
& $\{(i+\alpha_ip,0,j) \mid i\ttk{p},\,j\ttk{r} \}$  \\
\tn & $\{(i+\alpha_{ij}p,j,\beta_{ij}) \mid i\ttk{p},\,j\ttk{q} \}$
& $\{(ip,0,j) \mid i\ttk{p},\,j\ttk{r}\}$  \\
\tn & $\{(\alpha_i+jp,\beta_i,i) \mid i\ttk{r},\,j\ttk{p}\}$
& $\{(i+\alpha_{ij}p,j,0) \mid i\ttk{p},\,j\ttk{q}\}$  \\
\tn & $\{(i+\alpha_ip,\alpha_i,j) \mid i\ttk{p},\,j\ttk{r}\}$
& $\{(ip,j,\alpha_{ij}) \mid i\ttk{p},\,j\ttk{q}\}$  \\
\tn & $\{(i+\alpha_{ij}p,\alpha_i,j) \mid i\ttk{p},\,j\ttk{r}\}$
& $\{(\alpha_i+jp,i,0) \mid i\ttk{q},\,j\ttk{p} \}$  \\
\tn & $\{(i+\alpha_{ij}p,\alpha_i,j) \mid i\ttk{p},\,j\ttk{r}\}$
& $\{(jp,i,\alpha_i) \mid i\ttk{q},\,j\ttk{p} \}$  \\
\tn & $\{(i+\alpha_ip,\alpha_{ij},j) \mid i\ttk{p},\,j\ttk{r}\}$
& $\{(ip,j,\alpha_i) \mid i\ttk{p},\,j\ttk{q} \}$  \\
\tn & $\{(\alpha_i+jp,\alpha_{ij},i) \mid i\ttk{p},\,j\ttk{r}\}$
& $\{(i+\alpha_ip,j,0) \mid i\ttk{p},\,j\ttk{q} \}$  \\
\tn & $\{(i+\alpha_{ij}p,\beta_{ij},j) \mid i\ttk{p},\,j\ttk{r} \}$
& $\{(ip,j,0) \mid i\ttk{p},\,j\ttk{r}\}$  \\
\tn & $\{(\alpha_{ij}+kp,i,j) \mid i\ttk{q},\,j\ttk{r},\,k\ttk{p}\}$
& $\{(i+\alpha_ip,0,0) \mid i\ttk{p}\}$  \\
\tn & $\{(i+\alpha_{ij}p,k,j) \mid i\ttk{p},\,j\ttk{r},\,k\ttk{q}\}$
& $\{(ip,\alpha_i,0) \mid i\ttk{p}\}$  \\
\tn & $\{(i+\alpha_{ij}p,j,k) \mid i\ttk{p},\,j\ttk{q},\,k\ttk{r}\}$
& $\{(ip,0,\alpha_i) \mid i\ttk{p}\}$  \\
\tn & $\{(\alpha_i+pj,k,i) \mid i\ttk{r},\,j\ttk{p},\,k\ttk{q}\}$
& $\{(i+\alpha_ip,\alpha_i,0) \mid i\ttk{p}\}$  \\
\tn & $\{(\alpha_i+jp,i,k) \mid i\ttk{q},\,j\ttk{p},\,k\ttk{r}\}$
& $\{(i+\alpha_ip,0,\alpha_i) \mid i\ttk{p}\}$  \\
\tn & $\{(i+\alpha_i p,j,k) \mid i\ttk{p},\,j\ttk{q},\,k\ttk{r}\}$
& $\{(ip,\alpha_i,\alpha_i) \mid i\ttk{p}\}$  \\
\tn & $\{(i+\alpha_{ijk}p,j,k) \mid i\ttk{p},\,j\ttk{q},\,k\ttk{r}\}$
& $\{(ip,0,0) \mid i\ttk{p}\}$  \\
\hline
\end{longtable}

\begin{proof}
    Suppose $C$ is a perfect code in $\Cay(G,S)$. Then $G=\so\oplus C$ and so $|\so|\cdot |C|=|G|= p^2qr$. 
    Since $G$ is a connected non-complete graph, $1<|\so|<p^2qr$.
    
    \smallskip
    \textsf{Case 1.} $\so$ is aperiodic.
    \smallskip

    Since $G$ is Haj\'os and $G=\so\oplus C$, in this case $C$ must be periodic, that is, $\lc$ is nontrivial.
    Write $C=\lc\oplus D$ for some subset $D$ of $G$ whose existence is ensured by Lemma \ref{period}.
    Then $|\lc|\cdot |D|=|C|$ and $2\leq |\lc|\leq |C|$.
    If $|\so|$ is a prime, then $|\lc|=|C|$ by Lemma \ref{prime deg} part (c).
    If $|\so|$ is not a prime, then $|\lc|$ can be any divisor of $|C|$.
    Set $$ A=(\so+\lc)/\lc\quad\text{and}\quad B=(D+\lc)/\lc.$$
    Since $C$ is periodic, by Lemma \ref{quotient}, we have
    $$G/\lc=A\oplus B$$ and $B$ is aperiodic in $G/\lc$.
    Since $G/\lc$ is isomorphic to a subgroup of $G$ and subgroups of Haj\'os group are also Haj\'os groups, $A$ is periodic in $G/\lc$.
    We also know that $A$ contains $\lc$ and $A$ generates $G/\lc$ by Lemma \ref{generate-zero}.
    Furthermore, $|A|=|\so|$ and $|B|=|D|$.
    Note that if $|\lc|=|C|$, then $C=\lc$ as $0\in C$.
    This means we can take $D=\{0\}$ and hence $B=\lc/\lc$ and $G/\lc=A$.
    We use $C$ to get $A$ and then use Lemma \ref{d_i+l_i} to obtain $\so$.
    If $|\lc|<|C|$, then this factorization enables us to make use of previously established theorems to obtain the pair $(\so, C)$ via $(A,B)$.

    As in Table \ref{tap.p^2.q^2}, the first columns of Table \ref{tap.p^2.q.r} are $\left(|\so|,|C|,|\lc|\right)$ when $|\so|$ and $|C|$ are not relatively prime.
    The remaining columns contains $C$, $A$, and $\so$ when $|\lc|=|C|$ or $L_C$, $A$, $B$, $\so$, and $C$ when $|\lc|<|C|$; which previous theorem is used to obtain $(A,B)$ when $|\lc|<|C|$; and which row in Table \ref{t.p^2.q.r} the result corresponds to.
    
\begin{longtable}{|c|l|c|c|}
\caption{The case when $G=\Z_{p^2}\x\Z_{q}\x\Z_r $ and $\so$ is aperiodic}
\label{tap.p^2.q.r}
\\
\flap

& $C=\{(ip,j,k) \mid i\ttk{p},\,j\ttk{q},\,k\ttk{r}\} $ & & \\
$(p,pqr,pqr)$ 
& $A=\{(i,0,0)+\lc \mid i\ttk{p} \} $
& & row 1 \\
& $\so=\{(i+\alpha_ip,\alpha_i,\alpha_i) \mid i\ttk{p} \} $ & & \\ \hline

& $\lc=\<{p\e1} $ & & \\
& $A=\{(i,j,\alpha_i)+\lc \mid i\ttk{p},\,j\ttk{q}\} $ & & \\
$(pq,pr,p)$
& $B=\{(\alpha_i,0,i)+\lc \mid i\ttk{r} \} $ 
& \ref{p.q.2} \row4 & row 4 \\
& $\so=\{(i+\alpha_{ij}p,j,\alpha_i) \mid i\ttk{p},\,j\ttk{q}\} $ & & \\
& $C=\{(\alpha_i+jp,0,i) \mid i\ttk{r},\,j\ttk{p} \} $ & & \\ \hline

& $\lc=\<{p\e1} $ & & \\
& $A=\{(i,j,\alpha_i)+\lc \mid i\ttk{p},\,j\ttk{q}\} $ & & \\
$(pq,pr,p)$
& $B=\{(0,\alpha_i,i)+\lc \mid i\ttk{r} \} $ 
& \ref{p.q.2} \row5 & row 5 \\
& $\so=\{(i+\alpha_{ij}p,j,\alpha_i) \mid i\ttk{p},\,j\ttk{q}\} $ & & \\
& $C=\{(jp,\alpha_i,i) \mid i\ttk{r},\,j\ttk{p} \} $ & & \\ \hline

& $\lc=\<{\e3} $ & & \\
& $A=\{(i+\alpha_ip,j,0) +\lc \mid i\ttk{p},\,j\ttk{q}\} $ & & \\
$(pq,pr,r)$ 
& $B=\{(ip,\alpha_i,0)+\lc \mid i\ttk{p} \} $ 
& \ref{p^2.2} \row5 & row 6 \\
& $\so=\{(i+\alpha_ip,j,\alpha_{ij}) \mid i\ttk{p},\,j\ttk{q}\} $ & & \\
& $C=\{(ip,\alpha_i,j) \mid i\ttk{p},\,j\ttk{r} \} $ & & \\ \hline

& $\lc=\<{\e3} $ & & \\
& $A=\{(\alpha_i+jp,i,0) +\lc \mid i\ttk{q},\,j\ttk{p}\} $ & & \\
$(pq,pr,r)$ 
& $B=\{(i+\alpha_ip,0,0)+\lc \mid i\ttk{p} \} $ 
& \ref{p^2.2} \row6 & row 7 \\
& $\so=\{(\alpha_i+jp,i,\alpha_{ij}) \mid i\ttk{p},\,j\ttk{q}\} $ & & \\
& $C=\{(i+\alpha_ip,0,j) \mid i\ttk{p},\,j\ttk{r} \} $ & & \\ \hline

& $C=\{(ip,0,j) \mid i\ttk{p},\,j\ttk{r}\} $ & & \\
$(pq,pr,pr)$ 
& $A=\{(i,j,0)+\lc \mid i\ttk{p},\,j\ttk{q} \} $ 
& & row 8 \\
& $\so=\{(i+\alpha_{ij}p,j,\beta_{ij}) \mid i\ttk{p},\,j\ttk{q} \} $ & & \\ \hline

& $\lc=\<{p\e1} $ & & \\
& $A=\{(i,\alpha_i,j)+\lc \mid i\ttk{p},\,j\ttk{r}\} $ & & \\
$(pr,pq,p)$
& $B=\{(\alpha_i,i,0)+\lc \mid i\ttk{q} \} $ 
& \ref{p.q.2} \row6 & row 11 \\
& $\so=\{(i+\alpha_{ij}p,\alpha_i,j) \mid i\ttk{p},\,j\ttk{r}\} $ & & \\
& $C=\{(\alpha_i+jp,i,0) \mid i\ttk{q},\,j\ttk{p} \} $ & & \\ \hline

& $\lc=\<{p\e1} $ & & \\
& $A=\{(i,\alpha_i,j)+\lc \mid i\ttk{p},\,j\ttk{r}\} $ & & \\
$(pr,pq,p)$
& $B=\{(0,i,\alpha_i)+\lc \mid i\ttk{q} \} $ 
& \ref{p.q.2} \row7 & row 12 \\
& $\so=\{(i+\alpha_{ij}p,\alpha_i,j) \mid i\ttk{p},\,j\ttk{r}\} $ & & \\
& $C=\{(jp,i,\alpha_i) \mid i\ttk{q},\,j\ttk{p} \} $ & & \\ \hline

& $\lc=\<{\e2} $ & & \\
& $A=\{(i+\alpha_ip,0,j) +\lc \mid i\ttk{p},\,j\ttk{r}\} $ & & \\
$(pr,pq,q)$ 
& $B=\{(ip,0,\alpha_i)+\lc \mid i\ttk{p} \} $ 
& \ref{p^2.2} \row5 & row 13 \\
& $\so=\{(i+\alpha_ip,\alpha_{ij},j) \mid i\ttk{p},\,j\ttk{r}\} $ & & \\
& $C=\{(ip,j,\alpha_i) \mid i\ttk{p},\,j\ttk{q} \} $ & & \\ \hline

& $\lc=\<{\e2} $ & & \\
& $A=\{(\alpha_i+jp,0,i) +\lc \mid i\ttk{r},\,j\ttk{p}\} $ & & \\
$(pr,pq,q)$ 
& $B=\{(i+\alpha_ip,0,0)+\lc \mid i\ttk{p} \} $ 
& \ref{p^2.2} \row6 & row 14 \\
& $\so=\{(\alpha_i+jp,\alpha_{ij},i) \mid i\ttk{p},\,j\ttk{r}\} $ & & \\
& $C=\{(i+\alpha_ip,j,0) \mid i\ttk{p},\,j\ttk{q} \} $ & & \\ \hline

& $C=\{(ip,j,0) \mid i\ttk{p},\,j\ttk{r}\} $ & & \\
$(pr,pq,pq)$ 
& $A=\{(i,0,j)+\lc \mid i\ttk{p},\,j\ttk{q} \} $ 
& & row 15 \\
& $\so=\{(i+\alpha_{ij}p,\beta_{ij},j) \mid i\ttk{p},\,j\ttk{r} \} $ & & \\ \hline

& $C=\{(ip,0,0) \mid i\ttk{p}\} $ & & \\
$(pqr,p,p)$ 
& $A=\{(i,j,k)+\lc \mid i\ttk{p},\,j\ttk{q},\,k\ttk{r}\} $ 
& & row 22 \\
& $\so=\{(i+\alpha_{ijk}p,j,k) \mid i\ttk{p},\,j\ttk{q},\,k\ttk{r}\} $ & & \\ \hline

\end{longtable}

    \smallskip
    \textsf{Case 2.} $\so$ is periodic.
    \smallskip

    In this case we have $|\ls|\geq 2$ and $\so=D\oplus \ls$ for some subset $D$ of $G$ by Lemma \ref{period}.
    So $|D|\cdot|\ls|=|\so|$ and $1<|\ls|<|\so|$ by Lemma \ref{prime deg} part (b).
    Set $$X=(D+\ls)/\ls\quad\text{and}\quad Y=(C+\ls)/\ls.$$
    Since $\so$ is periodic, by Lemma \ref{quotient}, we have
    $$G/\lc=X\oplus Y$$ and $X$ is aperiodic in $G/\ls$.
    We also know that $X$ contains $\ls$ and $X$ generates $G/\ls$ by Lemma \ref{generate-zero}.
    This factorization enables us to make use of Lemma \ref{d_i+l_i} and previously established theorems to obtain the pair $(\so, C)$ via $(X,Y)$.
    
    As in Table \ref{tp.p^2.q^2}, the first columns of table \ref{tp.p^2.q.r} contains all possible values for $\left(|\so|,|C|,|\ls|\right)$ when $|\so|$ and $|C|$ are not relatively prime.
    The next three columns are $\ls$, $X$, $Y$, $\so$, and $C$; which previous theorem is used to obtain $(X,Y)$; and which row in Table \ref{t.p^2.q.r} the result corresponds to.

\begin{longtable}{|c|l|c|c|}
\caption{The case when $G=\Z_{p^2}\x\Z_{q}\x\Z_r $ and $\so$ is periodic}
\label{tp.p^2.q.r}
\\
\flp

& $\ls=\<{p\e1} $ & & \\
& $X=\{(\alpha_i,i,\beta_i)+\ls \mid i\ttk{q}\} $ & & \\
$(pq,pr,p)$
& $Y=\{(i,0,j)+\ls \mid i\ttk{p},\,j\ttk{r}\} $
& \ref{p.q.2} \row2 & row 2 \\
& $\so=\{(\alpha_i+jp,i,\beta_i) \mid i\ttk{q},\,j\ttk{p}\} $ & & \\
& $C=\{(i+\alpha_{ij}p,0,j) \mid i\ttk{p},\,j\ttk{r}\} $ & & \\ \hline

& $\ls=\<{\e2} $ & & \\
& $X=\{(i+\alpha_ip,0,\alpha_i)+\ls \mid i\ttk{p}\} $ & & \\
$(pq,pr,q)$
& $Y=\{(ip,0,j) +\ls \mid i\ttk{p},\,j\ttk{r}\} $ 
& \ref{p^2.2} \row1 & row 3 \\
& $\so=\{(i+\alpha_ip,j,\alpha_i) \mid i\ttk{p},\,j\ttk{q}\} $ & & \\
& $C=\{(ip,\alpha_{ij},j) \mid i\ttk{p},\,j\ttk{r}\} $ & & \\ \hline

& $\ls=\<{p\e1} $ & & \\
& $X=\{(\alpha_i,\beta_i,i)+\ls \mid i\ttk{r}\} $ & & \\
$(pr,pq,p)$
& $Y=\{(i,j,0)+\ls \mid i\ttk{p},\,j\ttk{q}\} $
& \ref{p.q.2} \row2 & row 9 \\
& $\so=\{(\alpha_i+jp,\beta_i,i) \mid i\ttk{r},\,j\ttk{p}\} $ & & \\
& $C=\{(i+\alpha_{ij}p,j,0) \mid i\ttk{p},\,j\ttk{q}\} $ & & \\ \hline

& $\ls=\<{\e3} $ & & \\
& $X=\{(i+\alpha_ip,\alpha_i,0)+\ls \mid i\ttk{p}\} $ & & \\
$(pr,pq,r)$
& $Y=\{(ip,j,0) +\ls \mid i\ttk{p},\,j\ttk{q}\} $ 
& \ref{p^2.2} \row1 & row 10 \\
& $\so=\{(i+\alpha_ip,\alpha_i,j) \mid i\ttk{p},\,j\ttk{r}\} $ & & \\
& $C=\{(ip,j,\alpha_{ij}) \mid i\ttk{p},\,j\ttk{q}\} $ & & \\ \hline

& $\ls=\<{p\e1} $ & & \\
& $X=\{(\alpha_{ij},i,j)+\ls \mid i\ttk{q},\,j\ttk{r}\} $ & & \\
$(pqr,p,p)$
& $Y=\{(i,0,0)+\ls \mid i\ttk{p}\} $ 
& \ref{p.q.2} \row{12} & row 16 \\
& $\so=\{(\alpha_{ij}+kp,i,j) \mid i\ttk{q},\,j\ttk{r},\,k\ttk{p}\} $ & & \\
& $C=\{(i+\alpha_ip,0,0) \mid i\ttk{p}\} $ & & \\ \hline

& $\ls=\<{\e2} $ & & \\
& $X=\{(i+\alpha_{ij}p,0,j)+\ls \mid i\ttk{p},\,j\ttk{r}\} $ & & \\
$(pqr,p,q )$ 
& $Y=\{(ip,0,0)+\ls \mid i\ttk{p}\} $ 
& \ref{p^2.2} \row7 & row 17 \\
& $\so=\{(i+\alpha_{ij}p,k,j) \mid i\ttk{p},\,j\ttk{r},\,k\ttk{q}\} $ & & \\
& $C=\{(ip,\alpha_i,0) \mid i\ttk{p}\} $ & & \\ \hline

& $\ls=\<{\e3} $ & & \\
& $X=\{(i+\alpha_{ij}p,j,0)+\ls \mid i\ttk{p},\,j\ttk{q}\} $ & & \\
$(pqr,p,r )$ 
& $Y=\{(ip,0,0)+\ls \mid i\ttk{p}\} $ 
& \ref{p^2.2} \row7 & row 18 \\
& $\so=\{(i+\alpha_{ij}p,j,k) \mid i\ttk{p},\,j\ttk{q},\,k\ttk{r}\} $ & & \\
& $C=\{(ip,0,\alpha_i) \mid i\ttk{p}\} $ & & \\ \hline

& $\ls=\<{p\e1+\e2} $ & & \\
& $X=\{(\alpha_i,0,i)+\ls \mid i\ttk{r}\} $ & & \\
$(pqr,p,pq)$ 
& $Y=\{(i,0,0)+\ls \mid i\ttk{p}\} $ 
& \ref{p.q} \row2 & row 19 \\
& $\so=\{(\alpha_i+pj,k,i) \mid i\ttk{r},\,j\ttk{p},\,k\ttk{q}\} $ & & \\
& $C=\{(i+\alpha_ip,\alpha_i,0) \mid i\ttk{p}\} $ & & \\ \hline

& $\ls=\<{p\e1+\e3} $ & & \\
& $X=\{(\alpha_i,i,0)+\ls \mid i\ttk{q}\} $ & & \\
$(pqr,p,pr)$ 
& $Y=\{(i,0,0)+\ls \mid i\ttk{p}\} $ 
& \ref{p.q} \row2 & row 20 \\
& $\so=\{(\alpha_i+jp,i,k) \mid i\ttk{q},\,j\ttk{p},\,k\ttk{r}\} $ & & \\
& $C=\{(i+\alpha_ip,0,\alpha_i) \mid i\ttk{p}\} $ & & \\ \hline

& $\ls=\<{\e2+\e3} $ & & \\
& $X=\{(i+\alpha_i p,0,0)+\ls \mid i\ttk{p}\} $ & & \\
$(pqr,p,qr)$ 
& $Y=\{(ip,0,0)+\ls \mid i\ttk{p}\} $ 
& \ref{p^2} & row 21 \\
& $\so=\{(i+\alpha_i p,j,k) \mid i\ttk{p},\,j\ttk{q},\,k\ttk{r}\} $ & & \\
& $C=\{(ip,\alpha_i,\alpha_i) \mid i\ttk{p}\} $ & & \\ \hline

\end{longtable}
\end{proof}

\noindent{\bf Acknowledgements}~~The first author was supported by the Melbourne Research Scholarship provided by The University of Melbourne. The second and third authors were supported by a Discovery Project (DP250104965) of the Australian Research Council.

\end{document}